\documentclass[a4paper,oneside,11pt]{article}
\usepackage[active]{srcltx}
\usepackage{hyperref}
\usepackage{authblk}
 \usepackage{bibspacing}

\newlength{\defbaselineskip}
\newcommand{\setlinespacing}[1]           {\setlength{\baselineskip}{#1 \defbaselineskip}}

\usepackage{amsmath}
\usepackage{amsthm}
\usepackage{amssymb}
\usepackage{latexsym}
\usepackage{mathtools}
\usepackage{mathrsfs}
\usepackage{graphics}
\usepackage{latexsym}
\usepackage{psfrag}
\usepackage{import}
\usepackage{verbatim}
\usepackage{enumerate}
\usepackage{enumitem}
\usepackage{graphicx}
\usepackage[usenames]{color}

\theoremstyle{plain}
\newtheorem{lemma}{Lemma}[section]
\newtheorem{theorem}[lemma]{Theorem}
\newtheorem{maintheorem}[lemma]{Main Theorem}
\newtheorem{proposition}[lemma]{Proposition}
\newtheorem{corollary}[lemma]{Corollary}
\newtheorem*{theorem*}{Theorem}
\newtheorem*{maintheorem*}{Main Theorem}

\theoremstyle{definition}
\newtheorem{assumption}[lemma]{Assumption}
\newtheorem{definition}[lemma]{Definition}
\newtheorem*{definition*}{Definition}
\newtheorem{remark}[lemma]{Remark}

\numberwithin{equation}{section}

\makeatletter
\newcommand{\mms}{m.m.s\@ifnextchar.{}{\@ifnextchar'{.}{\@ifnextchar){.}{.~}}}}
\makeatother

\newcommand{\dom}{{\textrm{Dom\,}}}
\newcommand{\R}{\mathbb{R}}
\newcommand{\N}{\mathbb{N}}
\newcommand{\NN}{\mathcal{N}}
\newcommand{\B}{\mathcal{B}}
\newcommand{\Q}{\mathbb{Q}}

\newcommand{\supp}{\text{\rm supp}}
\newcommand{\loc}{\text{\rm loc}}

\newcommand{\Eps}{\mathcal{E}}
\newcommand{\sing}{\text{sing}}

\newcommand{\gr}{\textrm{graph}}

\newcommand{\ve}{\varepsilon}

\newcommand{\enne}{\mathbb{N}}

\newcommand{\f}{\varphi}

\newcommand{\T}{\mathcal{T}}

\renewcommand{\L}{\mathcal{L}}
\newcommand{\F}{\mathcal{F}}
\newcommand{\RCD}{\mathsf{RCD}}

\newcommand{\CD}{\mathsf{CD}}
\newcommand{\BM}{\mathsf{BM}}
\newcommand{\Geo}{{\rm Geo}}
\newcommand{\MCP}{\mathsf{MCP}}
\newcommand{\MCPE}{\mathsf{MCP_{\eps}}}
\newcommand{\OptGeo}{{\rm OptGeo}}

\newcommand{\norm}[1]{\left\Vert#1\right\Vert}

\newcommand{\abs}[1]{\left\vert#1\right\vert}
\newcommand{\set}[1]{\left\{#1\right\}}
\newcommand{\brac}[1]{\left(#1\right)}
\newcommand{\scalar}[1]{\left \langle #1 \right \rangle}

\newcommand{\Real}{\mathbb{R}}

\newcommand{\eps}{\varepsilon}

\renewcommand{\L}{\mathcal{L}}

\newcommand{\q}{\mathfrak{q}}

\renewcommand{\P}{\mathbb P}
\newcommand{\ellc}{{\bar{\ell}}}
\newcommand{\varphic}{{\bar{\varphi}}}
\newcommand{\Phic}{{\bar{\Phi}}}
\newcommand{\pc}{{\bar{p}}}
\renewcommand{\P}{\mathcal{P}}
\newcommand{\len}{\ell}
\newcommand{\sfm}{\mathsf m}

\newcounter{mycounter}
\newcommand{\mm}{\mathfrak m}
\newcommand{\qq}{\mathfrak q}
\newcommand{\pp}{\mathfrak p}
\newcommand{\ee}{{\rm e}}
\newcommand{\QQ}{\mathfrak Q}

\newcommand{\sfd}{\mathsf d}
\newcommand{\Opt}{\mathrm{OptGeo}}

\author{
Fabio Cavalletti\thanks{SISSA, Trieste 34136, Italy. Email: cavallet@sissa.it.} \ and 
Emanuel Milman\thanks{Mathematics Department, Technion - I.I.T., Haifa 32000, Israel. Email: emilman@tx.technion.ac.il. \\
\textrm{The research leading to these results is part of a project that has received funding from the European Research Council (ERC) under the European Union's Horizon 2020 research and innovation programme (grant agreement No 637851).}
}}

\date{}                     

\title{The Globalization Theorem\\ for the Curvature-Dimension Condition}

\begin{document}

\maketitle

\bibliographystyle{plain}

\begin{abstract}
The Lott--Sturm--Villani Curvature-Dimension condition provides a synthetic notion for a metric-measure space to have Ricci-curvature bounded from below and dimension bounded from above. We prove that it is enough to verify this condition locally: 
an essentially non-branching metric-measure space $(X,\sfd,\mm)$ (so that $(\supp(\mm),\sfd)$ is a length-space and $\mm(X) < \infty$)
verifying the local Curvature-Dimension condition $\CD_{loc}(K,N)$ with parameters $K \in \R$ and $N \in (1,\infty)$, also verifies the global Curvature-Dimension condition $\CD(K,N)$. In other words, the Curvature-Dimension condition enjoys the globalization (or local-to-global) property, answering a question which had remained open since the beginning of the theory.
For the proof, we establish an equivalence between $L^1$ and $L^2$ optimal-transport--based interpolation.
The challenge is not merely a technical one, and several new conceptual ingredients which are of independent interest are developed: an explicit \emph{change-of-variables} formula for densities of Wasserstein geodesics depending on a second-order temporal derivative of associated Kantorovich potentials; a surprising \emph{third-order} theory for the latter Kantorovich potentials, which holds in complete generality on any proper geodesic space; and a certain \emph{rigidity} property of the change-of-variables formula, allowing us to bootstrap the a-priori available regularity. 
As a consequence, numerous variants of the Curvature-Dimension condition proposed by various authors throughout the years are shown to, in fact, all be equivalent in the above setting, thereby unifying the theory.
\end{abstract}

\tableofcontents

\section{Introduction}

The Curvature-Dimension condition $\CD(K,N)$ was first introduced in the 1980's by Bakry and \'Emery \cite{BakryEmery,BakryStFlour} in the context of diffusion generators, having in mind primarily the setting of weighted Riemannian manifolds, namely smooth Riemannian manifolds endowed with a smooth density with respect to the Riemannian volume. The $\CD(K,N)$ condition serves as a generalization of the classical condition in the non-weighted Riemannian setting of having Ricci curvature bounded below by $K \in \Real$ and dimension bounded above by $N \in [1,\infty]$ (see e.g. \cite{EMilmanNegativeDimension,Ohta-NegativeN} for further possible extensions). Numerous consequences of this condition 
have been obtained over the past decades, extending results from the classical non-weighted setting and at times establishing new ones directly in the weighted one. These include diameter bounds, volume comparison theorems, heat-kernel and spectral estimates, Harnack inequalities, topological implications, Brunn--Minkowski-type inequalities, and isoperimetric, functional and concentration inequalities -- see e.g. \cite{Ledoux-Book,BGL-Book,Vil} and the references therein. 

\smallskip
Being a differential and Hilbertian condition, it was for many years unclear how to extend the Bakry--\'Emery definition beyond the smooth Riemannian setting, as interest in (measured) Gromov-Hausdorff limits of Riemannian manifolds and other non-Hilbertian singular spaces steadily grew. In parallel, and apparently unrelatedly, the theory of Optimal-Transport was being developed in increasing generality following the influential work of Brenier \cite{BrenierMap} (see e.g. \cite{ambro:lecturenote,EvansSurveyOT,McCannGuillenLectureNotes,RachevRuschendorf-Book,UrbasLectureNotesOT,Vil:topics,Vil}). 
Given two probability measures $\mu_0,\mu_1$ on a common geodesic space $(X,\sfd)$ and a prescribed cost of transporting a single mass from point $x$ to $y$, the Monge-Kantorovich idea is to optimally couple $\mu_0$ and $\mu_1$ by minimizing the total transportation cost, and as a byproduct obtain a Wasserstein geodesic $[0,1] \ni t \mapsto \mu_t$ connecting $\mu_0$ and $\mu_1$ in the space of probability measures $\P(X)$. This gives rise to the notion of displacement convexity of a given functional on $\P(X)$ along Wasserstein geodesics, introduced and studied by McCann \cite{McCannConvexityPrincipleForGases}. 
Following the works of Cordero-Erausquin--McCann--Schmuckenschl\"ager \cite{corderomccann:brescamp}, Otto--Villani \cite{OttoVillaniHWI} and von Renesse--Sturm \cite{VonRenesseSturm}, it was realized that the $\CD(K,\infty)$ condition in the smooth setting may be equivalently formulated synthetically as a certain convexity property of an entropy functional along $W_2$ Wasserstein geodesics (associated to $L^2$-Optimal-Transport, when the transport-cost is given by the squared-distance function).

\smallskip
This idea culminated in the seminal works of Lott--Villani \cite{lottvillani:metric} and Sturm \cite{sturm:I, sturm:II}, where a synthetic definition of $\CD(K,N)$ was proposed on a general (complete, separable) metric space $(X,\sfd)$ endowed with a (locally-finite Borel) reference measure $\mm$ (``metric-measure space", or \mms); it was moreover shown that the latter definition coincides with the Bakry--\'Emery one in the smooth Riemannian setting (and in particular in the classical non-weighted one), that it is stable under measured Gromov-Hausdorff convergence of \mms's, and that it implies various geometric and analytic inequalities relating metric and measure, in complete analogy with the smooth setting. It was subsequently also shown \cite{Ohta-CDforFinsler,Petrunin-CDforAlexandrov} that Finsler manifolds and Alexandrov spaces satisfy the Curvature-Dimension condition. 
Thus emerged an overwhelmingly convincing notion of Ricci curvature lower bound $K$ and dimension upper bound $N$ for a general (geodesic) \mms $(X,\sfd,\mm)$, leading to a rich and fruitful theory exploring the geometry of \mms's by means of Optimal-Transport. 

\smallskip

One of the most important and longstanding open problems in the Lott--Sturm--Villani theory (see \cite{sturm:I, sturm:II} and \cite[pp. 888, 907]{Vil}) is whether the Curvature-Dimension condition on a general geodesic \mms (say, having full-support $\supp(\mm) = X$) enjoys the globalization (or local-to-global) property: if the $\CD(K,N)$ condition is known to hold on a neighborhood $X_o$ of any given point $o \in X$ (a property henceforth denoted by $\CD_{loc}(K,N)$), does it also necessarily hold on the entire space? 
Clearly this is indeed the case in the smooth setting, as both curvature and dimension may be computed locally (by equivalence with the differential $\CD$ definition). However, for reasons which we will expand on shortly, this is not at all clear and in some cases is actually false on general \mms's. An affirmative answer to this question would immensely facilitate the verification of the $\CD$ condition, which at present requires testing all possible $W_2$-geodesics on $X$, instead of locally on each $X_o$. The analogous question for sectional curvature on Alexandrov spaces (where the dimension $N$ is absent) does indeed have an affirmative answer, as shown by Topogonov, and in full generality, by Perelman (see \cite{BBI}).

\smallskip

Several partial answers to the local-to-global problem have already been obtained in the literature. A geodesic space $(X,\sfd)$ is called non-branching if geodesics are forbidden to branch at an interior-point into two separate geodesics. On a non-branching geodesic \mms $(X,\sfd,\mm)$ having full support, it was shown by Sturm in \cite[Theorem 4.17]{sturm:I} that the local-to-global property is satisfied when $N = \infty$ (assuming that the space of probability measures with finite $\mm$-relative entropy is geodesically convex; see also \cite[Theorem 30.42]{Vil} where the same globalization result was proved under a different condition involving the existence of a full-measure totally-convex subset of $X$ of finite-dimensional points). Still for non-branching geodesic \mms's having full support, a positive answer was also obtained by Villani in \cite[Theorem 30.37]{Vil} for the case $K=0$ and $N \in [1,\infty)$. 

\smallskip
We stress that in these results, the restriction to non-branching spaces is \emph{not} merely a technical assumption - an example of a heavily-branching \mms verifying $\CD_{loc}(0,4)$ which does not verify $\CD(K,N)$ for any fixed $K \in \Real$ and $N\in [1,\infty]$ was constructed by Rajala in \cite{R2016}. Consequently, a natural assumption is to require that $(X,\sfd)$ be non-branching, or more generally, to require that the $L^2$-Optimal-Transport on $(X,\sfd,\mm)$ be concentrated (i.e. up to a null-set) on a non-branching subset of geodesics, an assumption introduced by Rajala and Sturm in \cite{rajasturm:branch} under the name \emph{essentially non-branching} (see Section \ref{sec:PartII-prelim} for precise definitions). For instance, it is known \cite{rajasturm:branch} that measured Gromov-Hausdorff limits of Riemannian manifolds satisfying $\CD(K,\infty)$, and more generally, $\RCD(K,\infty)$ spaces, always satisfy the essentially non-branching assumption  (see Section \ref{S:Final}).

\medskip

In this work, we provide an affirmative answer to the globalization problem in the remaining range of parameters: for $N \in (1,\infty)$ and $K \in \Real$, the $\CD(K,N)$ condition verifies the local-to-global property on an essentially non-branching geodesic \mms $(X,\sfd,\mm)$ having finite total-measure and full support. The exclusion of the case $N=1$ is to avoid unnecessary pathologies, and is not essential. Our assumption that $\mm$ has finite total-measure (or equivalently, by scaling, that it is a probability measure) is most probably technical, but we did not verify it can be removed so as to avoid overloading the paper even further. 
This result is new even under the additional assumption that the space is infinitesimally Hilbertian (see \cite{gigli:laplacian}) -- we will say that such spaces verify $\RCD(K,N)$ -- in which case the assumption of being (globally) essentially non-branching is in fact superfluous. 

\smallskip

To better explain the difference between the previously known cases when $\frac{K}{N} = 0$ and the conceptual challenge which the newly treated case $\frac{K}{N} \neq 0$ poses, as well as to sketch our solution and its main new ingredients, which we believe are of independent interest, we provide some additional details below
and refer to Section \ref{sec:PartII-prelim} for precise definitions.

\subsection{Disentangling volume-distortion coefficients}

Roughly speaking, the $\CD(K,N)$ condition prescribes a synthetic second-order bound on how an infinitesimal volume changes when it is moved along a $W^{2}$-geodesic: 
the volume distortion (or transport Jacobian) $J$ along the geodesic should satisfy the following interpolation inequality for $t_0 = 0$ and $t_1 = 1$:
\begin{equation} \label{eq:intro-interpolate}
J^{\frac{1}{N}}(\alpha t_1 + (1-\alpha) t_0) \geq \tau_{K,N}^{(\alpha)}(\abs{t_1-t_0} \theta) J^{\frac{1}{N}}(t_1) + \tau_{K,N}^{(1-\alpha)}(\abs{t_1-t_0} \theta) J^{\frac{1}{N}}(t_0) \;\;\; \forall \alpha \in [0,1] ,
\end{equation}
where $\tau_{K,N}^{(t)}(\theta)$ is an explicit coefficient depending on the curvature $K \in \Real$, dimension $N \in [1,\infty]$, the interpolating time parameter $t \in [0,1]$ and the total length of the geodesic $\theta \in [0,\infty)$ (with an appropriate interpretation of (\ref{eq:intro-interpolate}) when $N=\infty$). When $N <\infty$, the latter coefficient is obtained by geometrically averaging two different volume distortion coefficients:
\begin{equation} \label{eq:intro-tau}
\tau_{K,N}^{(t)}(\theta) := t^{\frac{1}{N}} \sigma_{K,N-1}^{(t)}(\theta)^{\frac{N-1}{N}} ,
\end{equation}
where the $\sigma_{K,N-1}^{(t)}(\theta)$ term encodes an $(N-1)$-dimensional evolution orthogonal to the transport and thus affected by the curvature, and the linear term $t$
represents a one dimensional evolution tangential to the transport and thus independent of any curvature information. As with the Jacobi equation in the usual Riemannian setting, the function $[0,1] \ni t \mapsto \sigma(t) :=  \sigma_{K,N-1}^{(t)}(\theta)$ is explicitly obtained by solving the second-order differential equation:
\begin{equation} \label{eq:intro-ode}
 \sigma''(t) + \theta^{2} \frac{K}{N-1} \sigma(t) =0 \text{ on $t \in [0,1]$ } ~,~ \sigma(0) = 0 ~,~ \sigma(1) = 1 .
\end{equation}

The common feature of the previously known cases $\frac{K}{N} = 0$ for the local-to-global problem is the linear behaviour in time of the distortion coefficient: $\tau_{K,N}^{(t)}(\theta) = t$. A major obstacle with the remaining cases $\frac{K}{N} \neq 0$ is that the function $[0,1] \ni t \mapsto \tau_{K,N}^{(t)}(\theta)$ does not satisfy a second-order differential characterization such as (\ref{eq:intro-ode}). If it did, it would be possible to express the interpolation inequality (\ref{eq:intro-interpolate}) on $[t_0,t_1] \subset [0,1]$
as a second-order differential inequality for $J^{\frac{1}{N}}$ on $[t_0,t_1]$ (see Lemmas \ref{lem:point-CDKN} and \ref{lem:density-CDKN}), and so if (\ref{eq:intro-interpolate}) were known to hold for all $\set{[t_0^i,t_1^i]}_{i=1\ldots k}$ so that $\cup_{i=1}^k (t_0^i,t_1^i) = (0,1)$, it would follow that (\ref{eq:intro-interpolate}) also holds for $[t_0,t_1] = [0,1]$. However, a counterexample to the latter implication was constructed by Deng and Sturm in \cite{dengsturm}, 
thereby showing that:
\begin{equation} \label{eq:intro-test}
\begin{array}{l}
\text{the local-to-global property for $\frac{K}{N} \neq 0$, if true, cannot be obtained by a one-dimensional}\\
\text{bootstrap argument on a \emph{single} $W_2$-geodesic as above, and must follow from a deeper}\\
\text{reason involving a \emph{family} of $W_2$-geodesics \emph{simultaneously}.}
\end{array}
\end{equation}

\smallskip

On the other hand, the above argument does work if we were to replace $\tau$ by the slightly smaller $\sigma$ coefficients.
This motivated Bacher and Sturm in \cite{sturm:loc} to define for $K \in \Real$ and $N \in (1,\infty)$ the slightly weaker ``reduced'' Curvature-Dimension condition, denoted by $\CD^{*}(K,N)$, where the distortion coefficients $\tau_{K,N}^{(t)}(\theta)$ are indeed replaced by $\sigma_{K,N}^{(t)}(\theta)$. Using the above gluing argument (after resolving numerous technicalities), the local-to-global property for $\CD^*(K,N)$ was established in \cite{sturm:loc} on non-branching spaces (see also the work of Erbar--Kuwada--Sturm \cite[Corollary 3.13, Theorem 3.14 and Remark 3.26]{EKS-EntropicCD} for an extension to the essentially non-branching setting, cf. \cite{rajasturm:branch,CM3}).
Let us also mention here the work of Ambrosio--Mondino--Savar\'e \cite{AMS:locglob}, who independently of a similar result in \cite{EKS-EntropicCD}, established the local-to-global property for $\RCD^*(K,N)$ proper spaces, $K \in \Real$ and $N \in [1,\infty]$, without a-priori assuming any non-branching assumptions (but a-posteriori, such spaces must be essentially non-branching by \cite{rajasturm:branch}). 

\smallskip

Without requiring any non-branching assumptions, the $\CD^*(K,N)$ condition was shown in \cite{sturm:loc} to imply the same geometric and analytic inequalities as the $\CD(K,N)$ condition, but with slightly worse constants (typically missing the sharp constant by a factor of $\frac{N-1}{N}$), suggesting that 
the latter is still the ``right" notion of Curvature-Dimension. We conclude that the local-to-global challenge is to properly disentangle between the orthogonal and tangential components of the volume distortion $J$ before attempting to individually integrate them as above. This also highlights the geometric nature of the globalization problem, and demonstrates that it is not merely a technical challenge.

\subsection{Comparing $L^2$ and $L^1$ Optimal-Transport and Main result}

There have been a couple of prior attempts to disentangle the volume distortion into its orthogonal and tangential components, by comparing between $W_2$ and $W_1$ Wasserstein geodesics (associated to $L^2$ and $L^1$ Optimal-Transport, respectively). In \cite{cavasturm:MCP}, this strategy was implicitly employed by Cavalletti and Sturm to show that $\CD_{loc}(K,N)$ implies the measure-contraction property $\MCP(K,N)$, which in a sense is a particular case of $\CD(K,N)$ when one end of the $W_2$-geodesic is a Dirac delta at a point $o \in X$ (see \cite{sturm:II,Ohta1}). In that case, all of the transport-geodesics have $o$ as a common end point, so by considering a disintegration of $\mm$ on the family of spheres centered at $o$, and restricting the $W_2$-geodesic to these spheres, the desired disentanglement was obtained. 
In the subsequent work \cite{cava:decomposition}, Cavalletti generalized this approach to a particular family of $W_2$-geodesics, having the property that for a.e. transport-geodesic $\gamma$, its length $\ell(\gamma)$ is a function of $\f(\gamma_{0})$, where $\f$ is a Kantorovich potential associated to the corresponding $L^2$-Optimal-Transport problem. Here the disintegration was with respect to the individual level sets of $\f$, and again the restriction of the $W_2$-geodesic enjoying the latter property to these level sets (formally of co-dimension one) induced a $W_1$-geodesic, enabling disentanglement. 

\smallskip

Another application of $L^1$-Optimal-Transport, seemingly unrelated to disentanglement of $W_2$-geodesics, appeared in the recent breakthrough work of Klartag \cite{Klartag} on localization in the smooth Riemannian setting. The localization paradigm, developed by Payne--Weinberger \cite{PayneWeinberger}, Gromov--Milman \cite{Gromov-Milman} and Kannan--Lov\'asz--Simonovits \cite{KLS}, is a powerful tool to reduce various analytic and geometric inequalities on the space $(\Real^n,\sfd,\mm)$ to appropriate one-dimensional counterparts. The original approach by these authors was based on a bisection method, and thus inherently confined to $\Real^n$. In \cite{Klartag}, Klartag extended the localization paradigm to the weighted Riemannian setting, by disintegrating the reference measure $\mm$ on $L^1$-Optimal-Transport geodesics (or ``rays") associated to the inequality under study (cf. Feldman--McCann \cite{FeldmanMcCann-Manifold}), and proving that the resulting conditional one-dimensional measures inherit the Curvature-Dimension properties of the underlying manifold. 

\smallskip
Klartag's idea is quite robust, and permitted Cavalletti and Mondino in \cite{CM1} to avoid the smooth techniques used in \cite{Klartag} and to extend the localization paradigm to the framework of essentially non-branching geodesic \mms's $(X,\sfd,\mm)$ of full-support verifying $\CD_{loc}(K,N)$, $N \in (1,\infty)$. By a careful study of the structure of $W_1$-geodesics, Cavalletti and Mondino were able to transfer the Curvature-Dimension information encoded in the $W_2$-geodesics to the individual rays along which a given $W_1$-geodesic evolves, thereby proving that on such spaces,
\begin{equation} \label{eq:CD1-words}
\begin{array}{l}
\text{the conditional one-dimensional measures obtained by disintegration of $\mm$} \\
\text{on $L^1$-Optimal-Transport rays satisfy $\CD(K,N)$.} 
\end{array}
\end{equation} 
Note that the \emph{densities} of one-dimensional $\CD(K,N)$ spaces are characterized via the $\sigma$ (as opposed to $\tau$) volume-distortion coefficients (see the Appendix), so by applying the gluing argument described in the previous subsection, only local $\CD_{loc}(K,N)$ information was required in \cite{CM1} to obtain global control over the entire one-dimensional transport ray. 

\smallskip
This allowed Cavalletti and Mondino (see \cite{CM1,CM2}) to obtain a series of sharp geometric and analytic inequalities for $\CD_{loc}(K,N)$ spaces as above, in particular extending from the smooth Riemannian setting the sharp L\'evy-Gromov \cite{GromovGeneralizationOfLevy} and Milman \cite{EMilmanSharpIsopInqsForCDD} isoperimetric inequalities, as well as the sharp Brunn-Minkowski inequality of Cordero-Erausquin--McCann--Schmuckenschl\"ager \cite{corderomccann:brescamp} and Sturm \cite{sturm:II}, all in global form (see also Ohta \cite{OhtaNeedles}).

\smallskip
We would like to address at this point a certain general belief shared by some in the Optimal-Transport community, stating that the property $\BM(K,N)$ of satisfying the Brunn-Minkowski inequality (with sharp coefficients correctly depending on $K,N$), should be morally equivalent to the $\CD(K,N)$ condition. Rigorously establishing such an equivalence would immediately yield the local-to-global property of $\CD(K,N)$, by the Cavalletti--Mondino localization proof that $\CD_{loc}(K,N) \Rightarrow \BM(K,N)$. However, we were unsuccessful in establishing the missing implication $\BM(K,N) \Rightarrow \CD(K,N)$, and in fact a careful attempt in this direction seems to lead back to the circle of ideas we were ultimately able to successfully develop in this work.

\smallskip
Instead of starting our investigation from $\BM(K,N)$, our strategy is to directly start from a suitable modification of the property (\ref{eq:CD1-words}), 
which we dub $\CD^1(K,N)$, when \eqref{eq:CD1-words} is required to hold for transport rays associated to (signed) distance functions from level sets of continuous functions.
A stronger condition, when \eqref{eq:CD1-words} is required to hold for transport rays associated to \emph{all} $1$-Lipschitz functions, is denoted by $\CD^1_{Lip}(K,N)$ -- see Section \ref{S:CD1} for precise definitions. 
The main result of this work consists of showing that $\CD^1(K,N) \Rightarrow \CD(K,N)$, by means of transferring the one-dimensional $\CD(K,N)$ information encoded in a \emph{family} of suitably constructed $L^1$-Optimal-Transport rays, onto a given $W_2$-geodesic, thereby obtaining the correct disentanglement between tangential and orthogonal distortions.
This goes in exactly the \emph{opposite direction} to the one studied by Cavalletti and Mondino in \cite{CM1}, and completes the cycle: 
\[
\CD_{loc}(K,N) \Rightarrow \CD^1_{Lip}(K,N)  \Rightarrow \CD^1(K,N) \Rightarrow \CD(K,N) .
\]
To the best of our knowledge, this decisive feature of our work -- deducing $\CD(K,N)$ for a given $W_2$-geodesic by considering the $\CD_{\loc}(K,N)$ information encoded in \emph{family} (in accordance with (\ref{eq:intro-test})) of \emph{different} associated $W_2$-geodesics (manifesting itself in the $\CD^1(K,N)$ information along a family of different $L^1$-Optimal-Transport rays) -- has not been previously explored.

\begin{maintheorem} \label{thm:main}
Let $(X,\sfd,\mm)$ be an essentially non-branching \mms with $\mm(X) < \infty$, and let $K \in \Real$ and $N \in (1,\infty)$. Then the following statements are equivalent:
\begin{enumerate}
\item $(X,\sfd,\mm)$ verifies $\CD(K,N)$. 
\item $(X,\sfd,\mm)$ verifies $\CD^*(K,N)$. 
\item $(X,\sfd,\mm)$ verifies $\CD^1_{Lip}(K,N)$. 
\item $(X,\sfd,\mm)$ verifies $\CD^1(K,N)$. 
\end{enumerate}
If in addition $(\supp(\mm),\sfd)$ is a length-space, the above statements are equivalent to:
\begin{enumerate}
\setcounter{enumi}{4}
\item $(X,\sfd,\mm)$ verifies $\CD_{loc}(K,N)$. 
\end{enumerate}
\end{maintheorem}

To this list one can also add the entropic Curvature-Dimension condition $\CD^e(K,N)$ of Erbar--Kuwada--Sturm \cite{EKS-EntropicCD}, which is known to be equivalent to $\CD^*(K,N)$ for essentially non-branching spaces. 
In other words, all synthetic definitions of Curvature-Dimension are equivalent for essentially non-branching \mms's, and in particular, the local-to-global property holds for such spaces (recall that this is known to be false on \mms's where branching is allowed by \cite{R2016}). The equivalence with $\CD_{loc}(K,N)$ is clearly false without some global assumption ultimately ensuring that $(\supp(\mm),\sfd)$ is a geodesic-space, see Remark \ref{rem:false-without-length}.

\smallskip

As already mentioned, and being slightly imprecise (see Section \ref{S:Final} for precise statements), the implications $\CD(K,N) \Rightarrow \CD^*(K,N) \Rightarrow \CD_{loc}(K,N)$ follow from the work of Bacher and Sturm \cite{sturm:loc}, and the implication $\CD_{loc}(K,N) \Rightarrow \CD^1_{Lip}(K,N)$ follows by adapting to the present framework what was already proved by Cavalletti and Mondino in \cite{CM1} (after taking care of the important maximality requirement of transport-rays, see Theorem \ref{T:endpoints}). So almost all of our effort goes into proving that $\CD^1(K,N) \Rightarrow \CD(K,N)$. 
For a smooth weighted Riemannian manifold $(M,\sfd,\mm)$, it is an easy exercise to show the latter implication using the Bakry--\'Emery differential characterization of $\CD(K,N)$ -- simply use an appropriate umbilic hypersurface $H$ passing through a given point $p \in M$ and perpendicular to a given direction $\xi \in T_p M$, and apply the $\CD^1(K,N)$ definition to the distance function from $H$. Of course, this provides no insight towards how to proceed in the \mms setting, so it is natural to try and obtain an alternative synthetic proof, still in the smooth setting. While this is possible, it already poses a much greater challenge, which in some sense provided the required insight leading to the strategy 
we ultimately employ in this work.

\subsection{Main new ingredients of proof}

To achieve the right disentanglement, we are required to develop several new ingredients beyond the present state-of-the-art, which, being conceptual in nature, are in our opinion of independent interest. 

\begin{enumerate}
\item
The first is a \emph{change-of-variables} formula for the density of an $L^2$-Optimal-Transport geodesic in $X$ (see Theorem \ref{T:changeofV}), which depends on a second-order derivative of associated interpolating Kantorovich potentials. 

\smallskip

Let $\Geo(X)$ denote the collection of constant speed geodesics on $X$ parametrized on the interval $[0,1]$, and let $\ee_t : \Geo(X) \ni \gamma \mapsto \gamma_t \in X$ denote the evaluation map at time $t \in [0,1]$. Given two Borel probability measures $\mu_0,\mu_1 \in \P(X)$ with finite second moments, any $W_2$-geodesic $[0,1] \ni t\mapsto \mu_t \in \P(X)$ can be lifted to an optimal dynamical plan $\nu \in \P(\Geo(X))$, so that $(\ee_t)_{\sharp} \nu = \mu_t$ for all $t \in [0,1]$. Let $\varphi$ denote a Kantorovich potential associated to the $L^2$-transport problem between $\mu_0$ and $\mu_1$. Given $s,t\in (0,1)$, we introduce the time-propagated intermediate Kantorovich potential $\Phi_s^t$ by pushing forward $\varphi_s$ via $\ee_t \circ \ee_s^{-1}$,  where $\set{\varphi_t}_{t\in [0,1]}$ is the family of interpolating Kantorovich potentials obtained via the Hopf--Lax semi-group applied to $\varphi$. While $\ee_t^{-1}$ may be multi-valued, Theorem \ref{thm:order12-main} ensures that $\Phi_s^t  = \varphi_s \circ \ee_s \circ \ee_t^{-1}$ is well-defined on $\ee_t(G_\varphi)$, the set of $t$-mid-points of transport geodesics.

\smallskip
Theorem \ref{T:changeofV} states that if $(X,\sfd,\mm)$ is an essentially non-branching \mms verifying $\CD^{1}(K,N)$ ($\mm(X) < \infty$ and $N \in (1,\infty)$), and if $\mu_0,\mu_1 \ll \mm$, then for $\nu$-a.e. transport-geodesic $\gamma \in \Geo(X)$ of positive length:
\begin{equation}\label{E:intro-COV}
\frac{\rho_{s} (\gamma_{s})}{\rho_{t}(\gamma_{t})} = \frac{\ell^{2}(\gamma)}{\partial_{\tau}|_{\tau = t}\Phi_{s}^{\tau}(\gamma_{t})}  \cdot h^\gamma_s(t)  \;\;\;
\text{for a.e. $t,s \in (0,1)$, }
\end{equation}
where $\rho_t$ are appropriate versions of the densities $d\mu_t / d\mm$, and for every $s \in (0,1)$, $h^\gamma_s$ is a $\CD(\ell(\gamma)^2 K ,N)$ density on $[0,1]$ so that $h^\gamma_s(s) = 1$. In particular, for a.e. $t,s\in (0,1)$, $\partial_{\tau}|_{\tau = t}\Phi_{s}^{\tau}(\gamma_{t})$ exists and is positive.
Here $h^\gamma_s$ is obtained from the $\CD^1(K,N)$ condition applied to the transport-ray associated to the (signed) distance function from the level set $\set{\varphi_s = \varphi_s(\gamma_s)}$.  
\smallskip

Theorem \ref{T:changeofV} constitutes the culmination of  Part \ref{part2} of this work, which is mostly dedicated to introducing the $\CD^{1}(K,N)$ condition and rigorously establishing the change-of-variables formula \eqref{E:intro-COV}.
Note that we refrain from making any assumptions on (the challenging) spatial regularity of $\Phi_{s}^{t}$ when $t\neq s$,  so we are precluded from invoking the coarea formula in our derivation. Our main tool for deriving (\ref{E:intro-COV}) is a comparison between two disintegrations of appropriate measures, one encoding $W_2$ information and another encoding $W_1$ information -- see Section \ref{S:comparison1} for a heuristic derivation. 
\item 

To obtain disentanglement of the ``Jacobian" $t \mapsto 1/\rho_t(\gamma_t)$ into its orthogonal and tangential components, we need to understand the first-order variation of the change-of-variables formula (\ref{E:intro-COV}) at $t=s$, i.e. the second-order variation of $t \mapsto \Phi_s^t$ at $t=s$, which amounts to a third-order variation of $t \mapsto \varphi_t$. 
Our second main new ingredient in this work is a surprising \emph{third-order bound} on the variation of $t \mapsto \varphi_t$ along the Hopf--Lax semi-group (Theorem \ref{thm:z-c}), which holds in complete generality on any proper geodesic space. 

\smallskip

To this end, we develop in Part I of this work a first, second, and finally third order temporal theory of intermediate Kantorovich potentials in a purely metric setting $(X,\sfd)$, without specifying any reference measure $\mm$ and without assuming any non-branching assumptions. This part, which may be read independently of the other components of this work, is presented first (in Sections \ref{sec:prelim-partI}-\ref{sec:order3}), since its results are constantly used throughout the rest of this work. 

\smallskip

Our starting point here is the pioneering work by Ambrosio--Gigli--Savar\'e \cite{AGS-Book},\cite[Section 3]{ambrgisav:heat}, who already investigated in a very general (extended) metric space setting the first and second order temporal behaviour of the Hopf-Lax semi-group $Q_t$ applied to a general function $f : X \rightarrow \Real \cup \set{+\infty}$. However, the essential point we observe in our treatment is that when $f$ is itself a Kantorovich potential $\varphi$, characterized by the property that $\varphi = Q_1(-\varphi^c)$ and $\varphi^c = Q_1(-\varphi)$, much more may be said regarding the behaviour of $t \mapsto \varphi_t := -Q_t(-\varphi)$, even in first and second order. This is due to the fact that if we reverse time and define $\varphic_t := Q_{1-t}(-\varphi^c)$, then we obtain two-sided control over $\varphi_t$ on the set $\set{\varphi_t = \varphic_t}$, which turns out to coincide with the set $\ee_t(G_\varphi)$.
 So for instance, two apparently novel observations which we constantly use throughout this work are that for all $t \in (0,1)$, $\ell_t^2/2 := \partial_t \varphi_t$ exists on $\ee_t(G_\varphi)$, and that transport geodesics having a given $x \in X$ as their $t$-midpoint \emph{all} have the \emph{same} length $\ell_t(x)$. In Section \ref{sec:order12}, we establish Lipschitz regularity properties of $t \mapsto \ell^2_t(x)$ for all $x \in X$, as well as upper and lower derivative estimates, both pointwise and a.e., for appropriate times $t$. These are then transferred in Section \ref{S:Phi} to corresponding estimates for the function $\Phi_s^t$.

\smallskip

Part I culminates in Section \ref{sec:order3}, whose goal is to prove a quantitative version of the following (somewhat oversimplified) statement, which crucially provides second order information on $\ell_t$, or equivalently, third order information on $\varphi_t$, along $\gamma_t$:
\begin{equation} \label{eq:intro-3rd}
\begin{array}{l}
\text{If } \frac{1}{\ell(\gamma)^2} \partial_\tau|_{\tau=t} \frac{\ell_\tau^2}{2}(\gamma_t) \text{ exists a.e. in $t \in (0,1)$ and coincides with an absolutely} \\
\text{continuous function $z$, then } z'(t) \geq z(t)^2  \text{ for a.e. $t \in (0,1)$.}
\end{array}
\end{equation} 
Equivalently, this amounts to the statement that:
\begin{equation} \label{eq:intro-L}
(0,1) \ni r \mapsto L(r) := \exp\brac{- \frac{1}{\ell(\gamma)^2} \int^r_{r_0} \partial_\tau|_{\tau=t} \frac{\ell_\tau^2}{2}(\gamma_t) dt} \text{ is concave },
\end{equation}
since (formally):
\[
\frac{L''}{L} = (\log L) '' + ((\log L)')^2 = - z' + z^2 \leq 0 . \]
It turns out that $L(t)$ precisely corresponds to the tangential component of $1/\rho_t(\gamma_t)$, and its concavity ensures that it is synthetically controlled by the linear term appearing in the definition of $\tau^{(t)}_{K,N}(\theta)$ in (\ref{eq:intro-tau}). 
\smallskip
The novel observation that it is possible to extract in a general metric setting \emph{third order information} from the Hopf-Lax semi-group, which formally solves the first-order Hamilton-Jacobi equation, is in our opinion \emph{one of the most surprising parts of this work}. Even in the smooth Riemannian setting, we were not able to find a synthetic proof which is easier than the one in the general metric setting; a formal differential proof of (\ref{eq:intro-3rd}) assuming both temporal and (more challenging) spatial higher-order regularity of $\varphi_t$ is provided in Subsection \ref{subsec:order3-formal}, but the latter seems to wrongly suggest that it would not be possible to extend (\ref{eq:intro-3rd}) beyond a Hilbertian setting. Our proof in the general metric setting (Theorem \ref{thm:main-3rd-order}) is based on a careful comparison of second order expansions of $\eps \mapsto \varphi_{\tau+\eps}(\gamma_\tau)$ at $\tau=t,s$, and subtle differences between the usual second derivative and the second Peano derivative (see Section \ref{sec:prelim-partI}) come into play. 
\item
Our third main new ingredient, described in Part \ref{part3}, is a certain \emph{rigidity} property of the change-of-variables formula (\ref{E:intro-COV}), which allows us to bootstrap the a-priori available temporal regularity, and which in combination with the first and second ingredients, enables us to achieve disentanglement. 

\smallskip

Indeed, the definition of $\Phi_s^t$ may be naturally extended to an appropriate domain beyond $\ee_t(G_\varphi)$ as follows, allowing to easily (formally) calculate its partial derivative:
\[ \Phi_s^t = \varphi_t + (t-s) \frac{\ell_t^2}{2}  \;\;\;  , \;\;\; \partial_t \Phi_s^t = \ell_t^2 + (t-s) \partial_t\frac{\ell_t^2}{2} .
\] Evaluating at $x = \gamma_t$ and plugging this into the change-of-variables formula (\ref{E:intro-COV}), it follows that for $\nu$-a.e. geodesic $\gamma$:
\begin{equation} \label{eq:intro-rigid}
\frac{\rho_s(\gamma_s)}{\rho_t(\gamma_t)} = \frac{h^\gamma_s(t)}{1 + (t-s) \frac{\partial_\tau|_{\tau=t}\ell_\tau^2/2(\gamma_t)}{\ell^2(\gamma)}} \;\;\; \text{for a.e. } t,s \in (0,1) .
\end{equation}
Thanks to the idea of considering together both initial-point $s$ and end-point $t$, the latter formula takes on a very rigid structure:
note that on the left-hand-side the $s$ and $t$ variables are separated, and the denominator on the right-hand-side depends linearly is $s$. Consequently, we can easily bootstrap the a-priori available regularity in $s$ and $t$ of all terms involved. 
It follows that $\frac{1}{\ell^2(\gamma)} \partial_\tau|_{\tau=t}\ell_\tau^2/2(\gamma_t)$ must coincide for a.e. $t \in (0,1)$ with a locally-Lipschitz function $z(t)$, so that (\ref{eq:intro-3rd}) applies. In addition, by redefining $\set{h^\gamma_s}$ for $s$ in a null subset of $(0,1)$, we can guarantee that $(0,1) \ni s \mapsto h^\gamma_s(t)$ is locally Lipschitz (for any given $t \in (0,1)$), even though there is a-priori no relation between the different densities $\set{h^\gamma_s}_{s \in (0,1)}$.

At this point, if $\rho_t(\gamma_t)$ and $z(t)$ were known to be $C^2$ smooth, and equality were to hold in (\ref{eq:intro-rigid}) for all $s,t \in (0,1)$, we could then define:
\begin{equation} \label{intro:eq-Y}
Y(r) := \exp\brac{ \int_{r_0}^r \partial_t|_{t=s} \log h^\gamma_s(t) ds } ,
\end{equation}
and as $\partial_t|_{t=s} \log(1 + (t-s) z(t)) = z(s)$, it would follow, recalling the definition (\ref{eq:intro-L}) of $L$, that:
\begin{equation} \label{eq:intro-dis}
\frac{\rho_{r_0}(\gamma_{r_0})}{\rho_r(\gamma_r)} = L(r) Y(r) \;\;\; \forall r \in (0,1) . 
\end{equation}

Using the fact that all $\set{h^\gamma_s}_{s \in (0,1)}$ are $\CD(\ell(\gamma)^2 K,N)$ densities to control $\partial^2_t|_{t=r} \log h_r(t)$, and surprisingly, also the concavity of $L$ (again!) to 
control the mixed partial derivatives $\partial_s \partial_t|_{t=s=r} \log h^\gamma_s(t)$, a formal computation described in Subsection \ref{subsec:formal-rigid} then verifies that $Y$ is a $\CD(\ell(\gamma)^2 K,N)$ density itself. A rigorous justification without all of the above non-realistic assumptions turns out to be extremely tedious, due to the difficulty in applying an approximation argument while preserving the rigidity of the equation -- this is worked out in Section \ref{sec:LY} and the Appendix.
\end{enumerate}

\smallskip
After taking care of all these details, we finally obtain the desired disentanglement (\ref{eq:intro-dis}) of the Jacobian: $L$ is concave and so controlled synthetically by a linear distortion coefficient, whereas $Y$ is a  $\CD(\ell(\gamma)^2 K,N)$ density and so (by definition) $Y^{1/(N-1)}$ is controlled synthetically by the $\sigma^{(t)}_{K,N-1}(\ell(\gamma))$ coefficient. A standard application of H\"{o}lder's inequality then verifies that $J^{1/N}(r) = \rho_r(\gamma_r)^{-1/N}$ is controlled by the $\tau^{(t)}_{K,N}(\ell(\gamma))$ distortion coefficient, i.e. satisfies (\ref{eq:intro-interpolate}) -- in fact for all $t_0,t_1 \in [0,1]$ -- thereby establishing $\CD(K,N)$, see Theorem \ref{T:CD1-CD}.  

\smallskip

The definition (\ref{intro:eq-Y}) of $Y$ finally sheds light on the crucial role which the parameter $s \in (0,1)$ plays in our strategy -- its role is to vary between the different $W_2$-geodesics from which the $\CD_{loc}(K,N)$ information is extracted into the $\CD^1(K,N)$ information on the disintegration into transport-rays from the (signed) distance functions from level sets $\set{\varphi_s = \varphi_s(\gamma_s)}$, thereby coming full circle with the observation of (\ref{eq:intro-test}). 

\bigskip

Besides establishing the local-to-global property of $\CD(K,N)$ and the equivalence of its various variants (in our setting), we emphasize that as a by product of our proof, we obtain a remarkable new self-improvement property of $\CD(K,N)$: the $\tau_{K,N}$-concavity (\ref{eq:intro-interpolate}) of the transport Jacobian $J_t(\gamma_t)$ along \emph{all} $W_2$-geodesics implies the (a-priori) \emph{stronger} ``L-Y" decomposition  $J_t(\gamma_t) = L_\gamma(t) Y_\gamma(t)$, where $L_\gamma$ is concave and $Y_\gamma$ is a $\CD(\ell(\gamma)^2 K, N)$ density on $(0,1)$. As already mentioned above, this self-improvement is \emph{false} for a single $W_2$-geodesic. We believe that the stronger ``L-Y" information will prove to be of further use in the study of $\CD(K,N)$ essentially non-branching spaces.

\bigskip
We refer to Section \ref{S:Final} for the final details and for additional immediate corollaries of the Main Theorem \ref{thm:main} pertaining to $\RCD(K,N)$ and strong $\CD(K,N)$ spaces. We also provide there several concluding remarks and suggestions for further investigation.

\bigskip

\noindent
\textbf{Acknowledgment.} We would like to thank Theo Sturm and C\'edric Villani for numerous discussions and for encouraging us to pursue the globalization problem. We also thank the referees for their careful reading of the manuscript and helpful comments.

\bigskip

\part{Temporal Theory of Optimal Transport}\label{part1}

\section{Preliminaries} \label{sec:prelim-partI}

\subsection{Geodesics}

A metric space $(X,\sfd)$ is called a length space if for all $x,y \in X$, $\sfd(x,y) = \inf \ell(\sigma)$, where the infimum is over all (continuous) curves $\sigma : I \rightarrow X$ connecting $x$ and $y$, and $\ell(\sigma) := \sup \sum_{i=1}^k \sfd(\sigma(t_{i-1}),\sigma(t_{i}))$ denotes the curve's length, where the latter supremum is over all $k \in \N$ and $t_0 \leq \ldots \leq t_k$ in the interval $I \subset \Real$. A curve $\gamma$ is called a geodesic if $\ell(\gamma|_{[t_0,t_1]}) = \sfd(\gamma(t_0),\gamma(t_1))$ for all $[t_0,t_1] \subset I$. If $\ell(\gamma) = 0$ we will say that $\gamma$ is a null geodesic.
The metric space is called a geodesic space if for all $x,y \in X$ there exists a geodesic in $X$ connecting $x$ and $y$. 
We denote by $\Geo(X)$ the set of all closed directed constant-speed geodesics parametrized on the interval $[0,1]$:
\[
\Geo(X) : = \set{ \gamma : [0,1] \rightarrow X \; ; \;  \sfd(\gamma(s),\gamma(t)) = |s-t| \sfd(\gamma(0),\gamma(1)) \;\; \forall s,t \in [0,1] } .
\]
We regard $\Geo(X)$ as a subset of all Lipschitz maps $\text{Lip}([0,1], X)$ endowed with the uniform topology. We will frequently use $\gamma_t := \gamma(t)$.

\medskip

The metric space is called proper if every closed ball (of finite radius) is compact. It follows from the metric version of the Hopf-Rinow Theorem (e.g. \cite[Theorem 2.5.28]{BBI}) that for complete length spaces, local compactness is equivalent to properness, and that complete proper length spaces are in fact geodesic.

\medskip

Given a subset $D \subset X \times \Real$, we denote its sections by:
\[
D(t) := \set{ x \in X \;;\; (x,t) \in D} ~,~ D(x) := \set{t \in \Real \; ; \; (x,t) \in D} .
\]
Given a subset $G \subset \Geo(X)$, we denote by $\mathring{G} := \set{ \gamma|_{(0,1)} \;;\; \gamma \in G}$ the corresponding open-ended geodesics on $(0,1)$. For a subset of (closed or open) geodesics $\tilde{G}$, we denote:
\[
D(\tilde{G}) := \set{ (x,t) \in X \times \Real \; ; \; \exists \gamma \in \tilde{G} ~,~ t \in \text{Dom}(\gamma) \; , \; x = \gamma_t } .
\]
We denote by $\ee_t : \Geo(X) \ni \gamma \mapsto \gamma_t \in X$ the (continuous) evaluation map at $t \in [0,1]$, and abbreviate given $I \subset [0,1]$ as follows:
\begin{align*}
\ee_t(\tilde{G}) = \tilde{G}(t) & := D(\tilde{G})(t) =  \set{ \gamma_t \; ; \; \gamma \in \tilde{G} } ~,~ \ee_I(\tilde{G}) := \cup_{t \in I} \ee_t(\tilde{G}) ~,~\\
\tilde{G}(x) & := D(\tilde{G})(x) = \set{ t \in [0,1] \; ; \; \exists \gamma \in \tilde{G} ~,~ t \in \text{Dom}(\gamma) \; , \; \gamma_t = x} .
\end{align*}

\subsection{Derivatives} \label{subsec:prelim-derivatives}

For a function $g : A \rightarrow \Real$ on a subset $A \subset \Real$, denote its upper and lower derivatives at a point $t_0 \in A$ which is an accumulation point of $A$ by:
\[
\frac{\overline{d}}{dt} g(t_0) = \limsup_{A \ni t \rightarrow t_0} \frac{g(t) - g(t_0)}{t-t_0}  ~,~ \underline{\frac{d}{dt}} g(t_0) = \liminf_{A \ni t \rightarrow t_0} \frac{g(t) - g(t_0)}{t-t_0} .
\]
We will say that $g$ is differentiable at $t_0$ iff $\frac{d}{dt} g(t_0) := \frac{\overline{d}}{dt} g(t_0) = \underline{\frac{d}{dt}} g(t_0) < \infty$. 
This is a slightly more general definition of differentiability than the traditional one which requires that $t_0$ be an interior point of $A$.

\begin{remark} \label{R:diff-restriction}
Note that there are only a countable number of isolated points in $A$, so a.e. point in $A$ is an accumulation point. In addition, it is clear that if $t_0 \in B \subset A$ is an accumulation point of $B$ and $g$ is differentiable at $t_0$, then $g|_B$ is also differentiable at $t_0$ with the same derivative. In particular, if $g$ is a.e. differentiable on $A$ then $g|_B$ is also a.e. differentiable on $B$ and the derivatives coincide.
\end{remark}

\begin{remark}\label{R:differentiabilitydensity}
Denote by $A_1 \subset A$ the subset of density one points of $A$ (which are in particular accumulation points of $A$). By Lebesgue's Density Theorem $\L^1(A \setminus A_1) = 0$, where we denote by $\L^1$ the Lebesgue measure on $\Real$ throughout this work. If $g : A \rightarrow \Real$ is locally Lipschitz, consider any locally Lipschitz extension $\hat g : \Real \to \Real$ of $g$. Then it is easy to check that for $t_0 \in A_1$, $g$ is differentiable in the above sense at $t_0$ if and only if $\hat g$ is differentiable at $t_0$ in the usual sense, in which case the derivatives coincide. In particular, as $\hat g$ is a.e. differentiable on $\Real$, it follows that $g$ is a.e. differentiable on $A_1$ and hence on $A$, and it holds that $\frac{d}{dt} g = \frac{d}{dt} \hat g$ a.e. on $A$.
\end{remark}

\medskip

Let $f : I \rightarrow \Real$ denote a convex function on an open interval $I \subset \Real$. It is well-known that the left and right derivatives $f^{\prime,-}$ and $f^{\prime,+}$ exist at every point in $I$ and that $f$ is locally Lipschitz there; in particular, $f$ is differentiable at a given point iff the left and right derivatives coincide there. Denoting by $D \subset I$ the differentiability points of $f$ in $I$, it is also well-known that $I \setminus D$ is at most countable. Consequently, any point in $D$ is an accumulation point, and we may consider the differentiability in $D$ of $f' : D \rightarrow \Real$ as defined above. 
We will require the following elementary one-dimensional version (probably due to Jessen) of the well-known Aleksandrov's theorem 
about twice differentiability a.e. of convex functions on $\Real^n$ (see \cite[Theorem 5.2.1]{ConvexAnalysisBookI} or \cite[Section 2.6]{BorweinVanderwerff-Book}, and \cite[p. 31]{Schneider-Book} for historical comments). Clearly, all of these results extend to locally semi-convex and semi-concave functions as well; recall that a function $f : I \rightarrow \Real$ is called semi-convex (semi-concave) if there exists $C \in \Real$ so that $I \ni x \mapsto f(x) + C x^2$ is convex (concave).

\begin{lemma}[Second Order Differentiability of Convex Function] \label{lem:convex-2nd-diff}
Let $f : I \rightarrow \Real$ be a convex function on an open interval $I \subset \Real$, and let $\tau_0 \in I$ and $\Delta \in \Real$. Then the following statements are equivalent:
\begin{enumerate}
\item $f$ is differentiable at $\tau_0$, and if $D \subset I$ denotes the subset of differentiability points of $f$ in $I$, then $f' : D \rightarrow \Real$ is differentiable at $\tau_0$ with:
\[
(f')'(\tau_0) := \lim_{D \ni \tau \rightarrow \tau_0} \frac{f'(\tau) - f'(\tau_0)}{\tau-\tau_0} = \Delta .
\]
\item The right derivative $f^{\prime,+} : I \rightarrow \Real$ is differentiable at $\tau_0$ with $(f^{\prime,+})'(\tau_0) = \Delta$. 
\item The left derivative $f^{\prime,-}: I \rightarrow \Real$ is differentiable at $\tau_0$ with $(f^{\prime,-})'(\tau_0) = \Delta$. 
\item $f$ is differentiable at $\tau_0$ and has the following second order expansion there:
\[
f(\tau_0 + \eps) = f(\tau_0) + f'(\tau_0) \eps + \Delta \frac{\eps^2}{2} + o(\eps^2)  \text{ as $\eps \rightarrow 0$}. 
\]
In this case, $f$ is said to have a second Peano derivative at $\tau_0$. 
\end{enumerate}
\end{lemma}

We remark that even for a differentiable function $f$, while the implication $(1) \Rightarrow (4)$ follows by Taylor's theorem (existence of the second derivative at a point implies existence of the second Peano derivative there), the converse implication is in general false (see e.g. \cite{Oliver-ExactPeano} for a nice discussion). 
For a locally semi-convex or semi-concave function $f$, we will say that $f$ is twice differentiable at $\tau_0$ if any (all) of the above equivalent conditions hold for some $\Delta \in \Real$, and write $(\frac{d}{d\tau})^{2}|_{\tau = \tau_0} f(\tau) = \Delta$.

\smallskip
Finally, we will require the following slightly more refined notation. 
\begin{definition*}
Given an open interval $I \subset \Real$ and a function $f : I \rightarrow \Real$ which is differentiable at $\tau_0 \in I$, we define its upper and lower second Peano derivatives at $\tau_0$, denoted $\overline{\P}_2 f(\tau_0)$ and  $\underline{\P}_2 f(\tau_0)$ respectively, by:
\[
 \overline{\P}_2 f(\tau_0) := \limsup_{\eps\rightarrow 0} \frac{h(\eps)}{\eps^2} \geq \liminf_{\eps\rightarrow 0} \frac{h(\eps)}{\eps^2} =: \underline{\P}_2 f(\tau_0) ,
\]
where:
\[
h(\eps) := 2( f(\tau_0 + \eps) - f(\tau_0) - \eps f'(\tau_0)) .
\]
Clearly $f$ has a second Peano derivative at $\tau_0$ iff $\overline{\P}_2 f(\tau_0) = \underline{\P}_2 f(\tau_0) < \infty$. 
\end{definition*}

The following is a type of Stolz--Ces\`aro lemma:
\begin{lemma} \label{lem:peano-inq}
Given an open interval $I \subset \Real$ and a locally absolutely continuous function $f : I \rightarrow \Real$ which is differentiable at $\tau_0 \in I$, we have:
\[
\underline{\frac{d}{dt}} f'(\tau_0) \leq \underline{\P}_2 f(\tau_0) \leq \overline{\P}_2 f(\tau_0) \leq \frac{\overline{d}}{dt} f'(\tau_0)  .
\]
\end{lemma}
\begin{proof}
By local absolute continuity, $f$ is differentiable a.e. in $I$ and we have for small enough $\abs{\eps}$:
\[
\frac{1}{2} h(\eps) = f(\tau_0+\eps) - f(\tau_0) - \eps f'(\tau_0) =\int_{0}^{\eps} (f'(\tau_0 + \delta) - f'(\tau_0)) d\delta ,
\]
and hence:
\[
\frac{h(\eps)}{\eps^{2}} = \frac{1}{\eps^2} \int_{0}^{\eps} 2 \delta \frac{f'(\tau_0 + \delta) - f'(\tau_0)}{\delta}  d\delta .
\]
Taking appropriate subsequential limits as $\eps \rightarrow 0$, the asserted inequalities readily follow. 
\end{proof}

\section[Temporal Theory of Intermediate-Time Kantorovich Potentials.\\ First and Second Order]{Temporal Theory of Intermediate-Time Kantorovich Potentials. First and Second Order}\label{sec:order12}

In the next sections, we will only consider the quadratic cost function $c=\sfd^2/2$ on $X \times X$. 

\begin{definition*}[$c$-Concavity, Kantorovich Potential]
The $c$-transform of a function $\psi : X \rightarrow \Real \cup \set{\pm\infty}$ is defined as the following (upper semi-continuous) function:
\[
\psi^c(x) = \inf_{y \in X} \frac{\sfd(x,y)^2}{2} - \psi(y) . 
\]
A function $\varphi : X \rightarrow \Real \cup \set{\pm \infty}$ is called $c$-concave if $\varphi = \psi^c$ for some $\psi$ as above. 
It is well known \cite[Exercise 2.35]{Vil:topics} that $\varphi$ is $c$-concave iff $(\varphi^c)^c = \varphi$. In the context of optimal-transport 
with respect to the quadratic cost $c$, a $c$-concave function $\varphi : X \rightarrow \Real \cup \set{-\infty}$ which is not identically equal to $-\infty$ is also 
known as a Kantorovich potential, and this is how we will refer to such functions in this work. 
In that case, $\varphi^c : X \rightarrow \Real \cup \set{-\infty}$ is also a Kantorovich potential, called the dual or conjugate potential.
\end{definition*}

There is a natural way to interpolate between a Kantorovich potential and its dual by means of the Hopf-Lax semi-group, 
resulting in intermediate-time Kantorovich potentials $\set{\varphi_t}_{t \in (0,1)}$. The goal of the next three 
sections is to provide first, second and third order information on the time-behavior $t \mapsto \varphi_t(x)$ at intermediate times $t \in (0,1)$. 
In these sections, we only assume that $(X,\sfd)$ is a \textbf{proper geodesic metric space}.

\smallskip
In this section, we focus on first and second order information. The main new result is Theorem \ref{thm:order12-main}.

\subsection{Hopf-Lax semi-group}

We begin with several well-known definitions which we slightly modify and specialize to our setting. 
\begin{definition*}[Hopf-Lax Transform]
Given $f : X \rightarrow \Real \cup \set{\pm \infty}$ which is not identically $+\infty$ and $t > 0$, define the Hopf-Lax transform $Q_t f : X \rightarrow \Real \cup \set{-\infty}$ by:
\begin{equation} \label{eq:Hopf-Lax}
Q_t f (x) := \inf_{y \in X} \frac{\sfd(x,y)^2}{2t} + f(y) . 
\end{equation}
Clearly either $Q_t f \equiv -\infty$ or $Q_t f(x)$ is finite for all $x \in X$ (as our metric $\sfd$ is finite). Consequently, we denote:
\[
t_*(f) := \sup \set{t > 0 \;;\; Q_t f \not\equiv -\infty} ,
\]
setting $t_*(f) = 0$ if the supremum is over an empty set. Finally, we set $Q_0 f := f$. 
\end{definition*}

It is not hard to check (see e.g. \cite[Theorem 2.5 (i)]{lottvillani:length}) that when $(X,\sfd)$ is a length space (and in particular geodesic),
the Hopf-Lax transform is in fact a semi-group on $[0,\infty)$:
\[
Q_{s+t} f = Q_s \circ Q_t f \;\;\; \forall t,s \geq  0 .
\]

\begin{remark}
It is also possible to extend the definition of $Q_t f$ to negative times $t < 0$ by setting:
\[
Q_t f (x) := - Q_{-t}(-f)(x) = \sup_{y \in X} \frac{\sfd(x,y)^2}{2t} + f(y)  ~,~ t < 0 .  
\]
This is called the backwards Hopf-Lax semi-group on $(-\infty,0]$. However, $(\Real,+) \ni t \mapsto (Q_t,\circ)$ is in general \textbf{not} an abelian group homomorphism, not even for $t \in [0,1]$ when applied to a Kantorovich potential $\varphi$ (characterized by $Q_{-1} \circ Q_1(-\varphi) = -\varphi$) - see Subsection \ref{subsec:duality}. This will be a rather significant nuisance we will need to cope with in this work. 
\end{remark}

Clearly $(0,\infty) \times X \ni (t,x) \mapsto Q_t f(x)$ is upper semi-continuous as the infimum of continuous functions in $(t,x)$, and by definition $[0,\infty) \ni t \mapsto Q_t f(x)$ is monotone non-increasing for each $x \in X$. Consequently, $(0,\infty) \ni t \mapsto Q_t f(x)$ must be continuous from the left.  

It may also be shown (see \cite[Lemma 3.1.2]{AGS-Book}) that $X \times (0,t_*(f)) \ni (x,t) \mapsto Q_t f (x)$ is continuous (and in fact locally Lipschitz, see Theorem \ref{thm:AGS} below). 
Together with the left-continuity, we deduce that for every $x \in X$, $(0,t_*(f)] \ni t \mapsto Q_t f(x)$ is continuous. 

\medskip

Note that by definition $f^c = Q_1(-f)$, and that a Kantorovich pair of conjugate potentials $\varphi,\varphi^c : X \rightarrow \Real \cup \set{-\infty}$ are
characterized by not being identically equal to $-\infty$ and satisfying:
\[
\varphi = Q_1(-\varphi^c) ~,~ \varphi^c = Q_1(-\varphi) .
\]
In particular, $t_*(\varphi),t_*(\varphi^c) \geq 1$, and we a-posteriori deduce that $\varphi, \varphi^c$ are both finite on the entire space $X$ 
(we have used above the fact that the metric $\sfd$ is finite, which differs from other more general treatments).

\begin{definition*}[Interpolating Intermediate-Time Kantorovich Potentials]
Given a Kantorovich potential $\varphi : X \rightarrow \Real$, the interpolating Kantorovich potential at time $t \in [0,1]$, $\varphi_t : X \rightarrow \Real$, is defined for all $t \in [0,1]$ by: 
\[
\varphi_t(x) := Q_{-t}(\varphi) = -Q_t(-\varphi) . 
\]
Note that $\varphi_0 = \varphi$, $\varphi_1 = -\varphi^c$, and:
\[
-\varphi_t(x) = \inf_{y \in X} \frac{\sfd^2(x,y)}{2t} - \varphi(y) \;\;\;\; \forall t \in (0,1] .
\]
\end{definition*}

Applying the above mentioned general properties of the Hopf-Lax semi-group to $\varphi_t$, it will be useful to record:
\begin{lemma} \label{lem:order0}
\hfill
\begin{enumerate}
\item $(x,t) \mapsto \varphi_t(x)$ is lower semi-continuous on $X \times (0,1]$ and continuous on $X \times (0,1)$. 
\item For every $x \in X$, $[0,1] \ni t \mapsto \varphi_t(x)$ is monotone non-decreasing and continuous on $(0,1]$. 
\end{enumerate}
\end{lemma}

\begin{definition*}[Kantorovich Geodesic]
Given a Kantorovich potential $\varphi : X \rightarrow \Real$, a geodesic $\gamma \in \Geo(X)$ is called a $\varphi$-Kantorovich (or optimal) geodesic if:
\[ \varphi(\gamma_0) + \varphi^c(\gamma_1) = \frac{d(\gamma_0,\gamma_1)^2}{2} = \frac{\len(\gamma)^2}{2} .
\] We denote all $\varphi$-Kantorovich geodesics by $G_\varphi$. Note that $\gamma \in G_{\varphi}$ iff $\gamma^c \in G_{\varphi^c}$, where $\gamma^c(t) := \gamma(1-t)$ is the time-reversed geodesic. By upper semi-continuity of $\varphi$ and $\varphi^c$, it follows that $G_\varphi$ is a closed subset of  $\Geo(X)$.
\end{definition*}

The following is not hard to check (see e.g. \cite[Corollary 2.16]{cava:decomposition}): 
\begin{lemma} \label{lem:varphi-linear}
Let $\gamma$ be a $\varphi$-Kantorovich geodesic. Then:
\[
\varphi_s(\gamma_s) - \varphi_r(\gamma_r) = \frac{\sfd(\gamma_s,\gamma_r)^2}{2 (r-s)} = (r-s)\frac{\len(\gamma)^2}{2} \;\;\; \forall s,r \in [0,1] .
\]
\end{lemma}

\bigskip

\subsection{Distance functions}

The following important definition was given by Ambrosio--Gigli--Savar\'e \cite{AGS-Book, ambrgisav:heat}:
\begin{definition*}[Distance functions $D^{\pm}_f$] 
Given $f : X \rightarrow \Real \cup \set{+\infty}$ which is not identically $+\infty$, denote:
\[
D^+_{f}(x,t) := \sup \limsup_{n \rightarrow \infty} \sfd(x,y_n) \geq  \inf \liminf_{n \rightarrow \infty} \sfd(x,y_n)   =: D^-_{f}(x,t) ,
\]
where the supremum and infimum above run over the set of minimizing sequences $\set{y_n}$ in the definition of the Hopf-Lax transform (\ref{eq:Hopf-Lax}). 
A simple diagonal argument shows that the (outer) supremum and infimum above are in fact attained.
\end{definition*}

The following properties were established in \cite{AGS-Book},\cite[Chapter 3]{ambrgisav:heat}:
\begin{theorem}[Ambrosio--Gigli--Savar\'e] \label{thm:AGS}
For any metric space $(X,\sfd)$ (not necessarily proper, complete nor geodesic):
\begin{enumerate}
\item Both functions $D^{\pm}_f(x,t)$ are locally finite on $X \times (0,t_*(f))$, and $(x,t) \mapsto Q_t f(x)$ is locally Lipschitz there.
\item $(x,t) \mapsto D^{\pm}_f(x,t)$ is upper ($D^{+}_f(x,t)$) / lower ($D^{-}_f(x,t)$) semi-continuous on $X \times (0,t_*(f))$. 
\item For every $x \in X$, both functions $(0,t_*(f)) \ni t \mapsto D^{\pm}_f(x,t)$ are monotone non-decreasing and coincide except where they have (at most countably many) jump discontinuities. \item For every $x \in X$, $\partial_t^{\pm} Q_t f(x) = - \frac{(D^{\pm}_f(x,t))^2}{2 t^2}$ for all $t \in (0,t_*(f))$, where $\partial_t^{-}$ and $\partial_t^+$ denote the left and right partial derivatives, respectively. In particular, the map $(0,t_*(f)) \ni t \mapsto Q_t f(x)$ is locally Lipschitz and locally semi-concave, and differentiable at $t \in (0,t_*(f))$ iff $D^+_f(x,t) = D^-_f(x,t)$. 
\end{enumerate}
\end{theorem}

It may be instructive to recall the proof of property (3) above, which is related to some ensuing properties, so for completeness, we present it below. For simplicity, we restrict to the case of interest for us, and first record:

\begin{lemma} \label{lem:compactness}
Given a proper metric space $X$, a lower semi-continuous $f : X \rightarrow \Real$, $x \in X$ and $t \in (0,t_*(f))$, there exist $y^{\pm}_t \in X$ so that:
\[
Q_t f(x) = \frac{\sfd(x,y^{\pm}_t)^2}{2 t} + f(y^{\pm}_t) \; \text{ and } \;\; \sfd(x,y^{\pm}_t) = D^{\pm}_f(x,t) .
\]
\end{lemma}
\noindent Recall that $-\varphi$ is indeed lower semi-continuous for any Kantorovich potential $\varphi$.
\begin{proof}[Proof of Lemma \ref{lem:compactness}]
Let $\{y_t^{\pm,n}\}$ denote a minimizing sequence so that:
\[
Q_t f(x) = \lim_{n \rightarrow \infty}  \frac{\sfd(x,y_t^{\pm,n})^2}{2 t} + f(y_t^{\pm,n}) \text{ and }  D^{\pm}_f(x,t) = \lim_{n \rightarrow \infty} \sfd(x,y_t^{\pm,n}) .
\]
By property (1) we know that $D^{\pm}_f(x,t) < R < \infty$, and the properness implies that the closed geodesic ball $B_R(x)$ is compact. Consequently $\{y_t^{\pm,n}\}$ has a converging subsequence to $y^{\pm}_t$, and the lower semi-continuity of $f$ implies that:
\[
Q_t f(x) = \inf_{y \in X}  \frac{\sfd(x,y)^2}{2 t} + f(y)  = \min_{y \in B_R(x)} \frac{\sfd(x,y)^2}{2 t} + f(y)  = \frac{\sfd(x,y^{\pm}_t)^2}{2 t} + f(y^{\pm}_t) ,  
\]
as asserted. 
\end{proof}

\begin{proof}[Proof of (3) for proper $X$ and lower semi-continuous $f$]
The assertion will follow immediately after establishing:
\[
D^+_f(x,s) \leq D^-_f(x,t) \;\;\; \forall 0 < s < t < t_*(f) ,
\]
since trivially $D^-_f \leq D^+_f$ and since a monotone function can only have a countable number of jump discontinuities. 
By Lemma \ref{lem:compactness}, there exist $y^+_s$ and $y^-_t$ so that:
\[
Q_s f(x) = \inf_{y \in X}  \frac{\sfd(x,y)^2}{2 s} + f(y) = \frac{\sfd(x,y^+_s)^2}{2 s} + f(y^+_s) \text{ and } \sfd(x,y^+_s) = D^+_f(x,s) ,
\]
and:
\[
Q_t f(x) = \inf_{y \in X}  \frac{\sfd(x,y)^2}{2 t} + f(y) = \frac{\sfd(x,y^{-}_t)^2}{2 t} + f(y^{-}_t) \text{ and } \sfd(x,y^{-}_t) = D^-_f(x,t) .
\]
It follows that:
\begin{align*}
\frac{\sfd(x,y^+_s)^2}{2 s} + f(y^+_s) \leq \frac{\sfd(x,y^-_t)^2}{2 s} + f(y^-_t) , \\
\frac{\sfd(x,y^-_t)^2}{2 t} + f(y^-_t) \leq \frac{\sfd(x,y^+_s)^2}{2 t} + f(y^+_s) .
\end{align*}
Summing these two inequalities and rearranging terms, one deduces:
\[
D^+_f(x,s) \brac{\frac{1}{s} - \frac{1}{t}} \leq D^-_f(x,t) \brac{\frac{1}{s} - \frac{1}{t}} ,
\]
as required. 
\end{proof}

\bigskip

\subsection{Intermediate-time duality and time-reversed potential} \label{subsec:duality}

It is immediate to show by inspecting the definitions that we always have (e.g. \cite[Theorem 7.34 (iii)]{Vil} or \cite[Proposition 2.17 (ii)]{ambro:userguide}):
\[
Q_{-s} \circ Q_s f \leq f  \text{ on $X$} \;\;\; \forall s > 0 ; 
\]
this is an inherent group-structure incompatibility of the Hopf-Lax forward and backward semi-groups. Note that for $f = -\varphi$ where $\varphi$ is a Kantorovich potential, we do have equality for $s=1$, and in fact for all $s \in [0,1]$. 
However, for $f = Q_t(-\varphi)$, $t \in (0,1)$ and $s=1-t$, we can only assert an \emph{inequality} above
(\cite[Theorem 7.36]{Vil},\cite[Corollary 2.23 (i)]{ambro:userguide}):
\begin{equation} \label{eq:no-duality}
(\varphi^c)_{1-t} = Q_{-(1-t)} \circ Q_1 (-\varphi)   \leq  Q_t(-\varphi) = -\varphi_t \text{ on $X$,}
\end{equation}
and equality may not hold at every point of $X$ (cf. \cite[Remark 7.37]{Vil}).
Nevertheless, in our setting, the subset where equality is attained may be characterized as in the next proposition. We first introduce the following very convenient:

\begin{definition*}[Time-Reversed Interpolating Potential]
Given a Kantorovich potential $\varphi: X \rightarrow \Real$, define the time-reversed interpolating Kantorovich potential at time $t \in [0,1]$,
 $\varphic_t : X \rightarrow \Real$, as:
\[
\varphic_t := -(\varphi^c)_{1-t} =  Q_{1-t}(-\varphi^c)  = - Q_{-(1-t)} \circ Q_{1-t}(-\varphi_t) .
\]
Note that $\varphic_0 = \varphi$, $\varphic_1 = -\varphi^c$, and:
\[
\varphic_t(x) = \inf_{y \in X} \frac{\sfd^2(x,y)}{2(1-t)} - \varphi^c(y) \;\;\;\; \forall t \in [0,1) .
\]
\end{definition*}

\begin{proposition} \label{prop:duality-char}
\hfill
\begin{enumerate}
\item $\varphi_0 = \varphic_0 = \varphi$ and $\varphi_1 = \varphic_1 = -\varphi^c$.
\item For all $t \in [0,1]$, $\varphi_t \leq \varphic_t$. 
\item For any $t \in (0,1)$, $\varphi_t(x) = \varphic_t(x)$ if and only if $x \in \ee_t(G_\varphi)$. In other words:
\begin{equation} \label{eq:D-mathring-G}
D(\mathring{G}_\varphi) = \set{(x,t) \in X \times (0,1) \; ; \; \varphi_t(x) = \varphic_t(x) }.
\end{equation}
\end{enumerate}
\end{proposition}
$(1)$ is immediate by $c$-concavity, and $(2)$ is a reformulation of (\ref{eq:no-duality}), so the only assertion requiring proof is $(3)$. 
The if direction is well-known (e.g. \cite[Theorem 7.36]{Vil}, \cite[Corollary 2.23 (ii)]{ambro:userguide}), but the other direction appears to be new. It is based on the following simple lemma, which we will use again later on:
\begin{lemma} \label{lem:rigidity}
Assume that for some $x,y,z \in X$ and $t \in (0,1)$:
\[
\frac{\sfd(x,y)^2}{2 t} - \varphi(y) = \varphi^c(z) - \frac{\sfd(x,z)^2}{2 (1-t)} .
\]
Then $x$ is a $t$-intermediate point between $y$ and $z$:
\begin{equation} \label{eq:t-mid-point}
\sfd(y,z) = \frac{\sfd(x,y)}{t} = \frac{\sfd(x,z)}{1-t} ,
\end{equation}
and there exists a $\varphi$-Kantorovich geodesic $\gamma : [0,1] \rightarrow X$ with $\gamma(0) = y$, $\gamma(t) =x$ and $\gamma(1) = z$. 
\end{lemma}
\begin{proof}
Using that:
\begin{equation} \label{eq:varphi-rigidity}
\varphi(y) + \varphi^c(z) \leq \frac{\sfd(y,z)^2}{2} ,
\end{equation}
our assumption yields:
\[
\frac{\sfd(x,y)^2}{2 t}  + \frac{\sfd(x,z)^2}{2 (1-t)} \leq  \frac{\sfd(y,z)^2}{2} .
\]
On the other hand, the reverse inequality is always valid by the triangle and Cauchy--Schwarz inequalities:
\[
\frac{\sfd(y,z)^2}{2} \leq \frac{(\sfd(x,y) + \sfd(x,z))^2}{2} \leq \frac{\sfd(x,y)^2}{2 t}  + \frac{\sfd(x,z)^2}{2 (1-t)}  .
\]
It follows that we must have equality everywhere above, and (\ref{eq:t-mid-point}) amounts to the equality case in the Cauchy--Schwarz inequality. 
Consequently, the concatenation $\gamma : [0,1] \rightarrow X$ of any constant speed geodesic $\gamma_1 : [0,t] \rightarrow X$ between $y$ and $x$, with any constant speed geodesic $\gamma_2 : [t,1] \rightarrow X$ between $x$ and $z$, so that $\gamma(0) = y$, $\gamma(t) = x$ and $\gamma(1) = z$, must be a constant speed geodesic itself (by the triangle inequality). Lastly, the equality in (\ref{eq:varphi-rigidity}) implies that $\gamma \in G_\varphi$, thereby concluding the proof.
\end{proof}

\begin{proof}[Proof of Proposition \ref{prop:duality-char} (3)]
We begin with the known direction. Let $x = \gamma_t$ with $\gamma \in G_\varphi$. Apply Lemma \ref{lem:varphi-linear} to $\gamma$ with $s=0$ and $r=t$:
\[
\varphi(\gamma_0) - \varphi_t(\gamma_t)  = \varphi_0(\gamma_0) - \varphi_t(\gamma_t) = t \frac{\text{len}(\gamma)^2}{2} ,
\]
and to $\gamma^c \in G_{\varphi^c}$ with $s=1$ and $r=1-t$:
\[
-\varphi(\gamma_0) - (\varphi^c)_{1-t}(\gamma_t) =  (\varphi^c)_1(\gamma^c_1) - (\varphi^c)_{1-t}(\gamma^c_{1-t}) = -t \frac{\text{len}(\gamma^c)^2}{2} = -t \frac{\text{len}(\gamma)^2}{2},
\]
where we used that $(\varphi^c)_1 = -(\varphi^c)^c = -\varphi$. Summing these two identities, we obtain:
\[
\varphi_t(\gamma_t) = -  (\varphi^c)_{1-t}(\gamma_t) ,
\]
as asserted. 

For the other direction, assume that $\varphi_t(x) = - (\varphi^c)_{1-t}(x)$ for some $x \in X$ and $t \in (0,1)$. By Lemma \ref{lem:compactness} applied to the lower semi-continuous functions $-\varphi$ and $-\varphi^c$, there exist $y_t,z_t \in X$ so that:
\begin{align*}
-\varphi_t(x) & = Q_t(-\varphi)(x) = \frac{\sfd(x,y_t)^2}{2 t} - \varphi(y_t) , \\
\varphi_t(x) = -(\varphi^c)_{1-t}(x) & = Q_{1-t}(-\varphi^c)(x) = \frac{\sfd(x,z_t)^2}{2 (1-t)} -\varphi^c(z_t) .
\end{align*}
Summing the two equations, the assertion follows immediately from Lemma \ref{lem:rigidity}. 
\end{proof}

We also record the following immediate corollary of Lemma \ref{lem:order0}:
\begin{corollary} \label{cor:varphic-order0}
\hfill
\begin{enumerate}
\item $(x,t) \mapsto \varphic_t(x)$ is upper semi-continuous on $X \times [0,1)$ and continuous on $X \times (0,1)$. 
\item For every $x \in X$, $[0,1] \ni t \mapsto \varphic_t(x)$ is monotone non-decreasing and continuous on $[0,1)$. 
\end{enumerate}
\end{corollary}

Finally, in view of (\ref{eq:D-mathring-G}), we deduce for free:
\begin{corollary} \label{cor:G-closed} 
$D(\mathring{G}_\varphi)$ is a closed subset of $X \times (0,1)$. 
\end{corollary}
\begin{proof}
Immediate from (\ref{eq:D-mathring-G}) by the continuity of $\varphi_t(x)$ and $\varphic_t(x)$ on $X \times (0,1)$. 
\end{proof}

\bigskip

\subsection{Length functions $\ell_t^{\pm}$ and $\ellc_t^{\pm}$}

\begin{definition*}[Length functions $\ell_t^{\pm},\ellc_t^{\pm}$]
Given a Kantorovich potential $\varphi: X \rightarrow \Real$, denote:
\[
\ell^{\pm}_t(x) := \frac{D^{\pm}_{-\varphi}(x,t)}{t} \;\; , \;\;  \ellc^{\pm}_t(x) :=  \frac{D^{\pm}_{-\varphi^c}(x,1-t)}{1-t} \;\;, \;\; (x,t) \in X \times (0,1). 
\]
\end{definition*}
To provide motivation for these definitions, let us mention that we will shortly see that if $x = \gamma_t$ with $\gamma \in G_\varphi$ and $t \in (0,1)$, then:
\[
\ell^{+}_t(x) = \ell^{-}_t(x) = \ellc^{+}_t(x) = \ellc^{-}_t(x) = \len(\gamma) .
\]
In particular, all $\varphi$-Kantorovich geodesics having $x$ as their $t$-mid-point have the same length. These facts seem to not have been previously noted in the literature, 
and they will be crucially exploited in this work.

\begin{definition*}
For $\tilde{\ell} = \ell,\ellc$, introduce the following  set:
\[
D_{\tilde{\ell}} := \set{ (x,t) \in X \times (0,1) \; ; \; \tilde{\ell}_t^+(x) = \tilde{\ell}_t^-(x) } ,
\]
and on it define $\tilde{\ell}_t(x)$ as the common value $\tilde{\ell}_t^+(x) = \tilde{\ell}_t^-(x)$. 
\end{definition*}

Recalling that $\varphi_t = -Q_t(-\varphi)$ and $\varphic_t = Q_{1-t}(-\varphi^c)$, we begin by translating Theorem \ref{thm:AGS} into the following corollary. We freely use standard properties of semi-convex (semi-concave) functions, like twice a.e. differentiability, non-negativity (non-positivity) of the singular part of the distributional second derivative (see e.g. Lemma \ref{lem:Aleksandrov}), etc... \begin{corollary} \label{cor:AGS}
Let $\varphi : X \rightarrow \Real$ denote a Kantorovich potential. Then:
\begin{enumerate}

\item For $\tilde{\ell}= \ell,\ellc$ and $\tilde{\varphi}=\varphi,\varphic$, $\tilde{\ell}^{\pm}_t(x)$ are locally finite on $X \times (0,1)$, and $(x,t) \mapsto \tilde{\varphi}_t(x)$ is locally Lipschitz there.
\item For $\tilde{\ell} = \ell,\ellc$, $(x,t) \mapsto \tilde{\ell}^{\pm}_t(x)$ is upper ($\tilde{\ell}^{+}_t(x)$) / lower ($\tilde{\ell}^{-}_t(x)$) semi-continuous on $X \times (0,1)$. In particular, the subset $D_{\tilde{\ell}} \subset X \times (0,1)$ is Borel and $(x,t) \mapsto \tilde{\ell}_t(x)$ is continuous on $D_{\tilde{\ell}}$. 
\item For every $x \in X$ we have:
\[
\partial_t^{\pm} \varphi_t(x) = \frac{\ell^{\pm}_t(x)^2}{2} ~,~  \partial_t^{\pm} \varphic_t(x) = \frac{\ellc^{\mp}_t(x)^2}{2}\;\;\; \forall t \in (0,1) . 
\]
In particular, for $\tilde{\ell}= \ell,\ellc$ and $\tilde{\varphi}=\varphi,\varphic$, respectively, the map $(0,1) \ni t \mapsto \tilde{\varphi}_t(x)$ is locally Lipschitz, and it is differentiable at $t \in (0,1)$ iff $t \in D_{\tilde{\ell}}(x)$, the set on which both maps $(0,1) \ni t \mapsto \tilde{\ell}^{\pm}_t(x)$ coincide. $D_{\tilde{\ell}}(x)$ is precisely the set of continuity points of both maps, and thus coincides with $(0,1)$ with at most countably exceptions. In particular:
\[
\tilde{\varphi}_{t_2}(x) - \tilde{\varphi}_{t_1}(x) = \int_{t_1}^{t_2} \frac{\tilde{\ell}^2_\tau(x)}{2} d\tau \;\;\;  \forall t_1,t_2 \in (0,1) .
\] 
\item For every $x \in X$: \begin{enumerate}
\item
Both maps $(0,1) \ni t \mapsto t \ell^{\pm}_t(x)$ are monotone non-decreasing. In particular, $D_{\ell}(x) \ni t \mapsto \ell^2_t(x)$ is differentiable a.e., 
the singular part of its distributional derivative is non-negative, 
$(0,1) \ni t\mapsto \varphi_t(x)$ is locally semi-convex, and:
\begin{align}
\label{eq:4a} \underline{\partial}_t \frac{\ell_t^2(x)}{2}  & \geq -\frac{1}{t} \ell_t^2(x) \;\;\; \forall t \in D_{\ell}(x) . \end{align}
\item
Both maps $(0,1) \ni t \mapsto (1-t) \ellc^{\pm}_t(x)$ are monotone non-increasing. In particular, $D_{\ellc}(x) \ni t \mapsto \ellc^2_t(x)$ is differentiable a.e., 
the singular part of its distributional derivative is non-positive,
$(0,1) \ni t \mapsto \varphic_t(x)$ is locally semi-concave, and: \begin{align}
\label{eq:4b} \overline{\partial}_t \frac{\ellc_t^2(x)}{2}  & \leq \frac{1}{1-t} \ellc_t^2(x) \;\;\; \forall t \in D_{\ellc}(x) .  \end{align}
\end{enumerate}
\end{enumerate}
\end{corollary}
\begin{proof}
The only point requiring verification is that monotonicity of $t \mapsto t \ell_t(x)$ in (4a) and $t \mapsto (1-t) \ellc_t$ in (4b) implies (\ref{eq:4a}) and (\ref{eq:4b}), respectively. For instance, using the continuity of $t \mapsto \ell_t(x)$ on $D_{\ell}(x)$, (\ref{eq:4a}) is clearly equivalent to:
\begin{equation} \label{eq:4aa}
\underline{\partial}_t \ell_t(x) \geq -\frac{1}{t} \ell_t(x) \;\;\; \forall t \in D_{\ell}(x) .
\end{equation}
Now, if $\ell_t(x) = 0$ the monotonicity directly implies $\underline{\partial}_t \ell_t(x) \geq 0$ and establishes (\ref{eq:4aa}), whereas otherwise, (\ref{eq:4aa}) is equivalent by the chain-rule (and again the continuity of $t \mapsto \ell_t(x)$ on $D_{\ell}(x)$) to:
\[
\underline{\partial}_t \log(t \ell_t(x))  =  \frac{1}{t} + \underline{\partial}_t \log(\ell_t(x))  \geq 0 \;\;\; \forall t \in D_{\ell}(x) ,
\]
which in turn is a consequence of the aforementioned monotonicity. The proof of (\ref{eq:4b}) follows identically. 
\end{proof}

We now arrive to the main new result of this section, which will be constantly and crucially used in this work: 
\begin{theorem} \label{thm:order12-main}
Let $\varphi : X \rightarrow \Real$ denote a Kantorovich potential.  \begin{enumerate}
\item For all $x \in \ee_t(G_\varphi)$ with $t \in (0,1)$, we have:
\[
\ell^{+}_t(x) = \ell^{-}_t(x) = \ellc^{+}_t(x) = \ellc^{-}_t(x) = \len(\gamma) ,
\]
for any  $\gamma \in G_\varphi$ so that $\gamma_t = x$. 
In other words:
\[
 D(\mathring{G}_\varphi) = \set{(x,t) \in X \times (0,1) \; ; \; x = \gamma_t \; , \; \gamma \in G_\varphi} \subset D_\ell \cap D_{\ellc},
\]
 and moreover $\ell_t(x) = \ellc_t(x)$ there. 
 \item For all $x \in X$, $\mathring{G}_\varphi(x) \ni t \mapsto \ell_t(x)=\ellc_t(x)$ is locally Lipschitz:
\begin{multline} \label{eq:ell-Lipschitz}
\abs{ \sqrt{t (1-t)} \ell_{t}(x) - \sqrt{s (1-s)} \ell_{s}(x)} \\
\leq \sqrt{\ell_{t}(x)  \ell_{s}(x)} \abs{ \sqrt{t (1-s)} - \sqrt{s (1-t)} } \;\;\; \forall t,s \in \mathring{G}_\varphi(x)  .
 \end{multline}
\item For all $(x,t) \in D(\mathring{G}_\varphi) \subset D_\ell \cap D_{\ellc}$ we have for both $* = \underline{\P}_2 \varphic_t(x) ,\overline{\P}_2 \varphi_t(x)$: 
\[
 -\frac{1}{t} \ell_t^2(x) \leq \underline{\partial}_t \frac{\ell_t^2(x)}{2} \leq \underline{\P}_2 \varphi_t(x) \leq * \leq \overline{\P}_2 \varphic_t(x) \leq \overline{\partial}_t \frac{\ellc_t^2(x)}{2}  \leq \frac{1}{1-t} \ell_t^2(x)  ,
\]
where the Peano (partial) derivatives are with respect to the $t$ variable. 
\item For all $(x,t) \in D(\mathring{G}_\varphi) \subset D_\ell \cap D_{\ellc}$ we have: 
\begin{align*}
\overline{\partial}_t  \frac{\ell_t^2(x)}{2} &\leq  \overline{\partial}_t  \frac{\ellc_t^2(x)}{2} + \brac{\frac{1}{1-t} + \frac{1}{t}} \ell_t^2(x) \leq \brac{\frac{2}{1-t} + \frac{1}{t}} \ell_t^2(x) \\
\underline{\partial_t}  \frac{\ellc_t^2(x)}{2} &\geq \underline{\partial_t}  \frac{\ell_t^2(x)}{2} -\brac{\frac{1}{t} + \frac{1}{1-t}} \ell_t^2(x) \geq -\brac{\frac{2}{t} + \frac{1}{1-t}} \ell_t^2(x) .
\end{align*}
\end{enumerate}
In particular, for every $x \in X$, we have:
\[
\partial_t \varphi_t(x) = \partial_t \varphic_t(x) = \frac{\ell^2_t(x)}{2} = \frac{\ellc^2_t(x)}{2} \;\;\; \forall t \in \mathring{G}_\varphi(x)  ,
\]
with $t \mapsto \frac{\ell^2_t(x)}{2}$ and $t \mapsto \frac{\ellc^2_t(x)}{2}$ continuous on $D_\ell(x) \cap D_{\ellc}(x)$, differentiable a.e. there, and having locally bounded lower and upper derivatives on $\mathring{G}_{\varphi}(x) \subset D_\ell(x) \cap D_{\ellc}(x)$ as in (3) and (4). 
\end{theorem}

\begin{proof}To see (1), let $(x,t) \in D(\mathring{G}_\varphi)$. Equivalently, by Proposition \ref{prop:duality-char} (3), we know that $\varphi_t(x) = \varphic_t(x)$. In addition, Lemma \ref{lem:compactness} assures the existence of $y^{\pm}$ and $z^{\pm}$ in $X$ so that:
\begin{align*}
-\varphi_t(x) & = \frac{\sfd(x,y^{\pm})^2}{2t} - \varphi(y^{\pm})  \;\;, \;\; \sfd(x,y^{\pm}) = t \ell^{\pm}_t(x) \\
-\varphic_t(x) &= - \frac{\sfd(x,z^{\pm})^2}{2(1-t)} + \varphi^c(z^{\pm}) \;\;,\;\; \sfd(x,z^{\pm}) = (1-t) \ellc^{\pm}_t(x) .
\end{align*}
Equating both expressions and applying Lemma \ref{lem:rigidity}, we deduce that $x$ is the $t$-midpoint of a geodesic connecting $y^{\pm}$ and $z^{\pm}$ (for all 4 possibilities), and that:
\begin{equation} \label{eq:pmc-thesame}
\ell^{\pm}_t(x) = \frac{\sfd(x,y^{\pm})}{t} = \frac{\sfd(x,z^{\pm})}{1-t} = \ellc^{\pm}_{t}(x) ,
\end{equation}
so that all 4 possibilities above coincide. We remark in passing that this already implies in a non-branching setting that necessarily $y^+ = y^-$ and $z^+=z^-$, i.e. the uniqueness of a $\varphi$-Kantorovich geodesic with $t$-mid point $x$. 

Furthermore, if $x = \gamma_t$ for some $\gamma \in G_\varphi$, then by Lemma \ref{lem:varphi-linear}:
\[
-\varphi_t(x) = \frac{\sfd(x , \gamma_0)^2}{2t} -  \varphi(\gamma_0) .
\]
It follows by definition of $D^{\pm}_{-\varphi}(x,t)$ that:
\[
t \ell^-_t(x) = D^-_{-\varphi}(x,t) \leq \sfd(x,\gamma_0) = t \len(\gamma) \leq D^+_{-\varphi}(x,t) = t \ell^+_t(x) ,
\]
which together with (\ref{eq:pmc-thesame}) establishes that $\len(\gamma) = \ell_t(x) = \ellc_t(x)$. 

\smallskip
To see (2), let $\gamma^t, \gamma^s \in G_{\varphi}$ be so that $\gamma^t_t = \gamma^s_s = x$, for some $t,s \in (0,1)$. Then:
\[
\varphi^c(\gamma^p_1) = \frac{\len(\gamma^p)^2}{2} - \varphi(\gamma^p_0) \leq \frac{\sfd(\gamma^p_1,\gamma^q_0)^2}{2} - \varphi(\gamma^q_0) ,
\]
for $(p,q) = (t,s)$ and $(p,q) = (s,t)$. 
Summing these two inequalities, we obtain the well-known $c$-cyclic monotonicity of the set $\set{(\gamma^t_0,\gamma^t_1),(\gamma^s_0,\gamma^s_1)}$:
\[
\len(\gamma^t)^2 + \len(\gamma^s)^2 \leq \sfd(\gamma^t_0,\gamma^s_1)^2 + \sfd(\gamma^s_0,\gamma^t_1)^2 .
\]
To evaluate the right-hand-side, we simply pass through $x$ and employ the triangle inequality:
\[
\sfd(\gamma^p_0,\gamma^q_1) \leq \sfd(\gamma^p_0,x) + \sfd(x,\gamma^q_1) = p \; \len(\gamma^p) + (1-q) \; \len(\gamma^q) . 
\]
Plugging this above and rearranging terms, we obtain:
\[
t (1-t) \len(\gamma^t)^2 + s (1-s) \len(\gamma^s)^2 \leq \brac{t (1-s) + s (1-t)} \len(\gamma^t) \len(\gamma^s) .
\]
Completing the square by subtracting $2 \sqrt{t(1-t)s(1-s)} \len(\gamma^t) \len(\gamma^s)$ from both sides, and recalling that $\len(\gamma^p) = \ell_p(x)$ for $p=t,s$, we readily obtain  (\ref{eq:ell-Lipschitz}). In particular, using $t=s$, the above argument recovers the last assertion of (1) that $\len(\gamma)$ is the same for all $\gamma \in G_\varphi$ so that $\gamma_t = x$. 

\smallskip
To see (3), recall that given $x \in X$, we know by Proposition \ref{prop:duality-char} that $\varphi_t(x) \leq \varphic_t(x)$ for all $t \in (0,1)$ with equality iff $t \in \mathring{G}_\varphi(x)$. Since $\mathring{G}_\varphi(x) \subset D_\ell(x) \cap D_{\ellc}(x)$ by (1), we know that both maps $t \mapsto \tilde{\varphi}_t(x)$ are differentiable at $t_0 \in \mathring{G}_\varphi(x)$, and we see again that $\frac{\ell^2_{t_0}(x)}{2} = \partial_t \varphi_{t_0}(x) = \partial_t \varphic_{t_0}(x) = \frac{\ellc^2_{t_0}(x)}{2}$, since the derivatives of a function and its majorant must coincide at a mutual point of differentiability where they touch. Moreover, defining $\tilde{h} = h,\bar{h}$ as:
\[
 \tilde{h}(\eps) := 2\brac{\tilde{\varphi}_{t_0+\eps}(x) - \tilde{\varphi}_{t_0}(x) - \eps \partial_t \tilde{\varphi}_{t_0}(x)} ,
\]
it follows that $h \leq \bar{h}$ (on $(-t_0,1-t_0)$). Diving by $\eps^2$ and taking appropriate subsequential limits, we obviously obtain:
\[
\underline{\P}_2 \varphi_t(x)  \leq \underline{\P}_2 \varphic_t(x) ~,~ \overline{\P}_2 \varphi_t(x)  \leq \overline{\P}_2 \varphic_t(x) .
\]
Combining these inequalities with those of Lemma \ref{lem:peano-inq}, (\ref{eq:4a}) and (\ref{eq:4b}), the chain of inequalities in (3) readily follows. 

\smallskip
To see (4), let $t_0 \in \mathring{G}_\varphi(x)$. Consider the function $f(t) := \varphic_t(x) - \varphi_t(x)$ on $(0,1)$, which is locally semi-concave by Corollary \ref{cor:AGS}. By Proposition \ref{prop:duality-char}, we know that $f \geq 0$ with $f(t_0) = 0$. The function $f$ is differentiable on $D_{\ell}(x) \cap D_{\ellc}(x)$ and satisfies $f'(t) = \frac{\ellc^2_{t}(x)}{2} - \frac{\ell^2_{t}(x)}{2}$ there. In particular, this holds at $t_0 \in \mathring{G}_\varphi(x) \subset D_{\ell}(x) \cap D_{\ellc}(x)$ by (1) and $f'(t_0) = 0$. Note that by Corollary \ref{cor:AGS}:
\[
\overline{\partial}_t f'(t) \leq \overline{\partial}_t \frac{\ellc^2_{t}(x)}{2} - \underline{\partial}_t \frac{\ell^2_{t}(x)}{2} \leq \frac{1}{1-t} \ellc_t^2(x) + \frac{1}{t} \ell_t^2(x) .
\]
In particular, since both $D_{\tilde{\ell}}(x) \ni t \mapsto \tilde{\ell}_t(x)$ are continuous at $t = t_0 \in D_{\ell}(x) \cap D_{\ellc}(x)$, for $\tilde{\ell} = \ell,\ellc$, it follows that:
\[
\forall \eps > 0 \;\; \exists \delta > 0 \;\; \forall t \in (t_0-\delta,t_0 + \delta) \cap D_{\ell}(x) \cap D_{\ellc}(x) \;\;\; \overline{\partial}_t f'(t) \leq \frac{1}{1-t_0} \ell_{t_0}^2(x) + \frac{1}{t_0} \ell_{t_0}^2(x) + \eps .
\]
It follows that on the open interval $I_\delta := (t_0-\delta,t_0 + \delta) \cap (0,1)$, $f - C_\eps \frac{t^2}{2}$ is concave with $C_\eps$ defined as the constant on the right-hand-side above. Applying Lemma \ref{lem:positive-semi-concave} below to the translated function $f(\cdot + t_0)$ on the interval $I_\delta - t_0$, it follows that:
\[
\frac{1}{t-t_0} \brac{\frac{\ellc^2_{t}(x)}{2} - \frac{\ell^2_{t}(x)}{2}} = \frac{f'(t) - f'(0)}{t-t_0} \geq -C_\eps \;\;\; \forall t \in (t_0 - \frac{\delta}{2},t_0 + \frac{\delta}{2}) \cap D_{\ell}(x) \cap D_{\ellc}(x) .
\]
As $\ellc_{t_0}(x) = \ell_{t_0}(x)$ by (1), we obtain:
\[
\frac{\frac{\ell^2_{t}(x)}{2} - \frac{\ell^2_{t_0}(x)}{2}}{t-t_0} \leq \frac{\frac{\ellc^2_{t}(x)}{2} - \frac{\ellc^2_{t_0}(x)}{2}}{t-t_0} + C_{\eps} \;\;\; \forall t \in (t_0 - \frac{\delta}{2},t_0 + \frac{\delta}{2}) \cap D_{\ell}(x) \cap D_{\ellc}(x) .
\]
The assertion of (4) now follows by taking appropriate subsequential limits as $t \rightarrow t_0$ and using the fact that $\eps > 0$ was arbitrary. 
\end{proof}

\begin{lemma} \label{lem:positive-semi-concave}
Given $I \subset \Real$ an open interval containing $0$, let $f : I \rightarrow \Real$ denote a $C$-semi-concave function, so that $I \ni t \mapsto f - C \frac{t^2}{2}$ is concave, $C \geq 0$. Assume that $f \geq 0$ on $I$, that $f$ is differentiable at $0$ and that $f(0) = f'(0) = 0$. Then $\underline{\partial_t}|_{t=0} f'(t) \geq -C$, and moreover, $\frac{f'(t)}{t} \geq -C$ for all $t \in D \cap I/2$, where $D \subset I$ denotes the subset (of full measure) of differentiability points of $f$.
\end{lemma} 
Note that the $C$-semi-concavity is equivalent to $\overline{\partial}_t|_{t=0} f'(t) \leq C$, while the conclusion is from the opposite direction. It is not hard to verify that the asserted lower bound is in fact best possible. 
\begin{proof}[Proof of Lemma \ref{lem:positive-semi-concave}]
Set $g = f'$ on $D$. The $C$-semi-concavity is equivalent to the statement that $g(t) - C t$ is non-increasing on $D$, so that $g(t_2) \leq g(t_1) + C (t_2 - t_1)$ for all $t_1,t_2 \in D$ with $t_1 < t_2$. It follows that necessarily $g(t) \geq -C t$ for all $t \in D \cap I/2$ with $t \geq 0$, since:  
\[
0 \leq f(2t) - f(0) = \int_0^{2t} g(s) ds \leq \int_0^t (g(0) + C s) + \int_t^{2t} (g(t) + C(s-t)) ds = C \frac{t^2}{2} + t g(t) + C \frac{t^2}{2} .
\]
Repeating the same argument for $t \mapsto f(-t)$, we see that $-g(t) \geq C t$ for all $t \in D \cap I/2$ with $t \leq 0$. This concludes the proof. 
\end{proof}

\bigskip

\medskip

In a sense, Theorem \ref{thm:order12-main} (2) is the temporal analogue of the spatial $1/2$-H\"{o}lder regularity proved by Villani in \cite[Theorem 8.22]{Vil}. 
Formally taking $s \rightarrow t$ in (\ref{eq:ell-Lipschitz}), it is easy to check that one obtains (for both $\tilde{\ell} = \ell,\ellc$) stronger bounds than in Theorem \ref{thm:order12-main} (3) and (4):
\begin{equation} \label{eq:formal-3}
-\frac{1}{t} \ell^2_t(x) \leq \underline{\partial_t} \frac{\tilde{\ell}^2_t(x)|_{\mathring{G}_\varphi(x)}}{2} \leq \overline{\partial_t} \frac{\tilde{\ell}^2_t(x)|_{\mathring{G}_\varphi(x)}}{2} \leq \frac{1}{1-t} \ell^2_t(x) \;\;\; \forall t \in \mathring{G}_\varphi(x)  .  
\end{equation}
However, we do not know how to rigorously pass from (\ref{eq:ell-Lipschitz}) to (\ref{eq:formal-3}) or vice versa (by differentiation or integration, respectively), 
since we cannot exclude the possibility that the (relatively closed in $(0,1)$) set $\mathring{G}_\varphi(x)$ has isolated points, nor that it is disconnected. 
Instead, we can obtain the following stronger version of (\ref{eq:formal-3}) which only holds for a.e. $t \in \mathring{G}_\varphi(x)$, but will prove to be very useful later on.

\begin{corollary} \label{cor:accum1}
For all $x \in X$, for a.e. $t \in \mathring G_\varphi(x)$, $\partial_t \ell^2_t(x)$ and $\partial_t \ellc^2_t(x)$ exist, coincide, and satisfy:
\begin{equation} \label{eq:accumulation1}
-\frac{1}{t} \ell^2_t(x) \leq \partial_t \frac{\ell^2_t(x)}{2} = \partial_t \frac{\ell^2_t(x)|_{\mathring{G}_\varphi(x)}}{2} = \partial_t \frac{\ellc^2_t(x)|_{\mathring{G}_\varphi(x)}}{2} 
= \partial_t \frac{\ellc^2_t(x)}{2}  \leq \frac{1}{1-t} \ell^2_t(x)  .
\end{equation}
\end{corollary}
\begin{proof}
By Corollary \ref{cor:AGS}, for all $x \in X$ and $\tilde{\ell} = \ell,\ellc$, $t \mapsto \tilde{\ell}^2_t(x)$ is differentiable a.e. on $D_{\tilde{\ell}}(x)$. Consequently, the first and third equalities in (\ref{eq:accumulation1}) follow for a.e. $t \in \mathring G_\varphi(x) \subset D_{\ell}(x) \cap D_{\ellc}(x)$  by Remark \ref{R:diff-restriction}. The second equality follows since $\ell_t(x) = \ellc_t(x)$ for $t \in \mathring G_\varphi(x)$ by Theorem \ref{thm:order12-main}. 
The lower and upper bounds in (\ref{eq:accumulation1}) then follow from Theorem \ref{thm:order12-main} (3) (or as in (\ref{eq:formal-3}), by taking the limit as $s \rightarrow t$ in Theorem \ref{thm:order12-main} (2)). 
\end{proof}

\subsection{Null-Geodesics}

\begin{definition}[Null-Geodesics and Null-Geodesic Points]
Given a Kantorovich potential $\varphi : X \rightarrow \Real$, denote the subset of null $\varphi$-Kantorovich geodesics by:
\[
G_\varphi^0 := \set{ \gamma \in G_\varphi \; ; \; \len(\gamma) = 0 } .
\]
Its complement in $G_\varphi$ will be denoted by $G_\varphi^+$. The subset of $X$ of null $\varphi$-Kantorovich geodesic points is denoted by:
\[
X^0 := \set{ x \in X \; ; \; \exists \gamma \in G_\varphi^0 \;\; \gamma \equiv x} = \set{x \in X \; ; \; \varphi(x) + \varphi^c(x) = 0 } .
\]
Its complement in $X$ will be denoted by $X^+$. 
\end{definition}

The following provides a convenient equivalent characterization of $X^0$ and $X^+$:
\begin{lemma} \label{lem:X0}
Given $x \in X$, the following statements are equivalent:
\begin{enumerate}
\item $x \in X^0$, i.e. $\varphi(x) + \varphi^c(x) = 0$. 
\item $\forall t \in (0,1)$, $\varphi_t(x) = \varphic_t(x) = \varphi(x) = -\varphi^c(x)$.
\item $\forall t \in (0,1)$, $\varphi_t(x) = c$ and $\varphic_t(x) = \bar{c}$ for some $c,\bar{c} \in \Real$.  
\item $D_{\ell}(x) = D_{\ellc}(x) = (0,1)$ and $\;\forall t \in (0,1) \;\; \ell_t(x) = \ellc_t(x) = 0$. 
\item $\exists t_0 \in \mathring{G}_\varphi(x)$ so that $\varphi_{t_0}(x) = \varphi(x)$ or $\varphic_{t_0}(x) = \varphi(x)$ or $\varphi_{t_0}(x) = -\varphi^c(x)$ or $\varphic_{t_0}(x) = -\varphi^c(x)$. 
\item $\exists t_0 \in \mathring{G}_\varphi(x)$ so that $\ell_{t_0}^{-}(x) = 0$ or $\ell_{t_0}^{+}(x) = 0$ or $\ellc_{t_0}^{-}(x) = 0$ or $\ellc_{t_0}^+(x) = 0$. 
\end{enumerate}
In other words, we have the following dichotomy: \textbf{all} $\varphi$-Kantorovich geodesics having $x \in X$ as \textbf{some} interior mid-point have either strictly positive length (iff $x \in X^+$) or zero length (iff $x \in X^0$).  
\end{lemma}
\begin{remark} \label{rem:X0}
In fact, we always have $\varphi_t(x) = \varphic_t(x)$ and $\ell_t(x) = \ellc_t(x)$ for $t \in \mathring{G}_\varphi(x)  \subset D_{\ell}(x) \cap D_{\ellc}(x)$ by Theorem \ref{thm:order12-main}, so we may simply write ``$\varphi_{t_0}(x) = \varphi(x)$ or $\varphi_{t_0}(x)= -\varphi^c(x)$" and ``$\ell_{t_0}(x) = \ellc_{t_0}(x) = 0$" in statements (5) and (6), respectively. However, we chose to formulate these statements with the (a-priori) minimal requirements. 
\end{remark}
\begin{proof}[Proof of Lemma \ref{lem:X0}]
$(1) \Rightarrow (2)$ is straightforward: for instance, $(1)$ is by definition identical to  $\varphi_1(x) = \varphi_0(x)$ and (2) follows by the monotonicity of $[0,1] \ni t \mapsto \tilde{\varphi}_t(x)$ for both $\tilde{\varphi} = \varphi,\varphic$; alternatively, apply Lemma \ref{lem:varphi-linear} to the null geodesic $\gamma^0 \equiv x$ with respect to both Kantorovich potentials $\varphi$ and $\varphi^c$. \\
$(2) \Rightarrow (3)$ is trivial. \\
$(3) \Leftrightarrow (4)$ follows by using that $D_{\tilde{\ell}}(x)$ is characterized as the subset of $t$-differentiability points of $\varphi_t(x)$ on $(0,1)$ with $\partial_t \tilde{\varphi}_t(x) = \tilde{\ell}_t^2(x)/2$ there. \\
$(3) \Rightarrow (1)$:  by the continuity of $t \mapsto\varphi_t(x)$ from the left at $t=1$ it follows that $c = \varphi_1(x)$, and similarly the continuity of $t \mapsto\varphic_t(x)$ from the right at $t=0$ yields that $\bar{c} = \varphic_0(x) = \varphi(x)$. Since always $\varphi \leq \varphic$, we deduce $\varphi_1(x) = c \leq \bar{c} = \varphi(x)$. On the other hand, we always have $\varphi(x) \leq \varphi_1(x)$ by monotonicity, so we conclude that $\varphi(x) = \varphi_1(x)$, establishing statement (1). This concludes the proof of the equivalence $(1) \Leftrightarrow (2) \Leftrightarrow (3) \Leftrightarrow (4)$. \\
$(2) \Rightarrow (5)$ and $(4) \Rightarrow (6)$ are trivial. \\
$(5) \Rightarrow (6)$ is straightforward: for instance, if $\tilde{\varphi}_{t_0}(x) =  \tilde{\varphi}_0(x) = \varphi(x)$ for some $t_0 \in (0,1)$ and $\tilde{\varphi} \in \set{\varphi,\varphic}$, then by monotonicity, $\tilde{\varphi}_t(x) = \varphi(x)$ for all $t \in [0,t_0]$, and hence the left derivative at $t=t_0$ satisfies $\ell^{-}_{t_0}(x) = \partial_t^{-}|_{t=t_0} \varphi_t(x) = 0$ if $\tilde{\varphi} = \varphi$ and $\ellc^{+}_{t_0}(x) = \partial_t^{-}|_{t=t_0} \varphic_t(x) = 0$ if $\tilde{\varphi} = \varphic$. If $\tilde{\varphi}_{t_0}(x) = \tilde{\varphi}_1(x) = -\varphi^c(x)$, repeat the argument using the right derivative. \\
The only direction requiring second-order information on $\varphi_t$ is $(6) \Rightarrow (3)$. By Corollary \ref{cor:AGS}, $t \mapsto t \ell_{t}^{\pm}(x)$ and $t \mapsto (1-t) \ellc_t^{\pm}(x)$ are monotone non-decreasing and non-increasing on $(0,1)$, respectively. Since $t_0 \in \mathring{G}_\varphi$, in view of Remark \ref{rem:X0}, (5) is equivalent to $\ell^{\pm}_{t_0}(x) = \ellc^{\pm}_{t_0}(x) = 0$. The monotonicity implies that $\ell_t^{\pm}(x) = 0$ for all $t \in (0,t_0]$ and that $\ellc_t^{\pm}(x) = 0$ for all $t \in [t_0,1)$. It follows that $\varphi_t(x)$ is constant on $(0,t_0]$ and $\varphic_t(x)$ is constant on $[t_0,1)$. As $\varphi_{t_0}(x) = \varphic_{t_0}(x)$, the monotonicity of $t \mapsto \tilde{\varphi}_t(x)$ and the majoration $\varphi_t \leq \varphic_t$ forces both $t \mapsto \varphi_t(x)$ and $t \mapsto \varphic_t(x)$ to be constant on $(0,1)$, establishing $(3)$ (in fact with $c = \bar{c}$). 
\end{proof}

\begin{corollary} \label{cor:X+}
If $x \in X^+$ then $\ell_t(x) > 0$ for all $t \in [\inf \mathring{G}_\varphi(x), 1) \cap D_{\ell}(x)$ and  $\ellc_t(x) > 0$ for all $t \in (0,\sup \mathring{G}_\varphi(x)] \cap D_{\ellc}(x)$.
\end{corollary}
\begin{proof}
Immediate by (6) and the monotonicity of $D_{\ell}(x) \ni t \mapsto t \ell_t(x)$ and $D_{\ellc}(x) \ni t \mapsto (1-t) \ellc_t(x)$, together with the fact that $\mathring{G}_\varphi(x)$ is relatively closed in $(0,1)$ by Corollary \ref{cor:G-closed}. \end{proof}

\begin{corollary} \label{cor:X0}
Given $x \in X$, assume that $\exists t_1,t_2 \in \mathring{G}_\varphi(x)$ with $t_1 \neq t_2$. Then $x \in X^0$ iff $\varphi_{t_1}(x) = \varphi_{t_2}(x)$ (or equivalently, $\varphic_{t_1}(x) = \varphic_{t_2}(x)$). 
\end{corollary}
\begin{proof}
The ``only if" direction follows immediately by Lemma \ref{lem:X0}, whereas the ``if" direction follows by Corollary \ref{cor:X+}, after recalling that $\varphi_{t_2}(x) - \varphi_{t_1}(x) = \int_{t_1}^{t_2} \frac{\ell_\tau^2(x)}{2} d\tau$ by Corollary \ref{cor:AGS}. As usual, the equivalent condition follows by Theorem \ref{thm:order12-main}. 
\end{proof}

\bigskip

\section[Temporal Theory of Intermediate-Time Kantorovich Potentials.\\ Time-Propagation]{Temporal Theory of Intermediate-Time Kantorovich Potentials. Time-Propagation}
\label{S:Phi}
 
 The goal of this section is to introduce and study the following function(s):
 
 \begin{definition*}[Time-Propagated Intermediate Kantorovich Potentials]
 Given a Kantorovich potential $\varphi : X \rightarrow \Real$ and $s,t \in (0,1)$, define the $t$-propagated $s$-Kantorovich potential $\Phi_s^t$ on $D_{\ell}(t)$, and its time-reversed version $\Phic_s^t$ on $D_{\ellc}(t)$, by:
 \[
 \Phi_s^t := \varphi_t + (t-s) \frac{\ell_t^2}{2}  \text{ on $D_{\ell}(t)$} ~,~
 \Phic_s^t := \varphic_t + (t-s) \frac{\ellc_t^2}{2} \text{ on $D_{\ellc}(t)$} .
 \]
 \end{definition*}
 \noindent
 Observe that for all $s,t \in (0,1)$:
 \[
 \Phi_s^t = \Phic_s^t = \varphi_s \circ \ee_s \circ \ee_t^{-1}  \text{ on $\ee_t(G_\varphi)$} ;
 \]
 indeed, while $\ee_t^{-1} : \ee_t(G_\varphi) \rightarrow G_\varphi$ may be multi-valued, Theorem \ref{thm:order12-main} implies that $\len(\gamma) = \ell_t(x) = \ellc_t(x)$ for any $\gamma \in G_\varphi$ with $\gamma_t = x$, and consequently Lemma \ref{lem:varphi-linear} yields that $\varphi_s \circ \ee_s$ is single-valued for all such $\gamma$ and (also recalling Proposition \ref{prop:duality-char}):
  \[
 \Phi_s^t(\gamma_t) = \Phic_s^t(\gamma_t) = \varphi_s(\gamma_s) \;\;\; \forall \gamma \in G_{\varphi} .
 \]
 Consequently, on $\ee_t(G_\varphi)$, $\Phi_s^t=\Phic_s^t$ is identified as the push-forward of $\varphi_s$ via $\ee_t \circ \ee_s^{-1}$, i.e. its propagation along $G_\varphi$ from time $s$ to time $t$. 
 
\medskip

We will use the following short-hand notation. Given $s \in [0,1]$ and $a_s \in \Real$, we denote:
\[
G_{a_s} := \set{ \gamma \in G_\varphi \; ; \; \varphi_s(\gamma(s)) = a_s} ,
\]
suppressing the implicit dependence of $G_{a_s}$ on $s$. The above argument about why $\varphi_s \circ \ee_s \circ \ee_t^{-1}$ is well-defined can be rewritten as:

\begin{corollary}[Inter Level-Set Propagation] \label{cor:InterProp}
For all $s,t \in (0,1)$, $a_s, b_s \in \Real$, $a_s \neq b_s$, we have:
\[
\ee_t(G_\varphi) \cap \set{\Phi_s^t = a_s} \cap \set{\Phi_s^t = b_s} = \ee_{t}(G_{a_s}) \cap \ee_{t}(G_{b_s}) = \emptyset . 
\]
\end{corollary}
\noindent
Note that while typically disjoint sets remain disjoint under optimal-transport only under some additional non-branching assumptions, Corollary \ref{cor:InterProp} holds true in general.

\subsection{Monotonicity} 
 
\begin{lemma} \label{lem:propagation}
Let $x = \gamma^1_{t_1} = \gamma^2_{t_2}$ with $\gamma^1,\gamma^2 \in G_\varphi$ and $0 < t_1 < t_2 < 1$. Then for any $s \in (0,1)$:
 \begin{equation} \label{eq:propagation}
 \varphi_s(\gamma^2_s) - \varphi_s(\gamma^1_s) \geq  2 \min\brac{\frac{s}{t_2},\frac{1-s}{1-t_1}}(\varphi_{t_2}(x) - \varphi_{t_1}(x)) \geq 0 .
 \end{equation}
 Moreover, the left-hand-side is in fact strictly positive iff $x \in X^+$.  
   \end{lemma}
\begin{proof}
We know by Lemma \ref{lem:varphi-linear} and Theorem \ref{thm:order12-main} that:
\[
\varphi_s(\gamma^i_s) = \varphi_{t_i}(\gamma^i_{t_i}) + (t_i - s) \frac{\len^2(\gamma^i)}{2} =  \varphi_{t_i}(x) + (t_i - s) \frac{\ell_{t_i}^2(x)}{2} ~,~ i=1,2 .  
\]
Recall that $\varphi_{t_i}(x) = \varphic_{t_i}(x)$ and $\ell_{t_i}(x) = \ellc_{t_i}(x)$ by Proposition \ref{prop:duality-char} and Theorem \ref{thm:order12-main}, as $x = \gamma^i_{t_i}$. Now set $\bar{s} := (s \vee t_1) \wedge t_2$.  Since $\bar{s} \in \set{t_1,t_2,s}$, it follows that:
\begin{align*}
& \varphi_s(\gamma^2_s) - \varphi_s(\gamma^1_s)  -  \brac{\varphi_{t_2}(x) -  \varphi_{t_1}(x)} \\
& =   (t_2 - s) \frac{\ell_{t_2}^2(x)}{2} - (\bar{s} - s) \frac{\ell_{\bar{s}}^2(x)}{2} + (\bar{s} - s) \frac{\ellc_{\bar{s}}^2(x)}{2} -  (t_1-s) \frac{\ellc_{t_1}^2(x)}{2} .
\end{align*}
By Corollary \ref{cor:AGS}, we know 
for $\tilde{\ell} = \ell,\ellc$ that $D_{\tilde{\ell}}(x) \ni t \mapsto \tilde{\ell}^2_t(x)$ is differentiable a.e., and that the singular part of its distributional derivative is non-negative for $\tilde{\ell} = \ell$ and non-positive for $\tilde{\ell} = \ellc$. Consequently, we may proceed as follows:
\[
\geq \int_{\bar{s}}^{t_2} \underline{\partial}_{\tau} \brac{(\tau - s) \frac{\ell_{\tau}^2(x)}{2}} d\tau + 
\int_{t_1}^{\bar{s}} \overline{\partial}_{\tau} \brac{(\tau - s) \frac{\ellc_{\tau}^2(x)}{2}} d\tau ,
\]
where we used that $\tau-s \geq 0$ when $\bar{s} \leq \tau < t_2$ and that $\tau-s \leq 0$ when $\bar{s} \geq \tau >  t_1$. 
Using (\ref{eq:4a}) and (\ref{eq:4b}) to bound the above lower and upper derivatives on the sets (having full measure) $D_{\ell}(x)$ and $D_{\ellc}(x)$, respectively,  we obtain:
\begin{align*}
&\geq \int_{\bar{s}}^{t_2} \brac{1 - 2 \frac{\tau-s}{\tau}}\frac{\ell_\tau^2(x) }{2} d\tau + \int_{t_1}^{\bar{s}} \brac{1 + 2 \frac{\tau-s}{1-\tau}} \frac{\ellc_{\tau}^2(x)}{2} d\tau \\
& = \int_{\bar{s}}^{t_2} \brac{2 \frac{s}{\tau} - 1}\frac{\ell_\tau^2(x) }{2} d\tau + \int_{t_1}^{\bar{s}} \brac{2 \frac{1-s}{1-\tau} - 1} \frac{\ellc_{\tau}^2(x)}{2} d\tau \\
& \geq \brac{2 \frac{s}{t_2} - 1} \int_{\bar{s}}^{t_2} \frac{\ell_\tau^2(x) }{2} d\tau + \brac{2 \frac{1-s}{1-t_1} - 1} \int_{t_1}^{\bar{s}} \frac{\ellc_{\tau}^2(x)}{2} d\tau \\
& = \brac{2 \frac{s}{t_2} - 1}  \brac{\varphi_{t_2}(x) - \varphi_{\bar{s}}(x)} + \brac{2 \frac{1-s}{1-t_1} - 1} \brac{\varphic_{\bar{s}}(x) - \varphic_{t_1}(x)} .
\end{align*}
Summarizing, we have obtained:
\begin{align*}
\varphi_s(\gamma^2_s) - \varphi_s(\gamma^1_s)  & \geq \brac{2 \frac{s}{t_2} - 1}  \brac{\varphi_{t_2}(x) - \varphi_{\bar{s}}(x)} + \varphi_{t_2}(x) \\
& + \brac{2 \frac{1-s}{1-t_1} - 1}  \brac{\varphic_{\bar{s}}(x) - \varphi_{t_1}(x)} - \varphi_{t_1}(x) .
\end{align*}
We now use the inequality $\varphi_{\bar{s}}(x) \leq \varphic_{\bar{s}}(x)$ in the first line above when $2 \frac{s}{t_2} - 1 \geq 0$, and in the second line when $2 \frac{1-s}{1-t_1} - 1 \geq 0$, yielding:
\[
 \geq  \begin{cases} 
 2 \frac{s}{t_2} (\varphic_{t_2}(x) - \varphic_{\bar{s}}(x)) + 2 \frac{1-s}{1-t_1} (\varphic_{\bar{s}}(x) - \varphic_{t_1}(x)) & s \geq \frac{t_2}{2} \\
 2 \frac{s}{t_2} (\varphi_{t_2}(x) - \varphi_{\bar{s}}(x)) + 2 \frac{1-s}{1-t_1} (\varphi_{\bar{s}}(x) - \varphi_{t_1}(x)) &  1-s \geq \frac{1-t_1}{2}
 \end{cases} . 
\]
In particular, the first estimate applies whenever $s \geq \frac{1}{2}$ and the second one whenever $s \leq \frac{1}{2}$. Using that $[0,1] \ni \tau \mapsto \tilde{\varphi}_\tau(x)$ is monotone non-decreasing, the asserted  (\ref{eq:propagation}) is established in either case. Moreover, (\ref{eq:propagation}) implies that if $\varphi_s(\gamma^2_s) -\varphi_s(\gamma^1_s) = 0$ then $\varphi_{t_1}(x) = \varphi_{t_2}(x)$, and hence by Corollary \ref{cor:X0} that $x \in X^0$; and vice-versa, if $x \in X^0$ then all geodesics having $x$ as an interior point are null by Lemma \ref{lem:X0}, and hence $\gamma^1_s = \gamma^2_s = x$ and $\varphi_s(\gamma^2_s) -\varphi_s(\gamma^1_s) = 0$. 

\end{proof}

We can already deduce the following important consequence, complementing Corollary \ref{cor:InterProp}, which holds for any proper geodesic space $(X,\sfd)$, independently of any additional assumptions like various forms of non-branching:
\begin{corollary}[Intra Level-Set Propagation] \label{cor:IntraProp}
For any $s \in (0,1)$, $a_s \in \Real$, and $t_1, t_2 \in (0,1)$ with $t_1 \neq t_2$:
\[
e_{t_1}(G_{a_s} \setminus G_\varphi^0) \cap e_{t_2}(G_{a_s} \setminus G_\varphi^0) = e_{t_1}(G_{a_s}) \cap e_{t_2}(G_{a_s}) \cap X^+ = \emptyset .
\]
In other words, for each $x \in e_{(0,1)}(G_{a_s}) \cap X^+$, there exists a unique $t \in (0,1)$ so that $x \in \ee_t(G_{a_s})$. 
\end{corollary}
\begin{proof}
If $x = \gamma^1_{t_1} = \gamma^2_{t_2} \in X^+$, $0<t_1 < t_2< 1$, then  Lemma \ref{lem:propagation} yields $\varphi_s(\gamma^2(s)) > \varphi_s(\gamma^1(s))$, establishing the assertion. 
\end{proof}

\subsection{Properties of $\Phi_s^t$}

The following information will be crucially used when deriving the Change-Of-Variables formula in Section \ref{S:comparison1}:

\begin{proposition} \label{prop:Phi}
For any $s \in (0,1)$, the following properties of $\Phi_s^t$ and $\Phic_s^t$ hold:
\begin{enumerate}
\item The maps $(x,t) \mapsto \Phi_s^t(x)$ and $(x,t) \mapsto \Phic_s^t(x)$ are continuous on $D_{\ell}$ and $D_{\ellc}$, respectively. 
\item For each $x \in X$, $\tilde{\Phi} = \Phi,\Phic$ and $\tilde{\ell} = \ell,\ellc$, respectively, 
$D_{\tilde{\ell}}(x) \ni t \mapsto \tilde{\Phi}_s^t(x)$ is differentiable at $t$ iff $D_{\tilde{\ell}}(x) \ni t \mapsto \tilde{\ell}^2_t(x)$ is differentiable at $t$ or if $t=s \in D_{\tilde{\ell}}(x)$, so in particular $t \mapsto \tilde{\Phi}_s^t(x)$ is a.e. differentiable. 
At points $t$ of differentiability: 
\begin{equation} \label{eq:Phi-deriv}
\partial_t \tilde{\Phi}_s^t(x) = \tilde{\ell}_t^2(x) + (t-s) \frac{\partial_t \tilde{\ell}^2_t(x)}{2} .
\end{equation}
In particular, if $s \in D_{\tilde{\ell}}(x)$ then  $\exists \partial_t|_{t=s} \tilde{\Phi}_s^t(x) = \tilde{\ell}_s^2(x)$. 
\item For each $x \in X$, the map $\mathring{G}_\varphi(x) \ni t \mapsto \Phi_s^t(x) = \Phic_s^t(x)$ is locally Lipschitz and non-decreasing (if $\# \mathring{G}_\varphi(x)  \geq 2$, it is strictly increasing iff $x \in X^+$). 
\item For all $t \in (0,1)$:
\[
\begin{cases} \underline{\partial}_t \Phi_s^t(x) \geq \frac{s}{t} \ell_t^2(x) & t \geq s \\ \overline{\partial}_t \Phi_s^t(x) \leq \frac{s}{t} \ell_t^2(x) & t \leq s \end{cases} \;\; \forall x \in D_{\ell}(t) ~;~  \begin{cases} \overline{\partial}_t \Phic_s^t(x) \leq \frac{1-s}{1-t} \ellc_t^2(x)   & t \geq s \\ \underline{\partial}_t \Phic_s^t(x) \geq \frac{1-s}{1-t} \ellc_t^2(x) & t \leq s \end{cases} \;\; \forall x \in D_{\ellc}(t) .
\] 
\item For all $(x,t) \in D(\mathring{G}_\varphi)$:
\begin{align*}
 \min\brac{\frac{s}{t},\frac{1-s}{1-t} + \frac{t-s}{t(1-t)}} \ell_t^2(x) \leq \underline{\partial}_t \Phi_s^t(x)  \leq \overline{\partial}_t \Phi_s^t(x) \leq \max\brac{\frac{s}{t},\frac{1-s}{1-t} + \frac{t-s}{t(1-t)}}  \ell_t^2(x), \\
 \min\brac{\frac{1-s}{1-t} , \frac{s}{t} - \frac{t-s}{t(1-t)}} \ell_t^2(x) \leq  \underline{\partial}_t \Phic_s^t(x)  \leq \overline{\partial}_t \Phic_s^t(x) \leq \max\brac{\frac{1-s}{1-t} , \frac{s}{t} - \frac{t-s}{t(1-t)}}  \ell_t^2(x).
 \end{align*}
\end{enumerate}
\end{proposition}
\begin{proof}
Recall that:
\[
\tilde{\Phi}_s^t := \tilde{\varphi}_t(x) + (t-s) \frac{\tilde{\ell}^2_t(x)}{2}  \text{ on $D_{\tilde{\ell}}$.}
\]
The first and second statements follow by Lemma \ref{lem:order0} and Corollary \ref{cor:AGS}. As $t \mapsto \tilde{\varphi}_t(x)$ is differentiable on $D_{\tilde{\ell}}(x)$ with derivative $\frac{\tilde{\ell}_t^2(x)}{2}$, the points of differentiability of $t \mapsto \tilde{\Phi}_s^t(x)$ must coincide with those of $t \mapsto \tilde{\ell}_t^2(x)$ and (\ref{eq:Phi-deriv}) follows immediately, with the only possible exception being the point $t=s$ if $s \in D_{\tilde{\ell}}(x)$, where direct inspection and continuity of $t \mapsto \tilde{\ell}_t^2(x)$ on $D_{\tilde{\ell}}(x)$ verifies (\ref{eq:Phi-deriv}). 
The local Lipschitzness follows by Theorem \ref{thm:order12-main} (2). The monotonicity follows by Lemma \ref{lem:propagation}, since if $\gamma^t \in G_\varphi$ is such that $\gamma^t_t = x$, then $\Phi_s^t(\gamma^t_t) = \Phic_s^t(\gamma^t_t) = \varphi_s(\gamma^t_s)$.
The last two assertions follow as in the proof of Lemma \ref{lem:propagation}, after noting that:
\[
\begin{cases} \underline{\partial}_t \tilde{\Phi}_s^t(x) =\tilde{\ell}_t^2(x) + (t-s) \underline{\partial}_t \frac{\tilde{\ell}_t^2(x)}{2}  & t \geq s \\ \underline{\partial}_t \tilde{\Phi}_s^t(x) = \tilde{\ell}_t^2(x) + (t-s) \overline{\partial}_t \frac{\tilde{\ell}_t^2(x)}{2} & t \leq s \end{cases} \;\;\;\; \forall x \in D_{\tilde{\ell}}(t) ,
\]
and similarly for $\overline{\partial}_t$. Indeed, the estimates (\ref{eq:4a}) and (\ref{eq:4b}) of Corollary \ref{cor:AGS} yield (4),
which already yields half of the inequalities in (5) for all $(x,t) \in D_\ell \cap D_\ellc$. To get the other half, we must restrict to $D(\mathring{G}_\varphi)$ and use the estimates of Theorem \ref{thm:order12-main} (4), thereby concluding the proof. 
\end{proof}

As an immediate corollary of Proposition \ref{prop:Phi}, Corollary \ref{cor:accum1} and Lemma \ref{lem:X0}, we obtain: 

\begin{corollary} \label{cor:accum2}
For all $x \in X$, for a.e. $t \in \mathring G_\varphi(x)$, $\partial_t \Phi^t_s(x)$ and $\partial_t \Phic^t_s(x)$ exist, coincide, and satisfy:
\begin{align*} \min\brac{\frac{s}{t},\frac{1-s}{1-t}} \ell^2_t(x) & \leq \partial_t \Phi^t_s(x) = \partial_t \Phi^t_s(x)|_{\mathring{G}_\varphi(x)} \\
& = \partial_t \Phic^t_s(x)|_{\mathring{G}_\varphi(x)} = \partial_t \Phic^t_s(x) \leq \max\brac{\frac{s}{t},\frac{1-s}{1-t}} \ell_t^2(x)  .
\end{align*} In particular, if $x \in X^+$ then $\partial_t \Phi^t_s(x) > 0$ for a.e. $t \in \mathring G_\varphi(x)$.
\end{corollary}

We will also require the following consequence of Proposition \ref{prop:Phi} and Theorem \ref{thm:order12-main}:
\begin{lemma}\label{L:densityPhi}
For any $x \in X$, $s \in (0,1)$, and $\tilde{\Phi} = \Phi,\Phic$ and $\tilde{\ell} = \ell,\ellc$, respectively:
\[
\lim_{\eps \rightarrow 0} \frac{1}{2 \eps} \int_{(s-\eps,s+\eps) \cap \mathring G_\varphi(x)} \brac{\partial_t \tilde{\Phi}^t_s(x) - \tilde{\ell}_s^2(x)} dt = 0 .
\]
\end{lemma}
\begin{proof}
By (\ref{eq:Phi-deriv}), the claim boils down to proving:
\[
\lim_{\eps \rightarrow 0} \frac{1}{2\eps} \int_{(s-\eps,s+\eps) \cap \mathring G_\varphi(x)} (t-s) \partial_t \tilde{\ell}_t^2(x) dt = 0 . 
\]
Using Corollary \ref{cor:accum1}, it follows that:
\begin{align*}
\lim_{\eps \rightarrow 0} \frac{1}{2\eps} \abs{\int_{(s-\eps,s+\eps) \cap \mathring G_\varphi(x)} (t-s) \partial_t \tilde{\ell}_t^2(x) dt } 
&~ \leq  \lim_{\eps \rightarrow 0} \int_{(s-\eps,s+\eps) \cap \mathring G_\varphi(x)} \abs{\partial_t \frac{\tilde{\ell}_t^2(x)}{2}} dt  \\
&~  \leq  \frac{1}{\min(s,1-s)} \lim_{\eps \rightarrow 0} \int_{(s-\eps,s+\eps) \cap \mathring G_\varphi(x)}  \tilde{\ell}_t^2(x) dt .
\end{align*}
But the latter limit is clearly $0$ (e.g. by Corollary \ref{cor:AGS} (1)). 
\end{proof}

\bigskip
\bigskip

\section[Temporal Theory of Intermediate-Time Kantorovich Potentials.\\ Third Order]{Temporal Theory of Intermediate-Time Kantorovich Potentials. Third Order} \label{sec:order3}

Fix a non-null Kantorovich geodesic $\gamma \in G_\varphi^+$, and denote for short $\ell := \len(\gamma) > 0$.
Recall by the results of Section \ref{sec:order12} that for all $t \in (0,1)$, $\ell_t(\gamma_t) = \ellc_t(\gamma_t) = \ell$ and that $\partial_t \varphi_t(x) = \partial_t \varphic_t(x) =\ell_t^2(x)/2$ for all $x \in \ee_t(G_\varphi)$. Also, recall that given $x \in X$ and $\tilde{\ell} = \ell,\ellc$, the function $D_{\tilde{\ell}}(x) \ni t \mapsto \tilde{\ell}_t(x)$ is only a.e. differentiable, and even on $\mathring{G}_{\varphi}(x) \subset D_{\ell}(x) \cap D_{\ellc}(x)$, we only have at the moment upper and lower bounds on $\underline{\partial_t} \tilde{\ell}^2_t(x)/2$ and $\overline{\partial}_t \tilde{\ell}^2_t(x)/2$, i.e. second order information on $\tilde{\varphi}_t(x)$.

\medskip

The goal of this section is to rigorously make sense and prove the following formal statement, which provides second order information on $\ell_t$, or equivalently, third order information on $\varphi_t$, along $\gamma_t$:
\begin{equation} \label{eq:3rd-formal-goal}
z(t) := \partial_\tau|_{\tau=t} \frac{\ell_\tau^2}{2}(\gamma_t) \;\;\; \Rightarrow \;\;\; z'(t) \geq \frac{z(t)^2}{\ell^2} . 
\end{equation}
Equivalently, this amounts to the statement that the function:
\[
L(r) = \exp\brac{- \frac{1}{\ell^2} \int^r_{r_0} \partial_\tau|_{\tau=t} \frac{\ell_\tau^2}{2}(\gamma_t) dt} 
\]
is concave in $r \in (0,1)$, since formally:
\[
\frac{L''}{L} = (\log L) '' + ((\log L)')^2 = -\frac{z'}{\ell^2} + \frac{z^2}{\ell^4} \leq 0 . 
\]

\subsection{Formal Argument} \label{subsec:order3-formal} 

We start by providing a formal proof of (\ref{eq:3rd-formal-goal}) in an infinitesimally Hilbertian setting, which is rigorously justified on a Riemannian manifold if all involved functions are smooth (in time and space). 

Recall that the Hopf-Lax semi-group solves the Hamilton-Jacobi equation (e.g. \cite{ambrgisav:heat}):
\begin{equation} \label{eq:HJ}
\partial_t\varphi_t =  \frac{1}{2} \ell_t^2 = \frac{1}{2} \abs{\nabla \varphi_t}^2 .
\end{equation}
We evaluate all subsequent functions at $x = \gamma_t$. 
Since: \[
z(t) = \partial_t^2 \varphi_t(\gamma(t)) = \scalar{\nabla \partial_t \varphi_t , \nabla \varphi_t} , 
\]
and since $\gamma'(t) = -\nabla \varphi_t$ (see e.g. \cite{ambrgisav:heat} or Lemma \ref{L:dmonotonedistance}), 
\[
z'(t) = \partial_t^3 \varphi_t - \scalar{\nabla \partial_t^2 \varphi_t , \nabla \varphi_t} .
\]
But taking two time derivatives in (\ref{eq:HJ}), we know that:
\[
\partial_t^3 \varphi_t = \scalar{\nabla \partial_t^2 \varphi_t , \nabla \varphi_t} + \scalar{\nabla \partial_t \varphi_t , \nabla \partial_t \varphi_t} ,
\]
and so we conclude that:
\[
z'(t) = \abs{\nabla \partial_t \varphi_t}^2 . 
\]
It remains to apply Cauchy--Schwarz and deduce:
\[
z'(t) \geq \scalar{\nabla \partial_t \varphi_t , \frac{\nabla \varphi_t}{\abs{\nabla \varphi_t}}}^2 = \frac{z(t)^2}{\ell^2} ,
\]
as asserted. Note that in a general setting, we can try and interpret $z(t)$ as minus the directional derivative of $\ell_t^2/2 = \partial_t \varphi_t$ in the direction of $\gamma'(t)$ (by taking derivative of the identity $\frac{\ell^2_t}{2}(\gamma(t)) \equiv \frac{\ell^2}{2}$), and thus hope to justify the Cauchy--Schwarz inequality as the statement that the local Lipschitz constant of $\partial_t \varphi_t$ is greater than any unit-directional derivative. However, a crucial point in the above argument of identifying $z'(t)$ with $\abs{\nabla \partial_t \varphi_t}^2$ was to use the linearity of $\scalar{\cdot,\cdot}$ in both of its arguments, and so ultimately this formal proof is genuinely restricted to an infinitesimally Hilbertian setting.

\smallskip
The above discussion seems to suggest that there is no hope of proving (\ref{eq:3rd-formal-goal}) beyond the Hilbertian setting. Furthermore, it seems that the spatial regularity of $\varphi_t$ and $\partial_t \varphi_t = \frac{1}{2}\ell^2_t$ should play an essential role in any rigorous justification. Remarkably, we will see that this is not the case on both counts, and that an appropriate interpretation of (\ref{eq:3rd-formal-goal}) holds true on a general proper geodesic space $(X,\sfd)$. 

\subsection{Notation}

Recall that by the results of Section \ref{sec:order12}, $\tau \mapsto \varphi_\tau(x)$ and $\tau \mapsto \varphic_\tau(x)$ are locally semi-convex and semi-concave on $(0,1)$, respectively, and that $\partial_t^{\pm} \varphi_t(x) = \ell^{\pm}_t(x)^2/2$, $\partial_t^{\pm} \varphic_t(x) = \ellc^{\mp}_t(x)^2/2$ and $\ell^{\pm}_t(\gamma_t) = \ellc^{\pm}_t(\gamma_t) = \ell$ for all $t \in (0,1)$.
We respectively introduce $\tilde{p} = p,\pc$ by defining at $t \in (0,1)$:
\begin{align*}
\tilde{p}_{+}^\gamma(t) = \tilde{p}_{+}(t) & := \overline{\partial}_\tau|_{\tau=t} \tilde{\ell}^2_{\tau}(\gamma_t)/2 = \ell \cdot \overline{\partial}_\tau|_{\tau=t} \tilde{\ell}_{\tau}(\gamma_t)  = \ell \cdot \overline{\partial}_\tau|_{\tau=t} \tilde{\ell}^{\pm}_{\tau}(\gamma_t) ~,~ \\
\tilde{p}_{-}^\gamma(t) = \tilde{p}_{-}(t) & := \underline{\partial}_\tau|_{\tau=t} \tilde{\ell}^2_{\tau}(\gamma_t)/2 = \ell \cdot \underline{\partial}_\tau|_{\tau=t} \tilde{\ell}_{\tau}(\gamma_t) = \ell \cdot \underline{\partial}_\tau|_{\tau=t} \tilde{\ell}^{\pm}_{\tau}(\gamma_t) ~,~ 
\end{align*}
where the penultimate equalities in each of the lines above follow from the continuity of $D_{\tilde{\ell}}(\gamma_t) \ni \tau \mapsto \tilde{\ell}_\tau(\gamma_t)$ at $\tau = t \in G_\varphi(\gamma_t) \subset D_{\tilde{\ell}}(\gamma_t)$, and the last ones by the monotonicity of $\tau \mapsto \tau \ell^{\pm}_\tau(\gamma_t)$ and $\tau \mapsto  (1-\tau) \ellc^{\pm}_\tau(\gamma_t)$ and the density of $D_{\tilde{\ell}}$ in $(0,1)$. Clearly $\tilde{p}_{-}(t) \leq \tilde{p}_{+}(t)$, and $\tilde{p}_-(t) = \tilde{p}_+(t) = \tilde{p} \in \Real$ iff $D_{\tilde{\ell}}(\gamma_t) \ni \tau \mapsto \tilde{\ell}_\tau^2/2(\gamma_t)$ is differentiable at $\tau = t$ with derivative $\tilde{p}$.
In addition, for $\tilde{q} = q,\bar{q}$, set:
\[
\tilde{q}_+(t) := \overline{\P}_2 \tilde{\varphi}_t(x)|_{x=\gamma_t} \geq \underline{\P}_2 \tilde{\varphi}_t(x)|_{x=\gamma_t}  =: \tilde{q}_-(t) ,
\]
where the Peano (partial) derivatives are with respect to the $t$ variable. It will be useful to recall that if we define $\tilde{h} = h , \bar{h}$ by:
\begin{align*}
\tilde{h}(t,\eps) & := 2\brac{\tilde{\varphi}_{t+\eps}(\gamma_t) - \tilde{\varphi_t}(\gamma_t) - \eps \partial_t \tilde{\varphi}_t(\gamma_t)} \\
& \; = 2\brac{\tilde{\varphi}_{t+\eps}(\gamma_t) - \varphi_t(\gamma_t) - \eps \ell^2/2}  ,
\end{align*}
then:
\[
\tilde{q}_+(t) =  \limsup_{\eps \rightarrow 0} \frac{\tilde{h}(t,\eps)}{\eps^2}  \geq \liminf_{\eps \rightarrow 0} \frac{\tilde{h}(t,\eps)}{\eps^2} = \tilde{q}_-(t) .
\]
By definition, $\tilde{q}_-(t) = \tilde{q}_+(t) = \tilde{q} \in \Real$ if and only if $\tau \mapsto \tilde{\varphi}_{\tau}(\gamma_t)$ has second order Peano derivative at $\tau = t$ equal to $\tilde{q}$, and hence by Lemma \ref{lem:convex-2nd-diff}, iff $\tilde{p}_-(t) = \tilde{p}_+(t) = \tilde{q}$, or equivalently, iff any
of the other equivalent conditions for the second order differentiability of $(0,1) \ni \tau \mapsto \tilde{\varphi}_{\tau}(\gamma_t)$ at $\tau=t$ are satisfied. Moreover, Lemma \ref{lem:peano-inq} implies:
\[
\tilde{p}_-(t) \leq \tilde{q}_-(t) \leq \tilde{q}_+(t) \leq \tilde{p}_+(t)  \;\;\; \forall t \in (0,1) ,
\]
but we will not require this here. We summarize the above discussion in:
\begin{corollary} \label{cor:2nd-diff}
The following statements are equivalent for a given $t \in (0,1)$:
\begin{enumerate}
\item $\tilde{p}_-(t) = \tilde{p}_+(t) = \tilde{p} \in \Real$, i.e. $D_{\tilde{\ell}}(\gamma_t) \ni \tau \mapsto \frac{\tilde{\ell}^2_\tau}{2}(\gamma_t)$ is differentiable at $\tau = t$ with derivative $\tilde{p}$.  
\item $\tilde{q}_-(t) = \tilde{q}_+(t) = \tilde{q} \in \Real$, i.e. $(0,1) \ni \tau \mapsto \tilde{\varphi}_\tau(\gamma_t)$ has a second Peano derivative at $\tau = t$ equal to $\tilde{q}$. 
\end{enumerate}
In any of these cases $(0,1) \ni \tau \mapsto \tilde{\varphi}_\tau(\gamma_t)$ is twice differentiable at $\tau = t$, and we have:
\[
\partial_\tau^2|_{\tau = t} \tilde{\varphi}_\tau(\gamma_t) := \partial_\tau|_{\tau=t} \frac{\tilde{\ell}^2_\tau}{2}(\gamma_t) = \ell  \cdot \partial_\tau|_{\tau=t} \tilde{\ell}_t(\gamma_t) = \tilde{p} = \tilde{q} .
\]
\end{corollary}

\allowdisplaybreaks
\subsection{Main Inequality}

The following inequality and its consequences are the main results of this section. 

\begin{theorem} \label{thm:main-3rd-order}
For all $s < t$ and $\eps$ so that $s,t,s+\eps,t+\eps \in (0,1)$, we have (for both possibilities for $\pm$):
\[
\frac{h(t,\eps) - h(s,\eps)}{t-s} \geq \frac{s+\eps}{t+\eps} (\ell^{\pm}_{s+\eps}(\gamma_s) - \ell_s(\gamma_s))^2 ,
\]
and:
\[
\frac{\bar{h}(t,\eps) - \bar{h}(s,\eps)}{t-s} \geq \frac{1-t-\eps}{1-s-\eps} (\ellc^{\pm}_{t+\eps}(\gamma_t) - \ellc_t(\gamma_t))^2 . 
\]
\end{theorem}
\begin{proof}
By Lemma \ref{lem:compactness}, there exist $y^{\pm}_\eps \in X$ so that:
\[
 -\varphi_{s+\eps}(\gamma_{s}) = \frac{\sfd^{2}(y^{\pm}_{\eps},\gamma_{s})}{2(s+\eps)} - \varphi(y^{\pm}_{\eps}) ,
\]
with $\sfd(y^{\pm}_\eps,\gamma_s)  = D^{\pm}_{-\varphi}(\gamma_s,s+\eps) = (s+\eps) \ell^{\pm}_{s+\eps}(\gamma_s) =: D^{\pm}_{s+\eps}$. By definition, note that:
\[
-\varphi_{t+\eps}(\gamma_{t}) \leq \frac{\sfd^{2}(y^{\pm}_{\eps},\gamma_{t})}{2(t+\eps)} - \varphi(y^{\pm}_{\eps}) .
\]
We abbreviate $D_{r} := r \ell = \sfd(\gamma_r,\gamma_0)$, $r=s,t$.
The proof consists of subtracting the above two expressions and applying the triangle inequality:
\[
\sfd(y^{\pm}_\eps,\gamma_t) \leq \sfd(y^{\pm}_\eps,\gamma_s) + \sfd(\gamma_s,\gamma_t) = D^{\pm}_{s+\eps} + (D_t - D_s) = D_t + (D^{\pm}_{s+\eps} - D_{s}) .
\]
Indeed, we obtain after subtraction, recalling the definition of $h$, and an application of Lemma \ref{lem:varphi-linear}:
\begin{align*}
0 \leq &~ \varphi_{t+\eps}(\gamma_{t}) -\varphi_{s+\eps}(\gamma_{s}) + \frac{\sfd^{2}(y^{\pm}_{\eps},\gamma_{t})}{2(t+\eps)}  - \frac{\sfd^{2}(y^{\pm}_{\eps},\gamma_{s})}{2(s+\eps)}\\
=  &~ \frac{1}{2}\brac{h(t,\eps) - h(s,\eps)} +\varphi_{t}(\gamma_t) - \varphi_{s}(\gamma_s) +\frac{\sfd^{2}(y^{\pm}_{\eps}, \gamma_{t})}{2(t+\eps)}  - \frac{(D^{\pm}_{s+\eps})^2}{2(s+\eps)} \\
\leq  &~ \frac{1}{2}\brac{h(t,\eps) - h(s,\eps)} -\frac{\ell^{2}}{2}(t-s) - \frac{(D^{\pm}_{s+\eps})^2}{2(s+\eps)} \\
 &~ +\frac{ (D^{\pm}_{s+\eps} -D_{s})^{2} + D_{t}^{2} + 2(D^{\pm}_{s+\eps} - D_{s})D_{t} }{2(t+\eps)}.
\end{align*}
Carefully rearranging terms, we obtain: \begin{align*}
\frac{1}{2} \brac{h(t,\eps) - h(s,\eps)} & \geq  \frac{(D^{\pm}_{s+\eps})^2}{2(s+\eps)} - \frac{D^{2}_{s}}{2s} + \frac{D_{t}^{2}}{2}\left(\frac{1}{t} - \frac{1}{t+\eps} \right) 
-\frac{ (D^{\pm}_{s+\eps} -D_{s})^{2} + 2(D^{\pm}_{s+\eps} - D_{s})D_{t} }{2(t+\eps)} \crcr
 & = \frac{1}{2(s+\eps)} ((D^{\pm}_{s+\eps})^{2} - D_{s}^{2}) + \frac{D_{s}^{2}}{2} \left(\frac{1}{s+\eps} - \frac{1}{s} \right) + \frac{D_{t}^{2}}{2}\left(\frac{1}{t} - \frac{1}{t+\eps} \right) \crcr
 &  \phantom{=} ~ -\frac{ (D^{\pm}_{s+\eps} -D_{s})^{2} + 2(D^{\pm}_{s+\eps} - D_{s})D_{t} }{2(t+\eps)} \crcr
& =  (D^{\pm}_{s+\eps} - D_{s}) \brac{\frac{D^{\pm}_{s+\eps} + D_{s}}{2(s+\eps)} - \frac{D_{t}}{t+\eps} }
+ \frac{\ell^{2}}{2} \left( \frac{\eps t}{t+\eps} - \frac{\eps s}{s+\eps}   \right) - \frac{(D^{\pm}_{s+\eps} - D_{s})^{2}}{2(t+\eps)} \crcr
& = (D^{\pm}_{s+\eps} - D_{s}) \brac{\frac{D^{\pm}_{s+\eps} - D_{s} + 2D_{s} - 2(s+\eps)\ell}{2(s+\eps)} + D_{t}\left(\frac{1}{t}- \frac{1}{t+\eps}\right) }  \crcr
& \phantom{=} ~ + \eps^{2}\frac{\ell^{2}}{2} \brac{ \frac{1}{s+\eps} - \frac{1}{t+\eps} } - \frac{(D^{\pm}_{s+\eps} - D_{s})^{2}}{2(t+\eps)} \crcr
& = \frac{(D^{\pm}_{s+\eps} - D_{s})^{2}}{2}\left(\frac{1}{s+\eps} - \frac{1}{t+\eps} \right)  -\eps \frac{D^{\pm}_{s+\eps} - D_{s}}{s+\eps} \ell  \crcr
& \phantom{=}  ~ + (D^{\pm}_{s+\eps} - D_{s})D_{t}\left( \frac{1}{t} - \frac{1}{t+\eps} \right) + \eps^{2}\frac{\ell^{2}}{2} \brac{ \frac{1}{s+\eps} - \frac{1}{t+\eps} } \crcr
& = \left(\frac{1}{s+\eps} - \frac{1}{t+\eps} \right) \brac{ \frac{(D^{\pm}_{s+\eps} - D_{s})^{2}}{2}  -\eps \ell (D^{\pm}_{s+\eps} - D_{s}) + \eps^{2} \frac{\ell^{2}}{2} } \crcr
& = \frac{1}{2}\brac{\frac{1}{s+\eps} - \frac{1}{t+\eps}} \brac{ D^{\pm}_{s+\eps} - D_{s}  -\eps \ell }^{2} \crcr
& = \frac{1}{2}\left(\frac{1}{s+\eps} - \frac{1}{t+\eps} \right) (s+\eps)^{2} ( \ell^{\pm}_{s+\eps}(\gamma_s) - \ell)^{2} \crcr
& = \frac{t-s}{2} \frac{s+\eps}{t+\eps} ( \ell^{\pm}_{s+\eps}(\gamma_s) - \ell)^{2} ,
\end{align*}
and the first claim follows.
\allowdisplaybreaks[0]

The second claim follows by the duality between $\varphi$ and $\varphi^c$. Indeed, exchange $\varphi,\gamma,\eps,s,t$ with $\varphi^c,\gamma^c,-\eps,1-t,1-s$, and recall that $\varphic_t = -\varphi^c_{1-t}$. 
A straightforward inspection of the definitions verifies:
\[
h^{\varphi^c}(1-r,-\eps) = -\bar{h}^{\varphi}(r,\eps) ,
\]
and:
\[
\frac{(\ell_{1-t-\eps}^{\varphi^c,\pm}(\gamma^c_{1-t}))^2}{2} = - \partial_t^{\mp} \varphi^c_{1-t-\eps}(\gamma^c_{1-t}) = \partial_t^{\mp} \varphic_{t+\eps}(\gamma_t) = \frac{(\ellc^{\varphi,\pm}_{t+\eps}(\gamma_t))^2}{2} ,
\]
and so the second claim follows from the first one. Alternatively, one may repeat the above argument by subtracting the following two expressions: 
\begin{align*}
\varphic_{t+\eps}(\gamma_{t}) = &~ \frac{\sfd^{2}(z^{\pm}_{\eps},\gamma_{t})}{2(1-t-\eps)} - \varphi^c(z^{\pm}_{\eps}) , \crcr
\varphic_{s+\eps}(\gamma_{s}) \leq &~ \frac{\sfd^{2}(z^{\pm}_{\eps},\gamma_{s})}{2(1-s-\eps)} - \varphi^c(z^{\pm}_{\eps})  ,
\end{align*}
with $\sfd(z^{\pm}_\eps,\gamma_t) = D^{\pm}_{-\varphi^c}(\gamma_t,1-t-\eps) = (1-t-\eps) \ell^{\pm}_{t+\eps}(\gamma_t)$
and applying the triangle inequality $\sfd(z_\eps,\gamma_s) \leq \sfd(z_\eps,\gamma_t) + \sfd(\gamma_t,\gamma_s)$. 
\end{proof}

\subsection{Consequences}

As immediate corollaries of Theorem \ref{thm:main-3rd-order}, we obtain after diving both sides by $\eps^2$ and taking appropriate subsequential limits as $\eps\rightarrow 0$:
\begin{corollary}
For both $\tilde{q} = q,\bar{q}$, the functions $t \mapsto \tilde{q}_-(t)$ and $t \mapsto \tilde{q}_+(t)$ are monotone non-decreasing on $(0,1)$. 
\end{corollary}

\begin{corollary}
For all $0 < s < t < 1$ (and both possibilities for $\pm$):
\begin{equation} \label{eq:3rd-primal}
\frac{q_+(t) - q_-(s)}{t-s} \geq \frac{s}{t} \brac{\frac{p_{\pm}(s)}{\ell}}^2 ,
\end{equation}
and:
\begin{equation} \label{eq:3rd-dual}
\frac{\bar{q}_+(t) -\bar{q}_-(s)}{t-s} \geq \frac{1-t}{1-s} \brac{\frac{\bar{p}_{\pm}(t)}{\ell}}^2 .
\end{equation}
\end{corollary}

It will be convenient to use the above information in the following form:

\begin{theorem} \label{thm:z-c} 
Assume that for \textbf{a.e.} $t \in (0,1)$:
\begin{equation} \label{eq:twice-diff-assump}
 (0,1)\ni \tau \mapsto \tilde{\varphi}_\tau(\gamma_t) \text{ is twice differentiable at $\tau = t$ for \textbf{both} $\tilde{\varphi} = \varphi,\varphic$}
\end{equation}
in any of the equivalent senses given by Corollary \ref{cor:2nd-diff}, and that moreover:
\begin{equation} \label{eq:twice-diff-assump-2}
\partial_\tau^2|_{\tau = t} \varphi_\tau(\gamma_t) = \partial_\tau^2|_{\tau = t} \varphic_\tau(\gamma_t) \;\;\; \text{for a.e. } t \in (0,1) . 
\end{equation}
Furthermore, assume that the latter joint value coincides a.e. on $(0,1)$ with some continuous function $z_c$: 
\begin{equation} \label{eq:twice-diff-b-assump}
\partial_\tau^2|_{\tau = t} \varphi_\tau(\gamma_t) = \partial_\tau^2|_{\tau = t} \varphic_\tau(\gamma_t)  = z_c(t) \;\;\; \text{for a.e. } t \in (0,1) . 
\end{equation}
Then (\ref{eq:twice-diff-assump}) holds for \textbf{all} $t \in (0,1)$, and we have:
\begin{equation} \label{eq:twice-diff-concl}
\partial_\tau^2|_{\tau = t} \varphi_\tau(\gamma_t) = \partial_\tau^2|_{\tau = t} \varphic_\tau(\gamma_t) = \partial_\tau|_{\tau=t} \frac{\ell^2_\tau}{2}(\gamma_t) = \partial_\tau|_{\tau=t} \frac{\ellc^2_\tau}{2}(\gamma_t)  = z_c(t) \;\;\; \forall t \in (0,1) . 
\end{equation}
Moreover, we have the following third order information on $\varphi_t(x)$ at $x=\gamma_t$:
\begin{equation} \label{eq:z-c}
\frac{z_c(t) - z_c(s)}{t-s} \geq \sqrt{\frac{s}{t} \frac{1-t}{1-s}} \frac{\abs{z_c(s)} \abs{z_c(t)}}{\ell^2} \;\;\; \forall 0 < s < t < 1 . 
\end{equation}
In particular, for any point $t \in (0,1)$ where $z_c(t)$ is differentiable:
\[
z_c'(t) \geq \frac{z_c(t)^2}{\ell^2} .
\]
\end{theorem}
\begin{proof}
The assumptions imply by Corollary \ref{cor:2nd-diff} that $\tilde{q}_-(t) = \tilde{q}_+(t) = z_c(t)$ for a.e. $t \in (0,1)$. It follows that the same is true for every $t \in (0,1)$ by monotonicity of $\tilde{q}_{\pm}$ and the assumption that $z_c$ is continuous, yielding (\ref{eq:twice-diff-concl}). 
Furthermore, Corollary \ref{cor:2nd-diff} implies that $\tilde{p}_-(t) = \tilde{p}_+(t) = z_c(t)$ for both $\tilde{p} = p,\bar{p}$ and for all $t \in (0,1)$, and we obtain (\ref{eq:z-c}) by taking geometric mean of (\ref{eq:3rd-primal}) and (\ref{eq:3rd-dual}). The final assertion obviously follows by taking the limit in (\ref{eq:z-c}) as $s \rightarrow t$.
\end{proof}

We do not know whether all three assumptions (\ref{eq:twice-diff-assump}), (\ref{eq:twice-diff-assump-2}) and (\ref{eq:twice-diff-b-assump}) hold for a.e. $t \in (0,1)$ for a fixed Kantorovich geodesic $\gamma$. 
However, we can guarantee the first two assumptions, at least for almost all Kantorovich geodesics, in the following sense:

\begin{lemma} \label{lem:Fubini-diagonal} 
Let $\nu$ denote any $\sigma$-finite Borel measure concentrated on $G_\varphi$, so that for a.e. $t \in (0,1)$, $\mu_t := (\ee_t)_{\sharp}(\nu) \ll \sfm$ for some $\sigma$-finite Borel measure $\sfm$ on $X$.  Then for $\nu$-a.e. geodesic $\gamma$, (\ref{eq:twice-diff-assump}) and (\ref{eq:twice-diff-assump-2}) hold for a.e. $t \in (0,1)$. 
\end{lemma}
\begin{proof}
Recall that $D(\mathring G_\varphi)$ is closed in $X \times (0,1)$ by Corollary \ref{cor:G-closed}. Denote the following Borel subsets:
\[
P := \set{ (x,t) \in D(\mathring G_\varphi) \; ; \;  \exists \partial_t \ell^2_t(x)  ~,~ \exists \partial_t \ellc^2_t(x)  ~,~ \partial_t \ell^2_t(x)/2 = \partial_t \ellc^2_t(x)/2 } ~,~ B := D(\mathring G_\varphi) \setminus P .
\]
By Corollary \ref{cor:accum1}, we know that $\L^1(B(x)) = 0$ for all $x \in X$.
By Fubini:
\[
0 = \int \L^1(B(x)) \sfm(dx) = \int_0^1 \sfm(B(t)) \L^1(dt) ,
\]
and so for a.e. $t \in (0,1)$, $\sfm(B(t)) = 0$. Since $\mu_t \ll \sfm$ for a.e. $t \in (0,1)$, it follows that for a.e. $t \in (0,1)$, $\nu(\ee_t^{-1} B(t)) = \mu_t(B(t)) =  0$. In other words, for a.e. $t \in (0,1)$, the Borel set $\set{ \gamma \in G_\varphi \; ; \;  \gamma_t \in B(t) }$ has zero $\nu$-measure. 
Applying Fubini again as before: 
$$
0 = \int \nu(\set{ \gamma \in G_\varphi \; ; \;  \gamma_t \in B(t) }) 
\L^1(dt) = \int \L^1(\set{ t \in (0,1) \; ; \;  \gamma_t \in B(t)})
\nu(d\gamma) ,
$$
we conclude that for $\nu$-a.e. $\gamma \in G_\varphi$, the set $\set {t \in (0,1) \; ; \; \gamma_t \in B(t) }$
has zero Lebesgue measure, or equivalently, the set:
\[
\set{ t \in (0,1) \; ; \;  \exists \partial_\tau|_{\tau=t} \ell^2_\tau(\gamma_t) ~,~ \exists \partial_\tau|_{\tau=t} \ellc^2_\tau(\gamma_t) ~,~ \partial_\tau|_{\tau=t} \ell^2_\tau(\gamma_t)/2 = \partial_\tau|_{\tau=t} \ellc^2_\tau(\gamma_t)/2}
\]
has full Lebesgue measure. The asserted (\ref{eq:twice-diff-assump}) and  (\ref{eq:twice-diff-assump-2}) now directly follow from 
an application of Corollary \ref{cor:2nd-diff}.
\end{proof}

Finally, we obtain the following concise interpretation of the 3rd order information on $\tau \mapsto \varphi_\tau$ along $\gamma_t$, which will play a crucial role in this work:
\begin{lemma}  \label{lem:z-ac}
Assume that for some locally absolutely continuous function $z_{ac}$ on $(0,1)$ we have: 
\[
\exists \partial_\tau|_{\tau=t} \frac{\ell^2_\tau}{2} (\gamma_t) = z_{ac}(t) \;\;\; \text{for a.e. } t \in (0,1) . 
\]
Then
for any fixed $r_0 \in (0,1)$, the function:
\[
L(r) = \exp \brac{-\frac{1}{\ell^2} \int_{r_0}^r \partial_\tau|_{\tau=t} \frac{\ell^2_\tau}{2}(\gamma_t) dt } = \exp \brac{-\frac{1}{\ell^2} \int_{r_0}^r z_{ac}(t) dt } ,
\]
is concave on $(0,1)$. 
\end{lemma}
\begin{proof}
Since $L \in C^1(0,1)$, concavity of $L$ is equivalent to showing that the function:
\[
W(r) := - \ell^2 L'(r) = L(r) z_{ac}(r) 
\]
is monotone non-decreasing. But as this function is locally absolutely continuous, this is equivalent to showing that $W'(r) \geq 0$ for a.e. $r \in (0,1)$. Note that the points of differentiability of $W$ and $z_{ac}$ coincide. At these points (of full Lebesgue measure), we indeed have:
\[
W'(r) = L'(r) z_{ac}(r) + L(r) z'_{ac}(r) = L(r) (z_{ac}'(r) - z_{ac}(r)^2 / \ell^2) \geq 0,
\]
where the last inequality follows from Theorem \ref{thm:z-c}. This concludes the proof.
\end{proof}

We will subsequently show that under synthetic curvature conditions, the above assumption is indeed satisfied for $\nu$-a.e. geodesic $\gamma$.

\bigskip
\bigskip

\part{Disintegration Theory of Optimal Transport}\label{part2}

\section{Preliminaries} \label{sec:PartII-prelim}

So far we have worked without considering any reference measure over our metric space $(X,\sfd)$.
A triple $(X,\sfd, \mm)$ is called a metric measure space, \mms for short, if $(X, \sfd)$ is a complete and separable metric space and $\mm$ is a non-negative Borel measure over $X$. 
\textbf{In this work we will only be concerned with the case that $\mm$ is a probability measure, that is $\mm(X) =1$}, and hence $\mm$ is automatically a Radon measure (i.e. inner-regular). 
We refer to \cite{ambro:userguide,AGS-Book,GromovBook,Vil:topics, Vil} for background on metric measure spaces in general, and the theory of optimal transport on such spaces in particular.

\subsection{Geometry of Optimal Transport on Metric Measure Spaces}

The space of all Borel probability measures over $X$ will be denoted by $\mathcal{P}(X)$. It is naturally equipped with its weak topology, in duality with bounded continuous functions $C_b(X)$ over $X$.  The subspace of those measures having finite second moment will be denoted by $\mathcal{P}_{2}(X)$, and the subspace of $\mathcal{P}_{2}(X)$
of those measures absolutely continuous with respect to $\mm$ is denoted by $\mathcal{P}_{2}(X,\sfd,\mm)$.
The weak topology on $\mathcal{P}_{2}(X)$ is metrized by the $L^{2}$-Wasserstein distance $W_{2}$, defined as follows for any $\mu_0,\mu_1 \in \mathcal{P}(X)$:
\begin{equation}\label{eq:Wdef}
  W_2^2(\mu_0,\mu_1) := \inf_{ \pi} \int_{X\times X} \sfd^2(x,y) \, \pi(dx , dy),
\end{equation}
where the infimum is taken over all $\pi \in \mathcal{P}(X \times X)$ having $\mu_0$ and $\mu_1$ as the first and the second marginals, respectively; such candidates $\pi$ are called transference plans. 
It is known that the infimum in (\ref{eq:Wdef}) is always attained for any $\mu_0,\mu_1 \in \mathcal{P}(X)$, and the transference plans realizing this minimum are called optimal transference plans between $\mu_0$ and $\mu_1$.
When $W_2(\mu_0,\mu_1) < \infty$, it is known that given an optimal transference plan $\pi$ between $\mu_0$ and $\mu_1$, there exists a Kantorovich potential $\varphi : X \rightarrow \Real$ 
(see Section \ref{sec:order12}), which is associated to $\pi$, meaning that:
\begin{equation} \label{eq:OptimalKant}
\varphi(x) + \varphi^c(y) = \frac{\sfd(x,y)^2}{2} \;\;\;  \text{for $\pi$-a.e. $(x,y) \in X \times X$} .
\end{equation}

In particular, when $\mu_0,\mu_1 \in \mathcal{P}_2(X)$, then necessarily $W_2(\mu_0,\mu_1) < \infty$ and the above discussion applies. Moreover, in this case, it is known that for \emph{any} Kantorovich potential $\varphi$ associated to an optimal transference plan between $\mu_0$ and $\mu_1$, (\ref{eq:OptimalKant}) in fact holds for \emph{all} optimal transference plans $\pi$ between $\mu_0$ and $\mu_1$. In addition, in this case a transference plan $\pi$ is optimal iff it is supported on a $\sfd^2$-cyclically monotone set. A set $\Lambda \subset X \times X$ is said to be $c$-cyclically monotone if for any finite set of points $\{(x_{i},y_{i})\}_{i = 1,\ldots,N} \subset \Lambda$ it holds
$$
\sum_{i = 1}^{N} c(x_{i},y_{i}) \leq \sum_{i = 1}^{N} c(x_{i},y_{i+1}),
$$
with the convention that $y_{N+1} = y_{1}$.

\smallskip

As $(X,\sfd)$ is a complete and separable metric space then so is $(\mathcal{P}_2(X), W_2)$. 
Under these assumptions, it is known that $(X,\sfd)$ is geodesic if and only if $(\mathcal{P}_2(X), W_2)$ is geodesic. Recall that ${\rm e}_{t}$ denotes the (continuous) evaluation map at $t \in [0,1]$:
$$
  {\rm e}_{t} : \Geo(X) \ni \gamma\mapsto \gamma_t \in X. $$
A measure $\nu \in  \mathcal{P}(\Geo(X))$ is called an optimal dynamical plan if $({\rm e}_0,{\rm e}_1)_{\sharp} \nu$ is an optimal transference plan; it easily follows in that case that $[0,1] \ni t \mapsto ({\rm e}_t)_\sharp \nu$ is a geodesic in $(\mathcal{P}_2(X), W_2)$.
It is known that any geodesic $(\mu_t)_{t \in [0,1]}$ in $(\mathcal{P}_2(X), W_2)$ can be lifted to an optimal dynamical plan $\nu$ so that $({\rm e}_t)_\sharp \, \nu = \mu_t$ for all $t \in [0,1]$ (see for instance \cite[Theorem 2.10]{ambro:userguide}). 
 We denote by $\Opt(\mu_{0},\mu_{1})$ the space of all optimal dynamical plans $\nu$ so that $({\rm e}_i)_\sharp \, \nu = \mu_i$, $i=0,1$.
Consequently, whenever $(X,\sfd)$ is geodesic, the set $\Opt(\mu_{0},\mu_{1})$ is non-empty for all $\mu_0,\mu_1\in \mathcal{P}_2(X)$, and for any Kantorovich potential $\varphi$ associated to an optimal transference plan between $\mu_0$ and $\mu_1$, we have $\nu(G_\varphi) = 1$ for all $\nu \in \Opt(\mu_{0},\mu_{1})$.

In order to consider restrictions of optimal dynamical plans, for any $s,t \in [0,1]$ with $s \leq t$ we consider the restriction map
$$
\text{restr}^{t}_{s} : C([0,1]; X) \ni \gamma \mapsto \gamma\circ f^t_s \in C([0,1]; X), 
$$
where $f^{t}_{s} : [0,1] \to [s,t]$ is defined by $f^{t}_{s}(\tau) = s + (t-s) \tau$. During this work we will use the following facts: 
if $\nu\in \Opt(\mu_{0}, \mu_{1})$ then the restriction  
$(\text{restr}^{t}_{s})_{\sharp}\nu$ is still an optimal dynamical plan, now between $\mu_s$ and $\mu_t$ where $\mu_{r}:=(\ee_{r})_{\sharp}\nu$. 
Moreover, any probability measure $\nu' \in \P(\Geo(X))$ with $\supp(\nu') \subset \supp(\nu) ( \subset G_\varphi)$ is also an optimal dynamical plan, between $(\ee_{0})_{\sharp} \nu'$ and $(\ee_{1})_{\sharp} \nu'$.

\smallskip
On several occasions we will use the following standard lemma (whose proof is a straightforward adaptation of e.g. \cite[Lemma 4.4]{CM3}, relying on the Arzel\`a--Ascoli and Prokhorov theorems):

\begin{lemma} \label{lem:nu-compactness}
Assume that $(X,\sfd)$ is a Polish and proper space. Let $\set{\mu_0^i},\set{\mu_1^i} \subset \P_2(X)$ denote two sequences of probability measures weakly converging to $\mu^\infty_0 ,\mu^\infty_1 \in \P_2(X)$, respectively. Assume that $\nu^i \in \Opt(\mu_0^i,\mu_1^i)$. Then there exists a subsequence $\set{\nu^{i_j}}$ weakly converging to $\nu^\infty \in \Opt(\mu_0^\infty,\mu_1^\infty)$. 
\end{lemma}

\begin{definition*}[Essentially Non-Branching \mms]
A subset $G \subset \Geo(X)$ of geodesics is called non-branching if for any $\gamma^{1}, \gamma^{2} \in G$ the following holds:
$$
\exists t \in (0,1) \;\;\; \gamma^{1}_{s} = \gamma^{2}_{s} \;\;\; \forall s \in [0,t] 
\quad 
\Longrightarrow 
\quad 
\gamma^{1}_{s} = \gamma^{2}_{s} \;\;\; \forall s \in [0,1].
$$
$(X,\sfd)$ is called \emph{non-branching} if $\Geo(X)$ is non-branching. 
$(X,\sfd, \mm)$ is called \emph{essentially non-branching} \cite{rajasturm:branch} if for all $\mu_{0},\mu_{1} \in \mathcal{P}_{2}(X,\sfd,\mm)$, any $\nu \in \Opt(\mu_{0},\mu_{1})$ is concentrated on a Borel non-branching set $G\subset \Geo(X)$.
\end{definition*}

\noindent Recall that a measure $\nu$ on a measurable space $(\Omega,\F)$ is said to be concentrated on $A \subset \Omega$ if $\exists B \subset A$ with $B \in \F$ so that $\nu(\Omega \setminus B) = 0$. 

\medskip

\subsection{Curvature-Dimension Conditions}\label{Ss:CD}

We now turn to describe various synthetic conditions encapsulating generalized Ricci curvature lower bounds coupled with generalized dimension upper bounds. 

\begin{definition}[$\sigma_{K,\NN}$-coefficients] \label{def:sigma}
Given $K \in \Real$ and $\NN \in (0,\infty]$, define:
\[
D_{K,\NN} := \begin{cases}  \frac{\pi}{\sqrt{K/\NN}}  & K > 0 \;,\; \NN < \infty \\ +\infty & \text{otherwise} \end{cases} .
\]
In addition, given $t \in [0,1]$ and $0 < \theta < D_{K,\NN}$, define:
\[
\sigma^{(t)}_{K,\NN}(\theta) := \frac{\sin(t \theta \sqrt{\frac{K}{\NN}})}{\sin(\theta \sqrt{\frac{K}{\NN}})} = 
\begin{cases}   
\frac{\sin(t \theta \sqrt{\frac{K}{\NN}})}{\sin(\theta \sqrt{\frac{K}{\NN}})}  & K > 0 \;,\; \NN < \infty \\
t & K = 0 \text{ or } \NN = \infty \\
 \frac{\sinh(t \theta \sqrt{\frac{-K}{\NN}})}{\sinh(\theta \sqrt{\frac{-K}{\NN}})} & K < 0 \;,\; \NN < \infty 
\end{cases} ,
\]
and set $\sigma^{(t)}_{K,\NN}(0) = t$ and $\sigma^{(t)}_{K,\NN}(\theta) = +\infty$ for $\theta \geq D_{K,\NN}$. \\
\end{definition}

\begin{definition}[$\tau_{K,N}$-coefficients]
Given $K \in \Real$ and $N \in (1,\infty]$, define:
\[
\tau_{K,N}^{(t)}(\theta) := t^{\frac{1}{N}} \sigma_{K,N-1}^{(t)}(\theta)^{1 - \frac{1}{N}} .
\]
When $N=1$, set $\tau^{(t)}_{K,1}(\theta) = t$ if $K \leq 0$ and $\tau^{(t)}_{K,1}(\theta) = +\infty$ if $K > 0$. 
\end{definition}

The synthetic Curvature-Dimension condition $\CD(K,N)$ has been defined on a general \mms independently 
in several seminal works by Sturm and Lott--Villani: the case $N=\infty$ and $K \in \Real$ was defined in \cite{sturm:I} and \cite{lottvillani:metric}, the case $N \in [1,\infty)$ in \cite{sturm:II} for  $K \in \Real$ and in \cite{lottvillani:metric} for $K=0$ (and subsequently for $K \in \Real$ in \cite{LottVillani-WeakCurvature}). 
Our treatment in this work excludes the case $N=\infty$ (for which the globalization result we are after is in any case known \cite{sturm:I}). To exclude possible pathological behavior when $N=1$, we will always assume, unless otherwise stated, that $K \in \Real$ and $N \in (1,\infty)$.

We will use the following definition introduced in \cite{sturm:II}. Recall that given $N \in (1,\infty)$, the $N$-R\'enyi relative-entropy functional $\Eps_N : \P(X) \rightarrow [0,1]$ (since $\mm(X) = 1$) is defined as:
\[
\Eps_N(\mu) := \int \rho^{1 - \frac{1}{N}} d\mm ,
\]
where $\mu = \rho \mm + \mu^{\sing}$ is the Lebesgue decomposition of $\mu$ with $\mu^\sing \perp \mm$. It is known \cite{sturm:II} that $\Eps_N$ is upper semi-continuous with respect to the weak topology on $\P(X)$. 

\begin{definition}[$\CD(K,N)$] \label{def:CDKN}
A \mms $(X,\sfd,\mm)$ is said to satisfy $\CD(K,N)$ if for all $\mu_0,\mu_1 \in \P_2(X,\sfd,\mm)$, there exists $\nu \in \Opt(\mu_0,\mu_1)$ so that for all $t\in[0,1]$, $\mu_t := (\ee_t)_{\#} \nu \ll \mm$, and for all $N' \geq N$:
\begin{equation} \label{eq:CDKN-def}
\Eps_{N'}(\mu_t) \geq \int_{X \times X} \brac{\tau^{(1-t)}_{K,N'}(\sfd(x_0,x_1)) \rho_0^{-1/N'}(x_0) + \tau^{(t)}_{K,N'}(\sfd(x_0,x_1)) \rho_1^{-1/N'}(x_1)} \pi(dx_0,dx_1) ,
\end{equation}
where $\pi = (\ee_0,\ee_1)_{\sharp}(\nu)$ and $\mu_i = \rho_i \mm$, $i=0,1$. 
\end{definition}

\begin{remark}
When $\mm(X) < \infty$ as in our setting, it is known \cite[Proposition 1.6 (ii)]{sturm:II} that $\CD(K,N)$ implies $\CD(K,\infty)$, and hence the requirement $\mu_t \ll \mm$ for all intermediate times $t \in (0,1)$ is in fact superfluous, as it must hold automatically by finiteness of the Shannon entropy (see \cite{sturm:I,sturm:II}). 
\end{remark}

The following is a local version of $\CD(K,N)$:

\begin{definition}[$\CD_{loc}(K,N)$]\label{D:CDloc}
A \mms $(X,\sfd,\mm)$ is said to satisfy $\CD_{loc}(K,N)$ if for any $o \in \supp(\mm)$, there exists a neighborhood $X_o \subset X$ of $o$, so that for all $\mu_0,\mu_1 \in \P_2(X,\sfd,\mm)$ supported in $X_o$, there exists $\nu \in \Opt(\mu_0,\mu_1)$ so that for all $t\in[0,1]$, $\mu_t := (\ee_t)_{\#} \nu \ll \mm$, and for all $N' \geq N$, (\ref{eq:CDKN-def}) holds. 
\end{definition}
\noindent
Note that $(\ee_t)_{\sharp} \nu$ is not required to be supported in $X_o$ for intermediate times $t \in(0,1)$ in the latter definition. 

\medskip

The following pointwise density inequality is a known equivalent definition of $\CD(K,N)$ on essentially non-branching spaces (the equivalence follows by combining the results of \cite{CM3} and \cite{GSR:maps}, see the proof of Proposition \ref{P:MCP-density}):

\begin{definition}[$\CD(K,N)$ for essentially non-branching spaces] \label{def:CDKN-ENB}
An essentially non-branching \mms $(X,\sfd,\mm)$ satisfies $\CD(K,N)$ if and only if for all $\mu_0,\mu_1 \in \P_2(X,\sfd,\mm)$, there exists a unique $\nu \in \Opt(\mu_0,\mu_1)$, $\nu$ is induced by a map (i.e. $\nu = S_{\sharp}(\mu_0)$ for some map $S : X \rightarrow \Geo(X)$),  $\mu_t := (\ee_t)_{\#} \nu \ll \mm$ for all $t \in [0,1]$, and writing $\mu_t = \rho_t \mm$, we have for all $t \in [0,1]$:
\[
\rho_t^{-1/N}(\gamma_t) \geq  \tau_{K,N}^{(1-t)}(\sfd(\gamma_0,\gamma_1)) \rho_0^{-1/N}(\gamma_0) + \tau_{K,N}^{(t)}(\sfd(\gamma_0,\gamma_1)) \rho_1^{-1/N}(\gamma_1) \;\;\; \text{for $\nu$-a.e. $\gamma \in \Geo(X)$} .
\]
\end{definition}

\bigskip

The Measure Contraction Property $\MCP(K,N)$ was introduced independently by Ohta in \cite{Ohta1} and Sturm in \cite{sturm:II}. 
The idea is to only require the $\CD(K,N)$ condition to hold when $\mu_1$ degenerates to $\delta_o$, a delta-measure at $o \in \supp(\mm)$. However, there are several possible implementations of this idea. We start with the following one, which is a variation of the one used in \cite{CM3}:

\begin{definition}[$\MCPE(K,N)$] \label{D:MCPE}
A \mms $(X,\sfd,\mm)$ is said to satisfy $\MCPE(K,N)$ if for any $o \in \supp(\mm)$ and $\mu_0 \in \P_2(X,\sfd,\mm)$ with bounded support,
there exists $\nu \in \Opt(\mu_0, \delta_{o} )$, such that for all $t \in [0,1)$, if $\mu_t  := (\ee_t)_{\#} \nu$ then $\supp(\mu_t) \subset \supp(\mm)$, and:
\begin{equation} \label{eq:MCPE-def}
\Eps_{N}(\mu_t) \geq \int_X \tau_{K,N}^{(1-t)} (\sfd(x_0,o)) \rho_0^{1-\frac{1}{N}}(x_0) \mm(dx_0) ,
\end{equation}
where $\mu_0 = \rho_0 \mm$. 
\end{definition}

The variant proposed in \cite{Ohta1} is as follows:

\begin{definition}[$\MCP(K,N)$] \label{D:Ohta1}
A \mms $(X,\sfd,\mm)$ is said to satisfy $\MCP(K,N)$ if for any $o \in \supp(\mm)$ and  $\mu_0 \in \P_2(X,\sfd,\mm)$ of the form $\mu_0 = \frac{1}{\mm(A)} \mm\llcorner_{A}$ for some Borel set $A \subset X$ with $0 < \mm(A) < \infty$,
there exists $\nu \in \Opt(\mu_0, \delta_{o} )$ such that:
\begin{equation} \label{eq:MCP-def}
\frac{1}{\mm(A)} \mm \geq (\ee_{t})_{\sharp} \big( \tau_{K,N}^{(1-t)}(\sfd(\gamma_{0},\gamma_{1}))^{N} \nu(d \gamma) \big) \;\;\; \forall t \in [0,1] .
\end{equation}
\end{definition}

\begin{remark} \label{R:BonnetMyers}
Note that in \cite{Ohta1} it was assumed in addition that $\supp(\mm) = X$ and that $(X,\sfd)$ is a length-space, but (\ref{eq:MCP-def}) was only required to hold for $A \subset B(o, D_{K,N-1})$ if $K>0$; both our version and the one from \cite{Ohta1} imply that the diameter of $\supp(\mm)$ is bounded above by $D_{K,N-1}$ (this follows in our version since $\tau_{K,N}(\theta) = +\infty$ if $\theta \geq D_{K,N-1}$, and by \cite[Theorem 4.3]{Ohta1} in the version from \cite{Ohta1}), and also that $\supp(\mm)$ is a geodesic-space (see Lemma \ref{L:proper-support} below), and therefore both versions are ultimately equivalent. 
\end{remark}

\medskip

When either the $\MCP(K,N)$ or $\MCPE(K,N)$ conditions hold for a given $o \in \supp(\mm)$, we will say that the space satisfies the corresponding condition \emph{with respect to $o$}. 

\begin{remark} \label{rem:supp-nu}
The $\CD(K,N)$, $\CD_{loc}(K,N)$, $\MCPE(K,N)$ and $\MCP(K,N)$ conditions all ensure that for all $t \in [0,1]$, $\supp((\ee_t)_\sharp \nu) \subset \supp(\mm)$ for the appropriate $\nu \in \OptGeo(\mu_0,\mu_1)$ appearing in the corresponding definition. Consequently, for a fixed dense countable set of times $t \in (0,1)$, $\gamma_t \in \supp(\mm)$ for $\nu$-a.e. $\gamma \in \Geo(X)$; since $\supp(\mm)$ is closed, this in fact holds for all $t \in [0,1]$, and hence $\gamma \in \Geo(\supp(\mm))$ for $\nu$-a.e. $\gamma \in \Geo(X)$, i.e. $\supp(\nu) \subset \Geo(\supp(\mm))$. It follows that $(X,\sfd,\mm)$ satisfies $\CD(K,N)$, $\CD_{loc}(K,N)$, $\MCPE(K,N)$ or $\MCP(K,N)$ iff $(\supp(\mm),\sfd,\mm)$ does.
\end{remark}

The following simple lemma will be useful for quickly establishing that $(\supp(\mm),\sfd)$ is proper and geodesic:

\begin{lemma}\label{L:proper-support} 
Let $(X,\sfd,\mm)$ be a \mms verifying $\CD(K,N)$, $\MCPE(K,N)$ or $\MCP(K,N)$. Then $(\supp(\mm),\sfd)$ is a Polish, proper and geodesic space. 
The same holds for $\CD_{loc}(K,N)$ if $(\supp(\mm),\sfd)$ is assumed to be a length space.
\end{lemma}

\begin{proof}
As $\text{supp}(\mm) \subset X$ is closed, $(\text{supp}(\mm),\sfd)$ is Polish. It was shown in \cite[Lemma 2.5, Theorem 5.1]{Ohta1} for $\MCP(K,N)$ (and hence $\MCPE(K,N)$) and in \cite[Corollary 2.4]{sturm:II} for $\CD(K,N)$ that these conditions imply a doubling condition, so that every closed bounded ball in $(\text{supp}(\mm),\sfd)$ is totally bounded. Together with completeness, this already implies that the latter space is proper. 
By Remark \ref{rem:supp-nu}, $(\supp(\mm),\sfd,\mm)$ verifies the same corresponding condition as $(X,\sfd,\mm)$.   
In particular, if $(X,\sfd,\mm)$ and hence $(\supp(\mm),\sfd,\mm)$ verifies $\CD(K,N)$, $\MCPE(K,N)$ or $\MCP(K,N)$, then for any $x,y \in \supp(\mm)$, there is at least one geodesic in $\supp(\mm)$ from $B(y,\eps) \cap \supp(\mm)$ to $x$; together with properness and completeness, this already implies that $(\text{supp}(\mm),\sfd)$ is geodesic.
On the other hand, if $(X,\sfd,\mm)$ and hence $(\supp(\mm),\sfd,\mm)$ verifies $\CD_{loc}(K,N)$, the above argument shows that $(\supp(\mm),\sfd)$ is complete and locally compact. Together with the assumption that the latter space is a length-space, the Hopf-Rinow theorem implies that it is proper and geodesic. 
\end{proof}

\begin{lemma} \label{lem:CD-MCPE-MCP}
The following chain of implications is known:
\[
\CD(K,N) \Rightarrow \MCPE(K,N) \Rightarrow \MCP(K,N) . 
\]
\end{lemma}
\begin{proof}
By Remark \ref{rem:supp-nu}, we may reduce to the case $\supp(\mm) = X$. 
Fixing $\mu_0 \ll \mm$ with bounded support and $o \in X$, let $\nu^\eps$ be an element of $\Opt(\mu_0,\mu_1^\eps)$ satisfying the $\CD(K,N)$ condition for $\mu_1^\eps = \mm(B(o,\eps))^{-1}\mm \llcorner_{B(o,\eps)}$. By Lemma \ref{lem:nu-compactness} (which applies since the space is proper by Lemma \ref{L:proper-support}),  $\set{\nu^\eps}$ has a converging subsequence to $\nu^0 \in \Opt(\mu_0,\delta_o)$ as $\eps \rightarrow 0$. The upper semi-continuity of $\Eps_N$ and the continuity of the evaluation map $\ee_t$ ensure that $\nu^0$ satisfies the $\MCPE(K,N)$ condition (\ref{eq:MCPE-def}). The second implication follows by the arguments of \cite[Section 5]{R2012} (without any types of essential non-branching assumptions). 
\end{proof}
\begin{remark}
We will show in Proposition \ref{P:MCP-density} that for essentially non-branching spaces, $\MCP(K,N)$ implies back $\MCPE(K,N)$. We remark that for non-branching spaces, the implication $\CD(K,N) \Rightarrow \MCP(K,N)$ was first proved in \cite{sturm:II}. 
\end{remark}

\medskip
Many additional useful results on the structure of $W_{2}$-geodesics can be obtained just from the $\MCP$ condition.
The following has been shown in \cite[Theorem 1.1 and Appendix]{CM3} (when $\supp(m) = X$; the formulation below is immediately obtained from Remark \ref{rem:supp-nu}):

\begin{theorem}[\cite{CM3}]\label{T:optimalmapMCP}
Let $(X,\sfd,\mm)$ be an essentially non-branching \mms satisfying $\MCP(K,N)$. Given any pair $\mu_{0},\mu_{1} \in \mathcal{P}_{2}(X)$ with 
$\mu_{0} \ll \mm$ and $\supp(\mu_{1}) \subset \supp(\mm)$, the following holds:  
\begin{itemize}
\item[-] there exists a unique $\nu \in \Opt(\mu_{0},\mu_{1})$ and hence a unique optimal transference plan between $\mu_0$ and $\mu_1$; \item[-] there exists a map $S : X \supset \dom(S) \to \Geo(X)$ such that $\nu = S_{\sharp} \mu_{0}$; 
\item[-] for any $t \in [0,1)$ the measure $(\ee_{t})_{\sharp} \nu$ is absolutely continuous with respect to $\mm$.
\end{itemize}
\end{theorem}

The following is a standard corollary of the fact that the optimal dynamical plan is induced by a map (see e.g. the comments after \cite[Theorem 1.1]{GSR:maps}); as we could not find a reference, we sketch the proof for completeness.

\begin{corollary} \label{C:injectivity} 
With the same assumptions as in Theorem \ref{T:optimalmapMCP}, the unique optimal transference plan $\nu$ is concentrated on a (Borel) set $G \subset \Geo(X)$, so that for all $t \in [0,1)$, the evaluation map $\ee_t|_G : G \rightarrow X$ is injective. In particular, for any Borel subset $H \subset G$: \[
(\ee_t)_\sharp(\nu\llcorner_{H}) = (\ee_t)_\sharp(\nu)\llcorner_{\ee_t(H)} \;\;\; \forall t \in [0,1) .
\]
\end{corollary}
\begin{proof}[Sketch of proof]
First, we claim the existence of $X_1 \subset X$ with $\mu_0(X_1) = 1$, so the for all $x \in X_1$, there exists a unique $\gamma \in G_\varphi$ with $\gamma_0 = x$. Otherwise, if $A \subset X$ is a set of positive $\mu_0$-measure where this is violated, there are at least two distinct geodesics in $G_\varphi$ emanating from every $x \in A$. As these geodesics must be different at some rational time in $(0,1)$, it follows that there exists a rational $\bar t \in (0,1)$ and $B \subset A$ still of positive $\mu_0$-measure so that both pairs of geodesics emanating from $x$ are different at time $\bar t$ for all $x \in B$. Consider $\bar \mu_0 = \mu_0 \llcorner_{B} / \mu_0(B) \ll \mm$, and transport to time $\bar t$ half of its mass along one geodesic and the second half along the other one (see e.g. the proof of \cite[Theorem 5.1]{CM3}). The latter transference plan is optimal but is not induced by a map, yielding a contradiction. 

Now denote $G := S(X_1)$ (and hence $\nu(G) = 1$), so that the injectivity of $\ee_0|_G$ is already guaranteed. To see the injectivity of $\ee_t|_G$ for all $t \in (0,1)$, suppose in the contrapositive the existence of $\gamma^{1}, \gamma^{2} \in G$ with $\gamma^{1}_{t} = \gamma^{2}_{t}$. Denoting by $\eta$ the gluing of $\gamma^{1}$ restricted to $[0,t]$ with $\gamma^{2}$ restricted to $[t,1]$, it follows by $\sfd^{2}$-cyclic monotonicity (see e.g. the proof of \cite[Lemma 2.6]{sturm:loc} or that of Lemma \ref{lem:rigidity}) that $\eta \in G_{\f}$ with $\eta_{0} = \gamma_{0}^{1}$ and $\eta \neq \gamma^{1}$. But this is in contradiction to the definition of $X_1$, thereby concluding the proof. 
\end{proof}

\bigskip

\subsection{Disintegration Theorem}

We include here a version of the Disintegration Theorem that we will use. 
We will follow \cite[Appendix A]{biacar:cmono} where a  self-contained approach (and a proof) of the Disintegration Theorem in countably generated measure spaces can be found. 
An even more general version of the Disintegration Theorem can be found in \cite[Section 452]{Fre:measuretheory4}.

\medskip

Recall that given a measure space $(X,\mathscr{X},\mm)$, a set $A \subset X$ is called $\mm$-measurable if $A$ belongs to the completion of the $\sigma$-algebra $\mathscr{X}$, generated by adding to it all subsets of null $\mm$-sets; similarly, a function $f : (X,\mathscr{X},\mm) \rightarrow \Real$ is called $\mm$-measurable if all of its sub-level sets are $\mm$-measurable.

\begin{definition}[Disintegation on sets] \label{def:disintegration}
\label{defi:dis}
Let $(X,\mathscr{X},\mm)$ denote a measure space. 
Given any family $\set{X_\alpha}_{\alpha \in Q}$ of subsets of $X$, a \emph{disintegration of $\mm$ on $\set{X_\alpha}_{\alpha \in Q}$} is a measure-space structure 
$(Q,\mathscr{Q},\qq)$ and a map
\[
Q \ni \alpha \longmapsto \mm_{\alpha} \in \mathcal{P}(X,\mathscr{X})
\]
so that:
\begin{enumerate}
\item for $\qq$-a.e. $\alpha \in Q$, $\mm_\alpha$ is concentrated on $X_\alpha$; \item  for all $B \in \mathscr{X}$, the map $\alpha \mapsto \mm_{\alpha}(B)$ is $\qq$-measurable;  
\item for all $B \in \mathscr{X}$, $\mm(B) = \int \mm_{\alpha}(B)\, \qq(d\alpha)$.
\end{enumerate}
The measures $\mm_\alpha$ are referred to as \emph{conditional probabilities}.
\end{definition}

Given a measurable space $(X, \mathscr{X})$ and a function $\QQ : X \to Q$, with $Q$ a general set, we endow $Q$ with the \emph{push forward $\sigma$-algebra} $\mathscr{Q}$ of $\mathscr{X}$:
$$
C \in \mathscr{Q} \quad \Longleftrightarrow \quad \QQ^{-1}(C) \in \mathscr{X},
$$
i.e. the biggest $\sigma$-algebra on $Q$ such that $\QQ$ is measurable. 
Moreover, given a measure  $\mm$  on $(X,\mathscr{X})$, define a   
measure $\qq$ on $(Q,\mathscr{Q})$  by pushing forward $\mm$ via $\QQ$, i.e. $\qq := \QQ_\sharp \, \mm$.  

\begin{definition}[Consistent and Strongly Consistent Disintegation] 
\label{defi:dis}
A \emph{disintegration} of $\mm$ \emph{consistent with} $\QQ : X \rightarrow Q$ is a map: 
\[ Q \ni \alpha \longmapsto \mm_{\alpha} \in \mathcal{P}(X,\mathscr{X})
\] such that the following requirements hold:
\begin{enumerate}
\item  for all $B \in \mathscr{X}$, the map $\alpha \mapsto \mm_{\alpha}(B)$ is $\qq$-measurable; 
\item for all $B \in \mathscr{X}$ and $C \in \mathscr{Q}$, the following consistency condition holds:
$$
\mm \left(B \cap \QQ^{-1}(C) \right) = \int_{C} \mm_{\alpha}(B)\, \qq(d\alpha).
$$
\end{enumerate}
A disintegration of $\mm$ is called \emph{strongly consistent with respect to $\QQ$} if in addition:
\begin{enumerate}
\item[(3)] for $\qq$-a.e. $\alpha \in Q$, $\mm_\alpha$ is concentrated on $\QQ^{-1}(\alpha)$; \end{enumerate}
\end{definition}

The above general scheme fits with the following situation: given a measure space $(X,\mathscr{X},\mm)$, suppose 
a \emph{partition} of $X$ is given into \emph{disjoint} sets $\{ X_{\alpha}\}_{\alpha \in Q}$ so that $X = \cup_{\alpha \in Q} X_\alpha$. 
Here $Q$ is the set of indices and $\QQ : X \to Q$ is the quotient map, i.e.  
$$
\alpha = \QQ(x) \iff x \in X_{\alpha}.
$$
We endow $Q$ with the quotient $\sigma$-algebra $\mathscr{Q}$ and the quotient measure $\qq$ as described above, obtaining 
the quotient measure space $(Q, \mathscr{Q}, \qq)$. 
When a disintegration $\alpha \mapsto \mm_\alpha$ of $\mm$ is (strongly) consistent with the quotient map $\QQ$, we will simply say that it is (strongly) consistent with the partition. Note that any disintegration $\alpha \mapsto \mm_\alpha$ of $\mm$ on a \emph{partition} $\{ X_{\alpha}\}_{\alpha \in Q}$ (as in Definition \ref{def:disintegration}) is automatically strongly consistent with the partition (as in Definition \ref{defi:dis}), and vice versa. 

\medskip

We now formulate the Disintegration Theorem (it is formulated for probability measures but clearly holds for any finite non-zero measure):

\begin{theorem}[Theorem A.7, Proposition A.9 of \cite{biacar:cmono}] \label{T:disintegrationgeneral}
Assume that $(X,\mathscr{X},\mm)$ is a countably generated probability space and that $\set{X_{\alpha}}_{\alpha \in Q}$ is a partition of $X$. 
\medskip

Then the quotient probability space $(Q, \mathscr{Q},\qq)$ is essentially countably generated and 
there exists an essentially unique disintegration $\alpha \mapsto \mm_{\alpha}$ consistent with the partition. 
\medskip

If in addition $\mathscr{X}$ contains all singletons, then the disintegration is strongly consistent if and only if there exists a $\mm$-section $S_{\mm} \in \mathscr{X}$ of the partition such that the $\sigma$-algebra on $S_{\mm}$ induced by the quotient-map contains the trace $\sigma$-algebra $\mathscr{X} \cap S_{\mm} := \set{A \cap S_{\mm} ; A \in \mathscr{X}}$. 
\end{theorem}

Let us expand on the statement of Theorem \ref{T:disintegrationgeneral}. Recall that a $\sigma$-algebra $\mathcal{A}$ is \emph{countably generated} if there exists a countable family of sets so that $\mathcal{A}$ coincides with the smallest $\sigma$-algebra containing them. On the measure space $(Q, \mathscr{Q},\qq)$, the $\sigma$-algebra $\mathscr{Q}$ is called \emph{essentially countably generated} if there exists a countable family of sets $Q_{n} \subset Q$ such that for any $C \in \mathscr{Q}$ there exists $\hat C \in \hat{\mathscr{Q}}$, where $\hat{\mathscr{Q}}$ is the $\sigma$-algebra generated by $\{ Q_{n} \}_{n \in \N}$, such that $\qq(C\, \Delta \, \hat C) = 0$.

Essential uniqueness is understood above in the following sense: if $\alpha\mapsto \mm^{1}_{\alpha}$ and $\alpha\mapsto \mm^{2}_{\alpha}$ 
are two consistent disintegrations with the partition then $\mm^{1}_{\alpha}=\mm^{2}_{\alpha}$ for $\qq$-a.e. $\alpha \in Q$.

Finally, a set $S \subset X$ is a section for the partition $X = \cup_{\alpha \in Q}X_{\alpha}$ if for any $\alpha \in Q$, $S \cap X_\alpha$ is a singleton $\set{x_\alpha}$.  
By the axiom of choice, a section $S$ always exists, and we may identify $Q$ with $S$ via the map $Q \ni \alpha \mapsto x_\alpha \in S$.  
A set $S_{\mm}$ is an $\mm$-section if there exists $Y \in \mathscr{X}$ with $\mm(X \setminus Y) = 0$ such that the partition $Y = \cup_{\alpha\in Q_\mm} (X_{\alpha} \cap Y)$ has section $S_{\mm}$, where $Q_{\mm} = \set{\alpha \in Q ; X_{\alpha} \cap Y \neq \emptyset}$. As $\qq = \QQ_{\sharp} \mm$, clearly $\qq(Q \setminus Q_{\mm}) = 0$. 
As usual, we identify between $Q_{\mm}$ and $S_{\mm}$, so that now $Q_{\mm}$ carries two measurable structures: $\mathscr{Q} \cap Q_{\mm}$ (the push-forward of $\mathscr{X} \cap Y$ via $\QQ$), and also $\mathscr{X} \cap S_{\mm}$ via our identification. The last condition of Theorem \ref{T:disintegrationgeneral} is that $\mathscr{Q} \cap Q_{\mm} \supset \mathscr{X} \cap S_{\mm}$, i.e. that the restricted quotient-map $\QQ|_{Y} : (Y,\mathscr{X} \cap Y) \rightarrow (S_\mm , \mathscr{X} \cap S_{\mm})$ is measurable, so that the full quotient-map $\QQ : (X,\mathscr{X}) \rightarrow (S , \mathscr{X} \cap S)$ is $\mm$-measurable.

\medskip

We will typically apply the Disintegration Theorem to $(E,\B(E),\mm\llcorner_{E})$, where $E \subset X$ is an $\mm$-measurable subset (with $\mm(E) > 0$) of the \mms $(X,\sfd,\mm)$. 
As our metric space is separable, $\B(E)$ is countably generated, and so Theorem \ref{T:disintegrationgeneral} applies. 
In particular, when $Q \subset \Real$, $E$ is a closed subset of $X$, the partition elements $X_\alpha$ are closed and the quotient-map $\QQ : E \rightarrow Q$ is known to be Borel (for instance, this is the case when $\QQ$ is continuous), \cite[Theorem 5.4.3]{Srivastava} guarantees the existence of a Borel section $S$ for the partition so that $\QQ : E \rightarrow S$ is Borel measurable, thereby guaranteeing by Theorem \ref{T:disintegrationgeneral} the existence of an essentially unique disintegration strongly consistent with $\QQ$.

\bigskip

\section{$L^{1}$ Optimal Transportation Theory}\label{S:L1-OT}

In this section we recall various results from the theory of $L^1$ optimal-transport which are relevant to this work, and add some new information we will subsequently require. 
We refer to \cite{ambro:lecturenote, AmbrosioPratelliL1, biacava:streconv, cava:MongeRCD,  EvansGangbo,FeldmanMcCann-Manifold, Klartag, Vil:topics} for more details.

\subsection{Preliminaries}\label{Ss:preli}

To any $1$-Lipschitz function $u : X \to \R$ there is a naturally associated $\sfd$-cyclically monotone set: 

\begin{equation}\label{E:Gamma}  
\Gamma_{u} : = \{ (x,y) \in X\times X : u(x) - u(y) = \sfd(x,y) \}.
\end{equation}
Its transpose is given by $\Gamma^{-1}_{u}= \{ (x,y) \in X \times X : (y,x) \in \Gamma_{u} \}$. We define the \emph{transport relation} $R_u$ and the \emph{transport set} $\mathcal{T}_{u}$, as:
\begin{equation}\label{E:R}
R_{u} := \Gamma_{u} \cup \Gamma^{-1}_{u} ~,~ \mathcal{T}_{u} := P_{1}(R_{u} \setminus \{ x = y \}) ,
\end{equation}
where $\{ x = y\}$ denotes the diagonal $\{ (x,y) \in X^{2} : x=y \}$ and $P_{i}$ the projection onto the $i$-th component. Recall that $\Gamma_u(x) = \set{y \in X \; ;\; (x,y) \in \Gamma_u}$ denotes the section of $\Gamma_u$ through $x$ in the first coordinate, and similarly for $R_u(x)$ (through either coordinates by symmetry). 
Since $u$ is $1$-Lipschitz, $\Gamma_{u}, \Gamma^{-1}_{u}$ and $R_{u}$ are closed sets, and so are $\Gamma_u(x)$ and $R_u(x)$. 
Consequently $\mathcal{T}_{u}$ is a projection of a Borel set and hence analytic; it follows that it is universally measurable, and in particular, $\mm$-measurable \cite{Srivastava}.

\medskip

The following is immediate to verify (see \cite[Proposition 4.2]{ambro:lecturenote}):
\begin{lemma}\label{L:cicli}
Let $(\gamma_0,\gamma_1) \in \Gamma_{u}$ for some  $\gamma \in \Geo(X)$. Then $(\gamma_{s},\gamma_{t}) \in \Gamma_{u}$ 
for all $0\leq s \leq t \leq 1$.
\end{lemma}

Also recall the following definitions, introduced in \cite{cava:MongeRCD}:
\begin{align*}
	A_{+}	: = 	&~\{ x \in \mathcal{T}_{u} : \exists z,w \in \Gamma_{u}(x), (z,w) \notin R_{u} \}, \nonumber \\ 
	A_{-}		: = 	&~\{ x \in \mathcal{T}_{u} : \exists z,w \in \Gamma^{-1}_{u}(x), (z,w) \notin R_{u} \}.
\end{align*}
$A_{\pm}$ are called the sets of forward and backward branching points, respectively. Note that both $A_{\pm}$ are analytic sets;
for instance:
$$
A_{+} = P_{1} (\{ (x,z,w) \in \mathcal{T}_{u}\times X \times X \colon (x,z), (x,w) \in \Gamma_{u}, \ (z,w) \notin R_{u} \}),
$$
showing that $A_{+}$ is a projection of an analytic set and therefore analytic.
If $x \in A_{+}$ and $(y,x) \in \Gamma_{u}$ necessarily also $y \in A_{+}$ (as $\Gamma_{u}(y) \supset \Gamma_{u}(x)$ by the triangle inequality);
similarly, if $x \in A_{-}$ and $(x,y) \in \Gamma_{u}$ then necessarily $y \in A_{-}$.

Consider the \emph{non-branched transport set} 
$$
\T_{u}^{b} : = \T_{u} \setminus (A_{+} \cup A_{-}),
$$
which belongs to the sigma-algebra $\sigma(\mathcal{A})$ generated by analytic sets and is therefore $\mm$-measurable. Define the \emph{non-branched transport relation}:
\[
R_u^b := R_u \cap (\T_u^b \times \T_u^b) .
\]
In was shown in \cite{cava:MongeRCD} (cf. \cite{biacava:streconv}) that $R_u^b$ is an equivalence relation over $\T_{u}^{b}$ and that for any $x \in \T_{u}^{b}$, $R_{u}(x) \subset (X,\sfd)$ is isometric to a closed interval in $(\Real,\abs{\cdot})$. 

\begin{remark}
Note that even if $x \in \T_{u}^{b}$, the transport ray $R_{u}(x)$ need not be entirely contained in $\T_{u}^{b}$. 
However, we will soon prove that  almost every transport ray (with respect to an appropriate measure) has interior part contained in $\T_{u}^{b}$. 
\end{remark}

It will be very useful to note that whenever the space $(X,\sfd)$ is proper (for instance when $(X,\sfd,\mm)$ verifies $\MCP(K,N)$ and $\supp(\mm) = X$), $\T_{u}$ and $A_{\pm}$ are $\sigma$-compact sets:  
indeed writing $R_{u} \setminus \{ x= y\} = \cup_{\ve >0} R_{u} \setminus \{ \sfd(x,y) > \ve \}$ it follows that $R_{u} \setminus \{ x= y\}$ is $\sigma$-compact. Hence
$ \mathcal{T}_{u}$ is $\sigma$-compact. Moreover:
$$
A_{+} = P_{1} \Big( \{ (x,z,w) \in \T_{u} \times (R_{u})^{c} \colon (x,z), (x,w) \in \Gamma_{u} \} \Big);
$$
since $(R_{u})^{c}$ is open and open sets are $F_{\sigma}$ in metric spaces, it follows that $\{ (x,z,w) \in \T_{u} \times (R_{u})^{c} \colon (x,z), (x,w) \in \Gamma_{u} \}$ is $\sigma$-compact 
and therefore $A_{+}$ is $\sigma$-compact; the same applies to $A_{-}$. Consequently, $\T_u^b$ and $R_u^b$ are Borel. 

\medskip

Now, from the first part of the Disintegration Theorem \ref{T:disintegrationgeneral} applied to $(\T_u^b , \B(\T_u^b), \mm\llcorner_{\T_{u}^{b}})$, we obtain an essentially unique disintegration of $\mm\llcorner_{\T_{u}^{b}}$ consistent with the partition 
of $\T_{u}^{b}$ given by the equivalence classes $\set{R_u^b(\alpha)}_{\alpha \in Q}$ of $R_{u}^{b}$: 
$$
\mm\llcorner_{\T_{u}^{b}} = \int _{Q} \mm_{\alpha}\,\qq(d\alpha ),
$$
with corresponding quotient space $(Q, \mathscr{Q},\qq)$ ($Q \subset \T_u^b$ may be chosen to be any section of the above partition). 
The next step is to show that the disintegration is strongly consistent. By the  Disintegration Theorem, this is equivalent to the existence of a $\mm\llcorner_{\T_{u}^{b}}$-section $\bar Q \in \B(\T_u^b)$ (which by a mild abuse of notation we will call $\mm$-section), 
such that the quotient map associated to the partition is $\mm$-measurable, 
where we endow $\bar Q$ with the trace $\sigma$-algebra. This has already been shown in \cite[Proposition 4.4]{biacava:streconv} in the framework of non-branching metric spaces; 
since its proof does not use any non-branching assumption, we can conclude that:
$$
\mm\llcorner_{\T_{u}^{b}} = \int _{Q} \mm_{\alpha}\,\qq(d\alpha ), \quad \text{and for } \qq-\text{a.e. } \alpha \in Q, \quad \mm_{\alpha}(R_u^b(\alpha)) =1,
$$
where now $Q \supset \bar Q \in \B(\T_u^b)$ with $\bar Q$ an $\mm$-section for the above partition (and hence $\qq$ is concentrated on $\bar Q$).
For a more constructive approach under the additional assumption of properness of the space, see also \cite[Proposition 4.8]{cava:overview}.

\medskip

A-priori the non-branched transport set $\T_u^b$ can be much smaller than $\T_u$. However, under fairly general assumptions one can prove that the sets $A_{\pm}$ of forward and backward branching are both $\mm$-negligible. In \cite{cava:MongeRCD} this was shown for a \mms $(X,\sfd,\mm)$ verifying $\RCD(K,N)$ and $\supp(\mm) = X$. 
The proof only relies on the following two properties which hold for the latter spaces (see also \cite{cava:overview}):
\begin{itemize}
\item[-] $\supp(\mm) = X$. 
\item[-] Given $\mu_{0}, \mu_{1} \in \P_{2}(X)$ with $\mu_{0}\ll \mm$, there exists a unique optimal transference plan for the $W_{2}$-distance and it is induced by an optimal transport \emph{map} .
\end{itemize}
By Theorem \ref{T:optimalmapMCP} these properties are also verified for an essentially non-branching \mms $(X,\sfd,\mm)$ satisfying $\MCP(K,N)$ and $\supp(\mm)= X$. We summarize the above discussion in:

\begin{corollary} \label{C:disintMCP}
Let $(X,\sfd,\mm)$ be an essentially non-branching \mms satisfying $\MCP(K,N)$ and $\supp(X) = \mm$. Then for any $1$-Lipschitz function $u : X \to \R$, we have $\mm(\T_u \setminus \T_u^b) = 0$. In particular, we obtain the following essentially unique disintegration $(Q,\mathscr{Q},\qq)$ of $\mm\llcorner_{\T_{u}} = \mm\llcorner_{\T^b_{u}}$ strongly consistent with the partition 
of $\T_{u}^{b}$ given by the equivalence classes $\set{R_u^b(\alpha)}_{\alpha \in Q}$ of $R_{u}^{b}$: 
\begin{equation}\label{E:disintMCP}
\mm\llcorner_{\T_{u}} = \int_{Q} \mm_{\alpha} \,\qq(d\alpha), \quad \text{and for } \qq-\text{a.e. } \alpha \in Q, \quad \mm_{\alpha}(R_u^b(\alpha)) =1 .
\end{equation}
Here $Q$ may be chosen to be a section of the above partition so that $Q \supset \bar Q \in \B(\T_u^b)$ with $\bar Q$ an $\mm$-section with $\mm$-measurable quotient map. In particular, $\mathscr{Q} \supset \B(\bar Q)$ and $\qq$ is concentrated on $\bar Q$. \end{corollary}

\begin{remark}
By modifying the definitions of $A_+,A_-$ to only reflect branching inside $\supp(\mm)$, it is possible to remove the assumption that $\supp(X) = \mm$, but we refrain from this extraneous generality here. 
\end{remark}

\begin{remark}\label{R:cutlocus}
If we consider $u = \sfd(\cdot,o)$, it is easy to check that the set $A_{+}$ coincides with the cut locus $C_{o}$, i.e.~the set of those $z \in X$ such that there exists at least two distinct geodesics starting at $z$ and ending in $o$. 
Hence the previous corollary implies that for any $o \in X$, the cut locus has $\mm$-measure zero: $\mm(C_{o}) = 0$.
This in particular implies that an  essentially non-branching \mms verifying $\MCP(K,N)$ and $\supp(\mm) = X$ also supports a local $(1,1)$-weak Poincar\'e inequality, see \cite{MVR}.
\end{remark}

\subsection{Maximality of transport rays on non-branched transport-set}

It is elementary to check that $\Gamma_{u}$ induces a partial order relation on $X$: 
\[
y \leq_{u} x \;\;\; \Leftrightarrow \;\;\; (x,y) \in \Gamma_{u} .
\]
Note that by definition:
\begin{align*}
x \in A_+ ~,~ y \geq_u x \;\; & \Rightarrow \;\; y \in A_+ \;\; , \\
x \in A_- ~,~ y \leq_u x \;\;  & \Rightarrow \;\; y \in A_-  \;\; .
\end{align*}
Recall that for any $x \in \T_{u}^{b}$, $(R_{u}(x),\sfd)$ is isometric to a closed interval in $(\Real,\abs{\cdot})$. This isometry induces a total ordering on $R_u(x)$ which must coincide with either $\leq_u$ or $\geq_u$, implying that $(R_u(x), \leq_u)$ is totally ordered. 

\begin{lemma}\label{L:Newlemma}
For any $x \in \T_{u}^{b}$, $(R_u^b(x) = R_u(x) \cap \T_u^b,\sfd)$ is isometric to an interval in $(\R,|\cdot|)$.
\end{lemma}
\begin{proof}
Consider $z,w \in R_{u}(x) \cap \T_{u}^{b}$; as $(R_u(x), \leq_u)$ is totally ordered, assume without loss of generality that $z \leq_u w$. 
Given $y \in R_u(x)$ with $z \leq_u y \leq_u w$, we must prove that $y \in \T_{u}^{b}$. 
Indeed, since $w \geq_u y$ and $w \notin A_{+}$, necessarily $y \notin A_{+}$, and since $z \leq_u y$ and $z \notin A_{-}$, necessarily $y \notin A_{-}$. Hence $y \in \T_{u}^{b}$ and the claim follows. 
\end{proof}

Recall that given a partially ordered set, a chain is a totally ordered subset. A chain is called maximal if it is maximal with respect to inclusion. We introduce the following:

\begin{definition}[Transport Ray] \label{def:transport-ray}
A maximal chain $R$ in $(X,\sfd,\leq_u)$ is called a \emph{transport ray} if it is isometric to a closed interval $I$ in $(\Real,\abs{\cdot})$ of positive (possibly infinite) length. \end{definition}

In other words, a transport ray $R$ is the image of a closed non-null geodesic $\gamma$ parametrized by arclength on $I$ so that the function $u \circ \gamma$ is affine with slope $1$ on $I$, and so that $R$ is maximal with respect to inclusion.

\begin{lemma} \label{L:RayInTub}
Given $x \in \T_u^b$, $R$ is a transport ray passing through $x$ if and only if $R = R_u(x)$.
\end{lemma}
\begin{proof}
Recall that for any $x \in \T_{u}^{b}$, $(R_u(x), \sfd,\leq_u)$ is order isometric to a closed interval in $(\Real,\abs{\cdot})$. 
As $R_u(x)$ is by definition maximal in $X$ with respect to inclusion, it follows that it must be a transport ray. \\
Conversely, note that for any transport ray $R$ we always have $R \subset \cap_{w \in R} R_u(w)$. Indeed, for any $w,z \in R$, we have $z \leq_u w$ or $z \geq_u w$, and hence by definition $(w,z) \in R_u$ so that $z \in R_u(w)$. If $x \in R \cap \T_u^b$, we already showed above that $R_{u}(x)$ is a transport ray. Since $R \subset R_u(x)$ and $R$ is assumed to be maximal with respect to inclusion, it follows that necessarily $R = R_u(x)$. 
\end{proof}

\begin{corollary} \label{C:RayInTub}
If $R_1$ and $R_2$ are two transport rays which intersect in $\T_u^b$ then they must coincide. 
\end{corollary}

\medskip

In this subsection, we reconcile between the crucial maximality property of $R_u(\alpha)$ which we will require for the definition of $\CD^1$ in the next section, and the fact that the disintegration in (\ref{E:disintMCP}) is with respect to (the possibly non-maximal) $R_u^b(\alpha) = R_u(\alpha) \cap \T_u^b$. 
We will show that under $\MCP$, for $\qq$-a.e. $\alpha$, the only parts of $R_{u}(\alpha)$ which are possibly not contained in $\T_{u}^{b}$ are its end points -- this fact is the main new result of this section. 

\medskip

To rigorously state this new observation, we recall the classical definition of initial and final points, $\mathfrak{a}$ and $\mathfrak{b}$, respectively: 
\begin{align*}
\mathfrak{a} : = &~\{ x \in \T_{u} \colon  \nexists y \in \T_{u}, \ (y,x) \in \Gamma_{u}, \ y \neq x  \}, \\
\mathfrak{b} : = &~\{ x \in \T_{u} \colon  \nexists y \in \T_{u}, \ (x,y) \in \Gamma_{u}, \ y \neq x  \}.
\end{align*}
Note that:
$$
\mathfrak{a} = \T_{u} \setminus P_{1}\big(\{ \Gamma_u \setminus \{ x= y \} \}) ,$$
so $\mathfrak{a}$ is the difference of analytic sets and consequently belongs to $\sigma(\mathcal{A})$; similarly for $\mathfrak{b}$.
As in the previous subsection, whenever $(X,\sfd)$ is proper, $\mathfrak{a}, \mathfrak{b}$ are in fact Borel sets.

\begin{theorem}[Maximality of transport rays on non-branched transport-set]\label{T:endpoints}
Let $(X,\sfd, \mm)$ be an essentially non-branching \mms verifying $\MCP(K,N)$ and $\supp(\mm) = X$. Let $u : (X,\sfd) \rightarrow \Real$ be
any $1$-Lipschitz function, with \eqref{E:disintMCP} the associated disintegration of $\mm \llcorner_{\T_u}$.  \\
Then there exists 
 $\hat Q \subset Q$ such that $\qq(Q \setminus \hat Q) = 0$ and for any $\alpha \in \hat Q$ it holds: 
$$
R_{u}(\alpha) \setminus \T_{u}^{b} \subset \mathfrak{a} \cup \mathfrak{b}.
$$
In particular, for every $\alpha \in \hat Q$: 
\[
R_u(\alpha) = \overline{R_u^b(\alpha)} \supset R_u^b(\alpha) \supset \mathring{R}_u(\alpha) ,
\]
(with the latter interpreted as the relative interior). 
\end{theorem}

\begin{proof}
\hfill

\medskip
{\bf Step 1.}
Consider the $\mm$-section $\bar Q$ from Corollary \ref{C:disintMCP} so that $Q \supset \bar Q \in \mathcal{B}(\T_{u}^{b})$, $\mathscr{Q} \supset \B(\bar Q)$ and 
$\qq(Q \setminus \bar Q) = 0$. Consider the set:
$$
Q_{1} : = \{ \alpha \in \bar Q \colon R_{u}(\alpha) \setminus \T_{u}^{b} \nsubseteq \mathfrak{a}\cup \mathfrak{b} \}. 
$$
The claim will be proved once we show that $\qq(Q_{1})=0$. 
First, observe that 
$$
Q_{1} = \bar Q \cap P_{1} \Big( R_{u} \cap \big( \T_{u}^{b} \times (( A_{+}\setminus \mathfrak{a}) \cup (A_{-}\setminus \mathfrak{b})) \big) \Big),
$$
and therefore $Q_{1} \subset \bar Q$ is analytic; since $\mathscr{Q} \supset \B(\bar Q)$, it follows that $Q_1$ is $\qq$-measurable. Now suppose by contradiction that $\qq(Q_{1})> 0$. 

We can divide $Q_{1}$ into two sets:
$$
Q_{1}^{+} : = \{ \alpha \in Q_1 \colon \Gamma_{u}(\alpha) \setminus \T_{u}^{b} \nsubseteq  \mathfrak{b} \}, \quad
Q_{1}^{-} : =  \{ \alpha \in Q_1 \colon \Gamma^{-1}_{u}(\alpha) \setminus \T_{u}^{b} \nsubseteq \mathfrak{a} \} .
$$
Since $Q_{1} = Q_{1}^{+} \cup Q_{1}^{-}$, without any loss in generality let us assume $\qq(Q_{1}^{+}) > 0$, and for ease of notation assume further that $Q_{1}^{+} = Q_{1}$.

Hence, for any $\alpha \in Q_{1}$, there exists $z \in \Gamma_{u}(\alpha)$ such that $z \notin \T_{u}^{b}$ and $z \notin \mathfrak{b}$; note that necessarily $z \in A_{-}$.
Recall that for all $\alpha \in Q$, $R_{u}(\alpha)$ and hence $\Gamma_u(\alpha)$ are isometric via the map $u$ to closed intervals, and hence $\Gamma_u(\alpha) \setminus (\{ \alpha \} \cup \mathfrak{b})$ is isometric to an open interval. 
Since $\Gamma_u(\alpha) \cap  \T_{u}^{b}$ is isometric to an interval and contains $\alpha$, it follows that for $\alpha \in Q_1$, there exist distinct $a_\alpha,b_\alpha \in \Gamma_{u}(\alpha) \setminus \T_{u}^{b}$ so that:
\[
(u(b_{\alpha}), u(a_{\alpha})) \subset u (\Gamma_{u}(\alpha) \setminus \T_{u}^{b}) 
\]
is a non-empty open interval.  Moreover, we may select $a_{\alpha}$ and $b_\alpha$ to be $\qq$-measurable functions of $Q_1$. To see this, consider the set 
$\Sigma : = \{ (\alpha, x ,y) \in Q_{1} \times  \Gamma_{u}  \colon  x \in A_{-} , \  (\alpha,x) \in \Gamma_{u},\ \sfd(x,y) > 0\}$, 
and observe that it is analytic (being the intersection of analytic sets), and that $P_1(\Sigma) = Q_1$. By von Neumann's selection Theorem (see \cite[Theorem 5.5.2]{Srivastava}), there exists a 
$\sigma(\mathcal{A})$-measurable selection of $\Sigma$: $$
Q_{1} \ni \alpha \to (a_{\alpha},b_{\alpha}),
$$
and so in particular these functions are $\qq$-measurable. 
It follows that
$$
Q_{1} \ni \alpha \to u(a_{\alpha}), \qquad Q_{1} \ni \alpha \to u(b_{\alpha}),
$$
are also $\sigma(\mathcal{A})$-measurable and hence $\qq$-measurable. Possibly restricting $Q_{1}$, by Lusin's Theorem we can also assume that the above functions are continuous. 
  
\medskip

{\bf Step 2.}
By Fubini's Theorem
$$
0 <  \int_{Q_{1}} (u(a_{\alpha})-u(b_{\alpha}))  \, \qq(d\alpha) = \int_{\R} \qq \Big(\{\alpha \in Q_{1} \colon u (b_{\alpha})<  t < u (a_{\alpha})\} \Big)\, dt .
$$
Hence there exists $c \in \R$ and $Q_{1,c} \subset Q_{1}$ with $\qq(Q_{1,c}) > 0$,  such that for any $\alpha \in Q_{1,c}$ it holds $c \in (u(b_{\alpha}), u(a_{\alpha}))$; 
in particular for any $\alpha \in Q_{1,c}$ there exists a unique $z_{\alpha} \in \Gamma_{u}(\alpha)$ such that $u(z_{\alpha}) =c$. 
Furthermore, we can assume that $Q_{1,c}$ is compact, and hence by continuity of $u(a_\alpha)$ it follows that: 
$$
\exists \ve > 0 \;\;\; \forall \alpha \in Q_{1,c} \;\;\; u(a_{\alpha}) - c > \ve .
$$
Then define the following set: 
$$
\Lambda : = \{ (\alpha, x,z) \in Q_{1,c} \times \Gamma_{u} \colon (\alpha,x) \in R_{u}^{b}, \ u(z) =c\}.
$$
Recall that $R_u^b$ is Borel since $(X,\sfd)$ is proper, and therefore $\Lambda$ is Borel. Note by the aforementioned discussion that $P_1(\Lambda) = Q_{1,c}$.
Also note that for $(\alpha,x,z) \in \Lambda$, since $R_{u}(\alpha)$ is isometric to a closed interval, necessarily $z = z_{\alpha}$. Finally, we claim that $P_{2,3} (\Lambda)$ is $\sfd^{2}$-cyclically monotone: for $(x_{1},z_{1}), (x_{2},z_{2}) \in P_{2,3} (\Lambda)$ observe that 
$$
\sfd(x_{1},z_{1}) = u(x_{1}) - u(z_{1}) = u(x_{1}) - c =  u(x_{1}) - u(z_{2})   \leq \sfd(x_{1},z_{2}).
$$
Hence for $\{(x_{i}, z_{i})\}_{i \leq n } \subset P_{2,3} (\Lambda)$, setting $z_{n+1} = z_{1}$,
$$
\sum_{i\leq n} \sfd^{2}(x_{i},z_{i}) \leq \sum_{i\leq n} \sfd^{2}(x_{i},z_{i+1}),
$$
and the monotonicity follows.
We can then define a function $T$ by imposing $\gr(T) = P_{2,3}(\Lambda)$; note that $P_{2,3}(\Lambda)$ is analytic and therefore $T$ is Borel measurable 
(see \cite[Theorem 4.5.2]{Srivastava}).

\medskip

{\bf Step 3.} 
Consider now the measure 
$$
\eta_{0} : = \int_{Q_{1,c}} \mm_{\alpha}\,\qq(d\alpha),  
$$
and since $\qq(Q_{1,c})> 0$ it follows that $\eta_{0}(X) > 0$; note that $\eta_{0}$ is concentrated on 
$\dom(T) = \cup_{\alpha \in Q_{1,c}} R_{u}^b(\alpha)$.
Hence there exists $x \in X$ and $r > 0$ such that $\eta_{0}(B_{r}(x)) > 0$, and we redefine $\eta_0$ to be the probability measure obtained by conditioning $\eta_0$ to $B_r(x)$. Clearly 
$\eta_{0} \ll \mm$.
Finally we define $\eta_{1} : = T_{\sharp} \,\eta_{0}$. By {\bf Step 2} and Theorem \ref{T:optimalmapMCP}, the map $T$ is the unique optimal transport map between $\eta_{0}$ and $\eta_{1}$ for the $W_{2}$-distance (as it is supported on a $\sfd^2$-cyclically monotone set).
Consider moreover $\nu$ the unique element of $\Opt(\eta_{0},\eta_{1})$ -- then $\nu$-a.e. $\gamma$ it holds that:
$$
\gamma_{0} \in \dom(T) \cap B_{r}(x) \subset \T_u^b \;\;  ,  \;\; u(\gamma_{1}) = c \;\; , \;\; (\gamma_0,\gamma_1) \in \Gamma_u .
$$
It follows in particular by Lemma \ref{L:cicli} that $\gamma_{s} \in \Gamma_{u}(\gamma_{0})$ for all $s \in [0,1]$. 

Recalling that $u(a_{\alpha}) - c > \ve$ for all $\alpha \in Q_{1,c}$, that $a_{\alpha} \leq M$ by continuity on $Q_{1,c}$, and that the support of $\eta_0$ is bounded, it follows that there exists $\bar t \in (0,1)$ such that $\nu$-a.e. $\gamma_{\bar t} \in \mathcal{T}_{u} \setminus \T_{u}^{b} \subset A_{+} \cup A_{-}$.
Since $\mm(A_{+} \cup A_{-}) = 0$, necessarily $(\ee_{\bar t})_{\sharp} \nu \perp \mm$, but this is in contradiction with the assertion of Theorem \ref{T:optimalmapMCP} that $(\ee_{\bar t})_{\sharp} \nu \ll \mm$  since $\eta_{0} \ll \mm$ and $\bar t < 1$. 
The claim follows.
\end{proof}

\bigskip
\bigskip

\section{The $\CD^{1}$ Condition}\label{S:CD1}

In this section we introduce the $\CD^1(K,N)$ condition, which plays a cardinal role in this work. As a first step towards understanding this new condition, we show that it always implies $\MCPE(K,N)$ (and $\MCP(K,N)$), without requiring any types of non-branching assumptions. By analogy, we also introduce the $\MCP^1(K,N)$ condition, which may be of independent interest. 

\subsection{Definitions of $\CD^{1}$ and $\MCP^1$}\label{Ss:definitions}

We first assume that $\supp(\mm) = X$. Note that we do not assume that the transport rays $\{X_{\alpha}\}_{\alpha \in Q}$ below are disjoint or have disjoint relative interiors, in an attempt to obtain a useful definition also for \mms's which may have significant branching. However, throughout most of this work, we will typically assume in addition that the space is essentially non-branching, in which case an equivalent definition will be presented in Proposition \ref{P:CD1-ENB} below. 

\begin{definition}[$\CD^1_{u}(K,N)$ when $\supp(\mm) = X$] \label{D:CD1-u}
Let $(X,\sfd,\mm)$ denote a \mms with $\supp(\mm) = X$,  let $K \in \Real$ and $N \in [1,\infty]$, and let $u : (X,\sfd) \rightarrow \Real$ denote a $1$-Lipschitz function. 
$(X,\sfd,\mm)$ is said to verify the $\CD^{1}_{u}(K,N)$ condition if there exists a family $\{X_{\alpha}\}_{\alpha \in Q} \subset X$, such that: 
\begin{enumerate}
\item There exists a disintegration of $\mm\llcorner_{\mathcal{T}_{u}}$ on $\{X_{\alpha}\}_{\alpha \in Q}$:
\begin{equation}\label{E:New-disintCD1}
\mm\llcorner_{\mathcal{T}_{u}} = \int_{Q} \mm_{\alpha} \, \qq(d\alpha), \quad \text{with } \quad \mm_{\alpha}(X_{\alpha}) = 1,  \text{ for } \qq\text{-a.e. }\alpha \in Q .
\end{equation}
\item For $\qq$-a.e. $\alpha \in Q$, $X_\alpha$ is a transport ray for $\Gamma_u$ (recall Definition \ref{def:transport-ray}).
\item For $\qq$-a.e. $\alpha \in Q$, $\mm_\alpha$ is supported on $X_\alpha$. 
\item For $\qq$-a.e. $\alpha \in Q$, the \mms $(X_{\alpha}, \sfd,\mm_{\alpha})$ verifies $\CD(K,N)$.
\end{enumerate}
\end{definition}

We take this opportunity to define an analogous variant of $\MCP$:
\begin{definition}[$\MCP^1_{u}(K,N)$ when $\supp(\mm) = X$] \label{D:MCP1-u}
Let $(X,\sfd,\mm)$ denote a \mms with $\supp(\mm) = X$,  let $K \in \Real$ and $N \in [1,\infty]$, let $o \in X$ and denote the $1$-Lipschitz function $u := \sfd(\cdot,o)$. 
$(X,\sfd,\mm)$ is said to verify the $\MCP^{1}_{u}(K,N)$ condition if there exists a family $\{X_{\alpha}\}_{\alpha \in Q} \subset X$, such that conditions (1)-(3) above hold, together with:
\begin{enumerate}
\item[(4')] For $\qq$-a.e. $\alpha \in Q$, the \mms $(X_{\alpha}, \sfd,\mm_{\alpha})$ verifies $\MCP(K,N)$ with respect to $o \in X_\alpha$. 
\end{enumerate}
\end{definition}

\begin{remark} \label{R:o-dist}
Note that when $u= \sfd(\cdot,o)$ then necessarily $\mathcal{T}_{u} = X$ (if $X$ is not a singleton). In addition $(x,o) \in \Gamma_{u}$ for any $x \in X$, and hence by maximality of a transport ray, we must have $o \in X_{\alpha}$ for $\qq$-a.e. $\alpha \in Q$, and by condition (3) we deduce that $o \in \supp(\mm_\alpha)$ for $\qq$-a.e. $\alpha \in Q$. As $\CD(K,N)$ implies $\MCP(K,N)$ (in the one-dimensional case this is a triviality), we obviously see that $\CD^1_u(K,N)$ implies $\MCP^1_u(K,N)$ for all $u = \sfd(\cdot,o)$. 
\end{remark}

We will focus on a particular class of $1$-Lipschitz functions. 

\begin{definition*}[Signed Distance Function]
Given a continuous function $f : (X,\sfd) \to \R$ so that $\set{f = 0} \neq \emptyset$, the function:
\begin{equation}\label{E:levelsets}
d_{f} : X \to \R, \qquad d_{f}(x) : = \text{dist}(x, \{ f = 0 \} ) sgn(f),
\end{equation}
is called the signed distance function (from the zero-level set of $f$).  
\end{definition*}

\begin{lemma} \label{lem:df-Lip}
$d_f$ is $1$-Lipschitz on $\set{f \geq 0}$ and $\set{f \leq 0}$. If $(X,\sfd)$ is a length space, then $d_f$ is $1$-Lipschitz on the entire $X$. 
\end{lemma}
\begin{proof}
Given $x,y \in X$ with $f(x) f(y) \geq 0$, the assertion follows by the usual triangle inequality, valid for any metric space:
\[
\abs{d_f(x) - d_f(y)} = \abs{\text{dist}(x,\set{f=0}) - \text{dist}(y,\set{f=0})} \leq \sfd(x,y) .
\]
When $f(x) f(y) < 0$, and given $\eps > 0$, let $\gamma : [0,1] \rightarrow X$ denote a continuous path with $\gamma_0 = x$, $\gamma_1=y$ and $\len(\gamma) \leq \sfd(x,y) + \eps$. By continuity, it follows that there exists $t \in (0,1)$ so that $f(\gamma_t) = 0$. It follows that:
\[
\abs{d_f(x) - d_f(y)} =\text{dist}(x,\set{f = 0}) + \text{dist}(y,\set{f=0}) \leq \sfd(x,\gamma_t) + \sfd(y,\gamma_t) \leq \len(\gamma) \leq \sfd(x,y) + \eps.
\]
As $\eps > 0$ was arbitrary, the assertion is proved. 
\end{proof}

\begin{remark}\label{R:bigset}
To extend Remark \ref{R:o-dist} to more general signed distance functions, we will need to require that $(X,\sfd)$ is proper, and in that case $\mathcal{T}_{d_{f}} \supset X \setminus \{f =0\}$. Indeed, given $x \in X \setminus \{f =0 \}$, consider the distance minimizing $z \in \{ f= 0 \}$ (by compactness of bounded sets). Then $(x,z) \in R_{d_{f}}$ and as $x \neq z$ it follows that $x \in \mathcal{T}_{d_{f}}$. 
\end{remark}

We now remove the restriction that $\supp(\mm) = X$ and introduce the main new definitions of this work:

\begin{definition}[$\CD^1_{Lip}(K,N)$, $\CD^1(K,N)$ and $\MCP^1(K,N)$] \label{D:New-CD1}
Let $(X,\sfd,\mm)$ denote a \mms and let $K \in \Real$ and $N \in [1,\infty]$.
\begin{itemize}
\item[-] $(X,\sfd,\mm)$ is said to verify the $\CD^{1}_{Lip}(K,N)$ condition if $(\supp(\mm),\sfd,\mm)$ verifies $\CD^1_{u}(K,N)$ for all $1$-Lipschitz functions $u : (\supp(\mm),\sfd) \rightarrow \Real$. 
\item[-] $(X,\sfd,\mm)$ is said to verify the $\CD^{1}(K,N)$ condition if $(\supp(\mm),\sfd,\mm)$ verifies $\CD^1_{d_f}(K,N)$ for all continuous functions $f : (\supp(\mm),\sfd) \rightarrow \Real$ so that $\set{f=0} \neq \emptyset$ and $d_f : (\supp(\mm) , \sfd) \rightarrow \Real$ is $1$-Lipschitz. 
\item[-] $(X,\sfd,\mm)$ is said to verify $\MCP^{1}(K,N)$ if $(\supp(\mm),\sfd,\mm)$ verifies $\MCP_u^1(K,N)$ for all functions $u(x) = \sfd(x,o)$ with $o \in \supp(\mm)$. 
\end{itemize}
\end{definition}

\begin{remark}
Clearly $\CD^1_{Lip}(K,N) \Rightarrow \CD^1(K,N) \Rightarrow \MCP^1(K,N)$ in view of Remark \ref{R:o-dist}.
Note that we do not a-priori know that $d_f$ is $1$-Lipschitz, since we do not know that $(\supp(\mm),\sfd)$ is a length-space (see Lemma \ref{lem:df-Lip}); nevertheless, we will shortly see that the $\CD^1(K,N)$ condition implies that $(\supp(\mm),\sfd)$ must be a geodesic space, and hence the sentence ``so that $d_f$ is $1$-Lipschitz" is in fact redundant. 
\end{remark}

\begin{remark} \label{R:CD1-localizes}
By definition, the $\CD^1_{Lip}$, $\CD^1$ and $\MCP^1$ conditions hold for $(X,\sfd,\mm)$ iff they hold for $(\supp(\mm),\sfd,\mm)$. It is also possible to introduce a definition of $\CD^1_{u}$ and $\MCP^1_u$ which applies to $(X,\sfd,\mm)$ directly, without passing through $(\supp(\mm),\sfd,\mm)$ - this would involve requiring that the transport rays $\set{X_\alpha}$ are maximal \emph{inside $\supp(\mm)$}, and in the case of $\CD^1_u$ would only apply to functions $u$ which are $1$-Lipschitz on $\supp(\mm)$ (these may be extended to the entire $X$ by McShane's theorem). Our choice to use a tautological approach is motivated by the analogous situation for the more classical $W_2$ definitions of curvature-dimension (see Remark \ref{rem:supp-nu}) and is purely for convenience, so as not to overload the definitions. 
\end{remark}

\bigskip

\subsection{$\MCP^{1}$ implies $\MCP_\eps$}

\begin{proposition}\label{P:MCP}
Let $(X,\sfd,\mm)$ be a \mms verifying $\MCP^{1}(K,N)$ with $K \in \R$ and $N \in (1,\infty)$ (in particular, this holds if it verifies $\CD^1_{Lip}(K,N)$ or $\CD^1(K,N)$). Then it verifies $\MCPE(K,N)$. 
\end{proposition}

\begin{proof}
We will show that $(\supp(\mm),\sfd,\mm)$ satisfies $\MCPE(K,N)$, and consequently so will $(X,\sfd,\mm)$. 
By Remark \ref{R:CD1-localizes}, we may therefore assume that $\supp(\mm) = X$. 
Fix any $o \in X$ and consider the $1$-Lipschitz function $u (x) : = \sfd(x,o)$. From $\MCP^{1}(K,N)$ and Remark \ref{R:bigset} we deduce the existence of a disintegration of $\mm$ on $\T_u = X$ along a family of Borel sets $\{ X_{\alpha}\}_{\alpha \in Q}$: $$
\mm = \int_{Q} \mm_{\alpha} \, \qq(d\alpha), \quad \mm_{\alpha} (X_{\alpha}) = 1, \ \text{for } \qq- \text{a.e. } \alpha \in Q, 
$$
so that $X_\alpha$ is a transport ray for $\Gamma_u$, $\mm_\alpha$ is supported on $X_\alpha$ and $(X_{\alpha}, \sfd, \mm_{\alpha})$ verifies $\MCP(K,N)$ with respect to $o \in X_\alpha$, for $\qq$-a.e. $\alpha \in Q$.

Now consider any $\mu_0 \in \P(X)$ with $\mu_0 \ll \mm$, so that $\rho_0 := \frac{d\mu_0}{d\mm}$ has bounded support. 
By measurability of the disintegration, the function $Q \ni \alpha \mapsto z_\alpha := \int \rho_0(x) \mm_{\alpha}(dx)$ is $\qq$-measurable, and hence $\bar{Q} := \set{ \alpha \in Q \; ; \; z_\alpha \in (0,\infty)}$ is $\qq$-measurable. 
Clearly $\int_{\bar{Q}} z_\alpha \qq(d\alpha) = \int_{Q} z_\alpha \qq(d\alpha) = 1$ since $z_\alpha < \infty$ for $\qq$-a.e. $\alpha \in Q$. 

Define $\mu_0^\alpha := \frac{1}{z_\alpha} \rho_0 \mm_\alpha \in \mathcal{P}(X_\alpha)$ for all $\alpha \in \bar{Q}$. 
Since for $\qq$-a.e. $\alpha \in \bar{Q}$, the one-dimensional (non-branching) $(X_{\alpha},\sfd)$ contains $o$, 
there exists a unique element $\nu^\alpha$ of $\Opt(\mu_{0}^{\alpha}, \delta_{o}) \cap \mathcal{P}(\Geo(X_{\alpha}))$ where $\Geo(X_{\alpha})$ denotes the space of geodesics in $X_{\alpha}$. Define then:
\begin{equation} \label{E:big-nu}
\nu : = \int_{\bar{Q}} \nu^{\alpha} z_\alpha \,\qq(d\alpha),
\end{equation}
and observe that $(\ee_{0})_{\sharp} \nu = \rho_0 \mm = \mu_0$ and $(\ee_{1})_{\sharp} \nu = \delta_{o}$.
To conclude that $\nu \in \Opt(\mu_0, \delta_{o})$  we must show that $t \mapsto (\ee_{t})_{\sharp} \nu = : \mu_{t}$ is a $W_{2}$-geodesic. Indeed, for any $0 \leq s < t \leq 1$, consider the transference plan $(\ee_{s},\ee_{t})_{\sharp} \nu$ between $\mu_s$ and $\mu_t$, yielding:
\begin{align*}
W_{2}^{2}(\mu_{s}, \mu_{t}) &~ \leq  \int_{\bar{Q}}  \int_{X_{\alpha} \times X_{\alpha}} \sfd^{2}(x,y) (\ee_{s},\ee_{t})_{\sharp} \nu^{\alpha} (dxdy) z_\alpha \, \qq(d\alpha) \\&~ =  \int_{\bar{Q}}   (t-s)^{2} \int_{X_{\alpha} \times X_{\alpha}} \sfd^{2}(x,y) (\ee_{0},\ee_{1})_{\sharp} \nu^{\alpha} (dxdy) z_\alpha \, \qq(d\alpha) \\&~ = (t-s)^{2} \int_{\bar{Q}}   \int_{X_{\alpha}} \sfd^{2}(x,o)  \mu_0^\alpha(dx) z_\alpha \,\qq(d\alpha) \\ &~ = (t-s)^{2} \int_{Q}   \int_{X_{\alpha}} \sfd^{2}(x,o)  \rho_0(x) \mm_\alpha(dx) \,\qq(d\alpha) \\ 
&~ =  (t-s)^{2} \int_{X}  \sfd^{2}(x,o) \rho_0(x) \mm(dx) \\
&~ = (t-s)^{2} W_{2}^{2}(\mu_0,\delta_{o}). 
\end{align*}
By the triangle inequality, it follows that $t \mapsto \mu_{t}$ must indeed be a geodesic in $(\P_2(X),W_2)$. Note that this property is particular to transportation to a delta measure.

It remains to establish the $\MCPE$ inequality of Definition \ref{D:MCPE}. Fix $t \in (0,1)$, and recall that for $\qq$-a.e. $\alpha \in \bar{Q}$, the (one-dimensional, non-branching) $(X_{\alpha}, \sfd, \mm_{\alpha})$ verifies $\MCP(K,N)$ (and hence $\MCPE(K,N)$), and as $\mu_0^\alpha \ll \mm_\alpha$ and $o \in \supp(\mm_\alpha)$, in particular $\mu_t^\alpha := (\ee_t)_{\sharp}(\nu^\alpha) \ll \mm_\alpha$. Applying $\ee_t$ to both sides of (\ref{E:big-nu}), it follows that $\mu_t = (\ee_t)_{\sharp}(\nu) \ll \mm$. Writing $\mu_t = \rho_t \mm$ and $\mu_t^\alpha = \rho_t^\alpha \mm_\alpha$ for $\qq$-a.e. $\alpha \in \bar{Q}$, the $\MCPE$ condition implies that:
\begin{equation} \label{eq:big-MCPE}
\int_X (\rho_t^{\alpha}(x))^{1 - \frac{1}{N}} \mm_\alpha(dx) \geq \int_X \tau^{(1-t)}_{K,N}(d(x,o)) \brac{\frac{\rho_0(x)}{z_\alpha}}^{1-\frac{1}{N}} \mm_\alpha(dx) \;\;\; \forall \qq-\text{a.e. } \alpha \in \bar{Q} .
\end{equation}
In addition, the application of $\ee_t$ to both sides of (\ref{E:big-nu}) yields the following disintegration:
\begin{equation} \label{eq:big-disint}
\rho_t \mm = \int_{\bar{Q}} \rho_t^\alpha z_\alpha \mm_\alpha \qq(d\alpha) .
\end{equation}

Now consider the set $Y = \set{\rho_t > 0}$, and note that by (\ref{eq:big-disint}):
\begin{equation} \label{eq:big-vanishing}
\int_{X \setminus Y} \rho_t^\alpha(x) \mm_\alpha(dx) = 0 \;\;\; \forall \qq-\text{a.e. } \alpha \in \bar{Q} .
\end{equation}
Integrating  (\ref{eq:big-disint}) against $\rho_t^{-\frac{1}{N}}$ on $Y = \set{\rho_t > 0}$, applying H\"{o}lder's inequality on the interior integral for $\qq$-a.e. $\alpha \in \bar{Q}$,  using (\ref{eq:big-vanishing}), employing the one-dimensional $\MCPE$ inequality  (\ref{eq:big-MCPE}) and canceling $z_\alpha$, and finally applying H\"{o}lder's inequality again on the exterior integral, we obtain:
\begin{align*}
& \int_X \rho_t(x)^{1-\frac{1}{N}} \mm(dx)  = \int_Y \rho_t(x)^{1-\frac{1}{N}} \mm(dx) = \int_{\bar{Q}} \int_Y \rho_t^\alpha(x) \rho_t(x)^{-\frac{1}{N}} \mm_\alpha(dx) z_\alpha \qq(d\alpha) \\
 & \geq \int_{\bar{Q}} \brac{\int_Y (\rho_t^\alpha(x))^{1 - \frac{1}{N}} \mm_\alpha(dx)}^{\frac{N}{N-1}} \brac{\int_Y \rho_t(x)^{\frac{N-1}{N}} \mm_{\alpha}(dx)}^{-\frac{1}{N-1}} z_\alpha \qq(d \alpha) \\
&  = \int_{\bar{Q}} \brac{\int_X (\rho_t^\alpha(x))^{1 - \frac{1}{N}} \mm_\alpha(dx)}^{\frac{N}{N-1}} \brac{\int_X \rho_t(x)^{\frac{N-1}{N}} \mm_{\alpha}(dx)}^{-\frac{1}{N-1}} z_\alpha \qq(d \alpha) \\
& \geq \int_{\bar{Q}} \brac{\int_X \tau^{(1-t)}_{K,N}(d(x,o)) \rho_0(x)^{1-\frac{1}{N}} \mm_\alpha(dx)}^{\frac{N}{N-1}} \brac{\int_X \rho_t(x)^{\frac{N-1}{N}} \mm_{\alpha}(dx)}^{-\frac{1}{N-1}} \qq(d \alpha) \\
& \geq \brac{\int_{\bar{Q}}  \int_X \tau^{(1-t)}_{K,N}(d(x,o)) \rho_0(x)^{1-\frac{1}{N}} \mm_\alpha(dx) \qq(d \alpha)}^{\frac{N}{N-1}} \brac{\int_{\bar{Q}} \int_X \rho_t(x)^{\frac{N-1}{N}} \mm_{\alpha}(dx) \qq(d\alpha)}^{-\frac{1}{N-1}} \\
& \geq   \brac{\int_{Q}  \int_X \tau^{(1-t)}_{K,N}(d(x,o)) \rho_0(x)^{1-\frac{1}{N}} \mm_\alpha(dx) \qq(d \alpha)}^{\frac{N}{N-1}} \brac{\int_{Q} \int_X \rho_t(x)^{\frac{N-1}{N}} \mm_{\alpha}(dx) \qq(d\alpha)}^{-\frac{1}{N-1}} \\
& = \brac{\int_X \tau^{(1-t)}_{K,N}(d(x,o)) \rho_0(x)^{1-\frac{1}{N}} \mm(dx)}^{\frac{N}{N-1}} \brac{\int_X \rho_t(x)^{1 - \frac{1}{N}} \mm(dx)}^{-\frac{1}{N-1}} ,
\end{align*}
where the last inequality above follows since $\rho_0 \mm_\alpha = 0$ for $\alpha \in Q \setminus \bar{Q}$ and since the exponent on the second term is negative. Note that we applied H\"{o}lder's inequality above in reverse form:
\[
\int \abs{f}^{\alpha} \abs{g}^{\beta} d\omega \geq (\int \abs{f} d\omega)^{\alpha} (\int \abs{g} d\omega)^{\beta} ,
\]
which is valid as soon as $\alpha + \beta = 1$, $\beta < 0$, regardless of whether or not $\abs{g} > 0$ $\omega$-a.e.. 

Rearranging terms above and raising to the power of $\frac{N-1}{N}$, the desired inequality follows:
\[
\int_X \rho_t(x)^{1-\frac{1}{N}} \mm(dx) \geq \int_X \tau^{(1-t)}_{K,N}(d(x,o)) \rho_0(x)^{1-\frac{1}{N}} \mm(dx) .
\]

\end{proof}

\begin{remark}
Note that the above proof shows that, not only does it hold that $\supp(\mu_t) \subset \supp(\mm)$ for all $t\in [0,1)$, as required in the definition of $\MCPE(K,N)$, but in fact $\mu_t \ll \mm$. 
\end{remark}

\begin{remark} \label{rem:CD1-MCP}
Recalling that $\MCPE(K,N)$ always implies $\MCP(K,N)$, we deduce that $\MCP^1(K,N)$ implies $\MCP(K,N)$. In fact, a direct proof of the latter implication is elementary. Indeed, let $A \subset X$ be any Borel set with $0 < \mm(A) < \infty$,
and denote $\mu_0 = \frac{1}{\mm(A)} \mm\llcorner_{A}$. 
Recall that for $\qq$-a.e. $\alpha \in \bar{Q}$, $o \in X_\alpha$, $\supp(\mm_\alpha) = X_\alpha$ and
$(X_{\alpha},\sfd, \mm_{\alpha})$ verifies $\MCP(K,N)$. Defining $\nu$ as in (\ref{E:big-nu}) and continuing with the notation used there, 
it follows by  uniqueness of $\nu^\alpha$ and the $\MCP$ condition with respect to the point $o \in X_\alpha$, that for any Borel set $B \subset X$: $$
\mm_{\alpha} (B)  \geq \int_{\ee_{t}^{-1}(B)} \tau_{K,N}^{(1-t)} (\sfd(\gamma_{0},\gamma_{1}))^{N} \mm_{\alpha}(A) \nu^{\alpha}(d\gamma), 
$$
for $\qq$-a.e. $\alpha \in \bar{Q}$. Integrating over $\bar{Q}$ we obtain
\begin{align*}
\mm(B)  &~ \geq \int_{\bar{Q}} \mm_\alpha(B) \qq(d\alpha) \\
&~ \geq \int_{\ee_{t}^{-1}(B)} \int_{\bar{Q}} \tau_{K,N}^{(1-t)} (\sfd(\gamma_{0},\gamma_{1}))^{N}  \mm_{\alpha}(A) \nu^{\alpha}(d\gamma) \,\qq(d\alpha)  \\
&~ = \int_{\ee_{t}^{-1}(B)} \tau_{K,N}^{(1-t)} (\sfd(\gamma_{0},\gamma_{1}))^{N} \mm(A) \nu(d\gamma) ,
\end{align*}
and the claim follows. \end{remark}

As a consequence, we immediately obtain from Lemmas \ref{L:proper-support} and \ref{lem:df-Lip}:

\begin{corollary}\label{C:structure}
Let $(X,\sfd,\mm)$ be a \mms verifying $\CD^{1}(K,N)$ with $K \in \R$ and $N \in (1,\infty)$. 
Then $(\supp(\mm),\sfd)$ is a Polish, proper and geodesic space. 
In particular, for any continuous function $f : (\supp(\mm),\sfd) \to \R$ with $\{ f = 0 \} \neq \emptyset$, the function $d_{f} : (\supp(\mm),\sfd) \to \R$ is $1$-Lipschitz.
\end{corollary}

\subsection{On Essentially Non-Branching Spaces}

Having at our disposal $\MCP(K,N)$, we can now invoke the results of Section \ref{S:L1-OT} concerning $L^{1}$ Optimal Transportation theory, and obtain the following important equivalent definitions of $\CD^1_{Lip}(K,N)$, $\CD^1(K,N)$ and $\MCP^1(K,N)$ assuming that $(X,\sfd,\mm)$ is essentially non-branching.

\begin{proposition} \label{P:CD1-ENB}
Let $(X,\sfd,\mm)$ be an essentially non-branching \mms with $\supp(\mm) = X$. Given $K \in \Real$ and $N \in (1,\infty)$, the following statements are equivalent:
\begin{enumerate}
\item $(X,\sfd,\mm)$ verifies $\CD^1_{Lip}(K,N)$. 
\item For any 1-Lipschitz function $u : (X,\sfd) \rightarrow \Real$, let $\set{R_u^b(\alpha)}_{\alpha \in Q}$ denote the partition of $\T_{u}^{b}$ given  by the equivalence classes of $R_{u}^{b}$. Denote by $X_\alpha$ the closure $\overline{R_u^b(\alpha)}$. Then all the conditions (1)-(4) of Definition \ref{D:CD1-u} hold for the family $\set{X_\alpha}_{\alpha \in Q}$.  In particular, $X_\alpha = R_{u}(\alpha)$ is a transport-ray for $\qq$-a.e. $\alpha \in Q$. \\
Moreover, the sets $\set{X_\alpha}_{\alpha \in Q}$ have disjoint interiors $\{\mathring{R}_u^b(\alpha)\}_{\alpha \in Q}$ contained in $\T_u^b$, and the disintegration $(Q,\mathscr{Q},\qq)$ of $\mm\llcorner_{\T_u}$ on $\set{X_\alpha}_{\alpha \in Q}$ given by (\ref{E:New-disintCD1}) is essentially unique.  \\
Furthermore, $Q$ may be chosen to be a section of the above partition so that $Q \supset \bar Q \in \B(\T_u^b)$ with $\bar Q$ an $\mm$-section with $\mm$-measurable quotient map, so that in particular $\mathscr{Q} \supset \B(\bar Q)$ and $\qq$ is concentrated on $\bar Q$.
\end{enumerate}
An identical statement holds for $\CD^1(K,N)$ when only considering signed distance functions $u = d_f$.\\
An identical statement also holds for $\MCP^1(K,N)$ when only considering the functions $u = d(\cdot,o)$, after replacing above condition (4) of Definition \ref{D:CD1-u} with condition (4') of Definition \ref{D:MCP1-u}. 
\end{proposition}
\begin{proof}
The only direction requiring proof is $(1) \Rightarrow (2)$. Given a $1$-Lipschitz function $u$ as above, we may assume that $\mm(\T_u) > 0$, otherwise there is nothing to prove. The $\CD^1_u(K,N)$ condition ensures there exists a family $\set{Y_\beta}_{\beta \in P}$ of sets and a disintegration: \[
\mm \llcorner_{\T_u} = \int_P  \mm^P_\beta \pp(d \beta) ~, \quad \text{with } \quad \mm^P_{\beta}(Y_{\beta}) = 1,  \text{ for } \pp\text{-a.e. }\beta \in P ,
\]
so that for $\pp$-a.e. $\beta \in P$, $Y_\beta$ is a transport ray for $\Gamma_u$, $(Y_\beta,\sfd,\mm^P_\beta)$ satisfies $\CD(K,N)$ and $\supp(\mm^P_\beta) = Y_\beta$. By removing a $\pp$-null-set from $P$, let us assume without loss of generality that the above properties hold for all $\beta \in P$. 

As $\CD^1_{Lip}(K,N) \Rightarrow \CD^1 (K,N) \Rightarrow \MCP^1(K,N) \Rightarrow \MCP(K,N)$, and as our space is essentially non-branching with full-support, Corollary \ref{C:disintMCP} implies that  $\mm(A_+ \cup A_-) = 0$ and that there exists an essentially unique disintegration $(Q,\mathscr{Q},\qq)$ of $\mm\llcorner_{\T_{u}} = \mm\llcorner_{\T^b_{u}}$ strongly consistent with the partition of $\T_{u}^{b}$ given by $\set{R_u^b(\alpha)}_{\alpha \in Q}$:
\begin{equation} \label{E:ess-unique2}
\mm\llcorner_{\T_{u}} = \int_{Q} \mm_{\alpha} \,\qq(d\alpha), \quad \text{with } \quad \mm_{\alpha}(R_u^b(\alpha)) = 1,  \text{ for } \qq \text{-a.e. }\alpha \in Q .
\end{equation}
By Corollary \ref{C:disintMCP}, $Q$ may be chosen to be a section of the above partition satisfying the statement appearing in the formulation of Proposition \ref{P:CD1-ENB}.
Again, let us assume without loss of generality that $\mm_{\alpha}(R_u^b(\alpha)) = 1$ for all $\alpha \in Q$.

\smallskip
By Theorem \ref{T:endpoints}, there exists $Q_1 \subset Q$ of full $\qq$-measure so that $R_u(\alpha) = \overline{R_u^b(\alpha)} \supset R_u^b(\alpha) \supset \mathring{R}_u(\alpha)$ for all $\alpha \in Q_1$. In addition, since $\mm(\T_u \setminus \T_u^b) = 0$, there exists $P_1 \subset P$ of full $\pp$-measure so that $\mm^P_\beta(\T_u^b)=1$ for all $\beta \in P_1$. By Lemmas \ref{L:Newlemma} and \ref{L:RayInTub}, $(Y_\beta \cap \T_u^b,\sfd)$ is isometric to an interval in $(\Real,\abs{\cdot})$, and therefore $(\overline{Y_\beta \cap \T_u^b}, \sfd, (\mm^P_\beta)\llcorner_{\T_u^b})$ still satisfies $\CD(K,N)$, is of total measure $1$ and satisfies $\supp((\mm^P_\beta)\llcorner_{\T_u^b}) = \overline{Y_\beta \cap \T_u^b}$, for all $\beta \in P_1$. 

Now by Lemma \ref{L:RayInTub}, since $Y_\beta \cap \T_u^b \neq \emptyset$ for all $\beta \in P_1$, $Y_\beta = R_u(x)$ for all $x \in Y_\beta \cap \T_u^b$. In particular, for all $\beta \in P_1$, there exists a unique (since $R_u^b $ is an equivalence relation on $\T_u^b$ and by uniqueness of the section map) $\alpha = \alpha(\beta) \in Q$ so that $Y_\beta = R_u(\alpha)$. Denoting by $\tilde{Q} \subset Q$ the set of indices $\alpha$ obtained in this way, it is clear that $\tilde{Q}$ if of full $\qq$-measure, since:
$$
0 = \pp(P \setminus P_1) = \mm \left(\mathcal{T}_{u}^{b} \setminus \bigcup_{\beta \in P_{1}} Y_{\beta} \right) = \mm \left(\mathcal{T}_{u}^{b} \setminus \bigcup_{\alpha(\beta) \colon \beta \in P_{1}} R_{u}(\alpha(\beta)) \right) = \qq(Q \setminus \tilde{Q}).
$$
Consequently, $Q_2 := \tilde{Q} \cap Q_1$ is of full $\qq$-measure as well. Denoting $P_2 := \alpha^{-1}(Q_2)$ and repeating the above argument, it follows that $P_2 \subset P_1$ is of full $\pp$-measure and satisfies that for all $\beta \in P_2$, $Y_\beta = R_u(\alpha)$ for $\alpha = \alpha(\beta)\in Q_2$.

We conclude that there is a one-to-one correspondence: 
\[
\eta : P_2 \ni \beta \leftrightarrow \alpha \in Q_2 \;\;\; \text{whenever} \;\;\; Y_\beta \cap \T_u^b  = R_u^b(\alpha) ( = R_u(\alpha) \cap \T_u^b),
\]
so both of these representations yield an identical partition (up to relabeling) of the set:
\[
C := \bigcup_{\beta \in P_2} (Y_\beta \cap \T_u^b) = \bigcup_{\alpha \in Q_2} R_u^b(\alpha) .
\]
Clearly $\mm(\T_u^b \setminus C) = 0$ and so $C$ is $\mm$-measurable. Therefore, by the above two disintegration formulae:
\[
\mm \llcorner_{\T_u} = \mm\llcorner_{C} = \int_{P_2} (\mm^P_\beta)\llcorner_{\T_u^b} \pp(d\beta) = \int_{Q_2} \mm_\alpha \qq(d\alpha) . 
\]
After identifying between $P_2$ and $Q_2$ via $\eta$, it follows necessarily that $\qq \llcorner_{Q_2}=\pp \llcorner_{P_2}$ as they are both the push-forward of $\mm\llcorner_{C}$ under the partition map (since $(\mm^P_\beta)\llcorner_{\T_u^b}$ and $\mm_\alpha$ are both probability measures on $\T_u$). Applying the Disintegration Theorem \ref{T:disintegrationgeneral} to $(C,\B(C),\mm\llcorner_{C})$, we conclude that there is an essentially unique disintegration of $\mm\llcorner_{C}$ on the above partition of $C$. Consequently, there exist $P_3 \subset P_2$ of full $\pp$-measure and $Q_3 = \eta(P_3) \subset Q_2$ of full $\qq$-measure so that: 
\[
 (\mm^P_\beta)\llcorner_{\T_u^b} = \mm_{\alpha}
 \]
 for all pairs $(\beta,\alpha) \in P_3 \times Q_3$ related by the correspondence $\eta$. 

\smallskip
Recall that $X_\alpha := \overline{R_u^b(\alpha)}$. It follows that for all $\alpha \in Q_3$ (with corresponding $\beta \in P_3$):
\begin{enumerate}
\item $X_\alpha = \overline{R_u^b(\alpha)} = R_u(\alpha)$ is a transport ray. \item $(\overline{Y_\beta \cap \T_u^b} , \sfd, (\mm^P_\beta)\llcorner_{\T_u^b}) = (\overline{R_u^b(\alpha)} = X_\alpha,\sfd,\mm_\alpha)$ satisfies $\CD(K,N)$ with total measure $1$. \item Consequently:
\begin{equation} \label{E:ess-unique3}
 \mm\llcorner_{\T_{u}} = \int_{Q} \mm_{\alpha} \,\qq(d\alpha) ,
\end{equation}
 is a disintegration on $\set{X_\alpha}_{\alpha \in Q}$. 
\item $\mm_\alpha = (\mm^P_\beta)\llcorner_{\T_u^b}$ is supported on $\overline{Y_\beta \cap \T_u^b} = \overline{R_u^b(\alpha)} = X_\alpha$. 
\end{enumerate}
This confirms the 4 conditions of Definition \ref{D:CD1-u}, and the essential uniqueness of the disintegration (\ref{E:ess-unique3}) readily follows from that of the disintegration (\ref{E:ess-unique2}) and the arguments above. 

Finally, by Lemma \ref{L:Newlemma}, since $(R_u^b(\alpha) = R_u(\alpha) \cap \T_u^b , \sfd)$ is isometric to an interval in $(\Real,\abs{\cdot})$, then $\mathring{X}_\alpha = \mathring{R}_{u}^b(\alpha)$ for all $\alpha \in Q$. As $\set{R_{u}^b(\alpha)}_{\alpha \in Q}$ are equivalence classes, it follows that $\{\mathring{X}_\alpha \}_{\alpha \in Q}$ is a family of disjoint subsets of $\T_u^b$. This concludes the proof for the case of $\CD^1_{Lip}$ and $\CD^1$.

For $\MCP^1$, one just needs to note that if $u = \sfd(\cdot,o)$ then $o \in Y_\beta$ for all $\beta \in P$ (by Remark \ref{R:o-dist}, since $Y_\beta$ is a transport ray). Recalling the definition of $P_1 \subset P$, since $(Y_\beta \cap \T_u^b,\sfd)$ is isometric to an interval and $\mm^P_\beta(Y_\beta \cap \T_u^b) = 1$ for all $\beta \in P_1$, it follows necessarily that for those $\beta$, $o \in \overline{Y_\beta \cap \T_u^b}$ and $(\overline{Y_\beta \cap \T_u^b} , \sfd, (\mm^P_\beta)\llcorner_{\T_u^b})$ still satisfies $\MCP(K,N)$ with respect to $o$ and is of full support. The rest of the the argument is identical to the one presented above, concluding the proof.
\end{proof}

Recall moreover that we already derived several properties of $W_{2}$-geodesics in essentially non-branching \mms's verifying $\MCP(K,N)$. Hence from Proposition \ref{P:MCP}
we also obtain all the claims of Theorem \ref{T:optimalmapMCP} and Corollary \ref{C:injectivity}, as well as all of the results of the next section, provided the \mms is essentially non-branching and verifies $\CD^{1}(K,N)$ for $N \in (1,\infty)$.

\bigskip
\bigskip

\section{Temporal-Regularity under $\MCP$} \label{S:MCP}

In this section we deduce from the Measure Contraction and essentially non-branching properties various temporal-regularity results for the map $t \mapsto \rho_{t}(\gamma_{t})$ and related objects, which we will require for this work. By Proposition \ref{P:MCP}, these results also apply under the $\CD^1$ condition. 
While these properties are essentially standard consequences of recently available results and tools, they appear to be new and may be of independent interest. 

\medskip
As usual, we assume that $K \in \Real$ and $N \in (1,\infty)$. We begin with:

\begin{proposition} \label{P:MCP-density}
Let $(X,\sfd,\mm)$ denote an essentially non-branching \mms. Then the following are equivalent:
\begin{enumerate}
\item $(X,\sfd,\mm)$ verifies $\MCP(K,N)$.
\item $(X,\sfd,\mm)$ verifies $\MCPE(K,N)$.
\item 
For all $\mu_{0},\mu_{1} \in \mathcal{P}_{2}(X)$ with $\mu_{0} \ll \mm$ and $\supp(\mu_{1}) \subset \supp(\mm)$, there exists a unique $\nu \in \Opt(\mu_0,\mu_1)$, $\nu$ is induced by a map  (i.e. $\nu = S_{\sharp}(\mu_0)$ for some map $S : X \rightarrow \Geo(X)$), $\mu_t := (\ee_t)_{\#} \nu \ll \mm$ for all $t \in [0,1)$, and writing $\mu_t = \rho_t \mm$, we have for all $t \in [0,1)$:
\begin{equation} \label{E:MCP-density}
\rho_t^{-\frac{1}{N}}(\gamma_t) \geq  \tau_{K,N}^{(1-t)}(\sfd(\gamma_0,\gamma_1)) \rho_0^{-\frac{1}{N}}(\gamma_0) \;\;\; \text{for $\nu$-a.e. $\gamma \in \Geo(X)$} ,
\end{equation}
and (integrating with respect to $\nu$):
\begin{equation} \label{eq:strong-MCPE-def} 
\Eps_{N}(\mu_t) \geq \int \tau_{K,N}^{(1-t)} (\sfd(\gamma_0,\gamma_1)) \rho_0^{-\frac{1}{N}}(\gamma_0) \nu(d\gamma) .
\end{equation}
\item
For all $\mu_{0},\mu_{1} \in \mathcal{P}_2(X)$ of the form  $\mu_1 = \delta_o$ for some $o \in \supp(\mm)$ and $\mu_0 = \frac{1}{\mm(A)} \mm\llcorner_{A}$ for some Borel set $A \subset X$ with $0 < \mm(A) < \infty$,
 there exists a $\nu \in \Opt(\mu_0,\mu_1)$ so that for all $t \in [0,1)$, $\mu_t := (\ee_t)_{\#} \nu \ll \mm$ and (\ref{E:MCP-density}), (\ref{eq:strong-MCPE-def}) hold.
\end{enumerate}
Moreover, the equivalence $(1) \Leftrightarrow (4)$ does not require the essentially non-branching assumption.
\end{proposition}

\begin{remark}
In fact, for essentially non-branching spaces, it is also possible to add the $\MCP^1(K,N)$ condition to the above list of equivalent statements. Indeed, we have already seen in the previous section that $\MCP^1(K,N) \Rightarrow \MCPE(K,N)$ without any non-branching assumptions. The converse implication for non-branching spaces follows from \cite[Proposition 9.5]{biacava:streconv} (without identifying the $\MCP^1(K,N)$ condition by this name), and it is possible to extend this to essentially non-branching spaces by following the arguments of \cite[Proposition A.1]{cava:MongeRCD}. 
\end{remark}

\begin{remark} \label{rem:MCPE}
Note that in (3), one is allowed to test \emph{any} $\mu_1$ with $\supp(\mu_1) \subset \supp(\mm)$, not only $\mu_1 = \delta_o$ as in the other statements. 
By Theorem \ref{T:optimalmapMCP} (recall that $\MCPE(K,N)$ implies $\MCP(K,N)$), note that the $\MCPE(K,N)$ condition is precisely equivalent to the validity of (\ref{eq:strong-MCPE-def}) for all measures $\mu_{0},\mu_{1} \in \mathcal{P}_{2}(X)$ of the form $\mu_1 = \delta_o$ with $o \in \supp(\mm)$ and $\mu_0 \ll \mm$ with bounded support. 
\end{remark}

\begin{remark}
While the equivalence $(1) \Leftrightarrow (4)$ will not be directly used in this work, it is worthwhile remarking that this is the only instance we are aware of, where one can obtain information on the density along geodesics without assuming or a-posteriori concluding some type of non-branching assumption. Indeed, the proof of $(1) \Rightarrow (4)$ relies on the (newly available) Theorem \ref{thm:order12-main}. 
\end{remark}

\begin{proof}[Proof of Proposition \ref{P:MCP-density}]
\hfill

\textbf{$(1) \Rightarrow (4)$.} 
$(\supp(\mm),\sfd)$ is proper and geodesic by Lemma \ref{L:proper-support}. Given $\mu_0$ and $\mu_1 = \delta_o$ as in (4), any $\nu \in \Opt(\mu_0,\mu_1)$ is concentrated on $G_\varphi$ (where $\varphi$ is the associated Kantorovich potential), and so Theorem \ref{thm:order12-main} implies that $\sfd(\gamma_{0},\gamma_{1}) = \ell_t(\gamma_t)$ for $\nu$-a.e. $\gamma$. 
It follows that with the notation of Section \ref{sec:order12}:
\[
\frac{1}{\mm(A)} \mm \geq (\ee_{t})_{\sharp} \big( \tau_{K,N}^{(1-t)}(\sfd(\gamma_{0},\gamma_{1}))^{N} \nu(d \gamma) \big) =  \rho_t(x) \tau_{K,N}^{(1-t)}(\ell_t(x))^N \mm(dx) .
\]
The pointwise inequality between densities follows for $\mm$-a.e. $x$, and since $\ell_t < \infty$ (and hence $\tau_{K,N}^{(1-t)}(\ell_t(x)) > 0$) for $t\in (0,1)$, this in fact \emph{implies} that $(\ee_t)_{\sharp}(\nu) \ll \mm$ (without relying on Theorem \ref{T:optimalmapMCP}, which is unavailable without the essentially non-branching assumption). 
Since $(\ee_t)_{\sharp}(\nu) \ll \mm$, the inequality between densities is verified at $x = \gamma_t$ for $\nu$-a.e. $\gamma$. Noting that $\frac{1}{\mm(A)} = \rho_0(\gamma_0)$ for $\nu$-a.e. $\gamma$, (\ref{E:MCP-density}) and hence (\ref{eq:strong-MCPE-def}) are established for $\mu_0,\mu_1$ as above.

\textbf{$(4) \Rightarrow (1)$.}
This follows by applying (\ref{E:MCP-density}) to $\mu_0 = \frac{1}{\mm(A)} \mm \llcorner_{A}$ and $\mu_1 = \delta_o$, raising the resulting inequality to the power of $N$, and integrating it against $\nu\llcorner_{\set{\gamma_t \in B}}$ for all Borel sets $B \subset \supp(\mu_t)$, thereby verifying the $\MCP(K,N)$ inequality (\ref{eq:MCP-def}).

\textbf{$(4) \Rightarrow (2)$.} 
Let $o \in \supp(\mm)$ and let $\mu_{0}  = \rho_{0} \mm \in \mathcal{P}(X)$ with bounded support. 
As $(4)\Rightarrow (1)$, Lemma \ref{L:proper-support} implies that $(\supp(\mm),d)$ is proper, and in addition the assertions of Theorem \ref{T:optimalmapMCP} and Corollary \ref{C:injectivity} are in force. 

Now, there exists an non-decreasing sequence $\{f^{i}\}_{i\in \N}$ of simple functions, that is 
$$
f^{i} = \sum_{k \leq n(i)} \alpha^{i}_{k} \chi_{A^{i}_{k}}, \qquad \alpha^{i}_{k} > 0 ,\quad \mm(A^i_k) > 0, \quad  A^{i}_{k} \cap A^{i}_{j} = \emptyset, \textrm{ if } k\neq j,
$$
such that $\mu_{0}^{i} := \rho_{0}^{i} \mm := \frac{1}{z^i} f^i \mm \in \mathcal{P}(X)$ is of bounded support, $z^i := \int f^i d\mm \nearrow 1$, $f^i \nearrow \rho_0$ pointwise,  and $\mu_0^i \rightharpoonup \mu_0$ weakly, as $i \rightarrow \infty$. By Theorem \ref{T:optimalmapMCP} there exists a unique 
$\nu^{i} \in \Opt(\mu_{0}^{i}, \delta_{o})$, it is induced by a map, and can be written as:
$$
\nu^{i} = \sum_{k\leq n(i)} \frac{1}{z^i} \alpha^i_k \mm(A^i_k) \nu^{i}_{k}, 
$$
with each $\nu^{i}_{k}$ the unique optimal dynamical plan between $\mu_{0,k}^i := \rho_{0,k}^i \mm := \frac{1}{\mm(A_k^i)} \mm\llcorner_{A_k^i}$ and $\delta_{o}$. Moreover,  $(\ee_{t})_{\#}\nu^{i}_{k} \perp (\ee_{t})_{\#} \nu^{i}_{j}$ whenever $k\neq j$, for all $t \in [0,1)$ by Corollary \ref{C:injectivity}. Lastly, $\supp(\nu^i) \subset \Geo(\supp(\mm))$ by Remark \ref{rem:supp-nu}. It follows by (\ref{eq:strong-MCPE-def}) applied to $\nu^i_k$ that:
\[
\Eps_N((\ee_{t})_{\#}\nu^{i}_k) \geq  \int \tau_{K,N}^{(1-t)}(\sfd(x,o)) \left(\rho^{i}_{0,k}(x)\right)^{1-\frac{1}{N}} \, \mm(dx) .
\]
Multiplying by $\brac{\frac{1}{z^i} \alpha^i_k \mm(A^i_k)}^{1-\frac{1}{N}}$, summing over $k$, and using the mutual singularity of all corresponding measures, we obtain:
\begin{equation} \label{eq:MCPE-approx}
\Eps_N((\ee_{t})_{\#}\nu^{i}) \geq  \int \tau_{K,N}^{(1-t)}(\sfd(x,o)) \left(\rho^i_{0}(x)\right)^{1-\frac{1}{N}} \, \mm(dx) .
\end{equation}
Passing to a subsequence if necessary, Lemma \ref{lem:nu-compactness} implies that $\nu^i \rightharpoonup \nu^\infty \in \Opt(\mu_0,\delta_o)$, and hence $(\ee_{t})_{\#}\nu^{i} \rightharpoonup (\ee_{t})_{\#}\nu^{\infty}$. 
It follows by upper semi-continuity of $\Eps_N$ on the left-hand side of (\ref{eq:MCPE-approx}), and monotone convergence (and $z_i \rightarrow 1$) on the right hand side, that taking $i \to \infty$ yields the $\MCPE(K,N)$ inequality (\ref{eq:MCPE-def}). 
\textbf{$(2) \Rightarrow (3)$.} 
By Remark \ref{rem:supp-nu}, we may reduce to the case $\supp(\mm) = X$. 
In view of Remark \ref{rem:MCPE}, we first extend the validity of (\ref{eq:strong-MCPE-def}) by removing the (immaterial) restriction that $\mu_0$ has bounded support. 
When $K > 0$, $\supp(\mu_0)$ is automatically bounded since $\MCPE(K,N)$ implies $\MCP(K,N)$ which by Remark \ref{R:BonnetMyers} implies a Bonnet-Myers diameter estimate. When $K \leq 0$, we may weakly approximate a general $\mu_0 \in\P_2(X,\sfd,\mm)$ by measures $\mu_0^i \ll \mm$ having bounded support and repeat the argument presented above in the proof of \textbf{$(4) \Rightarrow (2)$}.

The case of a general $\mu_1 \in \P_2(X)$ with $\supp(\mu_1) \subset \supp(\mm)$ follows by approximating $\mu_1$ by a convex combination of delta-measures: 
$$
\mu_{1}^{i} = \sum_{k\leq n(i)} a^{i}_{k} \delta_{o^{i}_{k}} \;\; , \;\;  o^{i}_{k} \in \supp(\mm) \  \text{for} \ k \leq n(i), \ \text{and}  \ \sum_{k\leq n(i)} a^{i}_{k} = 1;
$$
with $W_{2}(\mu_{1}^{i}, \mu_{1}) \to 0$ as $i\to \infty$. By Theorem \ref{T:optimalmapMCP} (recall again that $\MCPE(K,N)$ implies $\MCP(K,N)$), for each $i$ there exists a unique $\nu^{i} \in \Opt(\mu_{0},\mu_{1}^{i})$, and we may write $\nu^{i} = \sum_{k \leq n(i)} \alpha^i_k \nu^{i}_{k}$ so that:
$$
\nu^{i}_{k} \in \Opt( (\ee_{0})_{\#} \nu^{i}_{k} , \delta_{o^{i}_{k}}).
$$
Moreover, as explained above, $(\ee_{t})_{\#} \nu^{i}_{k} \perp (\ee_{t})_{\#} \nu^{i}_{j}$ whenever $k\neq j$, for all $t \in [0,1)$. 
Furthermore, as $(\ee_{0})_{\#} \nu^{i}_{k} \ll \mm$ (since $(\ee_{0})_{\#} \nu^i = \mu_0 = \rho_0 \mm \ll \mm$), Theorem \ref{T:optimalmapMCP} implies that $(\ee_{t})_{\#} \nu^{i}_{k} \ll \mm$ for all $t \in [0,1)$. Writing $(\ee_{t})_{\#} \nu^{i}_{k}  = \rho_{k,t}^{i} \mm$, the $\MCPE(K,N)$ condition implies for all $t \in [0,1)$:
$$
\int  (\rho_{k,t}^{i})^{1-\frac{1}{N}}(x) \, \mm(dx) \geq \int \tau_{K,N}^{(1-t)} (\sfd(x,o^{i}_{k})) (\rho_{0,k}^{i})^{1-\frac{1}{N}}(x) \mm(dx);
$$
Multiplying by $(\alpha_k^i)^{1-1/N}$, summing over $k$ and using the mutual singularity of the corresponding measures, we obtain:
$$
\Eps_{N}((\ee_{t})_{\#} \nu^{i}) \geq \int_X \tau_{K,N}^{(1-t)} (\sfd(x,y)) \rho_0^{-\frac{1}{N}}(x) \,  (\ee_{0},\ee_{1})_{\#} \nu^{i} (dxdy).
$$
Passing as usual to a subsequence if necessary, Lemma \ref{lem:nu-compactness} implies that $\nu^i \rightharpoonup \nu^\infty \in \Opt(\mu_0,\mu_1)$, and hence $(\ee_{t})_{\#}\nu^{i} \rightharpoonup (\ee_{t})_{\#}\nu^{\infty}$. 
Invoking the upper semi-continuity of $\Eps_{N}$ on the left-hand-side, and lower semi-continuity of the right-hand-side (see \cite[Lemma 3.3]{sturm:II}, noting that the first marginal of $\nu^i$ is fixed to be $\mu_0 = \rho_0 \mm$), (\ref{eq:strong-MCPE-def}) finally follows in full generality.

The density estimate (\ref{E:MCP-density}) then follows using a straightforward variation of \cite[Proposition 3.1]{GSR:maps}, where it was shown how the existence of (a necessarily unique) transport map $S$ may be used to obtain a pointwise density inequality such as (\ref{E:MCP-density}) from an integral inequality such as (\ref{eq:strong-MCPE-def}) (the statement of \cite[Proposition 3.1]{GSR:maps} involves an assumption on infinitesimal Hilbertianity of the space, but the only property used in the proof is the existence of a transport map $S$ inducing a unique optimal dynamical plan). 

Finally, \textbf{$(3) \Rightarrow (4)$} is trivial. This concludes the proof. 
\end{proof}

\begin{corollary}\label{C:regularity3MCP}
Let $(X,\sfd,\mm)$ be an essentially non-branching \mms verifying $\MCP(K,N)$. Then with the same assumptions and notation as in Proposition \ref{P:MCP-density} (3), there exist versions of the densities $\rho_t = \frac{d\mu_t}{d\mm}$, $t \in [0,1)$, so that for $\nu$-a.e. $\gamma \in \Geo(X)$, for all $0\leq s \leq t <1$:
\begin{equation}\label{E:regularityrho}
\rho_s(\gamma_s) > 0 \;\; , \;\; \left( \tau_{K,N}^{(\frac{s}{t})} (\sfd(\gamma_{0},\gamma_{t})) \right)^{N} \leq \frac{\rho_{t}(\gamma_{t})}{\rho_{s}(\gamma_{s})} \leq 
\left( \tau_{K,N}^{(\frac{1-t}{1-s})} (\sfd(\gamma_{s},\gamma_{1})) \right)^{-N}
\end{equation}
(with $\frac{s}{t} = \frac{0}{0}$ interpreted as $1$ above). In particular, for $\nu$-a.e. $\gamma$, the map $t \mapsto \rho_{t}(\gamma_{t})$ is locally Lipschitz on $(0,1)$ and upper semi-continuous at $t=0$. 
\end{corollary}

\begin{proof}
\hfill

\medskip
\textbf{Step 1.}
Given $0 \leq s \leq t < 1$, observe that $(\text{restr}^{t}_{s})_{\sharp} \nu$ is the unique element of $\Opt(\mu_{s},\mu_{t})$; 
indeed $\mu_{s}$ is absolutely continuous with respect to $\mm$ and so Theorem \ref{T:optimalmapMCP} applies.
In particular, we deduce that for each $0 \leq s \leq t < 1$ and  $\nu$-a.e. $\gamma$:
$$
\rho_{t}(\gamma_{t})^{-1/N} \geq  \rho_{s}(\gamma_{s})^{-1/N}  \tau_{K,N}^{(\frac{1-t}{1-s})} (\sfd(\gamma_{s},\gamma_{1})), 
$$
with the exceptional set depending on $s$ and $t$. Reversing time and the roles of $\mu_s,\mu_t$, we similarly obtain for each $0 \leq s \leq t < 1$ and  $\nu$-a.e. $\gamma$ that:
$$
\rho_{s}(\gamma_{s})^{-1/N} \geq  \rho_{t}(\gamma_{t})^{-1/N}  \tau_{K,N}^{(\frac{s}{t})} (\sfd(\gamma_{0},\gamma_{t})), 
$$
with the exceptional set depending on $s$ and $t$ (the case $s=0$ is also included as the conclusion is then trivial). Note that given $s \in [0,1)$, as $\rho_s(x) > 0$ for $\mu_s$-a.e. $x$, we have that $\rho_s(\gamma_s) > 0$ for $\nu$-a.e. $\gamma$. Altogether, we see that for each $0 \leq  s \leq t < 1$, for $\nu$-a.e. $\gamma$:
\begin{equation}\label{E:MCP-a.e.}
\rho_s(\gamma_s) > 0 \;\;, \;\; \rho_{s}(\gamma_{s}) \left( \tau_{K,N}^{(\frac{s}{t})} (\sfd(\gamma_{0},\gamma_{t})) \right)^{N} \leq \rho_{t}(\gamma_{t}) \leq 
\rho_{s}(\gamma_{s}) \left( \tau_{K,N}^{(\frac{1-t}{1-s})} (\sfd(\gamma_{s},\gamma_{1})) \right)^{-N},
\end{equation}
with the exceptional set depending on $s$ and $t$. 

Together with an application of Corollary \ref{C:injectivity}, we deduce the existence of a Borel set $H \subset \Geo(X)$ with $\nu(H)=1$ such that $\ee_t|_H : H \rightarrow X$ is injective for all $t \in [0,1)$, and such that for every $\gamma \in H$, the double sided estimate \eqref{E:MCP-a.e.} holds for all $s,t \in [0,1) \cap \Q$.
We then define for $t \in [0,1)$ and $\gamma \in H$:
$$
\hat \rho_{t}(\gamma_t) := \begin{cases} \lim_{(0,1) \cap \Q \ni s \to t} \rho_{s}(\gamma_{s}) & t \in (0,1)  \\ \rho_{0}(\gamma_{0}) & t = 0 \end{cases} ,
$$
and $\hat \rho_t = 0$ outside of $\ee_t(H)$. 
By \eqref{E:MCP-a.e.} we see that for any $\gamma \in H$ and $t \in (0,1)$ the above limit always exists, and so by injectivity of $\ee_t|_H$, $\hat \rho_t$ is well-defined. Furthermore, \eqref{E:MCP-a.e.} implies that for all $\gamma \in H$, $\hat \rho_{\cdot}(\gamma_{\cdot})$ satisfies \eqref{E:MCP-a.e.} itself for all $0 \leq s \leq t < 1$. 
Finally, for each $t \in [0,1)$ consider any sequence $\{s_{n}\}\subset \Q$ 
converging to $t$; then \eqref{E:MCP-a.e.} is valid for $\nu$-a.e. $\gamma$
 at $t$ and $s_{n}$, with the exceptional set not depending on $n$. 
Taking the limit as $n\to \infty$ implies $
\rho_{t}(\gamma_{t}) = \hat \rho_{t}(\gamma_{t})$.
Hence we have obtained that for each $t \in [0,1)$, for $\nu$-a.e. $\gamma$:
$$
\rho_{t}(\gamma_{t}) = \hat \rho_{t}(\gamma_{t}),
$$
with the exceptional set depending only on $t$.

It follows that for all $t \in [0,1)$, $\rho_t(x) = \hat \rho_t(x)$ for $\mu_t$-a.e. $x$. As $\mu_t$ and $\mm$ are mutually absolutely continuous on $\set{\rho_t > 0}$, it follows that $\rho_t \mm =  \hat \rho_t 1_{\set{\rho_t > 0}}\mm$ for all $t \in [0,1)$. 

\medskip
\textbf{Step 2.}
We now claim that for all $t \in [0,1)$, $\mm(\{ \rho_t = 0 \} \cap \ee_t(H)) = 0$. This will establish that $\mu_t = \rho_t \mm = \hat \rho_t \mm$, so that $\hat \rho_t$ is indeed a density of $\mu_t$, thereby concluding the proof. 

Suppose in the contrapositive that the above is false, so that there exists $t \in [0,1)$ with $\mm(\{ \rho_t = 0 \} \cap \ee_t(H)) > 0$. As $\ee_t|_H$ is injective, there exist $K \subset H$ such that $K_{t} : = \ee_{t}(K) = \{ \rho_t = 0 \} \cap \ee_t(H)$. 

Set $K_s := \ee_s(K)$ for all $s \in [0,1)$. We claim that $\mm(K_s) > 0$ for all $s \in (0,1)$. Indeed, define $\eta_{t} : = \mm\llcorner_{K_{t}}/\mm(K_{t})$ and set $\bar \nu :=  (\ee_{t}|_H)^{-1}_{\#} \eta_{t}$ and $\eta_s := (\ee_s)_{\sharp} \bar \nu$. As $\bar \nu$ is concentrated on $K \subset H \subset \supp(\nu)$,  it follows that $(\text{restr}^{1}_{t})_{\sharp} \bar \nu$ must be an optimal dynamical plan between $\eta_t$ and $\eta_1$. As $\eta_t \ll \mm$, Theorem \ref{T:optimalmapMCP} implies that the latter plan is in fact the unique element of $\OptGeo(\eta_t,\eta_1)$, and that $\eta_s \ll \mm$ for all $s \in [t,1)$. As $\eta_s(K_s) = 1$, it follows that $\mm(K_s) > 0$. If $t > 0$, a similar argument applies to the range $s \in (0,t]$.

However, by definition, for all $s \in [0,1) \cap \Q$ we have $0 < \hat \rho_s = \rho_s$ on $\ee_s(H)$, and in particular on $\ee_s(K) = K_s$. Choosing any $s \in (0,1) \cap \Q$, we obtain the desired contradiction:
\[
0 < \int_{K_{s}} \rho_{s} \mm =\mu_{s}(K_{s}) = \mu_t(K_t) = \int_{K_t} \rho_t \mm = 0 .
\]
This concludes the proof.
\end{proof}

\medskip

\begin{proposition}\label{P:completeregularity}
Let $(X,\sfd,\mm)$  be an essentially non-branching \mms verifying $\MCP(K,N)$. 
Consider any $\mu_{0},\mu_{1} \in \mathcal{P}_{2}(X)$ with $\mu_0 \ll \mm$ and $\supp(\mu_1) \subset \supp(\mm)$, and let $\nu$ denote the unique element of  $\Opt(\mu_{0},\mu_{1})$. 
Then for any compact set $G \subset \Geo(X)$ with $\nu(G)> 0$, such that \eqref{E:regularityrho} holds for all $\gamma \in G$ and $0 \leq s \leq t < 1$, we have for all $s \in [0,1)$, $\mm(\ee_{s}(G)) > 0$, and for all $0\leq s \leq t <1$:
\begin{equation}\label{E:continuitymass}
\left(\frac{1-t}{1-s}\right)^{N} e^{-d(G) (t-s) \sqrt{(N-1)K^{-}}} \leq \frac{\mm(\ee_{t}(G))}{\mm(\ee_{s}(G))} \leq \left(\frac{t}{s}\right)^{N} e^{d(G) (t-s) \sqrt{(N-1)K^{-}}} ,
\end{equation}
where $d(G) = sup \{ \ell(\gamma) \colon \gamma \in G\} < \infty$ and $K^{-} = \max\{0, - K\}$ (and with $\frac{t}{s} = \frac{0}{0}$ interpreted as $1$ above).
In particular, the map $t \mapsto \mm(\ee_{t}(G))$ is locally Lipschitz on $(0,1)$ and lower semi-continuous at $t=0$. 
\end{proposition}

\begin{proof}
We proceed with the usual notation repeatedly used above. 
Fix $s \in [0,1)$. Since $\mu_{s}(\ee_{s}(G)) \geq \nu(G) >0$ and $\mu_s \ll \mm$, it follows that $\mm(\ee_{s}(G)) > 0$.  Define  $\bar \mu_{0} : = \mm\llcorner_{\ee_s(G)}/\mm(\ee_s(G))$.

By Corollary \ref{C:injectivity}, there exists a Borel set $H \subset G$ such that $\ee_{s}^{-1} : \ee_{s}(H) \to G$ is a single valued map and:
\begin{equation} \label{eq:stam}
\nu(G \setminus H) = 0 \quad, \quad \mm(\ee_{s}(G) \setminus \ee_{s}(H)) = 0 ,
\end{equation}
where the second assertion above follows since $\mm$ and $\mu_s$ are mutually absolutely continuous on $\set{ \rho_s > 0 }$, and since our assumption \eqref{E:regularityrho} guarantees that $\ee_{s}(G) \subset \{\rho_{s} > 0 \}$.
Now consider:
$$
\bar \nu :=  (\text{rest}^1_s \circ \ee_s^{-1})_{\sharp} (\bar \mu_0\llcorner_{\ee_s(H)}) =  \int_{\ee_s(H)} \delta_{\text{restr}^{1}_{s}(\ee_{s}^{-1}(x))} \bar \mu_{0}(dx) \in \mathcal{P}(\Geo(X)) .
$$
By construction and (\ref{eq:stam}), $(\ee_0)_{\sharp }\bar \nu = \bar \mu_{0}$; define $\bar \mu_{1} : = (\ee_1)_{\sharp }\bar \nu$ and note that necessarily $\bar \nu \in \Opt(\bar \mu_{0}, \bar \mu_{1})$ (since $\bar \nu$ is still supported on a $\sfd^2/2$-cyclically monotone set) and that it is induced by the map $T := \ee_1 \circ \ee_s^{-1}$. 
Theorem \ref{T:optimalmapMCP} then implies that $\bar \mu_{r} = \bar \rho_r \mm \ll \mm$ for all $r \in [0,1)$. Note that $\bar \mu_r$ is concentrated on the compact set $\ee_{t}(G)$ with $t := s + r(1-s)$, and therefore $\mm(\supp(\bar \mu_r)) \leq \mm(\ee_{t}(G) )$. It follows by Jensen's inequality together with the $\MCP(K,N)$ assumption that:
\begin{align*}
\mm(\ee_{t}(G) )^{1/N} & \geq \mm( \supp(\bar \mu_r) )^{1/N}  \geq \int \bar \rho_{r}^{1-1/N}(x) \,\mm(dx) \\
& \geq \mm(\ee_s(G))^{1/N - 1} \int_{\ee_s(G)} \tau_{K,N}^{(1-r)}(\sfd(x,T(x))) \, \mm(dx) \\
& \geq \mm(\ee_s(G))^{1/N} (1-r) e^{-(1-s) d(G) r \sqrt{(N-1)K^{-}}/N},
\end{align*}
where the last inequality follows from the lower bound (see e.g. \cite[Remark 2.3]{CM3}):
$$
\tau_{K,N}^{(1-r)} (\theta) = (1-r) \left( \frac{\sigma^{(1-r)}_{K,N-1} (\theta)}{ 1-r} \right)^{\frac{N-1}{N}} 
\geq (1-r)e^{ - \theta r \sqrt{ (N-1)K^{-}}/N  }.
$$
Substituting $r = \frac{t-s}{1-s}$, the left-hand side of (\ref{E:continuitymass}) is established. Reversing the time, the right-hand side of (\ref{E:continuitymass}) immediately follows, thereby concluding the proof. 
\end{proof}

\bigskip

The following two consequences of Proposition \ref{P:completeregularity} will be required for the proof of the change-of-variables formula in Section \ref{S:comparison1}. 
Recall that for any $G \subset \Geo(X)$,
$$
D(G) := \{ (x,t) \in X \times [0,1] \colon x = \gamma_t, \ \gamma \in G\},
$$
and that $D(G)(x) = \{ t \in [0,1] \colon x = \gamma_{t}, \ \gamma \in G \}$ and $D(G)(t) = \{ x \in X \colon x = \gamma_t,\ \gamma \in G \} = \ee_{t}(G)$.
To simplify the notation, we directly write $G(x)$ instead of $D(G)(x)$. 

\begin{proposition}\label{P:densityreversed}
With the same assumptions as in Proposition \ref{P:completeregularity}, we have for any $t \in (0,1)$:
$$
\lim_{\ve \to 0+} \frac{\mathcal{L}^{1} \big( G(x) \cap (t- \ve , t+\ve) \big)   }{2\ve} = 1 \;\;\; \text{ in $L^1(\ee_{t}(G),\mm)$.}
$$
The same result also holds for $t=0$ if we dispense with the factor of $2$ in the denominator. 
\end{proposition}

The proof follows the same line as the proof of \cite[Theorem 2.1]{cavhue:regular}. We include it for the reader's convenience. 

\begin{proof}
Fix $t \in (0,1)$.
Suppose in the contrapositive that the claim is false:
$$
\limsup_{\ve \to 0} \int_{\ee_{t}(G)} \left| 1 - \frac{\mathcal{L}^{1} (G(x) \cap (t-\ve, t+\ve)) }{2\ve} \right| \, \mm(dx) > 0.
$$
Consider the complement $G(x)^{c}  = \{ t \in [0,1] : x \notin \ee_{t}(G)\}$,
and deduce the existence of a sequence $\ve_{n} \to 0$ such that 
\begin{equation}\label{E:one}
\lim_{n\to \infty} \int_{\ee_{t}(G)} \frac{\mathcal{L}^{1} (G(x)^{c} \cap (t-\ve_{n}, t+\ve_{n})) }{2\ve_{n}} \, \mm(dx)>0.
\end{equation}
Now let:
$$
E: = \{ (x,s) \in \ee_{t}(G) \times (0,1) \; ; \;  s \in G(x)^{c} \}
$$
with $E(x)$, $E(s)$ the corresponding sections. By Fubini's Theorem and \eqref{E:one}
we obtain that:
\begin{align*}
& ~ \lim_{n \to \infty} \frac{1}{2\ve_{n}} \int_{(t-\ve_{n},t+\ve_{n})} \mm(E(s))\, \mathcal{L}^{1}(ds)  \crcr
				&~ =  \lim_{n \to \infty} \frac{1}{2\ve_{n}}  \mm \otimes \mathcal{L}^{1}\left(E \cap (\ee_{t}(G) \times (t-\ve_{n}, t+ \ve_{n}))\right)  \crcr
				&~ = \lim_{n \to \infty} \frac{1}{2\ve_{n}}  \int_{\ee_{t}(G)}  \mathcal{L}^{1}(G(x)^{c} \cap (t-\ve_{n}, t+ \ve_{n})) \, \mm(dx) > 0,
\end{align*}
so there must be a sequence of $\{s_{n}\}_{n \in \enne}$ converging to $t$ so that $\mm(E (s_{n})) \geq \kappa$, for some $\kappa>0$.
Repeating the above argument for the case $t=0$ with the appropriate obvious modifications, the latter conclusion also holds in that case as well. 
Note that:
$$
E(s_{n}) = \{ x\in \ee_{t}(G) \colon x \notin \ee_{s_{n}}(G) \} = \ee_{t}(G) \setminus \ee_{s_{n}}(G).
$$
The compact sets $\ee_{s_{n}}(G)$ converge to $\ee_{t}(G)$ in Hausdorff distance: indeed, $\sfd(\gamma_{t},\gamma_{s_{n}}) \leq C |t-s_{n}|$ where 
$C: =\sup_{\gamma \in G} \ell(\gamma)  < \infty$ by compactness of $G$. 
Hence,
for each $\ve > 0$ there exists $n(\ve)$ such that 
for all $n \geq n(\ve)$ it holds $\ee_{t}(G)^{\ve} \supset \ee_{s_n}(G)$ (and vice-versa), where $A^\eps := \set{y \in X \; ; \; \sfd(y,A) \leq \eps}$.
It follows that:
$$
\mm(\ee_{t}(G)^{\ve}) \geq \mm\big(\ee_{t}(G) \setminus \ee_{s_{n}}(G) \big) + \mm(\ee_{s_{n}}(G) ) \geq \kappa + \mm(\ee_{s_{n}}(G) ).
$$
Taking the limit as $n \to \infty$, the continuity property of Proposition \ref{P:completeregularity} (lower semi-continuity if $t=0$) implies that for each $\ve > 0$:
$$
\mm(\ee_{t}(G)^{\ve}) \geq \kappa + \mm(\ee_{t}(G) ),
$$
with $\kappa$ independent of $\ve$. 
Since $\mm(\ee_{t}(G)) = \lim_{\ve \to 0} \mm(\ee_{t}(G)^{\ve})$
we obtain a contradiction, and the claim is proved. 
\end{proof}

\begin{corollary} \label{cor:mu-on-X0}
With the same assumptions as in Proposition \ref{P:completeregularity}, and assuming that $\supp(\mm) = X$, we have:
$$
\nu (\ee_{0}^{-1}(X^{0})\cap  G^{+}_{\f} ) = 0 ,
$$
where $\f$ is an associated Kantorovich potential to the $c$-optimal-transport problem from $\mu_{0}$ to $\mu_{1}$ with $c = \sfd^2/2$. In particular:
\[  \mu_{t}\llcorner_{X^0} = \mu_0\llcorner_{X^0} \;\;\; \forall t \in [0,1) .
\] \end{corollary}

Recall from Section \ref{sec:order12} that $G_{\f}\subset \Geo(X)$ denotes the set of $\f$-Kantorovich geodesics, $G^{+}_{\f}$ denotes the subset of geodesics in $G_{\f}$ having positive length, and $X^0 = \ee_{[0,1]}(G_\varphi^0)$ denotes the subset of null geodesic points in $X$. Necessarily $\nu (G_{\f}) = 1$. The assumption $\supp(\mm) = X$ guarantees by Lemma \ref{L:proper-support} that $(X,\sfd)$ is proper and geodesic, so that the results of Part I are in force; by Remark \ref{rem:supp-nu} this poses no loss in generality. 

\begin{proof}[Proof of Corollary \ref{cor:mu-on-X0}]
Suppose by contradiction that $\nu (\ee_{0}^{-1}(X^{0})\cap  G^{+}_{\f} ) > 0$. By inner regularity, there exists a compact 
$G \subset \ee_{0}^{-1}(X^{0})\cap  G^{+}_{\f}$ with $\nu(G)>0$ verifying the hypothesis of Proposition \ref{P:completeregularity} and therefore also 
the conclusion of Proposition \ref{P:densityreversed} for $t =0$. 
In particular, for $\mm$-a.e. $x \in \ee_{0}(G) \subset X^{0}$ there exists $\gamma \in G \subset G^{+}_{\f}$ and $t \in (0,1)$ (sufficiently small) such that $x = \gamma_{t}$. 
But $\mu_0(\ee_0(G)) = \nu(\ee_0^{-1}(\ee_0(G))) \geq \nu(G) > 0$, and hence $\mm(\ee_0(G)) > 0$ as $\mu_0 \ll \mm$. It follows that there exists at least one $x \in \ee_{0}(G)$ as above, in direct contradiction to the characterization of $X^0$ given in Lemma \ref{lem:X0}. Hence we can conclude that $\nu$-almost-surely, $\ee_{t}^{-1}(X^{0})$ is contained in the set of null geodesics $G_{\varphi}^0$. For $t \in (0,1)$, $\ee_t^{-1}(X^0) \subset G_\varphi^0$ by Lemma \ref{lem:X0}, and so we conclude that $\mu_{t}\llcorner_{X^0} = \mu_0\llcorner_{X^0}$ for all $t \in [0,1)$. 
\end{proof}

\begin{remark}
When applying the results of this section, note that when both $\mu_0,\mu_1 \ll \mm$, then by reversing the roles of $\mu_0$ and $\mu_1$, we in fact obtain all the above results also at the right end-point $t=1$. 
\end{remark}

\section{Two families of conditional measures}\label{S:Conditional}

The next two sections will be devoted to the study of $W_{2}$-geodesics over $(X,\sfd,\mm)$, when $(X,\sfd,\mm)$ is assumed to be essentially non-branching and verifies $\CD^1(K,N)$. By Remark \ref{R:CD1-localizes}, we also assume $\supp(\mm) = X$. We will use Proposition \ref{P:CD1-ENB} as an equivalent definition for $\CD^1(K,N)$.
By Proposition \ref{P:MCP} and Remark \ref{rem:CD1-MCP}, $X$ also verifies $\MCP(K,N)$, and so Theorem \ref{T:optimalmapMCP} applies. In addition, it follows by Lemma \ref{L:proper-support} that $(X,\sfd)$ is geodesic and proper, and so the results of Part I apply. 

Fix $\mu_{0},\mu_{1} \in \mathcal{P}_{2}(X,\sfd,\mm)$, and denote by $\nu$ the unique element of $\Opt(\mu_{0},\mu_{1})$. As usual, we denote $\mu_t := (\ee_t)_{\sharp} \nu \ll  \mm$ for all $t \in [0,1]$, and set:
$$
\mu_t  =: \rho_{t} \mm \;\;\; \forall t \in [0,1] . 
$$
Fix also an associated Kantorovich potential $\f : X \to \R$ for the $c$-optimal transport problem from $\mu_{0}$ to $\mu_{1}$, with $c = \sfd^2/2$.
Recall that $G_{\f}\subset \Geo(X)$ denotes the set of $\f$-Kantorovich geodesics and that necessarily $\nu (G_{\f}) = 1$.
We further recall from Section \ref{sec:order12} that the interpolating Kantorovich potential and its time-reversed version at time $t \in (0,1)$ are defined for any $x \in X$ as:
$$
-\f_{t} (x) = \inf_{y \in X} \frac{\sfd^{2}(x,y)}{2t} - \f(y) ~,~ \varphic_t(x) = \inf_{y \in X} \frac{\sfd^2(x,y)}{2(1-t)} - \varphi^c(y) \;\;\;\; \forall t \in (0,1) ,
$$
with $\varphi_0 = \varphic_0 = \varphi$ and $\varphi_1 = \varphic_1 = -\varphi^c$. By Proposition \ref{prop:duality-char} we have, for all $t \in (0,1)$, $\varphi_t(x) \leq \varphic_t(x)$, with equality iff $x \in \ee_t(G_\varphi)$. 

\medskip

It will be convenient from a technical perspective to first restrict $\nu$, by inner regularity of Radon measures, Corollary \ref{C:regularity3MCP} (applied to both pairs $\mu_0,\mu_1$ and $\mu_1,\mu_0$), Proposition \ref{P:densityreversed} and Corollary \ref{C:injectivity}, to a suitable \emph{good} compact subset $G \subset G^+_{\f}$ with $\nu(G) \geq \nu(G_\f^+)-\ve$. Recall that $G_\f^+$ was defined in Section \ref{sec:order12} as the subset of geodesics in $G_\f$ having positive length, and
note that the length function $\ell : \Geo(X) \rightarrow [0,\infty)$ is continuous and hence is bounded away from $0$ and $\infty$ on a compact $G \subset G^+_{\f}$. 

\begin{definition}[Good Subset of Geodesics] \label{def:good}
A subset $G \subset G^+_{\f}$ is called \emph{good} if the following properties hold:
\begin{itemize}
\item[-] $G$ is compact;
\item[-] there exists $c > 0$ so that for every $\gamma \in G$:
\begin{equation} \label{eq:length-bound}
c \leq \ell(\gamma) \leq 1/c \; ;
\end{equation} 
\item[-] for every $\gamma \in G$, $\rho_s(\gamma_s) > 0$ for all $s \in [0,1]$ and $(0,1) \ni s \mapsto \rho_s(\gamma_s)$ is continuous; 
\item[-] the claim of Proposition \ref{P:densityreversed} holds true for $G$;
\item[-] The map $\ee_{t}|_G  : G \to X$ is injective (and we will henceforth restrict $\ee_t$ to $G$ or its subsets). 
\end{itemize}
\end{definition}

\begin{assumption} \label{ass:good-assumption}
We will assume in this section and in Subsection \ref{subsec:compare} that:
\[
 \textbf{$\nu$ is concentrated on a good $G \subset G^+_\varphi$.}
\]
We will dispose of this assumption in the Change-of-Variables Theorem \ref{T:changeofV}.
\end{assumption}

\bigskip

\subsection{$L^{1}$ partition}\label{Ss:L1condition}

For $s \in [0,1]$ and $a_{s} \in \R$, we recall the following notation (introduced in Section \ref{S:Phi} for $G = G_\varphi$, but now we treat a general $G \subset G_\varphi$ as above):
$$
G_{a_{s}} = G_{a_s,s} : = \{ \gamma \in G : \f_{s}(\gamma_{s}) = a_{s} \}.
$$
As $G$ is compact and $\ee_s : G \rightarrow X$ is continuous, $\ee_s(G)$ is compact. When $s \in (0,1)$, $\f_s : X \rightarrow \Real$ is continuous by Lemma \ref{lem:order0}, and hence $G_{a_s}$ is compact as well. 

The structure of the evolution of $G_{a_{s}}$, i.e. $\ee_{[0,1]}(G_{a_{s}}) = \{ \gamma_{t} \colon t \in [0,1], \ \gamma \in G_{a_{s}} \}$,
will be the topic of this subsection, so the properties we prove below are only meaningful for $a_s \in \varphi_s(\ee_s(G))$ (and moreover typically when $\mm(\ee_{[0,1]}(G_{a_{s}})) > 0$).
It will be convenient to use a short-hand notation for the signed-distance function from a level set of $\f_{s}$, 
$d_{a_{s}} : = d_{\f_{s} - a_{s}}$ (see \eqref{E:levelsets}).

\begin{lemma}\label{L:dmonotonedistance}
For any $s \in [0,1]$ and $a_{s} \in \varphi_s(\ee_s(G))$ the following holds:
for each $\gamma \in G_{a_{s}}$ and $0 \leq r \leq t \leq 1$, $(\gamma_{r},\gamma_{t}) \in \Gamma_{d_{a_{s}}}$.
In particular, the evolution of $G_{a_{s}}$ is a subset of the transport set associated to $d_{a_{s}}$:
$$
\ee_{[0,1]}(G_{a_{s}}) \subset \mathcal{T}_{d_{a_{s}}}.
$$
\end{lemma}

\begin{proof}
Fix $\gamma \in G_{a_{s}}$. If $s \in [0,1)$ then for any $p \in \{  \f_{s} = a_{s} \}$:
$$
\frac{\sfd^{2}(\gamma_{s}, \gamma_{1})}{2(1-s)} = \f_{s}(\gamma_{s}) + \f^{c}(\gamma_{1}) 
= \f_{s}(p) + \f^{c}(\gamma_{1}) \leq \varphic_{s}(p) + \f^{c}(\gamma_{1})  \leq \frac{\sfd^{2}(p, \gamma_{1})}{2(1-s)}
$$
by Lemma \ref{lem:varphi-linear} and Proposition \ref{prop:duality-char} (2), and hence $\sfd(\gamma_{s},\gamma_{1}) \leq \sfd(p,\gamma_{1})$; the latter also holds for $s=1$ trivially. Similarly, if $s \in (0,1]$ then for any $ q\in \{ \f_{s} = a_{s}\}$:
$$
\frac{\sfd^{2}(\gamma_{0}, \gamma_{s})}{2s} = \f(\gamma_{0})  - \f_{s}(\gamma_{s}) 
= \f(\gamma_{0}) - \f_{s}(q) \leq \frac{\sfd^{2}(\gamma_{0}, q)}{2s},
$$
and therefore $\sfd(\gamma_{0},\gamma_{s}) \leq \sfd(\gamma_{0},q)$, with the latter also holding for $s=0$ trivially. 
Consequently, for any $p,q \in \{ \f_{s} = a_{s} \}$:
$$
\sfd(\gamma_{0},\gamma_{1}) \leq \sfd(\gamma_{0},p) + \sfd(q,\gamma_{1}) .
$$
Taking infimum over $p$ and $q$ it follows that:
$$
 \sfd(\gamma_{0},\gamma_{1}) \leq d_{a_{s}}(\gamma_{0}) - d_{a_{s}}(\gamma_{1}) ,
$$
where the sign of $d_{a_s}$ was determined by the fact that $s \mapsto \varphi_s(\gamma_s)$ is decreasing (e.g. by Lemma \ref{lem:varphi-linear}). 
On the other hand:
\[
d_{a_{s}}(\gamma_{0}) - d_{a_{s}}(\gamma_{1}) \leq  \sfd(\gamma_{0},\gamma_{1}), 
\]
thanks to the $1$-Lipschitz regularity of $d_{a_{s}}$ ensured by Lemma \ref{lem:df-Lip} since $(X,\sfd)$ is geodesic.
Therefore equality holds and $(\gamma_{0},\gamma_{1}) \in \Gamma_{d_{a_{s}}}$. The assertion then follows by  Lemma \ref{L:cicli}. 
\end{proof}

Next, recall by Proposition \ref{P:CD1-ENB} applied to the function $u = d_{a_{s}}$, that according to the equivalent characterization of $\CD^{1}_u(K,N)$, the following disintegration formula holds:
\begin{equation} \label{eq:dis-prelim}
\mm\llcorner_{\T_{d_{a_s}}} = \int_{Q} \hat \mm_{\alpha}^{a_{s}} \,\hat \qq^{a_{s}} (d\alpha),
\end{equation}
where $Q$ is a section of the partition of $\T_{d_{a_s}}^b$ given by the equivalence classes $\{R_{d_{a_s}}^b(\alpha)\}_{\alpha \in Q}$, and for $\hat \qq^{a_s}$-a.e. $\alpha \in Q$, the probability measure $\hat \mm_{\alpha}^{a_{s}}$ is supported on the transport ray $X_{\alpha} = \overline{R^b_{d_{a_s}}(\alpha)} = R_{d_{a_s}}(\alpha)$ and $(X_{\alpha}, \sfd, \hat \mm_{\alpha}^{a_{s}})$ verifies $\CD(K,N)$. It follows by Lemma \ref{L:dmonotonedistance} that:
\begin{equation} \label{eq:dis-prelim2}
\mm\llcorner_{\ee_{[0,1]}(G_{a_{s}})} = \int_{Q} \hat \mm_{\alpha}^{a_{s}}\llcorner_{\ee_{[0,1]}(G_{a_{s}})} \,\hat \qq^{a_{s}}(d\alpha) .
\end{equation}
It will be convenient to make the previous disintegration formula a bit more explicit. We refer to the Appendix for the definition of $\CD(K,N)$ density and the (suggestive) relation to one-dimensional $\CD(K,N)$ spaces. Recall that $\ell_s(\gamma_s) = \len(\gamma)$ for all $\gamma \in G$. 

\begin{proposition}\label{P:L1disintegration}
For any $s \in (0,1)$ and $a_{s} \in \varphi_s(\ee_s(G))$, the following disintegration formula holds:
\begin{equation}\label{E:L1-1}
\mm\llcorner_{\ee_{[0,1]}(G_{a_{s}}) } = \int_{\ee_{s}(G_{a_{s}})} g^{a_{s}}(\beta, \cdot)_{ \#} \left( h^{a_{s}}_{\beta} \cdot \mathcal{L}^{1} \llcorner_{[0,1]} \right) 
\qq^{a_{s}}(d\beta),
\end{equation}
with $\qq^{a_{s}}$ a Borel measure concentrated on $\ee_{s}(G_{a_{s}})$ 
of mass $\mm(\ee_{[0,1]}(G_{a_{s}}))$,  $g^{a_{s}} : \ee_s(G_{a_s}) \times [0,1] \rightarrow X$ is defined by $g^{a_{s}}(\beta,t)  = \ee_t(\ee_{s}^{-1}(\beta))$ and is Borel measurable, for $\qq^{a_s}$-a.e. $\beta \in \ee_{s}(G_{a_{s}})$, $h^{a_{s}}_{\beta}$ is a $\CD(\ell_s(\beta)^2 K,N)$ probability density on $[0,1]$ vanishing at the end-points, and the map $\ee_{s}(G_{a_{s}}) \times [0,1] \ni (\beta,t) \mapsto  h^{a_{s}}_{\beta}(t)$ is $\qq^{a_s} \otimes \L^1\llcorner_{[0,1]}$-measurable. 
\end{proposition}

\begin{proof}

We will abbreviate $u = d_{a_s}$.

\smallskip
\textbf{Step 1.} We claim that:
\[
\forall \gamma \in G_{a_s} \;\; \forall \alpha \in Q ~,~  \ee_{[0,1]}(\gamma) \cap R_{u}^b(\alpha) \neq \emptyset  \; \Rightarrow \;  R_u(\alpha) \supset \ee_{[0,1]}(\gamma) . 
\]
Indeed, if $x \in \ee_{[0,1]}(\gamma)$, then $R_u(x) \supset \ee_{[0,1]}(\gamma)$ by Lemma \ref{L:dmonotonedistance}. But on the other hand, $R_u(x) = R_u(\alpha)$ for all $x \in R_u^b(\alpha)$, since any two transport rays intersecting in $\T_{u}^b$ must coincide by Corollary \ref{C:RayInTub}. Hence, if $\exists x \in \ee_{[0,1]}(\gamma) \cap R_{u}^b(\alpha)$, the assertion follows.

\smallskip
\textbf{Step 2.} We also claim that:
\[
\forall \gamma^1,\gamma^2 \in G_{a_s}\;\; \forall \alpha \in Q ~,~  \ee_{[0,1]}(\gamma^i) \cap R_{u}^b(\alpha) \neq \emptyset \; ,\;  i=1,2 \; \; \Rightarrow \; \gamma^1 = \gamma^2 .
\]
Indeed, since $\alpha \in Q \subset \T_u^b$ then $R_u(\alpha)$ is a transport ray by Lemma \ref{L:RayInTub}, and since $u = d_{a_s}$ is affine (with slope $1$) on a transport ray, $R_{u}(\alpha)$ must intersect $\{d_{a_s} = 0\}=\{\f_{s} = a_{s} \}$, and hence $\ee_{s} (G_{a_{s}})$, at most once. It follows by \textbf{Step 1} that $\gamma^1_s = \gamma^2_s$, and so by injectivity of $\ee_s|_G : G \rightarrow X$, that $\gamma^1 = \gamma^2$. 

\smallskip
\textbf{Step 3.} 
Denote:
\[
G^1_{a_s} := \set{ \gamma \in G_{a_s} \; ; \; \T_u^b \cap \ee_{[0,1]}(\gamma) \neq \emptyset} \; , \; Q^1 := \set{ \alpha \in Q \; ; \; R_u^b(\alpha) \cap \ee_{[0,1]}(G_{a_s})  \neq \emptyset} . \]
We claim that  there exists a bijective map:
\[
\eta : Q^1 \ni \alpha \mapsto \gamma^\alpha \in G^1_{a_s} ,
\]
 for which:
\[
R_u^b(\alpha) \cap \ee_{[0,1]}(G_{a_s}) = \T_u^b \cap \ee_{[0,1]}(\gamma^\alpha) = R_u^b(\alpha) \cap \ee_{[0,1]}(\gamma^\alpha)  .
\]
Indeed, for all $\alpha \in Q^1$, there exists precisely one $\gamma \in G_{a_s}$ (and hence $\gamma \in G^1_{a_s}$) so that $R_u^b(\alpha) \cap \ee_{[0,1]}(\gamma) \neq \emptyset$ by \textbf{Step 2}. And vice versa, given any $\gamma \in G^1_{a_s}$, there is at least one $\alpha \in Q$ (and hence $\alpha \in Q^1$) so that $R_u^b(\alpha) \cap \ee_{[0,1]}(\gamma) \neq \emptyset$, and it follows by \textbf{Step 1} that $\ee_{[0,1]}(\gamma) \subset R_u(\alpha)$ and hence $\T_u^b \cap \ee_{[0,1]}(\gamma) \subset R_u^b(\alpha)$; but this means that for all $\alpha \neq \beta \in Q$, $R_u^b(\beta) \cap \ee_{[0,1]}(\gamma) = \emptyset$, since $\set{R_u^b(\beta)}_{\beta \in Q}$ is a partition of $\T_u^b$, implying the uniqueness of $\alpha \in Q^1$.

Moreover, we claim that the map $\eta : (Q^1,\B(Q^1)) \rightarrow (G^1_{a_s} , \B(G^1_{a_s}))$ is measurable. Indeed, recall that $G_{a_s}$ is compact, and since $(X,\sfd)$ is proper, $\T_u^b$ and $R_u^b$ are Borel, and hence $G^1_{a_s}$ is analytic. 
Then write:
\[
\Lambda := P_{1,2}( \{ (y, \gamma , x, t) \in \T_u^b \times G_{a_s} \times X \times [0,1]  \; ; \; (y,x) \in R_u^b \; , \; x = \gamma_t  \} ) ,
\]
and:
\[
\gr(\eta) = \Lambda \cap (Q^1 \times G_{a_s}) = \Lambda \cap (Q^1 \times G^1_{a_s})  .
\]
Note that $\Lambda$ is analytic and that $\Lambda(x)$ is either an empty set or a singleton for all $x \in \T_u^b$ by \textbf{Step 2} (and the fact that $R_u^b$ is an equivalence relation on $\T_u^b$).
It follows that for any $B \in \B(G_{a_s})$, both $A_1 = P_1(\Lambda \cap (\T_u^b \times B))$ and $A_2 = P_1(\Lambda \cap (\T_u^b \times (G_{a_s} \setminus B)))$ are analytic, disjoint and $Q^1 = (Q^1 \cap A_1) \cup (Q^1 \cap A_2)$. By the Lusin separability principle \cite[Theorem 4.4.1]{Srivastava}, there exists a Borel subset $B_1 \subset \T_u^b$ containing $A_1$ which is still disjoint from $A_2$. Consequently $\eta^{-1}(B \cap G_{a_s}^1)  = \eta^{-1}(B) = Q^1 \cap A_1 = Q^1 \cap B_1 \in \B(Q^1)$, concluding the proof that $\eta$ is Borel measurable on $Q^1$. 

\smallskip
\textbf{Step 4.} Recall that for all $\alpha \in \bar Q$ of full $\hat \qq^{a_s}$ measure, $\hat \mm_\alpha^{a_s}$ is supported on the transport ray $R_u(\alpha) = \overline{R_u^b(\alpha)}$ and $(R_u(\alpha),\sfd,\hat \mm_\alpha^{a_s})$ verifies $\CD(K,N)$. Consequently, for such $\alpha$'s, $\hat \mm_\alpha^{a_s}$ gives positive mass to any relatively open subset of $R_u(\alpha)$ and does not charge points. It follows that for $\alpha \in \bar Q$, since $\ee_{[0,1]}(\gamma^\alpha) \subset R_u(\alpha)$ has non-empty relative interior, it holds that:\begin{align*}
& R_u^b(\alpha) \cap \ee_{[0,1]}(G_{a_s})  \neq \emptyset \; \Leftrightarrow \; \\
& \hat \mm^{a_s}_\alpha(\ee_{[0,1]}(G_{a_s})) = \hat \mm^{a_s}_\alpha(R_u^b(\alpha) \cap \ee_{[0,1]}(G_{a_s})) = \hat \mm^{a_s}_\alpha( R_u^b(\alpha) \cap \ee_{[0,1]}(\gamma^\alpha)) = \hat \mm^{a_s}_\alpha(\ee_{[0,1]}(\gamma^\alpha)) > 0 .
\end{align*}
In particular, $Q^1$ coincides up to a $\hat \qq^{a_s}$-null set with the $\hat \qq^{a_s}$-measurable set $Q^2 := \{ \alpha \in Q \; ; \; \hat \mm_\alpha^{a_s}(\ee_{[0,1]}(G_{a_s})) > 0 \}$, and thus $Q^1$ is itself $\hat \qq^{a_s}$-measurable. In fact, it is easy to see that $Q^1$ coincides with an analytic set up to a $\hat \qq^{a_s}$-null-set. 

\smallskip
\textbf{Step 5.} 
Recalling that $\ee_{[0,1]}(G_{a_s}) \subset \T_u$ by Lemma \ref{L:dmonotonedistance} and that $\mm(\T_u \setminus \T_u^b) = 0$ by Corollary \ref{C:disintMCP}, we obtain from (\ref{eq:dis-prelim}) the following disintegration of $\mm\llcorner_{\ee_{[0,1]}(G_{a_s})}$: \begin{align*}
& \mm\llcorner_{\ee_{[0,1]}(G_{a_s})}  = \mm\llcorner_{\T_u^b \cap \ee_{[0,1]}(G_{a_s})} = \int_{Q} \hat \mm_{\alpha}^{a_{s}}\llcorner_{\T_u^b \cap \ee_{[0,1]}(G_{a_s})} \,\hat \qq^{a_{s}}(d\alpha) \\
& =  \int_{\bar Q \cap Q^1} \hat \mm_{\alpha}^{a_{s}}\llcorner_{\ee_{[0,1]}(\gamma^\alpha)} \hat \qq^{a_{s}}(d\alpha) 
 =  \int_{\bar Q \cap Q^1} \frac{\hat \mm_{\alpha}^{a_{s}}\llcorner_{\ee_{[0,1]}(\gamma^\alpha)}}{\hat \mm_{\alpha}^{a_{s}}(\ee_{[0,1]}(\gamma^\alpha))}  \hat \mm_{\alpha}^{a_{s}}(\ee_{[0,1]}(\gamma^\alpha)) \hat \qq^{a_{s}}(d\alpha) ,
\end{align*}
where the last two transitions and the measurability of $\alpha \mapsto \hat \mm_{\alpha}^{a_{s}}(\ee_{[0,1]}(\gamma^\alpha)) > 0$ follow from \textbf{Step 4}. 
For all $\alpha \in \bar Q \cap Q^1$, define the probability measure:
\[
\bar \mm_{\alpha}^{a_s} := \frac{\hat \mm_{\alpha}^{a_{s}}\llcorner_{\ee_{[0,1]}(\gamma^\alpha)}}{\hat \mm_{\alpha}^{a_{s}}(\ee_{[0,1]}(\gamma^\alpha))} .
\]
Since $\ee_{[0,1]}(\gamma^\alpha)$ is a convex subset of $R_u(\alpha)$, it follows that the one-dimensional \mms $(\ee_{[0,1]}(\gamma^\alpha), \sfd,\bar \mm_{\alpha}^{a_s})$ verifies $\CD(K,N)$ and is of full support for all $\alpha \in \bar Q \cap Q^1$. Similarly, define:
\[
\bar \qq^{a_s} := \hat \mm_{\alpha}^{a_{s}}(\ee_{[0,1]}(\gamma^\alpha)) \hat \qq^{a_{s}} \llcorner_{Q^1} (d \alpha) .
\]

\smallskip
\textbf{Step 6.} Recall that our original disintegration (\ref{eq:dis-prelim}) was on $(Q,\mathscr{Q},\hat \qq^{a_s})$, so that there exists $\tilde{Q} \subset Q$ of full $\hat \qq^{a_s}$ measure so that $\tilde{Q} \in \B(\T_u^b)$ and $ \mathscr{Q} \supset \B(\tilde{Q})$. It follows that we may find $\mathscr{Q} \ni \tilde{Q}^1 \subset Q^1$ with $\hat \qq^{a_s}(Q^1 \setminus \tilde{Q}^1) = 0$ so that $\mathscr{Q} \supset \B(\tilde{Q}^1)$. Let us now push-forward the measure space $(Q^1 , \mathscr{Q} \cap Q^1 , \bar \qq^{a_s})$ via the Borel measurable map $\ee_s \circ \eta$ (by \textbf{Step 3}), yielding the measure space $(\ee_s(G_{a_s}^1), \mathscr{S}, \qq^{a_s})$, which is thus guaranteed to satisfy $\mathscr{S} \supset \B(\tilde{S})$, where $\tilde{S} := \ee_s \circ \eta (\tilde{Q}^1)$ is of full $\qq^{a_s}$ measure. 
Restricting the space to $\tilde{S}$ and abusing notation, we obtain $(\tilde{S}, \mathscr{S}, \qq^{a_s})$ with $\mathscr{S} \supset \B(\tilde{S})$,
 implying that $\qq^{a_s}$ is a Borel measure concentrated on $\tilde{S} \subset \ee_s(G^1_{a_s}) \subset \ee_s(G_{a_s})$. Note that $\hat \qq^{a_{s}}$, $\bar \qq^{a_s}$ and $\qq^{a_s}$ all have total mass 
$\mm(\ee_{[0,1]}(G_{a_{s}}))$. 

Denoting $\mm_{\gamma^\alpha_s}^{a_s} := \bar \mm_{\alpha}^{a_s}$, the disintegration from \textbf{Step 5} translates to:
\[
\mm\llcorner_{\ee_{[0,1]}(G_{a_s})} = \int_{\ee_s(G_{a_s})} \mm_{\beta}^{a_s} \qq^{a_s}(d \beta) .
\]
Furthermore, for $\qq^{a_s}$-a.e. $\beta$, the \mms $(\ee_{[0,1]}(\ee_s^{-1}(\beta)), \sfd, \mm_{\beta}^{a_s})$ verifies $\CD(K,N)$ and is of full support, and is therefore isometric to 
$(I^{a_{s}}_{\beta},\abs{\cdot}, \hat h^{a_{s}}_{\beta} \L^1\llcorner_{I^{a_{s}}_{\beta}})$, where $I^{a_{s}}_{\beta} : = [0,\ell_{s}(\beta)]$ and $\hat h^{a_{s}}_{\beta}$ is a $\CD(K,N)$ probability density on $I^{a_{s}}_{\beta}$ (see Definition \ref{def:CDKN-density}). To prevent measurability issues, we will use the convention that $\hat h^{a_{s}}_{\beta}$ vanishes at the end-points of $I^{a_{s}}_{\beta}$. 

\smallskip
\textbf{Step 7.} Next, we observe that $g^{a_s}$ is Borel. Indeed, note that by injectivity of $\ee_s$: \[
\gr(g^{a_s}) = P_{1,2,3} (\{ (\beta,t,x,\gamma) \in \ee_{s}(G_{a_{s}}) \times [0,1] \times X \times G_{a_s} \; ; \;  \gamma_s = \beta ~,~ \gamma_{t} = x  \}) .
\]
As $G_{a_s}$ is compact, it follows that $\gr(g^{a_s})$ is analytic, and hence (see \cite[Theorem 4.5.2]{Srivastava}) $g^{a_s}$ is Borel measurable. 

\smallskip

\textbf{Step 8.} It follows that $\mm_{\beta}^{a_s} = g^{a_s}(\beta,\cdot)_{\sharp}(h^{a_{s}}_{\beta} \L^1\llcorner_{[0,1]})$, where:
\[
[0,1] \ni t \mapsto h^{a_{s}}_{\beta} (t) : = \ell_s(\beta) \hat h^{a_{s}}_{\beta}(  t \ell_{s}(\beta)  ) .
\]
Clearly $h_{\beta}^{a_{s}}$ is now a $\CD(\ell_{s}(\beta)^{2}K,N)$ probability density on the interval $[0,1]$. The only remaining task is to prove that 
the map $\ee_s(G_{a_s}) \times [0,1] \ni (\beta,t) \mapsto h^{a_s}_\beta(t)$ is $\qq^{a_s} \otimes \L^1\llcorner_{[0,1]}$-measurable. 
By measurability of the disintegration (\ref{eq:dis-prelim2}) (recall Definition \ref{defi:dis}), the map $Q \ni \alpha \mapsto \hat \mm_\alpha^{a_s}(B)$ is $\hat \qq^{a_s}$-measurable for any Borel set $B \subset X$.
It follows that for any compact $I \subset (0,1)$, the map: 
$$
\ee_{s}(G_{a_{s}})\supset \tilde{S} \ni \beta \mapsto F (\beta) : = \int_{I} h^{a_{s}}_{\beta} (\tau) d\tau = \frac{\hat \mm^{a_s}_{\alpha(\beta)}(\ee_{I}(G_{a_s}))}{\hat \mm^{a_s}_{\alpha(\beta)}(\ee_{[0,1]}(G_{a_s}))} ,
$$
is $\qq^{a_{s}}$-measurable, where $\alpha(\beta) := (\ee_s \circ \eta)^{-1}(\beta)$ is $\qq^{a_s}$-measurable as a map 
$(\tilde{S},\mathscr{S},\qq^{a_s}) \rightarrow  (Q , \mathscr{Q} , \hat \qq^{a_s})$
by the construction from \textbf{Step 6}. As $h^{a_{s}}_{\beta}$ is continuous on $(0,1)$ for $\q^{a_s}$-a.e. $\beta$, we know that for such $\beta$ and all $t \in (0,1)$:
$$
h^{a_{s}}_{\beta}(t) = \lim_{\ve \to 0} \frac{1}{2\ve}\int_{[t-\ve,t+\ve]} h^{a_{s}}_{\beta} (\tau) d\tau .
$$
It follows by \cite[Proposition 3.1.27]{Srivastava} that for all $t \in (0,1)$, the map:
$$
\tilde{S} \ni \beta \mapsto h^{a_{s}}_{\beta}(t) 
$$
is $\qq^{a_{s}}$-measurable. As for $\q^{a_s}$-a.e. $\beta$, the map $(0,1) \ni t \mapsto h^{a_{s}}_{\beta}(t)$ is continuous, \cite[Theorem 3.1.30]{Srivastava} confirms the required measurability.
 
\smallskip

This concludes the proof.

\end{proof}

It will be convenient to invert the order of integration in \eqref{E:L1-1} using Fubini's Theorem: $$
\mm\llcorner_{\ee_{[0,1]}(G_{a_{s}}) } = \int_{[0,1]} g^{a_{s}}(\cdot,t)_{\sharp} \left( h^{a_{s}}_{\cdot}(t) \cdot \qq^{a_{s}} \right) \mathcal{L}^{1}(dt) .
$$
We thus define:
\[ 
\mm_{t}^{a_{s}} : = g^{a_{s}}(\cdot,t)_{\sharp} \left( h^{a_{s}}_{\cdot}(t) \cdot \qq^{a_{s}} \right),
\] 
so that the final formula is: 
\begin{equation}\label{E:L1-2}
\mm\llcorner_{\ee_{[0,1]}(G_{a_{s}}) } = \int_{[0,1]} \mm_{t}^{a_{s}} \, \mathcal{L}^{1}(dt). 
\end{equation}

\begin{remark} \label{rem:h-normalization}
Since for $\qq^{a_s}$-a.e. $\beta$, the $\CD(\ell_s^2(\beta) K , N)$ density $h^{a_s}_\beta$ must be strictly positive on $(0,1)$ (see Appendix),  
by multiplying and dividing $\qq^{a_s}$ by the positive $\qq^{a_s}$-measurable function $\beta \mapsto h^{a_s}_\beta(s)$ (recall that $s \in (0,1)$), we may always renormalize and assume that $h^{a_{s}}_{\beta}(s) = 1$. Note that this does not affect the definition of $\mm_{t}^{a_{s}}$ above. 
This normalization ensures that $\mm_s^{a_s} = \qq^{a_s}$ so that:
\begin{equation}  \label{eq:def-by-pushf}
 \mm_{t}^{a_{s}} : = g^{a_{s}}(\cdot,t)_{\sharp} \left( h^{a_{s}}_{\cdot}(t) \cdot \mm_s^{a_{s}} \right) .
\end{equation}
\end{remark}

\medskip

\begin{remark}\label{R:disjointsubset}
Note that since $\qq^{a_{s}}$ is concentrated on $\ee_s(G_{a_s})$, by definition $\mm_{t}^{a_{s}}$ is concentrated on $\ee_t(G_{a_s})$ for all $t \in (0,1)$. By Corollary \ref{cor:IntraProp}, the latter sets are disjoint for different $t$'s in $(0,1)$ (recall that $s \in (0,1)$ and that $G \subset G^+_{\varphi}$). 
Formula \eqref{E:L1-2} can thus be seen again as a disintegration formula over a partition. In particular, for any $s \in (0,1)$ and $0 < t , \tau < 1$ with $t \neq \tau$, the measures $\mm^{a_{s}}_{t}$ and $\mm^{a_{s}}_{\tau}$ are mutually singular.
\end{remark}

\medskip

\begin{proposition}\label{P:continuity}
For any $s\in (0,1)$ and $a_{s} \in \varphi_s(\ee_s(G))$, the map 
$$
(0,1) \ni t \mapsto \mm^{a_{s}}_{t}
$$
is continuous in the weak topology, we have: 
\[
\mm(\ee_{[0,1]}(G_{a_{s}})) > 0 \;\;\; \Rightarrow \;\;\; \forall t \in (0,1) \;\;\; \mm^{a_{s}}_{t}(\ee_{t}(G_{a_{s}}))  > 0 ,
\]
and:
$$
\forall t\in [0,1] \;\;\; \mm^{a_{s}}_{t}(\ee_{t}(G_{a_{s}})) =  \| \mm^{a_{s}}_{t} \| \leq C \; \mm(\ee_{[0,1]}(G_{a_{s}})),
$$
for some $C> 0$ depending only on $K$, $N$ and $c > 0$ from assumption (\ref{eq:length-bound}). 
\end{proposition}

\begin{proof}
Recall that the definition of $\mm^{a_{s}}_{t}$ does not depend on the last normalization we performed, when we imposed that $h^{a_s}_\beta(s) = 1$, so we revert to the normalization that $h^{a_s}_\beta$ is a $\CD(\ell_s(\beta)^2 K , N)$ probability density on $[0,1]$, and hence $\norm{\qq^{a_{s}}}= \mm(\ee_{[0,1]}(G_{a_{s}}))$. 
The second assertion follows since whenever the latter mass is positive, by positivity of a $\CD(K,N)$ density in the interior of its support (see Appendix):
\[
\forall t \in (0,1) \;\;\; \mm^{a_{s}}_{t}(\ee_{t}(G_{a_{s}})) = \norm{\mm^{a_{s}}_{t} } = \int h^{a_{s}}_{\beta}(t)  \qq^{a_s}(d\beta) > 0 .
\]
Similarly, it follows by Lemma \ref{lem:apriori0}, the lower semi-continuity of $h^{a_s}_\beta$ at the end-points (see Appendix), and assumption (\ref{eq:length-bound}), that $\max_{t \in [0,1]} h^{a_{s}}_{\beta}(t)$ is uniformly bounded in $a_{s}$ and $\beta$ for $\qq^{a_{s}}$-a.e. $\beta$ by a constant $C>0$ as above, implying that:
\[
\forall t \in [0,1] \;\;\; \norm{\mm^{a_{s}}_{t} } = \norm{h^{a_{s}}_{\cdot}(t) \cdot \qq^{a_{s}}} \leq C \norm{\qq^{a_{s}}} = C \; \mm(\ee_{[0,1]}(G_{a_{s}})) ,
\]
yielding the third assertion. 

Now note that the density $(0,1) \ni t \mapsto h^{a_{s}}_{\beta}(t)$ is continuous (see Appendix) for $\qq^{a_{s}}$-a.e. $\beta$, and the same trivially holds for the map $[0,1] \ni t \mapsto g^{a_{s}}(\beta,t)$. We conclude by Dominated Convergence that for any $f \in C_{b}(X)$ and any $t \in (0,1)$:
\begin{align*}
\lim_{\tau \to t}\int f(x) \, \mm^{a_{s}}_{\tau}(dx) =&~ \lim_{\tau \to t} \int f(g^{a_{s}}(\alpha,\tau)) h^{a_{s}}_{\beta}(\tau) \, \qq^{a_{s}}(d\beta)  \\
= &~ \int f(g^{a_{s}}(\beta,t)) h^{a_{s}}_{\alpha}(t) \, \qq^{a_{s}}(d\beta) 
= \int f(x) \, \mm^{a_{s}}_{t}(dx) , 
\end{align*}
yielding the first assertion, and concluding the proof. 
\end{proof}

\subsection{$L^{2}$ partition}
For each $t \in (0,1)$, we can find a natural partition of $\ee_{t}(G) \subset \ee_t(G_\varphi)$ consisting of level sets of the time-propagated intermediate Kantorovich potentials $\Phi_{s}^{t}$ introduced in Section \ref{S:Phi}. Recall that the function $\Phi_s^t$ ($s,t \in (0,1)$) was defined as:
$$
\Phi_{s}^{t} = \f_{t} + (t-s) \frac{\ell_{t}^{2}}{2} ,
$$
and interpreted on $\ee_{t}(G_\varphi)$ as the propagation of $\f_s$ from time $s$ to $t$ along $G_\varphi$, i.e. $\Phi_{s}^{t} = \f_{s} \circ \ee_{s} \circ \ee_{t}^{-1}$. In particular, for any $\gamma \in G$, $\Phi_{s}^{t}(\gamma_{t}) = \f_{s}(\gamma_{s})$, and $\ee_{t} (G_{a_{s}}) \cap \ee_{t} (G_{b_{s}}) = \emptyset$ as soon as $a_{s} \neq b_{s}$ (see Corollary \ref{cor:InterProp}). 
It follows that for any $s,t \in (0,1)$, we can consider the partition of the compact set $\ee_{t}(G)$ given by its intersection with the family $\{ \Phi_{s}^{t}= a_{s} \}_{a_{s} \in \R}$;
as usual, it will be sufficient to take $a_{s} \in \Phi_{s}^{t}(\ee_{t}(G)) = \f_{s}(\ee_{s}(G))$.

Since $\Phi_{s}^{t}$ is continuous, the Disintegration Theorem \ref{T:disintegrationgeneral} yields the following essentially unique disintegration 
of $\mm\llcorner_{\ee_{t}(G)}$ strongly consistent with respect to the quotient-map $\Phi_{s}^{t}$:
\begin{equation} \label{E:L2}
\mm\llcorner_{\ee_{t}(G)} = \int_{\f_{s}(\ee_{s}(G))} \hat \mm^{t}_{a_{s}} \, \qq^{t}_{s}(da_{s}), 
\end{equation}
so that for $\qq^{t}_{s}$-a.e. $a_{s}$, $\hat \mm^{t}_{a_{s}}$ is a probability measure concentrated on the set $\ee_t(G) \cap \set{\Phi_s^t = a_s} = \ee_{t}(G_{a_{s}})$. By definition, 
$\qq^{t}_{s} = (\Phi_{s}^{t})_{\#} \mm \llcorner_{\ee_{t}(G)}$. To make this disintegration more explicit, we show:

\begin{proposition}\label{P:2montone}
\hfill
\begin{enumerate}
\item For any $s,t,\tau \in (0,1)$, the quotient measures $\qq^{t}_{s}$ and $\qq^{\tau}_s$ are mutually absolutely continuous.
\item For any $s,t \in (0,1)$, the quotient measure $\qq^{t}_{s}$ is absolutely continuous with respect to Lebesgue measure $\mathcal{L}^1$ on $\R$. 
\end{enumerate}
\end{proposition}

\begin{proof}
Recall that $\qq^{t}_{s} = (\Phi_{s}^{t})_{\#} \mm \llcorner_{\ee_{t}(G)}$. 

\medskip

(1) For any Borel set  $I \subset \R$, note that:
\[
\qq^{t}_{s}(I) = \mm \left( \left\{ \gamma_{t} : \f_{s}(\gamma_{s}) \in I, \gamma \in G \right\} \right)  > 0 \;\; \Leftrightarrow \;\;  \mu_{t}\left( \left\{ \gamma_{t} : \f_{s}(\gamma_{s}) \in I, \gamma \in G \right\} \right)  > 0 ,
\]
since $\mu_t \ll \mm$ and its density $\rho_t$ is assumed to be positive on $\ee_t(G)$ where $\mu_t$ is supported (see Definition \ref{def:good}).
But $\mu_\tau = (\ee_\tau \circ \ee_t^{-1})_{\sharp} \mu_t$, and so:
\[
\mu_{\tau}\left( \left\{ \gamma_{\tau} : \f_{s}(\gamma_{s}) \in I, \gamma \in G \right\} \right) 
= \mu_{t}\left( \left\{ \gamma_{t} : \f_{s}(\gamma_{s}) \in I, \gamma \in G \right\} \right)  .
\] 
It follows that $\qq^{t}_{s}(I)  > 0$ iff $\qq^{\tau}_{s}(I)  > 0$, thereby establishing the first assertion. 

\medskip

(2) Thanks to the first assertion, it is enough to only consider the case $t=s$ in the second one. Recall that $\Phi_{s}^{s} = \f_{s}$.  Then the claim boils down to showing that  
$\mm(\f_{s}^{-1} (I) \cap \ee_{s}(G)) = 0$ whenever $I \subset \f_{s}(\ee_{s}(G))$ is a compact set with $\L^{1}(I) = 0$.

By compactness, we fix a ball $B_{r}(o)$ containing $\ee_{s}(G)$. Since $\f_{s}$ is Lipschitz continuous on bounded sets (Corollary \ref{cor:AGS} (1)), 
possibly using a cut-off Lipschitz function over $B_{r}(o)$, we may assume that $\f_{s}$ has bounded total variation measure  
$\| D \f_{s} \|$ (we refer to \cite{miranda:bvmetric} and \cite{ambromarino:bvgeneral} for all missing notions and background regarding BV-functions on metric-measure spaces).
From the local Poincar\'e inequality (see Remark \ref{R:cutlocus} and \cite[page 992]{miranda:bvmetric}) and the doubling property (see Lemma \ref{L:proper-support} and recall that $\supp(\mm) = X$), it follows that 
the total variation measure of $\f_{t}$ is absolutely continuous with respect to $\mm$, and that:
\begin{equation} \label{eq:BV-AC}
\exists c > 0 \;\;\;  c |\nabla \f_{s}| \, \mm \leq \| D \f_{s} \| \leq |\nabla \f_{s}| \, \mm  
\end{equation}
(see \cite[page 992]{miranda:bvmetric} or \cite[Section 4]{miranda:bvcoarea}), where: 
$$
|\nabla \f_{s}|(x) : = \liminf_{\delta \to 0} \sup_{y \in B_{\delta}(x)} \frac{|\f_{s}(y)-\f_{s}(x)|}{\delta}.
$$
By \cite[Theorem 6.1]{cheeger:Lip}, the previous quantity in fact coincides in our setting with the pointwise Lipschitz constant of $\f_{s}$ at $x$, which in turn coincides with $\ell^+_s(x)$
by \cite[Theorem 3.6]{ambrgisav:heat}; hence for $x = \gamma_s$ we have $|\nabla \f_{s}|(x) =\ell_s(x)$.
By the co-area formula (see \cite[Proposition 4.2]{miranda:bvmetric}), for any Borel set $A \subset B_{r}(o)$:
\begin{equation} \label{eq:coarea}
\int_{-\infty}^{+\infty} \| \partial \{ \f_{s} > \tau \} \| (A) \, d\tau = \| D\f_{s} \|(A),
\end{equation}
where $\| \partial \{ \f_{s} > \tau \} \|$ denotes the total variation measure associated to the set of finite perimeter $\{ \f_{s} > \tau \}$.
From \cite[Theorem 5.3]{ambro:perimeter} it follows that $\| \partial \{ \f_{s} > \tau \} \|$ is concentrated on $\{\f_{s} = \tau \}$ and therefore, 
for any Borel set $I \subset \f_{s}(\ee_{s}(G))$ with $\L^{1}(I) = 0$, it follows by (\ref{eq:coarea}) and (\ref{eq:BV-AC}):
$$
\| D \f_{s} \| (\f_{s}^{-1}(I)) = 0 \; , \; |\nabla\f_{s}| \,\mm (\f_{s}^{-1}(I)) = 0 .
$$
Since $|\nabla\f_{s}| = \ell_s(x) > 0$ on $\ee_{s}(G)$, it follows that $\mm (\f_{s}^{-1}(I) \cap \ee_{s}(G)) = 0$, thereby concluding the proof.
\end{proof}

\begin{remark}
Inspecting the proof of Proposition \ref{P:2montone}, from the co-area formula (\cite[Proposition 4.2]{miranda:bvmetric}) and the 
Hausdorff representation of the perimeter measure (\cite[Theorem 5.3]{ambro:perimeter}), it follows that for  $\qq^{s}_{s}$-a.e. $a_{s} \in \f_{s}(\ee_{s}(G))$ 
the measure $\mm^{s}_{a_{s}}$ is absolutely continuous with respect to the Hausdorff measure of codimension one (see \cite{ambro:perimeter} for more details). 
\end{remark}

Employing the previous proposition, we define:
\[
 \mm^{t}_{a_{s}} :=  (d\qq^t_s/d\mathcal{L}^1) \cdot \hat \mm^{t}_{a_{s}} , 
\]
obtaining from (\ref{E:L2}) the following disintegration (for every $s,t \in (0,1)$):
\begin{equation}\label{E:L2-1}
\mm\llcorner_{\ee_{t}(G)} = \int_{\f_{s}(\ee_{s}(G))} \mm^{t}_{a_{s}} \, \mathcal{L}^{1}(da_{s}), 
\end{equation}
with $\mm^{t}_{a_{s}}$ concentrated on $\ee_{t}(G_{a_{s}})$, for $\L^1$-a.e. $a_{s} \in \f_{s}(\ee_{s}(G))$.

\medskip

We now shed light on the relation of the above disintegration to $L^2$-Optimal-Transport, by relating it to another disintegration 
formula for $\nu$, the unique element of $\Opt(\mu_{0},\mu_{1})$. Observe that the family of sets $\{G_{a_{s}}\}_{a_{s}\in \R}$ is a partition of $G$ and that $G_{a_{s}} = \set{ \f_{s} \circ \ee_{s} = a_{s}}$.
Since the quotient-map $\f_{s} \circ \ee_{s} : \Geo(X) \to \R$ is continuous and $G$ is compact, the Disintegration Theorem \ref{T:disintegrationgeneral} ensures 
the existence of an essentially unique disintegration of $\nu$ strongly consistent with $\f_{s} \circ \ee_{s}$:
\begin{equation}\label{E:disintegrationu}
\nu  = \int_{\f_{s}(\ee_{s}(G))} \nu_{a_{s}} \, \qq^{\nu}_{s}(da_{s}), 
\end{equation}
so that for $\qq^{\nu}_{s}$-a.e. $a_{s} \in \f_{s}(\ee_{s}(G))$, the probability measure $\nu_{a_s}$ is concentrated on $G_{a_s}$. Clearly $\qq^{\nu}_{s}(\f_{s}(\ee_{s}(G))) = \norm{\nu}=1$.

\begin{corollary}\label{C:continuity}
\hfill
\begin{enumerate}
\item
For any $s \in (0,1)$, the quotient measure $\qq^{\nu}_{s}$ is mutually absolutely continuous with respect to $\qq^{s}_{s}$, 
and in particular it is absolutely continuous with respect to $\L^{1}$. 
\item 
For any $s,t \in (0,1)$ and $\L^1$-a.e. $a_{s} \in \f_{s}(\ee_{s}(G))$:
\begin{equation} \label{eq:continuity-second}
\rho_{t} \cdot \mm^{t}_{a_{s}} =  q^{\nu}_s(a_{s}) \cdot  (\ee_{t})_{\#} \nu_{a_{s}} ,
\end{equation}
where $q^{\nu}_s := d \qq^{\nu}_{s} / d\mathcal{L}^{1}$. In particular,  $\mm^{t}_{a_{s}}$  and $(\ee_{t})_{\#} \nu_{a_{s}}$ are mutually absolutely-continuous for  $\qq^{\nu}_{s}$-a.e. $a_{s} \in \f_{s}(\ee_{s}(G))$. 
\item
In particular, for any $s \in (0,1)$ and $\qq^{\nu}_{s}$-a.e. $a_{s} \in \f_{s}(\ee_{s}(G))$, the map: $$
[0,1] \ni t \mapsto \rho_{t} \cdot \mm^{t}_{a_{s}} 
$$
coincides for $\L^{1}$-a.e. $t \in [0,1]$ with the $W_2$-geodesic $t \mapsto (\ee_{t})_{\sharp} \nu_{a_{s}}$ up to a positive 
multiplicative constant depending only on $a_{s}$. 
\end{enumerate}
\end{corollary}

\begin{proof}
Recall that $\mu_s \ll \mm$ is supported on $\ee_s(G)$ and $\rho_{s} > 0$ there (see Definition \ref{def:good}), so that $\mu_s$ and $\mm\llcorner_{\ee_{s}(G)}$ are mutually absolutely-continuous. It immediately follows that the same holds for $(\f_{s} )_{\#} \mu_s$ and $\qq^s_s = (\f_{s} )_{\#} \mm\llcorner_{\ee_{s}(G)}$. But:
$$
(\f_{s} )_{\#} (\mu_s)  = (\f_{s} )_{\#} ( (\ee_{s})_{\#} \nu) = (\f_{s} \circ \ee_{s} )_{\#} ( \nu)  = \qq^{\nu}_s ,
$$
establishing (1).

Denoting the resulting probability density $q^{\nu}_s := d \qq^{\nu}_{s} / d\mathcal{L}^{1}$,
\eqref{E:disintegrationu} translates to: 
$$
\nu = \int_{\f_{s}(\ee_{s}(G))} q^{\nu}_{s}(a_{s}) \nu_{a_{s}} \, \L^{1}(da_{s}) .
$$
Pushing forward both sides via the evaluation map $\ee_t$ given $t \in (0,1)$, we obtain:
$$
\rho_{t} \mm = \int_{\f_{s}(\ee_{s}(G))} q^{\nu}_s(a_{s}) \cdot (\ee_{t})_{\#}\nu_{a_{s}}  \, \L^{1}(da_{s}),
$$
with  $q^{\nu}_s(a_{s}) \cdot (\ee_{t})_{\#}\nu_{a_{s}}$ concentrated on $\ee_{t}(G_{a_{s}})$ for $\L^1$-a.e. $a_s \in \f_{s}(\ee_{s}(G))$.
On the other hand, multiplying both sides of (\ref{E:L2-1}) by $\rho_t$ (which is supported on $\ee_t(G)$), we obtain:
\[
\rho_t \mm = \int_{\f_{s}(\ee_{s}(G))} \rho_t \cdot \mm^{t}_{a_{s}} \, \mathcal{L}^{1}(da_{s}), 
\]
with $\rho_t \cdot \mm^{t}_{a_{s}}$ concentrated on $\ee_{t}(G_{a_{s}})$ for $\L^1$-a.e. $a_s \in \f_{s}(\ee_{s}(G))$.
By the essential uniqueness of the disintegration (Theorem \ref{T:disintegrationgeneral}), noting that $\varphi_s(\ee_s(G))$ is compact, (\ref{eq:continuity-second}) immediately follows. 
As $\rho_t > 0$ on $\ee_t(G)$ (see Definition \ref{def:good}) and $q^\nu_s(a_s) \in (0,\infty)$ for $\qq^{\nu}_{s}$-a.e. $a_{s} \in \f_{s}(\ee_{s}(G))$, the ``in particular" part of (2) is also established. 

Finally, by Fubini's theorem, it follows that for each $s \in (0,1)$ and $\qq^{\nu}_{s}$-a.e. $a_{s} \in \f_{s}(\ee_{s}(G))$, (\ref{eq:continuity-second}) holds with $q^\nu_s(a_s) \in (0,\infty)$ for $\L^1$-a.e. $t \in (0,1)$. Note that for $\qq^{\nu}_{s}$-a.e. $a_{s} \in \f_{s}(\ee_{s}(G))$, the curve $t \mapsto (\ee_{t})_{\sharp} \nu_{a_{s}}$ is a $W_2$-geodesic (since $\nu_{a_s}$ is concentrated on $G_{a_s} \subset G$).
This establishes (3), thereby concluding the proof. 
\end{proof}

\bigskip

\section{Comparison between conditional measures}\label{S:comparison1}

So far we have proved, under Assumption \ref{ass:good-assumption}, that for each $s \in (0,1)$ we have the following two families of disintegrations:
\begin{equation}\label{E:representation1}
\mm \llcorner_{\ee_{t}(G)} = \int_{\f_{s}(\ee_{s}(G))}  \mm^{t}_{a_{s}} \, \mathcal{L}^{1} (da_{s}) \;\;\; \text{ and } \;\;\;
\mm \llcorner_{\ee_{[0,1]}(G_{a_{s}})} = \int_{[0,1]}  \mm^{a_{s}}_{t} \, \mathcal{L}^{1} (dt)
\end{equation}
for each $t \in (0,1)$ and each $a_{s} \in \f_{s}(\ee_{s}(G))$, respectively, corresponding to the partitions:
\[
\{\ee_t(G_{a_s}) \}_{a_s \in \f_{s}(\ee_{s}(G))} \;\;\; \text{ and } \;\;\;  \{  \ee_t(G_{a_s}) \}_{t \in (0,1)} .
\]
Moreover, 
both $\mm_{a_{s}}^{t}$ and $\mm_{t}^{a_{s}}$ are concentrated on $\ee_{t}(G_{a_{s}})$, for each $t \in (0,1)$ for $\L^1$-a.e. $a_{s} \in \f_{s}(\ee_{s}(G))$, and for each $a_s \in \f_{s}(\ee_{s}(G))$ and all $t \in (0,1)$, respectively, 
so that the above disintegrations are strongly consistent with respect to the corresponding partition. In addition, we have by (\ref{eq:def-by-pushf}) and (\ref{eq:continuity-second}) for all $s,t \in (0,1)$ and a.e. $a_s \in \varphi_s(G_{a_s})$:
\begin{equation} \label{eq:heuristic-COV1}
\mm_t^{a_s} = (\ee_t \circ \ee_s^{-1})_{\sharp}(h^{a_s}_{\cdot}(t) \mm_s^{a_s}) ~,~ \rho_t \mm_{a_s}^t = (\ee_t \circ \ee_s^{-1})_{\sharp}(\rho_s \mm_{a_s}^s) .
\end{equation}

The goal of the first subsection, in which we retain Assumption \ref{ass:good-assumption}, is to prove that $\mm^{t}_{a_{s}}$ and $\mm_{t}^{a_{s}}$ are in fact equivalent measures. We will prove in particular that for all $s \in (0,1)$:
\begin{equation} \label{eq:heuristic-COV2}
\mm_t^{a_s} = \partial_t \Phi_s^t \; \mm_{a_s}^t \;\;\; \text{for a.e. $t \in (0,1)$, $a_s \in \varphi_s(G_{a_s})$} .
\end{equation}
A heuristic formal argument for establishing (\ref{eq:heuristic-COV2}) may be seen as follows. Writing $\Phi_s^t(x) = \Phi_s(t,x)$, we have:
\[
\ee_t(G_{a_s}) = \ee_t(G) \cap \set{x\in X \; ; \; \Phi_s(t,x) = a_s} = \ee_t(G) \cap \set{ x \in X \; ; \; \Phi_s(\cdot,x)^{-1}(a_s) = t} .
\]
Formally applying the coarea formula (assuming spatial regularity), we have:
\[
\frac{\mm_t^{a_s}}{\mm_{a_s}^t} = \frac{\abs{\nabla_x \Phi_s(t,x)}}{\abs{\nabla_x \Phi_s(\cdot,x)^{-1}(a_s)}} = \abs{- \partial_t \Phi_s(t,x)} ,
\]
where the last transition follows by the implicit function theorem $\nabla_x \Phi_s + \partial_t \Phi_s \cdot \nabla_x\Phi^{-1}_s = 0$.

In the second subsection, we deduce the change-of-variables formula (\ref{E:intro-COV}) for the density along geodesics, discarding Assumption \ref{ass:good-assumption}. An insightful heuristic argument may be seen by combining (\ref{eq:heuristic-COV1}) and (\ref{eq:heuristic-COV2}) as follows:
\[
\frac{\partial_\tau|_{\tau=t}\Phi_s^\tau(\gamma_t)}{\rho_t(\gamma_t)} = \left . \frac{\mm_t^{a_s}}{\rho_t \mm_{a_s}^t}\right |_{\gamma_t} =
\left . \frac{h^{a_s}_{\cdot}(t) \mm_{s}^{a_s}}{\rho_s \mm_{a_s}^s} \right |_{\gamma_s}
= \frac{h^{a_s}_{\gamma_s}(t)}{\rho_s(\gamma_s)} {\partial_\tau|_{\tau=s} \Phi_s^\tau(\gamma_s)} = \frac{h^{a_s}_{\gamma_s}(t)}{\rho_s(\gamma_s)} {\ell(\gamma)^2} .
\]

\subsection{Equivalence of conditional measures} \label{subsec:compare}

Recall that Assumption \ref{ass:good-assumption} is still in force in this subsection.
We start with the following auxiliary:

\begin{lemma}\label{L:blowup1}
For every $s,t\in (0,1)$ and $a_{s} \in \f_{s}(\ee_{s}(G))$, the following limit:
$$
\mm_{t}^{a_{s}} = \lim_{\ve \to 0} \frac{1}{2\ve} \, \mm\llcorner_{\ee_{[t-\ve,t+\ve]}(G_{a_{s}})}
$$
holds true in the weak topology.  
Moreover, for any $f \in C_b(X)$, the map 
$\f_{s}(\ee_{s}(G)) \ni a_{s} \mapsto \int_{X} f \mm_{t}^{a_{s}}$ is Borel.
\end{lemma}
\begin{proof}
By Proposition \ref{P:continuity}, $(0,1) \ni t\mapsto \mm^{a_{s}}_{t}$ is continuous in the weak topology, and so together with (\ref{E:representation1}), we see that for any $f \in C_b(X)$:
 \[
\lim_{\ve \to 0} \frac{1}{2\ve} \int_X f(z) \mm\llcorner_{\ee_{[t-\ve,t+\ve]}(G_{a_{s}})}(dz)  = \lim_{\ve \to 0}\frac{1}{2\ve} \int_{t-\ve}^{t+\ve} \bigg(  \int_{X} f(z)\mm^{a_s}_{\tau}(dz) \bigg)  \mathcal{L}^{1}(d \tau) = \int_{X} f(z) \mm^{a_{s}}_{t}(dz) ,
\]
thereby concluding the proof of the first assertion. 
For the second assertion, given a compact set $I \subset [0,1]$, consider the 
compact set:
$$
K : = \{ (x,t,\gamma,a_{s}) \in X \times I \times G \times \f_{s}(\ee_s(G)) 
\colon
x = \gamma_{t}, \ \f_{s}(\gamma_{s}) = a_{s} \}.
$$
Hence $B := P_{14}(K) = \{ (\ee_I(G_{a_s}) , a_s) : a_s \in \f_{s}(\ee_s(G)) \} $ is compact as well. 
It follows by Fubini's theorem that 
the map $\f_{s}(\ee_{s}(G)) \ni a_{s} \mapsto \int_{\ee_I(G_{a_s})} f \, \mm$
is Borel. Taking $I = [t-\eps,t+\eps]$, employing the first assertion, and recalling that the pointwise limit of Borel functions is Borel, the second 
assertion follows.
\end{proof}

\begin{remark}
One may similarly show (employing an additional density argument) that for every $s,t\in (0,1)$ and $\L^{1}$-a.e. $a_{s} \in \f_{s}(\ee_{s}(G))$, the following limit:
\[
 \mm_{a_{s}}^{t} = \lim_{\ve \to 0} \frac{1}{2\ve} \, \mm\llcorner_{(\Phi_{s}^{t})^{-1}[a_{s}-\ve,a_{s}+\ve] \cap \ee_t(G) }
\]
holds true in the weak topology, but this will not be required. 
\end{remark}

\medskip

We now find explicit expressions for the densities.

\begin{theorem}\label{T:s=s} \label{T:any-t}
For any $s \in (0,1)$,
\begin{equation}\label{E:sim}
\mm^{a_{s}}_{s} = \ell_{s}^{2} \cdot \mm^{s}_{a_{s}} \;\;\; \text{ for $\mathcal{L}^{1}$-a.e. $a_{s} \in \f_{s}(\ee_{s}(G))$ } .
\end{equation}
Moreover, for any $s \in (0,1)$ and $\L^{1}$-a.e. $t \in (0,1)$ including at $t=s$, $\partial_{t}\Phi_{s}^{t}(x)$ exists and is positive for $\L^1$-a.e. $a_s \in \f_{s}(\ee_{s}(G))$ and $\mm_{a_{s}}^{t}$-a.e. $x$, and we have:
\begin{equation}\label{E:sim-t}
\mm^{a_{s}}_{t} = \partial_{t}\Phi_{s}^{t} \cdot \mm^{t}_{a_{s}} \;\;\; \text{ for $\mathcal{L}^{1}$-a.e. $a_{s} \in \f_{s}(\ee_{s}(G))$} .
\end{equation}
\end{theorem}

For the ensuing proof, it will be convenient to introduce the following notation. For all $t_0 \in \Real$ and $x_0 \in X$, denote:
\[
\i^1_{t_0} : X \ni x \mapsto (t_0,x) \in (\Real,X)  ~ , ~ \i^2_{x_0} : \Real  \ni t \mapsto (t,x_0)  \in (\Real,X) .
\]
Recall that $G(x)$ denotes the section $\set{t \in [0,1] \; ; \; \exists \gamma \in G \;,\;\gamma_t = x}$ and $\mathring G(x) = G(x) \cap (0,1)$. 

\begin{proof}[Proof of Theorem \ref{T:s=s}]
\hfill

\medskip
{\bf Step 1.}
Fix $s,t \in (0,1)$. By Lemma \ref{L:blowup1} and the boundedness of $\| \mm_{\tau}^{a_{s}}\|$ uniformly in $a_{s}$ and $\tau \in [0,1]$
(see Proposition \ref{P:continuity}), it is easy to deduce (e.g. by Dominated Convergence Theorem) the following limit of measures on $\f_{s}(\ee_{s}(G)) \times X$ in the weak topology (i.e. in duality with  $C_{b}(\f_{s}(\ee_{s}(G)) \times X)$):
$$
\int_{\f_{s}(\ee_{s}(G))} (\i^1_{a_s})_{\sharp} (\mm^{a_{s}}_{t}) \L^1(da_{s}) = \lim_{\ve \to 0} \frac{1}{2\ve} \int_{t-\ve}^{t+\ve} \int_{\f_{s}(\ee_{s}(G))}(\i^1_{a_s})_{\sharp} (\mm^{a_{s}}_{\tau})  \, \L^1(da_{s}) \, \L^1(d\tau) .
$$
Using Fubini's Theorem and \eqref{E:representation1}, we proceed as follows:
\begin{align}
\nonumber & = \lim_{\ve \to 0} \frac{1}{2 \eps} \int_{\f_{s}(\ee_{s}(G))} (\i^1_{a_s})_{\sharp}(\mm\llcorner_{\ee([t-\eps,t+\eps])(G_{a_s})}) \L^1(da_s) \\
\nonumber & = \lim_{\ve \to 0} \frac{1}{2\ve}(\mathcal{L}^{1}\otimes \mm)
\llcorner \{ (a_s,x) \in \f_{s}(\ee_{s}(G)) \times X \; ; \; \gamma_{\tau} = x,  \, \gamma \in G, \,\f_{s}(\gamma_{s}) = a_s, \tau \in (t-\ve,t+\ve) \}  \\
\nonumber & =  \lim_{\ve \to 0} \frac{1}{2\ve} (\mathcal{L}^{1}\otimes \mm)
\llcorner \{ (a_s,x) \in \f_{s}(\ee_{s}(G)) \times X \; ; \;  a_s = \Phi_{s}^{\tau}(x), \tau \in (t-\ve,t+\ve) \cap \mathring G(x) \}  \\
\label{eq:explain1} & = \lim_{\ve \to 0} \int_{\cup_{|\tau - t| < \ve} \ee_{\tau}(G)} \frac{1}{2\ve} (\i^2_{x})_{\sharp}(\mathcal{L}^{1}
\llcorner \{ \Phi_{s}^{\tau}(x) \; ; \; \tau \in (t-\ve,t+\ve) \cap \mathring G(x) \} )  \, \mm(dx) .
\end{align}
Moreover, we claim that it is enough to integrate on $\ee_{t}(G)$ above:
\begin{equation} \label{eq:explain2}
= \lim_{\ve \to 0} \int_{ \ee_{t}(G)} \frac{1}{2\ve} (\i^2_{x})_{\sharp}(\mathcal{L}^{1}
\llcorner \{ \Phi_{s}^{\tau}(x) \; ; \; \tau \in (t-\ve,t+\ve) \cap \mathring G(x) \} )  \, \mm(dx) .
\end{equation}
To see this, recall that by Proposition \ref{prop:Phi} (3) (relying on Theorem \ref{thm:order12-main} (2)), the map $(t-\ve,t+\ve)\cap \mathring G(x) \ni \tau \to \Phi_{s}^{\tau}(x)$ is Lipschitz with Lipschitz constant bounded uniformly in $\eps \in (0,t/2 \wedge (1-t)/2)$ and $x \in \cup_{|\tau - t|<\ve} \ee_{\tau}(G)$ (recall that for any $\gamma \in G$, $\ell(\gamma)\leq 1/c$); we denote the latter Lipschitz bound by $L$. Hence the family of measures 
$$
\frac{1}{2\ve}\mathcal{L}^{1}
\llcorner \{ \Phi_{s}^{\tau}(x) \; ; \; \tau \in (t-\ve,t+\ve) \cap \mathring G(x)\}, 
$$
is bounded in the total-variation norm by $L$, uniformly in $\ve$ and $x$ as above. But by continuity:
$$
\lim_{\ve \to 0} \mm (\cup_{|\tau- t|< \ve} \ee_{\tau}(G) \setminus \ee_{t}(G)) = 0, 
$$
and so we can modify the domain of integration in (\ref{eq:explain1}) yielding (\ref{eq:explain2}). 

\medskip

{\bf Step 2.} Fixing $x \in \ee_t(G)$, we now focus on the weak limit:
$$
\lim_{\ve \to 0} \frac{1}{2\ve}\mathcal{L}^{1}
\llcorner \{ \Phi_{s}^{\tau}(x) \; ; \; \tau \in (t-\ve,t+\ve) \cap \mathring G(x) \} . 
$$
Recall that $(t-\ve,t+\ve) \cap \mathring G(x) \ni \tau \mapsto \Phi_{s}^{\tau}(x)$ has Lipschitz constant bounded by $L$, and moreover, is increasing by Proposition \ref{prop:Phi} (3). Now extend it to the entire $(0,1)$ while preserving (non-strict) monotonicity and the bound on the Lipschitz constant,  e.g. $\hat \Phi_{s}^{\tau}(x) := \inf_{r \in (t-\ve,t+\ve) \cap \mathring G(x)} \Phi_{s}^{r}(x) + L (\tau - r)_+$. Then for any $f \in C_{b}(\R)$, by the change-of-variables formula for (monotone) Lipschitz functions: 
\begin{align*}
\frac{1}{2\ve}& \int_{\{\Phi_{s}^{\tau}(x) \; ; \; \tau \in (t-\ve,t+\ve) \cap \mathring G(x)\} } f(a) \, \L^1(da)  \crcr
&~ = \frac{1}{2\ve} \int_{ (t-\ve,t+\ve) \cap \mathring G(x)} f(\Phi_{s}^{\tau}(x))  \partial_{\tau} \hat \Phi_{s}^{\tau}(x) \, \L^1(d\tau) \crcr
&~ = \frac{1}{2\ve} \int_{ (t-\ve,t+\ve) \cap \mathring G(x)} f(\Phi_{s}^{\tau}(x))  \partial_{\tau}  \Phi_{s}^{\tau}(x) \, \L^1(d\tau) ;
\end{align*}
the last transition follows since $\tau \mapsto \Phi_{s}^{\tau}(x)$ is differentiable a.e. on $D_{\ell}(x)$ and hence $\partial_{\tau}  \Phi_{s}^{\tau}(x) = \partial_{\tau}  \Phi_{s}^{\tau}(x)|_{(t-\ve,t+\ve) \cap \mathring G(x)}$ for a.e. $\tau \in (t-\ve,t+\ve) \cap \mathring G(x)$ by Remark \ref{R:diff-restriction}, and in addition since $\partial_{\tau}  \Phi_{s}^{\tau}(x)|_{(t-\ve,t+\ve) \cap \mathring G(x)} = \partial_{\tau}  \hat \Phi_{s}^{\tau}(x)$ for a.e. $\tau \in (t-\ve,t+\ve) \cap \mathring G(x)$ by Remark \ref{R:differentiabilitydensity}. 
Recall that Proposition \ref{prop:Phi} ensures that for all $x \in X$, $\partial_t \Phi_{s}^{t}(x)$ exists for $\L^1$-a.e. $t \in \mathring G(x)$, including at $t=s$ if $s \in \mathring G(x)$ (in which case $\partial_t \Phi_{s}^{t}|_{t = s} = \ell_s^2(x)$). 
Moreover, Corollary \ref{cor:accum2} and our assumption that $G \subset G_\varphi^+$ ensure that $\partial_t \Phi_{s}^{t}(x) > 0$ for $\L^1$-a.e. $t \in \mathring G(x)$, including at $t=s$. 
Applying Fubini's theorem, we have: 
$$
0 =\int_{X}\mathcal{L}^{1}(\mathring G(x) \setminus\{ t \in \mathring G(x) \colon \exists \partial_{t}  \Phi_{s}^{t}(x)  > 0 \} )\mm(dx) 
= \int_0^1 \mm( \ee_t(G) \setminus \{ x \in \ee_t(G) \colon  \exists \partial_{t}  \Phi_{s}^{t}(x)  > 0 \} ) \L^1(dt).
$$
It follows that for $\L^1$-a.e. $t \in (0,1)$, $\partial_t \Phi_{s}^{t}(x)$ exists and is positive for $\mm$-a.e. $x \in \ee_t(G)$ (including at $t=s$ for all $x \in \ee_s(G)$).

\medskip
{\bf Step 3.}
We now claim that for $\L^1$-a.e. $t \in (0,1)$ including $t=s$, if $f \in C_b(\Real)$ and $\Psi \in C_b(X)$ then:
\[
\lim_{\eps \rightarrow 0} \int_{\ee_t(G)} \left [ \frac{1}{2 \eps} \int_{ (t-\ve,t+\ve) \cap \mathring G(x)} f(\Phi^\tau_s(x)) \partial_\tau \Phi_s^\tau(x) \L^1(d \tau) - f(\Phi^t_s(x)) \partial_t \Phi_s^t(x) \right ] \Psi(x) \mm(dx) =  0. 
\]
To this end, we will show that for such $t$'s, both:
\[
\textrm{I}_\eps(x) := \frac{1}{2 \eps} \int_{ (t-\ve,t+\ve) \cap \mathring G(x)} \brac{f(\Phi^\tau_s(x)) - f(\Phi^t_s(x)} \partial_\tau \Phi_s^\tau(x) \L^1(d \tau) ,
\]
and:
\[
\textrm{II}_\eps(x) := f(\Phi^t_s(x))  \left [ \frac{1}{2 \eps} \int_{ (t-\ve,t+\ve) \cap \mathring G(x)} \partial_\tau \Phi_s^\tau(x) \L^1(d \tau)   - \partial_t\Phi_s^t(x) \right ] ,
\]
tend to $0$ in $L^1(\ee_t(G),\mm)$ as $\eps \rightarrow 0$. 

\medskip
{\bf Step 4.}
To see the claim about $\textrm{I}_\eps$, since $\abs{\partial_\tau \Phi_s^\tau(x)} \leq L$ (uniformly in $\tau \in (t-\eps,t+\eps) \cap \mathring G(x)$ and $x \in \ee_t(G)$),
it is clear that $\lim_{\eps \rightarrow 0} \textrm{I}_\eps(x) = 0$ pointwise by continuity of $f$ and $\mathring G(x) \ni \tau \mapsto \Phi^\tau_s(x)$ (see Proposition \ref{prop:Phi}). To obtain convergence in $L^1(\ee_t(G),\mm)$, it is therefore enough to show by Dominated Convergence that:
\begin{equation} \label{eq:explain3}
\frac{1}{2 \eps} \int_{ (t-\ve,t+\ve) \cap \mathring G(x)} \brac{f(\Phi^\tau_s(x)) - f(\Phi^t_s(x)} \L^1(d \tau) \leq C ,
\end{equation}
uniformly in $x \in \ee_t(G)$. Since $f$ is uniformly continuous on the compact set $\varphi_s(\ee_s(G))$, the uniform estimate (\ref{eq:explain3}) follows since  $\mathring G(x) \ni \tau \mapsto \Phi^\tau_s(x)$ is Lipschitz on $[\delta,1-\delta]$, with Lipschitz constant depending only on $\delta > 0$ and an upper bound on $\{ \len(\gamma) \; ; \; \gamma \in G \}$ (see Proposition \ref{prop:Phi} (3) and Theorem \ref{thm:order12-main} (2)). 

\medskip
{\bf Step 5.}
To see the claim about $\textrm{II}_\eps$, it is clearly enough to show that:
\begin{equation} \label{eq:explain4}
\tilde{\textrm{II}}_\eps(x) := \frac{1}{2 \eps} \int_{ (t-\ve,t+\ve) \cap \mathring G(x)} \partial_\tau \Phi_s^\tau(x) \L^1(d \tau)  - \partial_t\Phi_s^t(x) \rightarrow 0  \text{ in $L^1(\ee_t(G),\mm)$.}
\end{equation}

\medskip
{\bf Step 5a.}
We first establish (\ref{eq:explain4}) for $\L^1$-a.e. $t \in (0,1)$ (independently of $f$ and $\Psi$). Since $\partial_\tau \Phi_s^\tau(x) \leq L$ uniformly in $\tau \in (t-\eps,t+\eps) \cap \mathring G(x)$ and $x \in \ee_t(G)$, by Dominated Convergence, it is enough to establish pointwise convergence in (\ref{eq:explain4}) for $\mm$-a.e. $x \in \ee_t(G)$.

For every $x \in X$, denote:
$$
Leb(x) : = \{ t \in \mathring G(x) \; ; \; t \text{ is a Lebesgue point of } \tau \mapsto \partial_{\tau}\Phi_{s}^{\tau}(x) 1_{\mathring G(x)}(\tau) \} .
$$
By Proposition \ref{prop:Phi} (based on Theorem \ref{thm:order12-main}), we know that for every $x \in X$, the map  $\tau \mapsto \partial_{\tau}\Phi_{s}^{\tau}(x)$ is in $L^{\infty}_{loc}(\mathring G(x))$, and so by Lebesgue's Differentiation Theorem, $\L^{1}(\mathring G(x) \setminus Leb(x)) = 0$. Integrating over $\mm$ and applying Fubini's Theorem, it follows that for $\L^{1}$-a.e. $t \in (0,1)$:
$$
\mm(\ee_{t}(G) \setminus \{ x \in \ee_{t}(G) \; ; \; t \text{ is a Lebesgue point of } \tau \mapsto \partial_{\tau}\Phi_{s}^{\tau}(x) 1_{\mathring G(x)}(\tau)  \}) = 0 ,
$$
thereby establishing (by definition) the pointwise convergence in (\ref{eq:explain4}) for $\mm$-a.e. $x \in \ee_t(G)$.

\medskip
{\bf Step 5b.}
We next establish (\ref{eq:explain4}) at $t=s$. Write:
\[
\tilde{\textrm{II}}_\eps(x) = 
\frac{1}{2 \eps} \int_{ (s-\ve,s+\ve) \cap \mathring G(x)} \brac{\partial_\tau \Phi_s^\tau(x) - \ell_s^2(x)} \L^1(d \tau) + \ell_s^2(x) \left [ 
\frac{1}{2 \eps} \int_{ (s-\ve,s+\ve) \cap \mathring G(x)} \L^1(d \tau) - 1 \right ] .
\]
The first expression tends to $0$ pointwise for all $x \in X$ by Lemma \ref{L:densityPhi}, and hence by Dominated Convergence also in $L^1(\ee_t(G),\mm)$ (since $\abs{\partial_\tau \Phi_s^\tau(x)} \leq L$ and $\ell_s(x) \leq 1/c$ uniformly). The second expression tends to $0$ in $L^1(\ee_t(G),\mm)$ by Proposition \ref{P:densityreversed} and the uniform boundedness of $\ell_s^2(x)$. 

\medskip
\textbf{Step 6.}
In other words, we have verified in \textbf{Steps 3-5} the following weak convergence, for $\L^1$-a.e. $t \in (0,1)$ including at $t=s$:
$$
\lim_{\ve \to 0} \int_{ \ee_{t}(G)} \frac{1}{2\ve} (\i^2_{x})_{\sharp}(\mathcal{L}^{1}
\llcorner \{ \Phi_{s}^{\tau}(x) \; ; \; \tau \in (t-\ve,t+\ve) \cap \mathring G(x)\} ) \, \mm(dx) = \int_{\ee_{t}(G)} (\i^2_{x})_{\sharp}(\delta_{\Phi_s^t(x)}) \partial_{t}\Phi_{s}^{t}(x)   \,\mm(dx) ,
$$
where recall $\Phi_s^s(x) = \varphi_s(x)$ and  $\partial_t \Phi_s^t |_{t=s} = \ell_{s}^{2}(x)$. Combining this with \textbf{Step 1}, we deduce that:
\[
\int_{\f_{s}(\ee_{s}(G))} (\i^1_{a_s})_{\sharp} (\mm^{a_{s}}_{t}) \L^1(da_{s}) = \int_{\ee_{t}(G)} (\i^2_{x})_{\sharp}(\delta_{\Phi_s^t(x)})) \partial_{t}\Phi_{s}^{t}(x)  \,\mm(dx)  .
\]
Integrating this identity against $1 \otimes \psi$ with $1 \in C_b(\Real)$ and $\psi \in C_{b}(X)$, we obtain:
\begin{align*}
& \int_{\f_{s}(\ee_{s}(G))} \int_{\ee_s(G)} \psi(x) \, \mm^{a_{s}}_{t}(dx) \, \L^1(d a_{s}) = \int_{\ee_{t}(G)} \psi(x) \partial_{t}\Phi_{s}^{t}(x) \, \mm(dx)  \\
 & = \int_{\f_{s}(\ee_{s}(G))} \int_{\ee_t(G)} \psi(x) \, \partial_{t}\Phi_{s}^{t}(x) \mm^{t}_{a_{s}}(dx) \, \L^1(d a_{s}) , 
\end{align*}
where we used that $\mm^{a_{s}}_{t}$ is concentrated on $\ee_t(G_{a_s}) \subset \ee_t(G)$ for all $t \in (0,1)$ and $a_s \in \f_{s}(\ee_{s}(G))$ in the first expression, and 
 the disintegration \eqref{E:representation1} of $\mm\llcorner_{\ee_{t}(G)}$ in the last transition. In other words, we obtained for $\L^1$-a.e. $t \in (0,1)$ including at $t=s$:
 \[
 \int_{\f_{s}(\ee_{s}(G))} \mm^{a_{s}}_{t} \L^1(d a_{s}) = \int_{\f_{s}(\ee_{s}(G))} \partial_{t}\Phi_{s}^{t}  \; \mm^{t}_{a_{s}} \L^1(d a_{s}) .
 \]
Since $\mm_{a_s}^t$ is also concentrated on $\ee_{t}(G_{a_{s}})$ for all $t \in (0,1)$ and $\L^1$-a.e. $a_s \in \f_{s}(\ee_{s}(G))$, the assertion follows 
by essential uniqueness of consistent disintegrations (Theorem \ref{T:disintegrationgeneral}). Note that by \textbf{Step 2}, $\partial_{t}\Phi_{s}^{t}(x)$ exists and is positive for $\L^1$-a.e. $t \in (0,1)$ including at $t=s$ for $\mm$-a.e. $x \in \ee_t(G)$, and so by (\ref{E:representation1}), the same holds for $\L^1$-a.e. $a_s \in \f_{s}(\ee_{s}(G))$ and $\mm_{a_{s}}^{t}$-a.e. $x$. 
\end{proof}

\bigskip

\subsection{Change-of-Variables Formula}\label{Ss:change}

We now obtain the following main result of Sections \ref{S:Conditional} and \ref{S:comparison1}. At this time, we dispense of Assumption \ref{ass:good-assumption}.

\begin{theorem}[Change-of-Variables]\label{T:changeofV}
Let $(X,\sfd,\mm)$ be an essentially non-branching \mms verifying $\CD^{1}(K,N)$ with $\supp(\mm) = X$, and let $\mu_0,\mu_1 \in \P_2(X,\sfd,\mm)$. Let $\nu$ denote the unique element of $\Opt(\mu_{0},\mu_{1})$, and set $\mu_t := (\ee_t)_{\sharp} \nu \ll \mm$ for all $t \in (0,1)$.

Then there exist versions of the densities $\rho_t := d\mu_t / d\mm$, $t \in [0,1]$, so that for $\nu$-a.e. $\gamma \in \Geo(X)$, (\ref{E:regularityrho}) holds for all $0 \leq s \leq t \leq 1$, and in particular, for $\nu$-a.e. $\gamma$, $t \mapsto \rho_t(\gamma_t)$ is positive and locally Lipschitz on $(0,1)$, and upper semi-continuous at $t=0,1$. 

Moreover, for any $s\in (0,1)$, for $\L^{1}$-a.e. $t \in (0,1)$ and $\nu$-a.e. $\gamma \in G_\varphi^+$, $\partial_{\tau}|_{\tau = t}\Phi_{s}^{\tau}(\gamma_{t})$ exists, is positive, and the following change-of-variables formula holds:
\begin{equation}\label{E:recap}
\frac{\rho_{t}(\gamma_{t})}{\rho_{s} (\gamma_{s})} = \frac{\partial_{\tau}|_{\tau = t}\Phi_{s}^{\tau}(\gamma_{t})}{\ell^{2}(\gamma)}  \cdot \frac{1}{ h^{\f_{s}(\gamma_{s})}_{\gamma_s}(t)} .
\end{equation}
Here $\f$ denotes a Kantorovich potential associated to the $c$-optimal-transport problem between $\mu_0$ and $\mu_1$ with cost $c = \sfd^2/2$, and $\Phi_s^t$ denotes the time-propagated intermediate Kantorovich potential introduced in Section \ref{S:Phi}; $h^{\f_{s}(\gamma_{s})}_{\gamma_{s}}$ is the $\CD(\ell(\gamma)^2 K ,N)$ density on $[0,1]$ from Proposition \ref{P:L1disintegration}, after applying the re-normalization from Remark \ref{rem:h-normalization}, so that $h^{\f_{s}(\gamma_{s})}_{\gamma_s}(s) = 1$. 
In particular, for $\nu\text{-a.e. } \gamma \in G_\varphi^+$, the above change-of-variables formula holds for $\L^{1}$-a.e. $t,s \in (0,1)$. 

Lastly, for all $\gamma \in G_\varphi^0$, we have:
\begin{equation} \label{E:recap2}
\rho_t(\gamma_t) = \rho_s(\gamma_s) \;\;\; \forall t,s \in [0,1] .
\end{equation}
\end{theorem}

Recall that $\nu$ is concentrated on $G_\varphi = G_\varphi^+ \cup G_\varphi^0$, where $G_\varphi^+$ and $G_\varphi^0$ denote the subsets of positive and zero length $\varphi$-Kantorovich geodesics, respectively. Note that $\partial_{t}|_{t=s}\Phi_{s}^{t}(\gamma_{s}) = \ell_s^2(\gamma_s) = \ell^2(\gamma)$ by Proposition \ref{prop:Phi}, so that together with our normalization that $h^{\f_{s}(\gamma_{s})}_{\gamma_s}(s) = 1$, we see that both sides of (\ref{E:recap}) are indeed equal to $1$ for $t=s$.

\begin{proof}[Proof of Theorem \ref{T:changeofV}]
\hfill

\medskip
\textbf{Step 0.}
As usual, by Proposition \ref{P:MCP} and Remark \ref{rem:CD1-MCP}, $(X,\sfd,\mm)$ also verifies $\MCP(K,N)$, and so Theorem \ref{T:optimalmapMCP} and all the results of Section \ref{S:MCP} apply. We will use the versions of the densities given by Corollary \ref{C:regularity3MCP}. On $X^0 = \ee_{[0,1]}(G_\varphi^0)$, we know by Corollary \ref{cor:mu-on-X0} that $\mu_0\llcorner_{X^0} = \mu_1\llcorner_{X^0} = \mu_t\llcorner_{X^0}$ for all $t \in [0,1]$, and so if necessary, we simply redefine $\rho_t|_{X^0} := \rho_0|_{X^0}$  for all $t \in (0,1]$, so that (\ref{E:recap2}) holds. Note that by Lemma \ref{lem:X0}, this will not affect $(0,1) \ni t \mapsto \rho_t(\gamma_t)$ for all $\gamma \in G_{\varphi}^+$, and Corollary \ref{cor:mu-on-X0} (applied to the pair $\mu_1,\mu_0$) ensures that the same is true for $\nu$-a.e. $\gamma \in G_{\varphi}^+$ at $t=1$. 

\medskip
\textbf{Step 1.}
As explained in the beginning of Section \ref{S:Conditional}, by inner regularity of Radon measures, Corollary \ref{C:regularity3MCP} (applied to both pairs $\mu_0,\mu_1$ and $\mu_1,\mu_0$), Proposition \ref{P:densityreversed} and Corollary \ref{C:injectivity}, there exists a \emph{good} compact subset $G^{\eps} \subset G^+_{\f}$ with $\nu(G^{\eps}) \geq \nu(G_\f^+)-\ve$ for any $\eps > 0$ (recall Definition \ref{def:good}). Of course, we may assume that $G^{\eps}$ is increasing as $\eps$ decreases to $0$ (say, along a fixed sequence). Fixing $\eps > 0$ and a good $G^{\eps}$, denote $\nu^{\eps} = \frac{1}{\nu(G^\eps)} \nu\llcorner_{G^{\eps}}$ and $\mu_t^{\eps} := (\ee_t)_{\sharp} \nu^{\eps} \ll \mm$, so that all of the results of Section \ref{S:Conditional} and Subsection \ref{subsec:compare} apply to $\nu^{\eps}$. Note that by Corollary \ref{C:injectivity}, we have that $\mu^{\eps}_t = \frac{1}{\nu(G^\eps)} (\mu_t)\llcorner_{\ee_t(G^{\eps})}$ for all $t \in [0,1]$, and therefore:
\[
\mu_t^{\eps} = \rho_t^{\eps} \mm ~,~  \rho_t^{\eps} := \frac{1}{\nu(G^\eps)} \rho_t|_{\ee_t(G^{\eps})} \;\;\; \forall t \in [0,1] .
\]
Also note that as $\nu$ is concentrated on $G^{\eps} \subset G_\varphi$, $\varphi$ is still a Kantorovich potential for the associated transport-problem. 

\medskip
\textbf{Step 2.}
Recall that by Corollary \ref{C:continuity} (3), for each $s \in (0,1)$ and $\qq^{\eps,s}_{s}$-a.e. $a_{s} \in \f_{s}(\ee_{s}(G^\eps))$, 
the map:
$$
[0,1] \ni t \mapsto \rho_{t} \cdot \mm^{\eps,t}_{a_{s}} 
$$
coincides for $\L^{1}$-a.e. $t \in [0,1]$ with the geodesic $t \mapsto (\ee_{t})_{\sharp} \nu^{\eps}_{a_{s}}$ up to a (positive) constant $C^{\eps}_{a_s}$ depending on $a_{s}$, where $\nu^{\eps}_{a_s}$ is the conditional measure from the disintegration in (\ref{E:disintegrationu}). Consequently, for such $s$ and $a_s$, for $\L^{1}$-a.e. $t \in [0,1]$ and any Borel $H \subset G^{\eps}_{a_{s}}$, the quantity:
\begin{equation} \label{eq:change-explain1}
\int_{\ee_{t}(H)}\rho^{\eps}_{t} (x) \mm^{\eps,t}_{a_{s}}(dx) = C^{\eps}_{a_s}  \int_{\ee_t(H)} (\ee_{t})_{\sharp} \nu^{\eps}_{a_{s}}( dx) = C^{\eps}_{a_s} \nu^{\eps}_{a_{s}}(H) 
\end{equation}
is constant (where we used the fact that $\ee_t|_{G^{\eps}} : G^{\eps} \rightarrow X$ is injective).

By Theorem \ref{T:any-t}, for $\L^{1}$-a.e. $t \in (0,1)$ and $\L^1$-a.e. $a_s \in \varphi_s(G^{\eps}_s)$ (and hence for $\q_s^{\eps,s}$-a.e. $a_s \in \varphi_s(G_s^{\eps})$ by Proposition \ref{P:2montone}), $\partial_{t}\Phi_{s}^{t}(x)$ exists and is positive for $\mm^{\eps,t}_{a_s}$-a.e. $x$, and $\mm_{t}^{\eps,a_{s}} = \partial_{t}\Phi_{s}^{t} \cdot \mm^{\eps,t}_{a_{s}}$. It follows that for those $t$ and $a_s$ for which this representation and (\ref{eq:change-explain1}) hold true:
\begin{align}
\label{eq:change-explain2}
 C^{\eps}_{a_s} \nu^{\eps}_{a_{s}}(H)  & = \int_{\ee_{t}(H)} \rho^{\eps}_{t}(x) \mm^{\eps,t}_{a_{s}}(dx)  =  \int_{\ee_{t}(H)} \rho^{\eps}_{t}(x) (\partial_{t}\Phi_{s}^{t}(x))^{-1} \, \mm^{\eps,a_{s}}_{t}(dx)  \\
\nonumber &~ = \int_{\ee_{s}(H)} \rho^{\eps}_{t}(g^{a_{s}}(\beta,t)) (\partial_{\tau}|_{\tau=t}\Phi_{s}^{\tau}(g^{a_{s}}(\beta,t)))^{-1} h^{a_{s}}_{\beta}(t)\, \mm_{s}^{\eps,a_{s}}(d\beta) \\
\nonumber &~  = \int_{\ee_{s}(H)}  \rho^{\eps}_{t}(g^{a_{s}}(\beta,t)) (\partial_{\tau}|_{\tau=t} \Phi_{s}^{\tau}(g^{a_{s}}(\beta,t)))^{-1} h^{a_{s}}_{\beta}(t) \ell^{2}_{s}(\beta) \mm_{a_{s}}^{\eps,s}(d\beta) ,
\end{align}
where the second transition follows from our normalization and Remark \ref{rem:h-normalization}, ensuring that $\mm_t^{\eps,a_s} = (g^{a_{s}}(\cdot,t))_{\sharp}\,  ( h^{a_{s}}_{\cdot}(t) \mm_s^{\eps,a_{s}})$,  and the last transition follows from Theorem \ref{T:s=s}.

Note that $g$ and $h$ above do \emph{not} depend on $\eps > 0$. For $g$, this follows by its very definition as $g^{a_s}(\beta,t) = \ee_t(\ee_s^{-1}(\beta))$ (and the injectivity of $\ee_s|_{G_\eps}$ for all $\eps > 0$). For $h$, this immediately follows by inspecting the proof of Proposition \ref{P:L1disintegration}, where $h^{a_s}_{\gamma_s}(t)$ was uniquely defined (for $t \in (0,1)$) as the continuous version of the density of $\hat \mm^{a_s}_\alpha$ from (\ref{eq:dis-prelim}) after conditioning it on $\ee_{[0,1]}(\gamma)$ and pulling it back to the interval $[0,1]$, where $\alpha \in Q^{1,\eps}$ was bijectively identified with $\gamma \in G_{a_s}^{\eps,1}$ via $\eta^{\eps}$; as $Q^{1,\eps}$ and $G_{a_s}^{\eps,1}$ clearly increase as $\eps$ decreases to $0$, with $\eta^{\eps}|_{Q^{1,\eps'}} = \eta^{\eps'}$ for $0 < \eps < \eps'$, we verify that $h$ indeed does not depend on $\eps>0$.

\medskip
\textbf{Step 3.}
As the left-hand-side of (\ref{eq:change-explain2}) does not depend on $t$,  it follows that for all $s \in (0,1)$ and for $\qq^{\eps,s}_{s}$-a.e. $a_{s} \in \f_{s}(\ee_{s}(G^{\eps}))$ (both of which we fix for the time being), there exists a subset $T \subset (0,1)$ of full $\L^1$ measure, so that for all $H \subset G^{\eps}_{a_s}$:
\[
T \ni t \mapsto \int_{\ee_{s}(H)}  \rho^{\eps}_{t}(g^{a_{s}}(\beta,t)) (\partial_{\tau}|_{\tau=t}\Phi_{s}^{\tau}(g^{a_{s}}(\beta,t)))^{-1} h^{a_{s}}_{\beta}(t) \ell^{2}_{s}(\beta) \mm_{a_{s}}^{\eps,s}(d\beta)
\]
is constant. As any Borel subset of $\ee_s(G_{a_s})$ may be written as $\ee_s(H)$, equality of measures follows, and hence equality of densities for $\mm_{a_{s}}^{\eps,s}$-a.e. $\beta$. We have therefore proved that for $t,t' \in T$:
\begin{equation} \label{eq:change-explain5}
 \rho^{\eps}_{t'}(\gamma_{t'}) (\partial_{\tau}|_{\tau=t'}\Phi_{s}^{\tau}(\gamma_{t'}))^{-1} h^{a_{s}}_{\gamma_s}(t')  =  \rho^{\eps}_{t}(\gamma_t) (\partial_{\tau}|_{\tau=t}\Phi_{s}^{\tau}(\gamma_t))^{-1} h^{a_{s}}_{\gamma_s}(t) ,
\end{equation}
for $\mm^{\eps,s}_{a_{s}}$-a.e. $\beta \in \ee_{s}(G^{\eps}_{a_{s}})$, where $\gamma = \gamma^\beta = \ee_s^{-1}(\beta) = g^{a_s}(\beta,\cdot)\in G^{\eps}_{a_{s}}$, with the exceptional set depending on $t,t'$. Note that given $t' \in T$, $\partial_{\tau}|_{\tau=t'}\Phi_{s}^{\tau}(\gamma^\beta_{t'})$ indeed exists for $\mm^{\eps,s}_{a_{s}}$-a.e. $\beta \in \ee_{s}(G^{\eps}_{a_{s}})$ by Corollary \ref{C:continuity} (2). 

It follows that for all $t \in T$, for $\mm^{\eps,s}_{a_{s}}$-a.e. $\beta \in \ee_{s}(G^{\eps}_{a_{s}})$, (\ref{eq:change-explain5}) holds simultaneously for a countable sequence $t' \in T^t \subset T$ which is dense in $(0,1)$. 
Taking the limit in (\ref{eq:change-explain5}) as $T^t \ni t' \rightarrow s$, using Proposition \ref{prop:Phi} (5) which entails:
\[ \lim_{T^t \ni t' \rightarrow s} \partial_{\tau}|_{\tau=t'}\Phi_{s}^{\tau}(\gamma^\beta_{t'}) = \ell_s(\gamma^\beta_s)^2 = \ell(\gamma^\beta)^2 ,
\]  employing the continuity of $(0,1) \ni t' \mapsto h^{a_{s}}_{\gamma_s}(t')$, our normalization $h^{a_{s}}_{\gamma_s}(s) = 1$, and the continuity of $(0,1) \ni t' \mapsto \rho^{\eps}_{t'}(\gamma_{t'})$ (as $G^{\eps}$ is good), it follows that for all $s \in (0,1)$, for $\qq^{\eps,s}_{s}$-a.e. $a_{s} \in \f_{s}(\ee_{s}(G^{\eps}))$ and $\L^1$-a.e. $t \in (0,1)$:
\begin{equation} \label{eq:change-explain3}
 \rho^{\eps}_{s}(\gamma_s) \ell(\gamma)^{-2}  =  \rho^{\eps}_{t}(\gamma_t) (\partial_{\tau}|_{\tau=t}\Phi_{s}^{\tau}(\gamma_t))^{-1} h^{a_{s}}_{\gamma_s}(t) 
\end{equation}
for $\mm^{\eps,s}_{a_{s}}$-a.e. $\beta \in \ee_{s}(G^{\eps}_{a_{s}})$, with $\gamma = \ee_s^{-1}(\beta) \in G^{\eps}_{a_{s}}$.

\medskip
\textbf{Step 4.}
Recall that by Corollary \ref{C:continuity} (2), $\mm^{\eps,s}_{a_s}$ and $(\ee_s)_{\sharp} \nu^{\eps}_{a_s}$ are mutually absolutely continuous for $\qq^{\eps,s}_{s}$-a.e. $a_{s} \in \f_{s}(\ee_{s}(G^{\eps}))$. It follows that  for all $s \in (0,1)$, for $\qq^{\eps,s}_{s}$-a.e. $a_{s} \in \f_{s}(\ee_{s}(G^{\eps}))$ and $\L^1$-a.e. $t \in (0,1)$, (\ref{eq:change-explain3}) holds for $\nu_{a_s}$-a.e. $\gamma$. By Corollary \ref{C:continuity} (1), note that $\qq^{\eps,s}_{s}$ and $\qq^{\eps,\nu}_s$ are mutually absolutely continuous, and hence the disintegration formula \eqref{E:disintegrationu} implies that for all $s \in (0,1)$ and $\L^{1}$-a.e. $t \in (0,1)$:
$$
\rho^{\eps}_{s} (\gamma_{s})   \ell(\gamma)^{-2}= \rho^{\eps}_{t}(\gamma_{t})   (\partial_{\tau}|_{\tau=t}\Phi_{s}^{\tau}(\gamma_{t}))^{-1} h^{\varphi_s(\gamma_s)}_{\gamma_s}(t), 
$$
for $\nu$-a.e. $\gamma \in G^{\eps}$, and in particular that $\partial_{\tau}|_{\tau=t}\Phi_{s}^{\tau}(\gamma_{t})$ exists and is positive for those $s$, $t$ and $\gamma$. 
Taking the limit as $\eps \rightarrow 0$ along a countable sequence, it follows for all $s \in (0,1)$, $\L^{1}$-a.e. $t \in (0,1)$ and $\nu$-a.e. $\gamma \in G_\varphi^+$, that:
$$
\rho_{s} (\gamma_{s})   \ell(\gamma)^{-2}= \rho_{t}(\gamma_{t})   (\partial_{\tau}|_{\tau=t}\Phi_{s}^{\tau}(\gamma_{t}))^{-1} h^{\varphi_s(\gamma_s)}_{\gamma_s}(t) ,
$$
 thereby concluding the proof of (\ref{E:recap}). As a consequence, an application of Fubini's Theorem verifies that for $\nu$-a.e. $\gamma \in G_\varphi^+$, (\ref{E:recap}) holds for $\L^1$-a.e. $s,t \in (0,1)$.

\end{proof}

\begin{remark} \label{rem:COV-for-Phic}
Observe that all of the results of this section also equally hold for $\Phic_s^t$ in place of $\Phi_s^t$. Indeed, recall that for all $x \in X$, $\Phi_s^t(x) = \Phic_s^t(x)$ for $t \in \mathring G_\varphi(x)$, and that by Corollary \ref{cor:accum2}, $\partial_t \Phi_s^t(x) = \partial_t \Phic_s^t(x)$ for a.e. $t \in \mathring G_\varphi(x)$. As these were the only two properties used in the above derivation (in particular, in \text{Step 2} of the proof of Theorem \ref{T:any-t}), the assertion follows. 
\end{remark}

\bigskip
\bigskip

\part{Putting it all together}\label{part3}

\section{Combining Change-of-Variables Formula with Kantorovich 3rd order information} \label{sec:LY}

Let $(X,\sfd,\sfm)$ denote an essentially non-branching \mms verifying $\CD^1(K,N)$. Let $\mu_0,\mu_1 \in \P_2(X,\sfd,\mm)$, and let $\nu$ be the unique element  of $\Opt(\mu_0,\mu_1)$ (by Proposition \ref{P:MCP}, Remark \ref{rem:CD1-MCP} and Theorem \ref{T:optimalmapMCP}). Recall that $\mu_t := (\ee_t)_{\sharp} \nu \ll \mm$ for all $t \in [0,1]$, and we subsequently denote by $\rho_t$ the versions of the corresponding densities given by Theorem \ref{T:changeofV} (resulting from Corollary \ref{C:regularity3MCP}). Finally, denote by $\varphi$ a Kantorovich potential associated to the corresponding optimal transference plan, so that $\nu(G_\varphi) = 1$. 

\subsection{Change-of-Variables Rigidity}

Recall that by the Change-of-Variables Theorem \ref{T:changeofV}, we know that for $\nu$-a.e. geodesic $\gamma \in G_\varphi^+$ and for a.e. $t,s \in (0,1)$,
$\partial_\tau|_{\tau=t} \Phi^\tau_s(\gamma_t)$ exists, is positive, and it holds that:
\begin{equation} \label{eq:main-formula0}
\frac{\rho_s(\gamma_s)}{\rho_t(\gamma_t)} = \frac{h^{\varphi_s(\gamma_s)}_{\gamma_s}(t)}{ \partial_\tau|_{\tau=t} \Phi^\tau_s(\gamma_t) / \len(\gamma)^2 } .
\end{equation}
In fact, by Remark \ref{rem:COV-for-Phic}, the same also holds with $\Phic$ in place of $\Phi$, so that in particular:
\begin{equation} \label{eq:same-Phi}
\partial_\tau|_{\tau=t} \Phi^\tau_s(\gamma_t) = \partial_\tau|_{\tau=t} \Phic^\tau_s(\gamma_t) \;\;\; \text{for $\nu$-a.e. } \gamma \in G_{\varphi}^+ \;\;\; \text{for a.e. } t,s \in (0,1) .
\end{equation}
Recall that given $t,s \in (0,1)$, for $\tilde{\Phi} = \Phi,\Phic$ and $\tilde{\ell} = \ell,\ellc$, respectively, $\tilde{\Phi}_s^t$ was defined on $D_{\tilde{\ell}}$ as: 
\[
\tilde{\Phi}_s^t = \tilde{\varphi}_t + (t-s) \frac{\tilde{\ell}_t^2}{2} ,
\]
and that by Proposition \ref{prop:Phi} (2), the differentiability points of  $t \mapsto \tilde{\Phi}_{s}^t(x)$ and $t \mapsto \tilde{\ell}^2_t(x)$ coincide for all $t \neq s$, and at those points:
\begin{equation} \label{eq:plugging-into-COV}
\partial_t \tilde{\Phi}_s^t (x) = \tilde{\ell}_t^2(x) + (t-s) \partial_t \frac{\tilde{\ell}_t^2}{2}(x) .
\end{equation}
It follows from (\ref{eq:same-Phi}) that for $\nu$-a.e. geodesic $\gamma \in G_\varphi^+$ and for a.e. $t \in (0,1)$:
\begin{equation} \label{eq:assumptions12}
\exists \partial_\tau|_{\tau=t}  \frac{\ell_\tau^2}{2}(\gamma_t)  \;\; ,\;\; \exists \partial_\tau|_{\tau=t}  \frac{\ellc_\tau^2}{2}(\gamma_t) \;\;,\;\; \partial_\tau|_{\tau=t}  \frac{\ell_\tau^2}{2}(\gamma_t) = \partial_\tau|_{\tau=t}  \frac{\ellc_\tau^2}{2}(\gamma_t) .
\end{equation}
Alternatively, (\ref{eq:assumptions12}) follows directly by Lemma \ref{lem:Fubini-diagonal}, in fact for $\nu$-a.e. $\gamma$ (not just $\gamma \in G_\varphi^+$).

\medskip

Plugging (\ref{eq:plugging-into-COV}) and (\ref{eq:assumptions12}) into (\ref{eq:main-formula0}), it follows that we may express the Change-of-Variables Theorem \ref{T:changeofV} as the statement that for $\nu$-a.e. geodesic $\gamma \in G_\varphi^+$, we have:
\begin{equation} \label{eq:main-formula}
\frac{\rho_s(\gamma_s)}{\rho_t(\gamma_t)} = \frac{h^{\varphi_s(\gamma_s)}_{\gamma_s}(t)}{1 + (t-s) \frac{\partial_\tau|_{\tau=t}\ell_\tau^2/2(\gamma_t)}{\len(\gamma)^2}} = \frac{h^{\varphi_s(\gamma_s)}_{\gamma_s}(t)}{1 + (t-s) \frac{\partial_\tau|_{\tau=t}\ellc_\tau^2/2(\gamma_t)}{\len(\gamma)^2}}
\;\;\; \text{for a.e. } t,s \in (0,1) .
\end{equation}
Note that the denominators on the right-hand-side of (\ref{eq:main-formula}) are always positive (when defined) for all $t,s \in (0,1)$ by Theorem \ref{thm:order12-main} (3). 
Fixing the geodesic $\gamma$, we denote for brevity $\rho(t) := \rho_t(\gamma_t)$, $h_s(t) := h^{\varphi_s(\gamma_s)}_{\gamma_s}(t)$ and $K_0 := K \cdot \len(\gamma)^2$. We then have the following additional information for $\nu$-a.e. $\gamma \in G_\varphi^+$, by Corollary \ref{C:regularity3MCP} and Proposition \ref{P:L1disintegration}, respectively:

\begin{enumerate}[label=(\Alph*)]
\item $(0,1) \ni t \mapsto \rho(t)$ is locally Lipschitz and strictly positive.
\item For all $s \in (0,1)$, $h_s$ is a $\CD(K_0,N)$ density on $[0,1]$, satisfying $h_s(s) = 1$. In particular, it is locally Lipschitz continuous on $(0,1)$ and strictly positive there.
\end{enumerate}

\begin{remark}
It is in fact possible to deduce (A) just from the Change-of-Variables formula (\ref{eq:main-formula}) and without referring to Corollary \ref{C:regularity3MCP}. 
This may be achieved by a careful bootstrap argument, exploiting the separation of variables on the left-hand-side of (\ref{eq:main-formula}) and the a-priori estimates of Lemma \ref{lem:apriori} in the Appendix on the logarithmic derivative of $\CD(K_0,N)$ densities. But since we already know (A), and since (A) was actually (mildly) used in the proof of the Change-of-Variables Theorem \ref{T:changeofV}, we only mention this possibility in passing. Note that Corollary \ref{C:regularity3MCP} applies to all $\MCP(K,N)$ essentially non-branching spaces, whereas the Change-of-Variables formula requires knowing the stronger $\CD^1(K,N)$ condition.
\end{remark}
Fix a geodesic $\gamma \in G_\varphi^+$ satisfying (\ref{eq:main-formula}), (A) and (B) above. 
Let $I \subset (0,1)$ be the set of full measure where \eqref{eq:main-formula} holds for all $s \in I$. 
It follows from (\ref{eq:main-formula}) that for all $s \in I$, $t \mapsto \frac{\partial_\tau|_{\tau=t} \tilde{\ell}_\tau^2/2(\gamma_t)}{\len(\gamma)^2}$ coincide a.e. on $(0,1)$ for both $\tilde{\ell} = \ell,\ellc$ with the same locally Lipschitz function $t \mapsto z_s(t)$ defined on $(0,1) \setminus \set{s}$:
\[
z_s(t) := \frac{\frac{1}{\rho_s(\gamma_s)} h^{\varphi_s(\gamma_s)}_{\gamma_s}(t) \rho_t(\gamma_t)  - 1}{t-s} .
\]
 By continuity, it follows that the functions $\set{z_s}_{s \in I}$ must all coincide on their entire domain of definition with a single function $t \mapsto z(t)$ defined on $(0,1)$; the latter function must therefore be locally Lipschitz continuous, and satisfy:
\begin{equation} \label{eq:z-partial-lt}
z(t) = \frac{\partial_\tau|_{\tau=t} \ell_\tau^2/2(\gamma_t)}{\len(\gamma)^2} = \frac{\partial_\tau|_{\tau=t} \ellc_\tau^2/2(\gamma_t)}{\len(\gamma)^2} \;\;\; \text{for a.e. } t \in (0,1) . 
\end{equation}
By Theorem \ref{thm:z-c}, which provides us with 3rd order information on intermediate-time Kantorovich potentials, we obtain the following additional information on $z$:
\begin{enumerate}[label=(\Alph*)]
\setcounter{enumi}{2}
\item $(0,1) \ni t \mapsto z(t)$ is locally Lipschitz. \\ For any $\delta \in (0,1/2)$, there exists $C_\delta > 0$ so that:
\[
\frac{z(t) - z(s)}{t-s} \geq (1 - C_\delta (t-s)) \abs{z(s)} \abs{z(t)} \;\;\; \forall 0 < \delta \leq s < t \leq 1 - \delta < 1 . 
\]
In particular, $z'(t) \geq z^2(t)$ for a.e. $t \in (0,1)$. 
\end{enumerate}

\begin{remark}
By Theorem \ref{thm:z-c}, we obtain the following interpretation for $z(t)$ -- it coincides for \textbf{all} $t \in (0,1)$ with the second Peano derivative of $\tau \mapsto \varphi_{\tau}(\gamma_t)$ and of $\tau \mapsto \varphic_{\tau}(\gamma_t)$ at $\tau = t$.  In particular, these second Peano derivatives are guaranteed to exist for all $t \in (0,1)$ and are a continuous function thereof. 
\end{remark}

We have already seen above how (\ref{eq:main-formula}) enabled us to deduce (\ref{eq:z-partial-lt}), thereby gaining (by Theorem \ref{thm:z-c}) an additional order of regularity for $\partial_\tau|_{\tau=t} \ell_\tau^2/2(\gamma(t))$. The purpose of this section is to show that the combination of the Change-of-Variables Formula:
\begin{equation} \label{eq:rigid-assump}
\frac{\rho(s)}{\rho(t)} = \frac{h_{s}(t)}{1 + (t-s) z(t)} \;\;\; \text{for a.e. } t,s\in (0,1) ,
\end{equation}
together with properties (A), (B) and (C) above, forms a very rigid condition, and already implies the following representation for $\frac{1}{\rho_t(\gamma_t)}$; we formulate this independently of the preceding discussion as follows:

\begin{theorem}[Change-of-Variables Rigidity] \label{thm:rigid}
Assume that (\ref{eq:rigid-assump}) holds,
where $\rho$, $\set{h_s}$ and $z$ satisfy (A), (B) and (C) above. Then:
\[
\frac{1}{\rho(t)} = L(t) Y(t) \;\;\; \forall t \in (0,1) ,
\]
where $L$ is concave and $Y$ is a $\CD(K_0,N)$ density on $(0,1)$.
\end{theorem}

\bigskip

\subsection{Formal Argument} \label{subsec:formal-rigid}

To better motivate the ensuing proof of Theorem \ref{thm:rigid}, we begin with a formal argument.

Assume that the functions $\rho$ and $z$ are $C^2$ smooth and that equality holds in (\ref{eq:rigid-assump}) for all $t,s \in (0,1)$. It follows that the mapping $(s,t) \mapsto h_s(t)$ is also $C^2$ smooth. Fix any $r_0 \in (0,1)$, and define the functions $L$ and $Y$ by:
\[
\log L(r) := - \int_{r_0}^r z(s) ds  ~,~  \log Y(r) := \int_{r_0}^r \partial_t|_{t=s} \log h_s(t) ds .
\]
Note that by (\ref{eq:rigid-assump}):
\begin{align*}
& \log \frac{\rho(r_0)}{\rho(r)} = \int_{r_0}^r \partial_t |_{t=s} \log \frac{\rho(s)}{\rho(t)}  ds \\ 
=  & \int_{r_0}^r \partial_t|_{t=s}  \log h_s(t) ds - \int_{r_0}^r \partial_t|_{t=s} \log(1+(t-s) z(t)) ds = \log Y(r) + \log L(r) . 
\end{align*}
As already noted in Lemma \ref{lem:z-ac}, the concavity of $L$ follows from (C), since:
\[
\frac{L''}{L} = (\log L)'' + ((\log L)')^2 = -z' + z^2 \leq 0 .
\]
The more interesting function is $Y$. We have for all $r \in (0,1)$:
\begin{align*}
(\log Y)'(r) &= \partial_t|_{t=r} \log h_r(t) , \\
(\log Y)''(r) &= \partial_t^2|_{t=r} \log h_r(t) + \partial_s \partial_t |_{t=s=r} \log h_s(t) .
\end{align*}
To handle the last term on right-hand-side above, note that by the separation of variables on the left-hand-side of (\ref{eq:rigid-assump}), we have by (C) again, after taking logarithms and calculating the partial derivatives in $t$ and $s$:
\begin{equation} \label{eq:partials-negative}
\partial_s \partial_t|_{t=s=r} \log h_s(t)  = \partial_s \partial_t |_{t=s=r} \log(1 + (t-s) z(t)) = -z'(r) + z^2(r) \leq 0.
\end{equation}
We therefore conclude that for all $r \in (0,1)$:
\[
(\log Y)''(r) + \frac{((\log Y)'(r))^2}{N-1} \leq \partial_t^2 |_{t=r} \log h_r(t) + \frac{(\partial_t |_{t=r} \log h_r(t))^2}{N-1} \leq -K_0 ,
\]
where the last inequality follows from (B) and the differential characterization of $\CD(K_0,N)$ densities (applied to $h_r(t)$ at $t=r$). Applying the characterization again, we deduce that $Y$ is a ($C^2$-smooth) $\CD(K_0,N)$ density on $(0,1)$. This concludes the formal proof that:
\[
\frac{\rho(r_0)}{\rho(r)} = L(r) Y(r) \;\;\; \forall r \in (0,1) ,
\]
with $L$ and $Y$ satisfying the desired properties. In a sense, the latter argument has been tailored to ``reverse-engineer" the smooth Riemannian argument, where the separation to orthogonal and tangential components of the Jacobian is already encoded in the Jacobi equation, (B) is a \emph{consequence} of the corresponding Riccati equation, and (C) is a \emph{consequence} of Cauchy--Schwarz (cf. \cite[Proof of Theorem 1.7]{sturm:II}).

\bigskip

\subsection{Rigorous Argument}

It is surprisingly very tedious to modify the above formal argument into a rigorous one. It seems that an approximation argument cannot be avoided, since the definition of $Y$ above is inherently differential, and so on one hand we do not know how to check the $\CD(K_0,N)$ condition for $Y$ synthetically, but on the other hand $Y$ is not even differentiable, so it is not clear how to check the $\CD(K_0,N)$ condition by taking derivatives. The main difficulty in applying an approximation argument here stems from the fact that we do not know how to approximate $\{h_s\}$ and $z$ by smooth functions $\{h^\eps_s\}$ and $z^\eps$, so that simultaneously:
\begin{itemize}
\item[-] $\{h^\eps_s\}$ are $\CD(K_0-\eps,N)$ densities ;
\item[-] $z^\eps$ is a function of $t$ only, and not of $s$ ; 
\item[-] and the separation of variables structure of (\ref{eq:rigid-assump}) is preserved. 
\end{itemize}
Our solution is to note that the main role of the separation of variables in the above formal argument was to ensure that (\ref{eq:partials-negative}) holds, and so we will replace the rigid third requirement with the following relaxed one:
\begin{itemize}
\item[-] $\partial_s \partial_t |_{t=s=r} \log h^\eps_s(t) \leq B_\delta \eps$ for all $r \in [\delta,1-\delta]$ and $\delta > 0$.  
\end{itemize}

\begin{proof}[Proof of Theorem \ref{thm:rigid}] 
\hfill \\
\noindent \textbf{Step 1 - Redefining $h_s(t)$.}\\
First, observe that there exists $I_y \subset (0,1)$ of full measure so that for all $s \in I_y$, (\ref{eq:rigid-assump}) is satisfied for a.e. $t \in (0,1)$, and hence for \textbf{all} $t \in (0,1)$, since all the functions $\rho$, $\set{h_s}$ and $z$ are assumed to be continuous on $(0,1)$. Unfortunately, we cannot extend this to \textbf{all} $s \in (0,1)$ as well, since there may be a null set of $s$'s for which the densities $h_s(t)$ do not comply at all with the equation (\ref{eq:rigid-assump}). To remedy this, we simply force (\ref{eq:rigid-assump}) to hold for all $s,t \in (0,1)$ by defining:
\begin{equation} \label{eq:rigid-assump2}
\tilde{h}_s(t) := \frac{\rho(s)}{\rho(t)} (1 + (t-s) z(t)) \;\;\; s,t \in (0,1) ,
\end{equation}
and claim that for all $s \in (0,1)$, $\tilde{h}_s$ is a $\CD(K_0,N)$ density on $(0,1)$. Indeed, for $s \in I_y$, $\tilde{h}_s = h_s$ and there is nothing to check. If $s_0 \in (0,1) \setminus I_y$, simply note that $\tilde{h}_s(t)$ is locally Lipschitz in $s \in (0,1)$ (since $\rho(s)$ is), and hence:
\[
\tilde{h}_{s_0}(t) = \lim_{s \rightarrow s_0} \tilde{h}_{s}(t) = \lim_{I_y \ni s \rightarrow s_0} \tilde{h}_s(t) = \lim_{I_y \ni s \rightarrow s_0} h_s(t) \;\;\; \forall t \in (0,1) . 
\]
But the family of $\CD(K_0,N)$ densities on $(0,1)$ is clearly closed under pointwise limits (it is characterized by a family of inequalities between 3 points), and so $\tilde{h}_{s_0}$ is a $\CD(K_0,N)$ density, as asserted. 

\medskip
\noindent \textbf{Step 2 - Properties of $z$ and $\{\tilde{h}_s\}$.}\\
We next collect several additional observations regarding the functions $z$ and $\{\tilde{h}_s\}$. 
Recall that $\rho$ (by assumption) and $\tilde{h}_s$ (as $\CD(K_0,N)$ densities) are strictly positive in $(0,1)$. 
Together with (\ref{eq:rigid-assump2}) (or directly from (\ref{eq:rigid-assump})), this implies that $1 + (t-s) z(t) > 0$ for all $t,s \in (0,1)$, and hence:
\begin{enumerate}[label=(\Alph*)]
\setcounter{enumi}{3}
\item $-\frac{1}{t} \leq z(t) \leq \frac{1}{1-t} \;\;\; \forall t \in (0,1)$. 
\end{enumerate}
In fact, we already knew this by Theorem \ref{thm:order12-main} (3) but refrained from including this into our assumption (C) since this is a consequence of the other assumptions. Furthermore:
\begin{enumerate}[label=(\Alph*)]
\setcounter{enumi}{4}
\item $I_x := \{ t \in (0,1) \; ; \; \tau \mapsto \tilde{h}_s(\tau) \text{ is differentiable at } \tau=t \text{ for all } s \in (0,1)\}$ is of full measure.
\end{enumerate}
Indeed, this follows directly from the definition (\ref{eq:rigid-assump2}) by considering the set all points $t$ where $\rho(t)$ and $z(t)$ are differentiable. In addition, we clearly have:
\begin{enumerate}[label=(\Alph*)]
\setcounter{enumi}{5}
\item $\forall t \in I_x$, $(0,1) \ni s \mapsto \partial_t \tilde{h}_s(t)$ is continuous.
\end{enumerate}

\medskip
\noindent \textbf{Step 3 - Defining $L$ and $Y$.}\\
Now fix $r_0  \in (0,1)$, and define the functions $L,Y$ on $(0,1)$ as follows:
\[
\log L(r) := - \int_{r_0}^r z(s) ds  ~,~  \log Y(r) := \int_{r_0}^r \partial_t |_{t=s} \log \tilde{h}_s(t) ds .
\]
Clearly, the function $L$ is well defined for all $r \in (0,1)$ as $z$ is assumed locally Lipschitz. As for the function $Y$, (E) implies that $\partial_t |_{t=s} \log \tilde{h}_s(t)$ exists for a.e. $s \in (0,1)$, and the fact that the latter integrand is locally integrable on $(0,1)$ is a consequence of Lemma \ref{lem:apriori} in the Appendix, which guarantees a-priori locally-integrable estimates on the logarithmic derivative of $\CD(K_0,N)$ densities. 

Consequently, as in our formal argument, we may write (since $\log \rho$ is locally absolutely continuous on $(0,1)$):
\begin{align*}
& \log \frac{\rho(r_0)}{\rho(r)} = \int_{r_0}^r \partial_t |_{t=s} \log \frac{\rho(s)}{\rho(t)}  ds \\ 
=  & \int_{r_0}^r \partial_t|_{t=s} \log \tilde{h}_s(t) ds - \int_{r_0}^r \partial_t |_{t=s} \log(1+(t-s) z(t)) ds = \log Y(r) + \log L(r) ,
\end{align*}
and hence:
\[
\frac{\rho(r_0)}{\rho(r)} = L(r) Y(r) \;\;\; \forall r \in (0,1). 
\]
We have already verified in Lemma \ref{lem:z-ac} that the property $z'(s) \geq z^2(s)$ a.e. in $s \in (0,1)$ implies that $L$ is concave on $(0,1)$, so it remains to show that $Y$ is a $\CD(K_0,N)$ density on $(0,1)$. 

\medskip
\noindent \textbf{Step 4 - Approximation argument.}\\
We now arrive to our approximation argument. Given $\eps_1 , \eps_2 >0$, $t \in (\eps_1,1-\eps_1)$ and $s \in (\eps_2,1-\eps_2)$, define the double logarithmic mollification of $\tilde{h}_s(t)$ by:
\[
\log \tilde{h}^{\eps_1,\eps_2}_s(t) := \int \int \log \tilde{h}_y(x) \psi_{\eps_1}(t-x) \psi_{\eps_2}(s-y) dx dy ,
\]
where $\psi_\eps(x) = \frac{1}{\eps} \psi(x/\eps)$ and $\psi$ is a $C^2$-smooth non-negative function on $\Real$ supported on $[-1,1]$ and integrating to $1$. Since for all $\eta \in (0,1/2)$, we clearly have by (\ref{eq:rigid-assump2}) (and, say, (D)):
\[
\int_{\eta}^{1-\eta} \int_{\eta}^{1-\eta} \abs{\log \tilde{h}_y(x)} dx dy  < \infty ,
\]
it follows by Proposition \ref{prop:log-convolve-2d} in the Appendix on logarithmic convolutions that $\{\tilde{h}^{\eps_1,\eps_2}_s(t)\}_{s \in (\eps_2,1-\eps_2)}$ is a $C^2$-smooth (in $(t,s)$) family of $\CD(K_0,N)$ densities on $(\eps_1,1-\eps_1)$. 

\medskip
\noindent \textbf{Step 5 - Concluding the proof assuming (H1) and (H2).}\\
We will subsequently show the following two additional properties of the family $\{h^{\eps_1,\eps_2}_s(t)\}$:
\begin{enumerate}[label=(H\arabic*)]
\item $\lim_{\eps_2 \rightarrow 0} \lim_{\eps_1 \rightarrow 0} \partial_t |_{t=s} \log \tilde{h}^{\eps_1,\eps_2}_s(t) = \partial_t |_{t=s} \log \tilde{h}_s(t)$ for a.e. $s \in (0,1)$. 
\item $\forall \delta \in (0,1/2) \; \exists C_\delta > 0 \;\; \forall \eps \in (0,\frac{\delta}{8}] \;\; \forall \eps_1,\eps_2 \in (0,\eps]$:
\[
\partial_s \partial_t |_{t=s=r} \log \tilde{h}^{\eps_1,\eps_2}_s(t) \leq 2 C_\delta \eps \;\;\; \forall r \in [\delta,1-\delta] .
\]
\end{enumerate}

Assuming these additional properties, let us show how to conclude the proof of Theorem \ref{thm:rigid}. Set $\eps = \max(\eps_1,\eps_2)$, and assuming that $\eps < \min(r_0,1-r_0)$, define the function $Y^{\eps_1,\eps_2}$ on $(\eps,1-\eps)$ given by:
\[
\log Y^{\eps_1,\eps_2}(r) := \int_{r_0}^r \partial_t |_{t=s} \log \tilde{h}^{\eps_1,\eps_2}_s(t) ds .
\]

First, we claim to have the following pointwise convergence for all $r \in (0,1)$:
\begin{equation} \label{eq:Y-converge}
\lim_{\eps_2 \rightarrow 0} \lim_{\eps_1 \rightarrow 0} \log Y^{\eps_1,\eps_2}(r) = \lim_{\eps_2 \rightarrow 0} \lim_{\eps_1 \rightarrow 0} \int_{r_0}^r 
\partial_t |_{t=s} \log \tilde{h}^{\eps_1,\eps_2}_s(t) ds = \int_{r_0}^r \partial_t |_{t=s}\log \tilde{h}_s(t) ds = \log Y(r)  .
\end{equation}
Indeed, the pointwise convergence of the integrands is ensured by property (H1), and as soon as $r_0,r \in (\eta , 1-\eta)$ for some $\eta > 0$, we obtain by  the a-priori estimates of Lemma \ref{lem:apriori} in the Appendix (since $\tilde{h}^{\eps_1,\eps_2}_s$ is a $\CD(K_0,N)$ density on $(\eta,1-\eta)$ for all $\eps_1,\eps_2 \in (0,\eta]$ and $s \in (\eta,1-\eta)$):
\[
\forall t,s \in [r_0,r] \;\; \forall \eps_1,\eps_2 \in (0,\eta] \;\;  \abs{\partial_t \log \tilde{h}^{\eps_1,\eps_2}_s(t)} \leq C(r,r_0,\eta,K_0,N) .
\]
Consequently, (\ref{eq:Y-converge}) follows by Lebesgue's Dominated Convergence theorem.

\smallskip
Now $Y^{\eps_1,\eps_2}$ is $C^2$-smooth, and so as in our formal argument, we have for all $r \in (\eps,1-\eps)$:
\begin{align*}
(\log Y^{\eps_1,\eps_2})'(r) &= \partial_t |_{t=r}\log \tilde{h}^{\eps_1,\eps_2}_r(t) , \\
(\log Y^{\eps_1,\eps_2})''(r) &= \partial_t^2 |_{t=r} \log \tilde{h}^{\eps_1,\eps_2}_r(t) + \partial_s \partial_t |_{t=s=r} \log \tilde{h}^{\eps_1,\eps_2}_s(t) .
\end{align*}
As $\tilde{h}^{\eps_1,\eps_2}_r$ is a $\CD(K_0,N)$ density on $(\eps,1-\eps)$, we know by the differential characterization of such densities that:
\[
\partial_t^2 |_{t=r} \log \tilde{h}^{\eps_1,\eps_2}_r(t) + \frac{1}{N-1} (\partial_t |_{t=r} \log \tilde{h}^{\eps_1,\eps_2}_r(t))^2 \leq -K_0 . 
\]
Combining this with property (H2), we conclude that for any $\delta \in (0,1/2)$, whenever $\eps = \max(\eps_1,\eps_2) \in (0,\min(r_0,1-r_0,\frac{\delta}{8}))$:
\[
(\log Y^{\eps_1,\eps_2})''(r) + \frac{1}{N-1} ((\log Y^{\eps_1,\eps_2})'(r))^2 \leq -K_0 + 2 C_{\delta} \eps \;\;\; \forall r \in [\delta , 1 - \delta],
\]
and hence $Y^{\eps_1,\eps_2}$ is a $C^2$-smooth $\CD(K_0 - 2 C_{\delta} \eps,N)$ density on $[\delta,1- \delta]$.

\smallskip
Combining all of the preceding information, since (as before) the family of $\CD(K_0',N)$ densities is closed under pointwise limits, we conclude from (\ref{eq:Y-converge}) that $Y$ is a $\CD(K_0 - 2 C_\delta \eps , N)$ density on $[\delta,1-\delta]$, for any $\delta \in (0,1/2)$ and $\eps \in (0,\min(r_0,1-r_0,\frac{\delta}{8}))$. Taking the limit as $\eps \rightarrow 0$ and then as $\delta \rightarrow 0$, we confirm that $Y$ must be a $\CD(K_0,N)$ density on $(0,1)$, concluding the proof.

\smallskip
It remains to establish properties (H1) and (H2).

\medskip
\noindent \textbf{Step 6 - proof of (H1).}\\
Given $y \in (0,1)$ and $t \in (\eps_1,1-\eps_1)$, denote:
\[
\log \tilde{h}_y^{\eps_1}(t) := \int \log \tilde{h}_y(x) \psi_{\eps_1}(t-x) dx ,
\]
so that for every $s \in (\eps_2,1-\eps_2)$:
\begin{equation} \label{eq:double-layer}
\log \tilde{h}^{\eps_1,\eps_2}_s(t) = \int \log \tilde{h}^{\eps_1}_y(t) \psi_{\eps_2}(s-y) dy .
\end{equation}
By Proposition \ref{prop:log-convolve} in the Appendix, $\tilde{h}_y^{\eps_1}$ is a $\CD(K_0,N)$ density on $(\eps_1,1-\eps_1)$ for all $y \in (0,1)$. Consequently, Lemma \ref{lem:apriori} implies that $t \mapsto \log \tilde{h}_y^{\eps_1}(t)$ is locally Lipschitz on $(\eps_1,1-\eps_1)$, uniformly in $y \in (0,1)$:
\begin{equation} \label{eq:H1-uniform}
\sup_{y \in (0,1)} \abs{\partial_t \log \tilde{h}^{\eps_1}_y(t)} \leq C(t,\eps_1,K_0,N) .
\end{equation}
 In particular, it follows that we may differentiate in $t$ under the integral in (\ref{eq:double-layer}) at any $t_0 \in (\eps_1,1-\eps_1)$:
\begin{equation} \label{eq:H1-under-integral}
\partial_t |_{t = t_0} \log \tilde{h}^{\eps_1,\eps_2}_s(t) = \int \partial_t |_{t=t_0} \log \tilde{h}^{\eps_1}_y(t) \psi_{\eps_2}(s-y) dy .
\end{equation}

Now, by a standard argument (see Lemma \ref{lem:Lip-deriv} at the end of this section), we know that the derivative of an $\eps$-mollification of a Lipschitz function converges to the derivative itself, at all points where the derivative exists, namely:
\[
\forall t_0 \in I_x \;\; \forall y \in (0,1) \;\; \lim_{\eps_1 \rightarrow 0} \partial_t |_{t=t_0} \log \tilde{h}_y^{\eps_1}(t) = \partial_t |_{t=t_0} \log \tilde{h}_y(t)  .
\]
Together with (\ref{eq:H1-uniform}) and (\ref{eq:H1-under-integral}), it follows by Dominated Convergence theorem that:
\[
\forall t_0 \in I_x \;\; \forall s \in (\eps_2 , 1-\eps_2) \;\;\; \lim_{\eps_1 \rightarrow 0} \partial_t |_{t = t_0} \log \tilde{h}^{\eps_1,\eps_2}_s(t) = \int \partial_t |_{t=t_0} \log \tilde{h}_y(t) \psi_{\eps_2}(s-y) dy .
\]
But by property (F), we know that $(0,1) \ni y \mapsto \partial_t |_{t=t_0} \log \tilde{h}_y(t)$ is continuous for all $t_0 \in I_x$, and therefore taking the limit as $\eps_2 \rightarrow 0$:
\[
\forall t_0 \in I_x \;\; \forall s \in (0,1) \;\;\;  \lim_{\eps_2 \rightarrow 0} \lim_{\eps_1 \rightarrow 0}  \partial_t |_{t = t_0} \log \tilde{h}^{\eps_1,\eps_2}_s(t) 
= \partial_t  |_{t=t_0} \log \tilde{h}_s(t) .
\]
By property (E), $I_x$ has full measure, thereby concluding the proof of (an extension of) property (H1). 

\medskip
\noindent \textbf{Step 7 - proof of (H2).}\\
We will require the following:
\begin{lemma} \label{lem:z-incr-ratio}
Let $z$ satisfy (C) and (D). Then for all $\delta \in (0,1/2)$, there exists $C_\delta > 0$, so that for all
$\eps \in (0,\frac{\delta}{4}]$, $r \in [\delta,1-\delta]$, $r-\eps \leq t_1 < t_2 \leq r + \eps$ and $r-\eps \leq s_1 < s_2 \leq r+\eps$, we have:
\begin{align*}
& (1 + (t_1 - s_1) z(t_1)) (1 + (t_2 - s_2) z(t_2)) \\
& \leq (1 + C_\delta \eps (t_2 - t_1) (s_2 - s_1))(1 + (t_2 - s_1) z(t_2)) (1 + (t_1 - s_2) z(t_1)) . 
\end{align*}
\end{lemma}
\begin{proof}
Opening the various brackets, the assertion is equivalent to the statement:
\begin{align*}
& z(t_1) (s_2 - s_1) - z(t_2)(s_2 - s_1) + z(t_1) z(t_2) (t_2 - t_1)(s_2 - s_1)  \\
& \leq C_\delta \eps (t_2 - t_1) (s_2 - s_1) (1 + (t_2 - s_1) z(t_2)) (1 + (t_1 - s_2) z(t_1)) ,
\end{align*}
and after dividing by $(t_2 - t_1) (s_2 - s_1)$, we see that our goal is to establish:
\begin{equation} \label{eq:z-incr-ratio}
z(t_1) z(t_2) - \frac{z(t_2) - z(t_1)}{t_2 - t_1} \leq C_\delta \eps (1 + (t_2 - s_1) z(t_2)) (1 + (t_1 - s_2) z(t_1)) ,
\end{equation}
for an appropriate $C_\delta$. Note that the right-hand-side of (\ref{eq:z-incr-ratio}) is always positive by (D). 
As $\min(t_i,1-t_i) \geq \delta - \eps \geq \frac{3}{4} \delta$, by our assumption (C), (\ref{eq:z-incr-ratio}) would follow from:
\[
\abs{z(t_1)} \abs{z(t_2)} B_{\frac{3}{4} \delta} 2 \eps \leq C_\delta \eps (1 - 2 \eps \abs{z(t_2)})(1 - 2 \eps \abs{z(t_1)}) ,
\]
or equivalently (assuming $\abs{z(t_1)} \abs{z(t_2)} > 0$, otherwise there is nothing to prove):
\begin{equation} \label{eq:z-incr-ratio-final}
2 B_{\frac{3}{4} \delta} \leq C_{\delta} \brac{\frac{1}{\abs{z(t_1)}} - 2 \eps} \brac{\frac{1}{\abs{z(t_2)}} - 2 \eps} .
\end{equation}
But $\frac{1}{\abs{z(t_i)}} \geq \min(t_i,1-t_i) \geq \frac{3}{4} \delta$ by  (D), and as $\eps \in (0,\frac{\delta}{4}]$, we see that (\ref{eq:z-incr-ratio-final}) is ensured by setting:
\[
C_{\delta} := \frac{32}{\delta^2} B_{\frac{3}{4} \delta}  .
\]
\end{proof}

Translating the statement of Lemma \ref{lem:z-incr-ratio} into a statement for $\tilde{h}_s(t)$ using (\ref{eq:rigid-assump2}), we obtain that for all $\delta \in (0,1/2)$, there exists $C_\delta > 0$, so that for all
$\eps \in (0,\frac{\delta}{8}]$, $r \in [\delta,1-\delta]$, $r-\eps \leq t,s\leq r + \eps$ and $\Delta t,\Delta s \in [0,\eps]$, we have:
\[ \log \tilde{h}_{s}(t) + \log \tilde{h}_{s+\Delta s}(t + \Delta t) \leq \log \tilde{h}_{s}(t + \Delta t) + \log \tilde{h}_{s+\Delta s}(t) + 2 C_\delta \eps \; \Delta t \; \Delta s .
\]  Integrating the above in $t$ against $\psi_{\eps_1}(r-t)$ and in $s$ against $\psi_{\eps_2}(r-s)$ with $\eps_1,\eps_2 \in (0,\eps]$, we obtain that under the same assumptions as above:
\[
\log \tilde{h}^{\eps_1,\eps_2}_{r}(r) + \log \tilde{h}^{\eps_1,\eps_2}_{r+\Delta s}(r + \Delta t) \leq \log \tilde{h}^{\eps_1,\eps_2}_{r}(r + \Delta t) + \log \tilde{h}^{\eps_1,\eps_2}_{r+\Delta s}(r) + C_\delta \eps \; \Delta t \; \Delta s.
\]
Exchanging sides, dividing by $\Delta t > 0$ and taking limit as $\Delta t \rightarrow 0$, and then dividing by $\Delta s > 0$ and taking limit as $\Delta s \rightarrow 0$, we obtain precisely:
\[
\partial_s \partial_t |_{t=s=r} \log \tilde{h}^{\eps_1,\eps_2}_{s}(t) \leq 2 C_\delta \eps ,
\]
thereby confirming (H2). 
\end{proof}

For completeness, we provide a proof of the following lemma, used in \textbf{Step 6} above.  

\begin{lemma} \label{lem:Lip-deriv} Let $f$ be a locally Lipschitz function on an open interval $I \subset \Real$. Let $\psi$ denote a $C^1$-smooth compactly supported function on $\Real$ which integrates to $1$. Denote by $\psi_\eps(x) = \frac{1}{\eps} \psi(x/\eps)$, $\eps > 0$, the corresponding family of mollifiers. Then:
\[
\lim_{\eps \rightarrow 0} (f \ast \psi_\eps)'(x) = f'(x)  ,
\]
at all points $x \in I$ where $f$ is differentiable. 
\end{lemma}
\begin{proof}
Without loss of generality, assume that $0 \in I$, that $f$ is differentiable at $0$ and that $f(0)=0$. Assume that $\psi$ is supported in $[-M,M]$, and let $\eps > 0$ be small enough so that $[-M \eps , M \eps] \subset I$. Then:
\[
(f \ast \psi_\eps)'(0) = \left . \frac{d}{dx} \right |_{x=0} \int f(x+y) \psi_\eps(y) dy = \int f'(y) \psi_\eps(y) dy ,
\]
where the differentiation under the integral is justified since $f$ is locally Lipschitz. Integrating by parts (which is justified as $f \psi_\eps$ is absolutely continuous), we obtain:
\[
(f \ast \psi_\eps)'(0) = - \int_{-M\eps}^{M\eps} f(y) \psi_\eps'(y) dy = - \int_{-M}^M \frac{f(\eps z)}{\eps z} z \psi'(z) dz .
\]
But for each $z \in [-M , M] \setminus \set{0}$, $\lim_{\eps \rightarrow 0} \frac{f (\eps z)}{\eps z} = f'(0)$, and since $f$ is Lipschitz on $[-\eps M , \eps M]$, we obtain by Lebesgue's Dominated Convergence Theorem that:
\[
\lim_{\eps \rightarrow 0} (f \ast \psi_\eps)'(0)  = - \int_{-M}^M f'(0) z \psi'(z) dz = f'(0) \int_{-M}^M \psi(z) dz = f'(0) ,
\]
as asserted. 
\end{proof}

\bigskip
\bigskip

\section{Final Results} \label{S:Final}

In this final section, we combine the results obtained in Parts \ref{part1}, \ref{part2} and the previous section, establishing at last the Main Theorem \ref{thm:main} and the globalization theorem for the $\CD(K,N)$ condition. We also treat the case of an infinitesimally Hilbertian space. 

\smallskip
Throughout this section, recall that we assume $K \in \Real$ and $N \in (1,\infty)$.

\subsection{Proof of the Main Theorem \ref{thm:main}}

\begin{theorem}\label{T:CDloctoCD1}
Let $(X,\sfd,\mm)$ be an essentially non-branching \mms, so that $(\supp(\mm),\sfd)$ is a length space. Then:
$$
\CD_{loc}(K,N) \Rightarrow \CD^{1}_{Lip}(K,N) .
$$
\end{theorem}

\begin{proof}
By Remark \ref{rem:supp-nu}, $(X,\sfd,\mm)$ satisfies $\CD_{loc}(K,N)$ if and only if $(\supp(\mm),\sfd,\mm)$ does. By Remark \ref{R:CD1-localizes}, the same is true for $\CD^1_{Lip}(K,N)$. Consequently, we may assume that $\supp(\mm) = X$. By Lemma \ref{L:proper-support} we deduce that $(X,\sfd)$ is proper and geodesic (note that this would be false without the length space assumption above). 
Note that for geodesic essentially non-branching spaces, it is known that $\CD_{loc}(K,N)$ implies $\MCP(K,N)$ -- see \cite{cavasturm:MCP} for a proof assuming non-branching, 
but the same proof works under essentially non-branching, see the comments after \cite[Corollary 5.4]{CM3}.
 Consequently, the results of Section \ref{S:L1-OT} apply.

Recall that given a $1$-Lipschitz function $u : X \rightarrow \Real$, the equivalence relation $R^b_{u}$ on the transport set $\T_{u}^{b}$ induces a partition 
$\{R_u^b(\alpha)\}_{\alpha \in Q}$ of $\T_{u}^{b}$. By Corollary \ref{C:disintMCP}, we know that 
$\mm(\T_u \setminus \T_u^b) = 0$ with associated strongly consistent disintegration:
$$
\mm\llcorner_{\T_{u}}  = \mm\llcorner_{\T_{u}^{b}} = \int_{Q} \mm_{\alpha} \,\qq(d\alpha), \qquad \text{with } \mm_{\alpha} (R_u^b(\alpha)) = 1, \ \text{for } \qq\text{-a.e. }  \alpha \in Q .
$$
It was proved in \cite{CM1} that the $\CD_{loc}(K,N)$ condition ensures that for $\qq$-a.e. $\alpha \in Q$, $(\overline{R^b_u(\alpha)},\sfd,\mm_{\alpha})$ verifies $\CD(K,N)$ with $\supp(\mm_\alpha) = \overline{R^b_u(\alpha)}$.
Denoting by $X_{\alpha}$ the closure $\overline{R_u^b(\alpha)}$,
Theorem \ref{T:endpoints} ensures that $X_{\alpha}$ coincides with the transport ray $R_{u}(\alpha)$ for $\qq$-a.e. $\alpha \in Q$. Consequently, all 4 conditions of the $\CD^1_u(K,N)$ Definition \ref{D:CD1-u} are verified, and the assertion follows.  
\end{proof}

\begin{theorem}\label{T:CD1-CD}
Let $(X,\sfd,\mm)$ be an essentially non-branching \mms . Then:
\[
\CD^1(K,N) \Rightarrow \CD(K,N) .
\]
\end{theorem}

\begin{proof}
By Remark \ref{R:CD1-localizes}, $(X,\sfd,\mm)$ satisfies $\CD^1(K,N)$ if and only if $(\supp(\mm),\sfd,\mm)$ does. By Remark \ref{rem:supp-nu}, the same is true for $\CD(K,N)$. Consequently, we may assume that $\supp(\mm) = X$. 

By Proposition \ref{P:MCP} and Remark \ref{rem:CD1-MCP}, $X$ also verifies $\MCP(K,N)$, and so Theorem \ref{T:optimalmapMCP} applies.
Given $\mu_{0},\mu_{1} \in \mathcal{P}_{2}(X,\sfd,\mm)$, consider the unique $\nu \in \Opt(\mu_{0},\mu_{1})$, and denote $\mu_t := (\ee_t)_{\sharp}(\nu) \ll \mm$ for all $t \in [0,1]$. Let $\rho_t := d\mu_t / d\mm$ denote the versions of the densities guaranteed by Corollary \ref{C:regularity3MCP}. 

Denote an associated Kantorovich potential by $\varphi$, and recall that $\nu$ is concentrated on $G_\varphi = G_\varphi^+ \cup G_\varphi^0$, where $G_\varphi^+$ and $G_\varphi^0$ denote the subsets of positive and zero length $\varphi$-Kantorovich geodesics, respectively.
The change-of-variables Theorem \ref{T:changeofV} and Proposition \ref{prop:Phi} yield that for $\nu$-a.e. geodesic $\gamma \in G_\varphi^+$:
\begin{equation} \label{eq:final1}
\frac{\rho_s(\gamma_s)}{\rho_t(\gamma_t)} = \frac{h^{\varphi_s(\gamma_s)}_{\gamma_s}(t)}{1 + (t-s) \frac{\partial_\tau|_{\tau=t}\ell_\tau^2/2(\gamma_t)}{\len(\gamma)^2}} = \frac{h^{\varphi_s(\gamma_s)}_{\gamma_s}(t)}{1 + (t-s) \frac{\partial_\tau|_{\tau=t}\ellc_\tau^2/2(\gamma_t)}{\len(\gamma)^2}}
\;\;\; \text{for a.e. } t,s \in (0,1) .
\end{equation}
where for all $s \in (0,1)$, $h_{s} = h^{\f_{s}(\gamma_{s})}_{\gamma_s}$ is a $\CD(K_{0},N)$ density, with $K_{0} = \ell^{2}(\gamma) K$ and $h_{s}(s)=1$. Together with Corollary \ref{C:regularity3MCP}, which ensures the Lipschitz regularity (and positivity) of $(0,1) \ni t \mapsto \rho_t(\gamma_t)$, this verifies assumptions (A) and (B) of Theorem \ref{thm:rigid}. 
As explained in Section \ref{sec:LY}, the 3rd order information on the Kantorovich potential $\varphi$ asserted by Theorem \ref{thm:z-c} verifies assumption (C) of Theorem \ref{thm:rigid}. It follows by Theorem \ref{thm:rigid} (and the discussion preceding it) that the rigidity of (\ref{eq:final1}) necessarily implies that for those $\gamma \in G_\varphi^+$ satisfying (\ref{eq:final1}), it holds:
$$
\frac{1}{\rho_{t}(\gamma_{t})} = L(t) Y(t) \;\;\; \forall t \in (0,1) ,
$$
where $L$ is concave and $Y$ is a $\CD(K_0,N)$ density on $(0,1)$. Noting that $\sigma_{K_0,N}^{(\alpha)}(\theta) = \sigma_{K,N}^{(\alpha)}(\theta \ell(\gamma))$, we obtain by a standard application of H\"{o}lder's inequality that for any $t_0,t_1 \in (0,1)$, $\alpha \in [0,1]$ and $t_\alpha = \alpha t_1 + (1-\alpha) t_0$:
\begin{align}
\nonumber & \rho_{t_\alpha}^{-\frac{1}{N}}(\gamma_{t_\alpha})  =  L^{\frac{1}{N}}(t_\alpha) Y^{\frac{1}{N}}(t_\alpha) \\
\nonumber & \geq \Big( \alpha L(t_1) + (1-\alpha) L(t_0) \Big)^{\frac{1}{N}} \cdot \Big( \sigma_{K_0,N-1}^{(\alpha)}(\abs{t_1-t_0}) Y^{\frac{1}{N-1}}(t_1) + \sigma_{K_0,N-1}^{(1-\alpha)}(\abs{t_1-t_0}) Y^{\frac{1}{N-1}}(t_0) \Big)^{\frac{N-1}{N}} \\
\nonumber & \geq  
 \alpha^{\frac{1}{N}} \sigma_{K_0,N-1}^{(\alpha)}(\abs{t_1-t_0})^{\frac{N-1}{N}} L^{\frac{1}{N}} (t_1)Y^{\frac{1}{N}}(t_1) + 
         (1-\alpha)^{\frac{1}{N}} \sigma_{K_0,N-1}^{(1-\alpha)}(\abs{t_1-t_0})^{\frac{N-1}{N}} L^{\frac{1}{N}} (t_0)Y^{\frac{1}{N}}(t_0)   \\
\nonumber & = 
 \alpha^{\frac{1}{N}} \sigma_{K,N-1}^{(\alpha)}(\abs{t_1-t_0} \ell(\gamma))^{\frac{N-1}{N}} \rho_{t_1}^{-\frac{1}{N}}(\gamma_{t_1})  + 
         (1-\alpha)^{\frac{1}{N}} \sigma_{K,N-1}^{(1-\alpha)}(\abs{t_1-t_0} \ell(\gamma))^{\frac{N-1}{N}} \rho_{t_0}^{-\frac{1}{N}}(\gamma_{t_0})    \\
\label{eq:synthetic-inq-anyt} & =  \tau_{K,N}^{(\alpha)}(\sfd(\gamma_{t_0}, \gamma_{t_1})) \rho_{t_1}^{-\frac{1}{N}}(\gamma_{t_1}) + 
\tau_{K,N}^{(1-\alpha)}(\sfd(\gamma_{t_0}, \gamma_{t_1})) \rho_{t_0}^{-\frac{1}{N}}(\gamma_{t_0}) .
\end{align}
Using the upper semi-continuity of $t \mapsto \rho_t(\gamma_t)$ at the end-points $t=0,1$ ensured by Corollary \ref{C:regularity3MCP} (as both $\mu_0,\mu_1 \ll \mm$), we conclude that for $\nu$-a.e. $\gamma \in G^+_{\varphi}$,  the previous inequality in fact holds for all $t_0,t_1 \in [0,1]$. In particular, for $t_0 = 0$, $t_1=1$ and all $\alpha \in [0,1]$:
\begin{equation} \label{eq:grand-conclusion}
\rho_{\alpha}^{-1/N}(\gamma_{\alpha})  \geq 
\tau_{K,N}^{(\alpha)}(\sfd(\gamma_0, \gamma_1)) \rho_{1}^{-\frac{1}{N}}(\gamma_{1}) + 
\tau_{K,N}^{(1-\alpha)}(\sfd(\gamma_0, \gamma_1)) \rho_{0}^{-\frac{1}{N}}(\gamma_{0}) .
\end{equation}
As for null-geodesics $\gamma \in G_\varphi^0$ (having zero length), note that $\tau_{K,N}^{(s)}(0) = s$ and that $[0,1] \ni t \mapsto\rho_t(\gamma_t)$ remains constant by Theorem \ref{T:changeofV}, and therefore (\ref{eq:grand-conclusion}) holds trivially with equality for all $\gamma \in G_\varphi^0$. In conclusion, (\ref{eq:grand-conclusion})  holds for $\nu$-a.e. geodesic $\gamma$, thereby confirming the validity of Definition \ref{def:CDKN-ENB} and verifying $\CD(K,N)$.
\end{proof}

As an immediate consequence of the previous two theorems, we obtain the Local-to-Global Theorem for the Curvature-Dimension condition.

\begin{theorem}\label{T:localtoglobal}
Let $(X,\sfd,\mm)$ be an essentially non-branching \mms so that $(\supp(\mm),\sfd)$ is a length space.  Then:
$$
\CD_{loc}(K,N) \iff \CD(K,N).
$$
\end{theorem}
\begin{remark} \label{rem:false-without-length}
It is clear that the above globalization theorem is false without some global assumption ultimately ensuring that $(\supp(\mm),\sfd)$ is geodesic. 
Indeed, simply consider a $\CD(K,N)$ space, and restrict it to two disjoint geodesically-convex closed subsets of $(\supp(\mm),\sfd)$ (each having positive measure) --
 the resulting space clearly satisfies $\CD_{loc}(K,N)$ but not $\CD(K,N)$; it is also easy to construct similar examples where $(\supp(\mm),\sfd)$ is connected. 
In addition, as already mentioned in the Introduction, the globalization theorem is known to be false without some type of non-branching assumption (see \cite{R2016}). 
\end{remark}

As an interesting byproduct, we also obtain that $\CD^{1}$ and $\CD^{1}_{Lip}$ are equivalent conditions on essentially non-branching spaces:

\begin{corollary}\label{C:CDLip-CD1}
Let $(X,\sfd,\mm)$ be an essentially non-branching \mms.  Then:
$$
\CD(K,N) \iff \CD^{1}(K,N) \iff \CD_{Lip}^{1}(K,N).
$$
\end{corollary}

\begin{proof}
$\CD_{Lip}^{1}(K,N)$ is by definition stronger than $\CD^{1}(K,N)$, which in turn implies $\CD(K,N)$ by Theorem \ref{T:CD1-CD}. 
But $\CD(K,N)$ implies its local version $\CD_{loc}(K,N)$, as well as that $(\supp(\mm),\sfd)$ is geodesic by Lemma \ref{L:proper-support}. The cycle is then closed by Theorem \ref{T:CDloctoCD1}.
\end{proof}

Finally, we deduce a complete equivalence between the reduced and the classic Curvature-Dimension conditions on essentially non-branching spaces. Recall that the reduced version $\CD^*(K,N)$, introduced in \cite{sturm:loc} (in the non-branching setting), is defined exactly in the same manner as $\CD(K,N)$, with the only (crucial) difference being that one employs the slightly smaller $\sigma^{(t)}_{K,N}(\theta)$ coefficients instead of the $\tau^{(t)}_{K,N}(\theta)$ ones in Definition \ref{def:CDKN}. 

\begin{corollary}\label{C:CDstarCD}
Let $(X,\sfd,\mm)$ be an essentially non-branching \mms.  Then:
$$
\CD^{*}(K,N) \iff \CD(K,N).
$$
\end{corollary}

\begin{proof} 
By definition $\CD(K,N)$ is stronger than $\CD^*(K,N)$ (see \cite[Proposition 2.5 (i)]{sturm:loc}). For the converse implication, note that $\CD^*(K,N)$ implies that $(\supp(\mm),\sfd)$ is proper and geodesic, by verbatim repeating the proof of Lemma \ref{L:proper-support}. 
Then we observe that $\CD^*(K,N) \Rightarrow \CD_{loc}(K^{-},N)$, where 
$\CD_{loc}(K^{-},N)$ denotes that $(X,\sfd,\mm)$ verifies $\CD_{loc}(K',N)$ for every $K' < K$ (with the open neighborhoods possibly depending on $K'$).
For non-branching spaces, this was proved in \cite[Proposition 5.5]{sturm:loc} (see also \cite[Lemma 2.1]{dengsturm}), but the proof does not rely on any non-branching assumptions. Then, by Theorem \ref{T:localtoglobal}, we obtain $\CD(K',N)$ for any $K' <  K$. Finally, by uniqueness of dynamical plans (see Theorem \ref{T:optimalmapMCP} and Lemma \ref{lem:CD-MCPE-MCP}) and continuity of $\tau_{K',N}^{(t)}(\theta)$ in $K'$, the claim follows.
\end{proof}

\subsection{$\RCD(K,N)$ spaces}

We also mention the more recent Riemannian Curvature Dimension condition $\RCD^{*}(K,N)$. In the infinite dimensional case $N = \infty$, it was introduced in \cite{ambrgisav:Riemann} for finite measures $\mm$ and in \cite{AGMR} for $\sigma$-finite ones.
The class $\RCD^{*}(K,N)$ with $N<\infty$ has been proposed in \cite{gigli:laplacian} and extensively investigated 
in \cite{ambrgisav:Bakry, EKS-EntropicCD,AMS:RCD}. We refer to these papers and references therein for a general account 
on the synthetic formulation of the latter Riemannian-type Ricci curvature lower bounds.
Here we only briefly recall that it is a strengthening of the reduced Curvature Dimension condition:
a \mms verifies $\RCD^{*}(K,N)$ if and only if it satisfies $\CD^{*}(K,N)$ and is infinitesimally Hilbertian \cite[Definition 4.19 and Proposition 4.22]{gigli:laplacian}, meaning that
the Sobolev space $W^{1,2}(X,\mm)$ is a Hilbert space (with the Hilbert structure induced by the Cheeger energy).
Recall also that the local-to-global property for the $\RCD^{*}(K,N)$ condition (say for length spaces of full support) has already been established for $N=\infty$ in \cite[Theorem 6.22]{ambrgisav:Riemann} for non-branching spaces with finite second moment, for $N < \infty$ in \cite[Theorems 3.17 and 3.25]{EKS-EntropicCD} for strong $\RCD^*(K,N)$ spaces, and for all $N \in [1,\infty]$ in \cite[Theorems 7.2 and 7.8]{AMS:locglob} for proper spaces without any non-branching assumptions. 

\smallskip

We are now in a position to introduce the following (expected) definition: 

\begin{definition*}
We will say that a \mms $(X,\sfd,\mm)$ satisfies $\RCD(K,N)$ if it verifies $\CD(K,N)$ and is infinitesimally Hilbertian.
\end{definition*}

We can now immediately deduce:

\begin{corollary} \label{C:RCD}
$$
\RCD(K,N) \iff \RCD^{*}(K,N).
$$
\end{corollary}
\noindent
Note that $\CD^{*}(K,\infty)$ and $\CD(K,\infty)$ are the same condition, so the above also holds for $N = \infty$. 
\begin{proof}
Since $\CD(K,N)$ is stronger than $\CD^{*}(K,N)$, one implication is straightforward.
For the other implication, recall that $\RCD^{*}(K,N)$ forces the space to be essentially non-branching (see \cite[Corollary 1.2]{rajasturm:branch}), and so the assertion follows by Corollary \ref{C:CDstarCD}.
\end{proof}

\begin{corollary} \label{cor:RCD-loc}
Let $(X,\sfd,\mm)$ be an \mms so that $(\supp(\mm),\sfd)$ is a length space.  Then:
$$
\RCD_{loc}(K,N) \iff \RCD(K,N).
$$
\end{corollary}
\begin{proof}
One implication is trivial. For the converse, as usual, we may assume that $\supp(\mm) = X$ by Remark \ref{rem:supp-nu}. By Lemma \ref{L:proper-support}, we know that $(X,\sfd)$ is proper and geodesic (as usual, this would be false without the length space assumption above).
 As the local-to-global property has been proved for proper geodesic $\RCD^*(K,N)$ spaces without any non-branching assumptions in \cite{AMS:locglob}, it follows that:
\[
\RCD_{loc}(K,N) \Rightarrow \RCD^*_{loc}(K,N) \Rightarrow \RCD^*(K,N) \Rightarrow \RCD(K,N) ,
\]
where the last implication follows by Corollary \ref{C:RCD}. 
\end{proof}

\bigskip

\subsection{Concluding remarks} 

We conclude this work with several brief remarks and suggestions for further investigation.
\begin{itemize}
\item[-] Note that the proof of Theorem \ref{T:CD1-CD} in fact yields more than stated: not only does the synthetic inequality (\ref{eq:synthetic-inq-anyt}) hold (for all $t_0,t_1 \in [0,1]$), but in fact we obtain for $\nu$-a.e. geodesic $\gamma$ the a-priori stronger disentanglement (or ``L-Y" decomposition):
\begin{equation} \label{eq:LY-decomp}
\frac{1}{\rho_t(\gamma_t)} = L_\gamma(t) Y_\gamma(t) \;\;\; \forall t \in (0,1),
\end{equation}
where $L_\gamma$ is concave and $Y_\gamma$ is a $\CD(\ell(\gamma)^2 K, N)$ density on $(0,1)$. As explained in the Introduction, it follows from \cite{dengsturm} that for a fixed $\gamma$, (\ref{eq:LY-decomp}) is indeed strictly stronger than (\ref{eq:synthetic-inq-anyt}). In view of Main Theorem \ref{thm:main}, this constitutes a new characterization of essentially non-branching $\CD(K,N)$ spaces. 
\item[-] According to \cite[p. 1026]{EKS-EntropicCD}, it is possible to localize the argument of \cite{rajasturm:branch} and deduce from a \emph{strong} $\CD_{loc}(K,\infty)$ condition  (when $K$-convexity of the entropy is assumed along \emph{any} $W_2$-geodesic with end-points inside the local neighborhood), that the space is globally essentially non-branching. In combination with our results, it follows that the \emph{strong} $\CD(K,N)$ condition enjoys the local-to-global property, without \emph{a-priori} requiring any additional non-branching assumptions. 
\item[-] It would still be interesting to clarify the relation between the $\CD(K,N)$ condition and the property $\BM(K,N)$ of satisfying a Brunn-Minkowski inequality (with sharp dependence on $K,N$ as in \cite{sturm:II}). Note that by Main Theorem \ref{thm:main}, it is enough to understand this locally on essentially non-branching spaces.
\item[-] It would also be interesting to study the $\CD^1(K,N)$ condition on its own, when no non-branching assumptions are assumed, and to verify the usual list of properties desired by a notion of Curvature-Dimension (see \cite{lottvillani:metric,sturm:II,CM2}). 
\item[-] A natural counterpart of $\RCD(K,N)$ would be $\RCD^{1}(K,N)$: we will say that a \mms verifies $\RCD^{1}(K,N)$ if it verifies $\CD^{1}(K,N)$ and it is infinitesimally Hilbertian. Recall that an $\RCD(K,N)$ space is always essentially non-branching \cite{rajasturm:branch}, and hence Main Theorem \ref{thm:main} immediately yields:
$$
\RCD(K,N) \Rightarrow \RCD^{1}(K,N).
$$
The converse implication would be implied by the following claim which we leave for a future investigation: an $\RCD^{1}(K,N)$-space is always essentially non-branching.
\item[-] In regards to the novel third order temporal information on the intermediate-time Kantorovich potentials $\varphi_t$ we obtain in this work -- it would be interesting to explore whether it has any additional consequences pertaining to the \emph{spatial} regularity of solutions to the Hamilton-Jacobi equation in general, and of the transport map $T_{s,t} = \ee_t \circ \ee_{s}|_G^{-1}$ from an intermediate time $s \in (0,1)$ in particular (where $G \subset G_\varphi$ is the subset of injectivity guaranteed by Corollary \ref{C:injectivity}). In the smooth Riemannian setting, the map $T_{s,t}$ is known to be locally Lipschitz by Mather's regularity theory (see \cite[Chapter 8]{Vil} and cf. \cite[Theorem 8.22]{Vil}). A starting point for this investigation could be the following bound on the (formal) Jacobian of $T_{s,t}$, which follows immediately from (\ref{eq:main-formula}), Theorem \ref{thm:order12-main} (3) and Lemma \ref{lem:apriori}: for $\mu_s$-a.e. $x$, the Jacobian is bounded above by a function of $s,t,K,N,\l_s(x)$ only. 
\end{itemize}

\renewcommand{\thesection}{A}
\section{Appendix - One Dimensional $\CD(K,N)$ Densities} \label{Appendix A} 
\setcounter{lemma}{0}
\setcounter{equation}{0}  \setcounter{subsection}{0}

\begin{definition} \label{def:CDKN-density}
 A non-negative function $h$ defined on an interval $I \subset \Real$ is called a $\CD(K,N)$ density on $I$, for $K \in \Real$ and $N \in (1,\infty)$, if for all $x_0,x_1 \in I$ and $t \in [0,1]$:
 \[
 h(t x_1 + (1-t) x_0)^{\frac{1}{N-1}} \geq  \sigma^{(t)}_{K,N-1}(\abs{x_1-x_0}) h(x_1)^{\frac{1}{N-1}} + \sigma^{(1-t)}_{K,N-1}(\abs{x_1-x_0}) h(x_0)^{\frac{1}{N-1}} ,
 \]
(recalling the coefficients $\sigma$ from Definition \ref{def:sigma}). 
While we avoid in this work the case $N = \infty$, it will be useful in this section to also treat the case $N = \infty$, whence the latter condition is interpreted by subtracting 1 from both sides, multiplying by $N-1$, and taking the limit as $N \rightarrow \infty$, namely:  \[
 \log h (t x_1 + (1-t) x_0) \geq t \log h(x_1) + (1-t) \log h(x_0) + \frac{K}{2} t (1-t) (x_1-x_0)^2 .
 \]
For completeness, we will say that $h$ is a $\CD(K,1)$ density on $I$ iff $K \leq 0$ and $h$ is constant on the interior of $I$. 
\end{definition}

Unless otherwise stated, we assume in this appendix that $K \in \Real$ and $N \in (1,\infty]$. 
The following is a specialization to dimension one of a well-known result in the theory of $\CD(K,N)$ mm-spaces, which explains the terminology above. Here we do not assume that a \mms is necessarily equipped with a \emph{probability} measure. 

\begin{theorem}
If $h$ is a $\CD(K,N)$ density on an interval $I \subset \Real$ then the \mms $(I,\abs{\cdot},h(t) dt)$ verifies $\CD(K,N)$. Conversely, if the \mms $(\Real,\abs{\cdot},\mu)$ verifies $\CD(K,N)$ and $I = \supp(\mu)$ is not a point, then $\mu \ll \L^1$ and there exists a version of the density $h = d\mu / d\L^1$ which is a $\CD(K,N)$ density on $I$.\end{theorem}
\begin{proof}
The first assertion follows from e.g. \cite[Theorem 1.7 (ii)]{sturm:II}, and the second follows by considering the $\CD(K,N)$ condition for uniform measures $\mu_0,\mu_1$ on intervals of length $\eps$ and $\alpha \eps$, respectively, letting $\eps \rightarrow 0$, employing Lebesgue's differentiation theorem, and optimizing on $\alpha > 0$ (e.g. as in the proof of \cite[Theorem 4.3]{cavasturm:MCP}).
\end{proof}

Let $h$ be a $\CD(K,N)$ density on an interval $I \subset \Real$. A few standard and easy consequences of Definition \ref{def:CDKN-density} are:
\begin{itemize}
\item
$h$ is also a $\CD(K_2,N_2)$ density for all $K_2 \leq K$ and $N_2 \in [N,\infty]$ (this follows from the corresponding monotonicity of the coefficients $\sigma^{(t)}_{K,N-1}(\theta)$ in $K$ and $N$, see e.g. \cite{sturm:II, lottvillani:metric}). 
\item 
$h$ is lower semi-continuous on $I$ and locally Lipschitz continuous in its interior (this is easily reduced to a standard identical statement for concave functions on $I$). 
\item 
$h$ is strictly positive in the interior whenever it does not identically vanish (follows immediately from the definition).
\item 
$h$ is locally semi-concave in the interior, i.e. for all $x_0$ in the interior of $I$, there exists $C_{x_0} \in \Real$ so that $h(x) - C_{x_0} x^2$ is concave in a neighborhood of $x_0$ (easily checked for $\CD(K,\infty)$ densities). In particular, it is twice differentiable (in the sense of Lemma \ref{lem:convex-2nd-diff}) a.e. in $I$. 
\end{itemize}

\subsection{Differential Characterization}

The following is a well-known differential characterization of $C^2$-smooth $\CD(K,N)$ densities:
\begin{lemma} \label{lem:CDKN-C2}
Let $h \in C^2_{loc}(I)$ on some open interval $I \subset \Real$. The following are equivalent:
\begin{enumerate}
\item $h$ is a $\CD(K,N)$ density on $I$. 
\item For all $x \in I$:
\begin{equation} \label{eq:diff-CDKN}
(\log h)''(x) + \frac{1}{N-1} ((\log h)'(x))^2 = (N-1) \frac{(h^{\frac{1}{N-1}})''(x)}{h^{\frac{1}{N-1}}(x)} \leq - K . 
\end{equation}
where the left hand side is interpreted as $(\log h)''(x)$ when $N=\infty$. 
\end{enumerate}
\end{lemma}

\begin{remark}
The equality in (\ref{eq:diff-CDKN}) holds for any $N \in (1,\infty)$ by the Leibniz and chain rules at any point $x$ where $h(x)$ is positive and twice differentiable (and in particular, $h^{\frac{1}{N-1}}$ and $\log h$ are also twice differentiable at such a point $x$). The condition (\ref{eq:diff-CDKN}) is the one-dimensional specialization of the Bakry--\'Emery $\CD(K,N)$ condition for smooth weighted Riemannian manifolds \cite{BakryEmery,BakryStFlour}. 
\end{remark}

In fact, we will require a couple of extensions of the above standard claim, which in particular, together imply Lemma \ref{lem:CDKN-C2}; to avoid unnecessary generality, we only treat the case $N \in (1,\infty)$. 

\begin{lemma} \label{lem:point-CDKN}
Let $h$ denote a $\CD(K,N)$ density on an interval $I \subset \Real$, $N \in (1,\infty)$. Then $h$ satisfies (\ref{eq:diff-CDKN}) at any point $x$ in the interior where it is twice differentiable (in particular, (\ref{eq:diff-CDKN}) holds for a.e. $x \in I$). 
\end{lemma}
\begin{proof}
Let $x$ be a point as above. 
Observe that:
\[
\sigma^{(1/2)}_{K,N-1}(\theta) = \frac{1}{2} + \frac{\theta^2}{16} \frac{K}{N-1} + o(\theta^2) \;\;\; \text{as $\theta \rightarrow 0$} ,
\]
and so denoting $g = h^{\frac{1}{N-1}}$, the $\CD(K,N)$ condition with $x_0 = x - \eps$, $x_1 = x+\eps$ and $t = 1/2$ implies:
\[
2 g(x) \geq \brac{1 + \frac{\eps^2}{2} \frac{K}{N-1} + o(\eps^2)} \brac{g(x + \eps) + g(x-\eps)} \;\;\; \text{as $\eps \rightarrow 0$}.
\]
It follows by Taylor's theorem and continuity of $g$ in the interior of $I$ that:
\[
g''(x) = \lim_{\eps \rightarrow 0} \frac{g(x+\eps) + g(x-\eps) - 2 g(x)}{\eps^2} \leq \lim_{\eps \rightarrow 0}  -\frac{K}{N-1} \frac{g(x + \eps) + g(x-\eps)}{2} = -\frac{K}{N-1} g(x) ,
\]
confirming (\ref{eq:diff-CDKN}) and concluding the proof. 
\end{proof}

\begin{lemma} \label{lem:density-CDKN}
Let $h$ be a positive differentiable function on an open interval $I \subset \Real$ whose derivative is locally absolutely continuous there (and hence $h$ is twice differentiable a.e. in $I$). If $h$ satisfies (\ref{eq:diff-CDKN}) for a.e. $x \in I$ and $N \in (1,\infty)$, then $\len(I) \leq D_{K,N-1}$ and $h$ is a $\CD(K,N)$ density on $I$. 
\end{lemma}
\begin{remark}
The differentiability assumption at every point cannot be relaxed, as witnessed by the convex function $h(x) = \abs{x}$, which satisfies $h''(x) = 0$ for a.e. $x$ but nevertheless is not concave. 
\end{remark}
\begin{proof}
Given $x_0, x_1 \in I$ with $\abs{x_1 - x_0} < D_{K,N-1}$, consider the function $\Delta$ on $[0,1]$ given by:
\[
\Delta(t) :=  h(t x_1 + (1-t) x_0)^{\frac{1}{N-1}} - \sigma^{(t)}_{K,N-1}(\abs{x_1-x_0}) h(x_1)^{\frac{1}{N-1}} - \sigma^{(1-t)}_{K,N-1}(\abs{x_1-x_0}) h(x_0)^{\frac{1}{N-1}} .
\]
As $\Delta$ is positive and bounded away from zero on $[0,1]$, and since $y^{\frac{1}{N-1}}$ is Lipschitz on compact sub-intervals of $(0,\infty)$, it follows that $\Delta$ is differentiable with absolutely continuous derivative on $[0,1]$. 
In addition, clearly $\Delta(0) = \Delta(1) = 0$. Abbreviating $\sigma(t) = \sigma^{(t)}_{K,N-1}(\abs{x_1-x_0})$, it is immediate to verify that:
\begin{equation} \label{eq:max-principle-1}
\frac{d^2}{(dt)^2}  \sigma(t) = -\frac{K}{N-1} (x_1 - x_0)^2 \sigma(t) ,
\end{equation}
and therefore our assumption (\ref{eq:diff-CDKN}) for a.e. $x \in I$ implies:
\begin{equation} \label{eq:max-principle-2}
\frac{d^2}{(dt)^2} \Delta(t) \leq -\frac{K}{N-1} (x_1-x_0)^2 \Delta(t) \;\;\; \text{for a.e. } t \in [0,1] . 
\end{equation}

Now set $\Delta_0(t) = \Delta(t)$ and $\Delta_1(t) = \Delta(1-t)$, and for each $i \in \set{0,1}$, denote by $\beta_i$ the absolutely continuous function on $[0,1]$ given by:
\begin{equation} \label{eq:max-principle-beta}
\beta_i(t) := \Delta'_i(t) \sigma(t) - \Delta_i(t) \sigma'(t)  .
\end{equation}
It follows by the Leibniz rule that for any $i\in \set{0,1}$:
\[
\beta_i'(t) = \Delta''_i(t) \sigma(t) - \Delta_i(t) \sigma''(t) \leq 0 \;\;\; \text{for a.e. } t \in [0,1]  ,
\]
and since $\sigma(0) = 0$ we also have $\beta_i(0) = 0$. The absolute continuity implies that $\beta_i$ is monotone non-increasing, and hence $\beta_i(t) \leq 0$ for all $t \in [0,1]$. 

We are ready to conclude that $\Delta \geq 0$ on $[0,1]$, by showing that $\Delta(t_0) \geq 0$ for any local extremum point $t_0 \in (0,1)$ of $\Delta$.  Indeed, when $K \leq 0$, this is immediate, since $\sigma' > 0$ and:
\[
0 \geq \beta_0(t_0) =  \Delta'_0(t_0) \sigma(t_0)-\Delta_0(t_0) \sigma'(t_0) = -\Delta(t_0) \sigma'(t_0) .
\]
When $K > 0$, set $t_1 = 1 - t_0$ which is a local extremal point of $\Delta_1$ in $(0,1)$, and note that $t_{i^*} \in (0,1/2]$ for some $i^* \in \set{0,1}$. Since $\abs{x_1 - x_0} < D_{K,N-1}$, it follows that $\sigma' > 0$ on $[0,1/2]$, and so the same argument as for the case $K \leq 0$ but applied to $\Delta_{i^*}$ yields that $\Delta(t_0) = \Delta_{i^*}(t_{i^*}) \geq 0$, as asserted.

Finally, when $K > 0$, assume in the contrapositive that there exist $x_0,x_1 \in I$ with $x_1 - x_0 = D_{K,N-1}$. Denote $\Delta_0(t) = \Delta(t) := h(t x_1 + (1-t)x_0)$ and set $\sigma(t) := \sin(\pi t)$ for $t \in [0,1]$. Note that as before, (\ref{eq:max-principle-1}) and (\ref{eq:max-principle-2}) are satisfied, and so defining the function $\beta_0$ by (\ref{eq:max-principle-beta}), $\beta_0$ is again monotone non-increasing on $[0,1]$. But:
\[
\beta_0(0) = -\Delta_0(0) \sigma'(0) = -\pi h(x_0) < \pi h(x_1) = -\Delta_0(1) \sigma'(1) = \beta_0(1) ,
\]
yielding a contradiction to the monotonicity, and concluding the proof.  
\end{proof}

\subsection{A-priori estimates}

We will also require the following a-priori estimates on the supremum and logarithmic derivative of $\CD(K,N)$ densities. Here it is crucial that $N \in (1,\infty)$. 

\begin{lemma} \label{lem:apriori0}
Let $h$ denote a $\CD(K,N)$ density on a finite interval $(a,b)$, $N \in (1,\infty)$, which integrates to $1$. Then:
\[
\sup_{x_0 \in (a,b)} h(x_0) \leq \frac{1}{b-a} \begin{cases} N & K \geq 0  \\ (\int_0^1 (\sigma^{(t)}_{K,N-1}(b-a))^{N-1} dt)^{-1} & K < 0 \end{cases} .
\]
In particular, for fixed $K$ and $N$, $h$ is uniformly bounded from above as long as $b-a$ is uniformly bounded away from $0$ (and from above if $K < 0$).
\end{lemma}
\begin{proof}
Given $x_0 \in (a,b)$, we have by the $\CD(K,N)$ condition:
\begin{align*}
1 &= (x_0-a) \int_0^{1} h(t x_0 + (1-t) a) dt + (b-x_0) \int_{0}^1 h((1-t) b + t x_0) dt \\
&\geq h(x_0) \brac{(x_0 - a) \int_0^1 (\sigma^{(t)}_{K,N-1}(x_0-a))^{N-1} dt + (b-x_0) \int_0^1  (\sigma^{(t)}_{K,N-1}(b-x_0))^{N-1} dt } .
\end{align*}
When $K \geq 0$, the monotonicity of $K \mapsto \sigma^{(t)}_{K,N-1}(\theta)$ implies that $\sigma^{(t)}_{K,N-1}(\theta) \geq \sigma^{(t)}_{0,N-1}(\theta) = t$, and we obtain:
\[
1 \geq h(x_0) \frac{b-a}{N} .
\]
When $K < 0$, one may show that the function $\theta \mapsto \sigma^{(t)}_{K,N-1}(\theta)$ is decreasing on $\Real_+$, as this is equivalent to showing that the function $x \mapsto \log \sinh \exp(x)$ is convex on $\Real_+$, and the latter may be verified by direct differentiation (and using that $\sinh(x) \cosh(x) \geq x$). Consequently, we obtain:
\[
1 \geq h(x_0) (b-a) \int_0^1 (\sigma^{(t)}_{K,N-1}(b-a))^{N-1} dt ,
\]
as asserted. We remark that when $K > 0$, one may similarly show that the function $\theta \mapsto \sigma^{(t)}_{K,N-1}(\theta)$ is increasing on $[0,D_{K,N-1})$, and since $\sigma^{(t)}_{K,N-1}(0) = t$, we obtain the previous estimate we employed. 
\end{proof}

\begin{lemma} \label{lem:apriori}
Let $h$ denote a $\CD(K,N)$ density on a finite interval $(a,b)$, $N \in (1,\infty)$. Then:
\[
- \sqrt{K (N-1)} \cot((b-x) \sqrt{K / (N-1)})  \leq (\log h)'(x) \leq \sqrt{K (N-1)} \cot((x-a) \sqrt{K / (N-1)}) , 
\]
for any point $x \in (a,b)$ where $h$ is differentiable. In particular, $\log h(x)$ is locally Lipschitz on $x \in (a,b)$ with estimates depending continuously only on $x,a,b,K,N$.
\end{lemma}
\begin{proof}
Denote $\Psi = h^{\frac{1}{N-1}}$. The inequality on the right-hand-side follows since:
\[
\Psi(t x + (1-t) a) \geq \sigma^{(t)}_{K,N-1}(x-a) \Psi(x)  \;\;\; \forall t \in [0,1]
\]
with equality at $t=1$, and hence we may compare derivatives at $t=1$:
\[
(x-a) \Psi'(x) \leq \partial_t  |_{t=1} \sigma^{(t)}_{K,N-1}(x-a)  \Psi(x)  ,
\]
whenever $\Psi$ is differentiable at $x$. The inequality on the left-hand-side follows similarly. 
\end{proof}

\subsection{Logarithmic Convolutions}

We will require the following:
\begin{proposition} \label{prop:log-convolve}
Let $h$ denote a $\CD(K,N)$ density on an interval $(a,b)$. Let $\psi_\eps$ denote a non-negative $C^2$ function supported on $[-\eps,\eps]$ with $\int \psi_\eps = 1$. 
For any $\eps \in (0,\frac{b-a}{2})$, define the function $h^\eps$ on $(a+\eps,b-\eps)$ by:
\[
\log h^\eps := \log h \ast \psi_\eps .
\]
Then $h^\eps$ is a $C^2$-smooth $\CD(K,N)$ density on $(a+\eps ,b -\eps)$. 
\end{proposition}

For the proof, we will require the following general:
\begin{lemma} \label{lem:Aleksandrov}
Let $g$ denote a semi-concave function on an open interval $I$ (i.e. $g(x) - M\frac{x^2}{2}$ is concave for some $M \geq 0$). Let $\psi$ denote a $C^2$-smooth non-negative test function with compact support in $I$. Then:
\[
\int_I g(x) \psi''(x) dx \leq \int_I g''(x) \psi(x) dx .
\]
In other words, the singular part of $g$'s distributional second derivative is non-positive. 
\end{lemma}
The argument is identical to the one used by D.~Cordero--Erausquin in the proof of \cite[Lemma 1]{CorderoMassTransportAndGaussianInqs}. For completeness, we present the proof. 
\begin{proof}
Extend $g$ and $\psi$ to the entire $\Real$ by defining them as equal to zero outside of $I$. 
Given $\eps > 0$ and $x \in I$, denote:
\[
D^2_\eps g(x) := \frac{g(x+\eps) + g(x-\eps) - 2 g(x)}{\eps^2}  ,
\]
and similarly for $D^2_\eps \psi(x)$. By Taylor's theorem, for any point $x \in I$ where $g$ is twice differentiable we have $\lim_{\eps \rightarrow 0} D^2_\eps g(x) = g''(x)$. In fact, this holds at any point where $g$ has a second Peano derivative, see Subsection \ref{subsec:prelim-derivatives}; in the context of convex functions on $\Real^n$, such points are called points possessing a Hessian in the sense of Aleksandrov. 
Now since for small enough $\eps > 0$, $D^2_\eps g \leq M$ on the support of $\psi$ by semi-concavity (and since $\psi \geq 0$), we obtain by Fatou's lemma:
\[
\int_I g''(x) \psi(x) dx \geq \limsup_{\eps \rightarrow 0} \int_I D^2_\eps g(x) \psi(x) dx = \limsup_{\eps \rightarrow 0} \int_I g(x) D^2_\eps \psi(x) dx = \int_I g(x) \psi''(x) dx ,
\]
where the last equality follows by Lebesgue's Dominated Convergence theorem using the fact that $\abs{D^2_\eps \psi (x)} \leq \max \abs{\psi''}$ for all $x \in I , \eps > 0$, and the fact that $g$ is locally integrable. 
\end{proof}

\begin{proof}[Proof of Proposition \ref{prop:log-convolve}]
Note that $\log h$ is locally integrable on $(a,b)$, so that the integral:
\[
\log h^\eps(x) = \int \log h(y) \psi_{\eps}(x-y) dy ,
\]
is well-defined for all $x \in (a+\eps, b-\eps)$, and we may take two derivatives in $x$ under the integral (as $\psi_\eps$ is $C^2$-smooth with bounded corresponding derivatives), implying the asserted smoothness. 
In addition:
\[
(\log h^\eps)'(x) = \int \log h(y) \frac{d}{dx} \psi_{\eps}(x-y) dy = - \int \log h(y) \frac{d}{dy} \psi_{\eps}(x-y) dx = \int (\log h)'(y) \psi_{\eps}(x-y) dy ,
\]
where the last equality follows from the usual integration by parts formula and Leibniz rule since $(\log h(y)) \psi_{\eps}(x-y)$ is absolutely continuous. 
Furthermore:
\[
(\log h^\eps)''(x) = \int \log h(y) \frac{d^2}{(dx)^2} \psi_{\eps}(x-y) dy = \int \log h(y) \frac{d^2}{(dy)^2} \psi_{\eps}(x-y) dy \leq \int (\log h)''(y) \psi_{\eps}(x-y) dy ,
\]
where the last inequality follows by Lemma \ref{lem:Aleksandrov} applied to $g = \log h$, since $h$ is a $\CD(K,\infty)$ density (by monotonicity in $N$), and hence $\log h(x) + K \frac{x^2}{2}$ is concave on $(a,b)$. 

Putting everything together and applying Jensen's inequality, we obtain:
\begin{align*}
& (\log h^\eps)''(x) + \frac{1}{N-1} ((\log h^{\eps})'(x))^2 \\
\leq  & \int (\log h)''(y) \psi_\eps(x-y) dy + \frac{1}{N-1} \brac{\int (\log h)'(y) \psi_\eps(x-y) dy}^2 \\
\leq & \int \brac{(\log h)''(y) + \frac{1}{N-1}((\log h)'(y))^2} \psi_\eps(x-y) dy \leq 0  ,
\end{align*}
where the last inequality follows since the integrand is non-positive (where it is defined) by Lemma \ref{lem:point-CDKN}. A final application of Lemma \ref{lem:CDKN-C2} concludes the proof. 
\end{proof}

We will use Proposition \ref{prop:log-convolve} in the following form:

\begin{proposition} \label{prop:log-convolve-2d}
Let $\set{h_s(t)}_{s \in (c,d)}$ denote a Borel measurable family of $\CD(K,N)$ densities on $(a,b)$ (so that for every $t \in (a,b)$, $(c,d) \ni s \mapsto h_s(t)$ is Borel measurable). Assume in addition that:
\begin{equation} \label{eq:log-integrable}
\int_{c}^{d} \int_{a}^{b} \abs{\log h_y(x)} dx dy < \infty .
\end{equation}
Given $\eps_1,\eps_2 > 0$ and $s \in (c+\eps_2 , d-\eps_2)$, denote the following function:
\begin{equation} \label{eq:log-double-moll}
\log h^{\eps_1,\eps_2}_{s}(t) := \int \int \log h_y(x) \psi_{\eps_1}(t-x) \psi_{\eps_2}(s-y) dx dy \;,\; t \in (a+\eps_1,b-\eps_1) ,
\end{equation}
where as usual, $\psi_{\eps_i}$ denotes a non-negative $C^2$ function supported on $[-\eps_i,\eps_i]$ with $\int \psi_{\eps_i} = 1$. Then $\set{h^{\eps_1,\eps_2}_{s}(t)}_{s \in (c+\eps_2,d-\eps_2)}$ is a $C^2$-smooth (in $(s,t)$) family of $\CD(K,N)$ densities on $(a+\eps_1,b-\eps_1)$.
\end{proposition}
\begin{proof}
The proof is a repetition of the proof of the previous proposition, so we will be brief. 
Our assumption (\ref{eq:log-integrable}) implies that (\ref{eq:log-double-moll}) is well-defined, and justifies taking two derivatives in $t$ and $s$ under the integral, 
implying the assertion on smoothness. The first derivative in $t$ under the integral may be integrated by parts, whereas for the second derivative we apply Lemma \ref{lem:Aleksandrov}. A final application of Jensen's inequality as in Proposition \ref{prop:log-convolve} establishes the asserted differential characterization of $\CD(K,N)$, concluding the proof. 

\end{proof}

\bigskip
\bigskip

\setlinespacing{1.0}
\setlength{\bibspacing}{2pt}


\begin{thebibliography}{10}





\bibitem{ambro:perimeter} L.~Ambrosio.
\newblock Fine properties of sets of finite perimeter in doubling metric  measure spaces.
\newblock {\em Set Valued Analysis}, {10}:111--128, 2002.

\bibitem{ambro:lecturenote} L.~Ambrosio.
\newblock Lecture notes on optimal transport problems.
\newblock In {\em Mathematical aspects of evolving interfaces ({F}unchal,
  2000)}, volume 1812 of {\em Lecture Notes in Math.}, pages 1--52. Springer,
  Berlin, 2003.

\bibitem{ambro:userguide} L.~Ambrosio and N.~Gigli.
\newblock A user's guide to optimal transport.
\newblock {\em Modelling and Optimisation of Flows on Networks, Piccoli, B., Rascle, M. (eds)}, volume 2062 of {\em Lecture Notes in Math.}, pages 1--155. Springer, Heidelberg, 2013.



\bibitem{AGMR} L.~Ambrosio, N.~Gigli, A.~Mondino, and T.~Rajala.
\newblock Riemannian Ricci curvature lower bounds in metric measure spaces with $\sigma$-finite measure.
\newblock {\em Trans. Am. Math. Soc.},  {367}(7):4661--4701, 2015.



\bibitem{AGS-Book} L.~Ambrosio, N.~Gigli, and G.~Savar\'e.
\newblock Gradient Flows in Metric spaces and in the Space of Probability measures,
\newblock {\em Lectures in Mathematics ETH-Z\"urich}. Birkh\"auser Verlag, Basel, 2005.

\bibitem{ambrgisav:heat} L.~Ambrosio, N.~Gigli, and G.~Savar\'e.
\newblock Calculus and heat flow in metric measure spaces and application to  spaces with {R}icci curvature bounded from below.
\newblock {\em Invent. math.},  {195}:289--391, 2014.

\bibitem{ambrgisav:Riemann} L.~Ambrosio, N.~Gigli, and G.~Savar\'e.
\newblock Metric measure spaces with Riemannian Ricci curvature bounded from below.
\newblock {\em Duke Math. J.},  {163}:1405--1490, 2014.

\bibitem{ambrgisav:Bakry} L.~Ambrosio, N.~Gigli, and G.~Savar\'e.
\newblock Bakry-\'Emery curvature-dimension condition and Riemannian Ricci curvature bounds. 
\newblock {\em Ann. Probab.},  {43}:339--404, 2015.




\bibitem{ambromarino:bvgeneral} L.~Ambrosio and S.~Di Marino.
\newblock Equivalent definitions of BV space and of total variation on metric measure spaces.
\newblock {\em J. Funct. Anal.},  {266}:4150--4188, 2014.



\bibitem{AMS:locglob} L.~Ambrosio, A.~Mondino, and G.~Savar\'e.
\newblock On the {B}akry-\'Emery condition, the gradient estimates and the local-to-global property of $\RCD^{*}(K,N)$ metric measure spaces.
\newblock {\em J. Geom. Anal.},  {26}:24--56, 2016.


\bibitem{AMS:RCD} L.~Ambrosio, A.~Mondino, and G.~Savar\'e.
\newblock Nonlinear diffusion equations and curvature conditions in metric measure spaces, 
\newblock {\em  Preprint}, arXiv:1509.07273.


\bibitem{miranda:bvcoarea}
L.~Ambrosio, M.~Miranda Jr., and D.~Pallara.
\newblock Special functions of bounded variation in doubling metric measure  spaces.
\newblock {\em Calculus of variations: topics from the mathematical heritage of  E. De Giorgi}, pages 1--45, 2004.

\bibitem{AmbrosioPratelliL1}
L.~Ambrosio and A.~Pratelli.
\newblock Existence and stability results in the {$L^1$} theory of optimal
  transportation.
\newblock In {\em Optimal transportation and applications ({M}artina {F}ranca,
  2001)}, volume 1813 of {\em Lecture Notes in Math.}, pages 123--160.
  Springer, Berlin, 2003.


\bibitem{sturm:loc}
K.~Bacher and K.T. Sturm.
\newblock Localization and tensorization properties of the
  {C}urvature-{D}imension condition for metric measure spaces.
\newblock {\em J. Funct. Anal.},  {259}(1):28--56, 2010.

\bibitem{BakryStFlour}
D.~Bakry.
\newblock L'hypercontractivit\'e et son utilisation en th\'eorie des
  semigroupes.
\newblock In {\em Lectures on probability theory ({S}aint-{F}lour, 1992)},
  volume 1581 of {\em Lecture Notes in Math.}, pages 1--114. Springer, Berlin,
  1994.

\bibitem{BakryEmery}
D.~Bakry and M.~{\'E}mery.
\newblock Diffusions hypercontractives.
\newblock In {\em S\'eminaire de probabilit\'es, XIX, 1983/84}, volume 1123 of
  {\em Lecture Notes in Math.}, pages 177--206. Springer, Berlin, 1985.

\bibitem{BGL-Book}
D.~Bakry, I.~Gentil, and M.~Ledoux.
\newblock {\em Analysis and geometry of {M}arkov diffusion operators}, volume
  348 of {\em Grundlehren der Mathematischen Wissenschaften [Fundamental
  Principles of Mathematical Sciences]}.
\newblock Springer, Cham, 2014.



\bibitem{biacar:cmono}
S.~Bianchini and L.~Caravenna.
\newblock On the extremality, uniqueness and optimality of transference plans.
\newblock {\em Bull. Inst. Math. Acad. Sin.(N.S.)},  {4}(4):353--454, 2009.

\bibitem{biacava:streconv}
S.~Bianchini and F.~Cavalletti.
\newblock The {M}onge problem for distance cost in geodesic spaces.
\newblock {\em Comm. Math. Phys},  {318}:615 -- 673, 2013.

\bibitem{BorweinVanderwerff-Book}
J.~M. Borwein and J.~D. Vanderwerff.
\newblock {\em Convex functions: constructions, characterizations and
  counterexamples}, volume 109 of {\em Encyclopedia of Mathematics and its
  Applications}.
\newblock Cambridge University Press, Cambridge, 2010.

\bibitem{BrenierMap}
Y.~Brenier.
\newblock Polar factorization and monotone rearrangement of vector-valued
  functions.
\newblock {\em Comm. Pure Appl. Math.}, 44(4):375--417, 1991.


\bibitem{BBI} D.~Burago, Y.~Burago and S.~Ivanov.
\newblock{A course in Metric geometry},
\newblock{\em Graduate Studies in Mathematics},  AMS, 2001. 


\bibitem{cava:MongeRCD}  F.~Cavalletti.
\newblock  Monge problem in metric measure spaces with Riemannian curvature-dimension condition,
\newblock{\em Nonlinear Anal.}  {99}:136--151, 2014.


\bibitem{cava:decomposition} F.~Cavalletti. \newblock Decomposition of geodesics in the {W}asserstein space and the  globalization property.
\newblock {\em Geom. Funct. Anal.},  {24}:493 -- 551, 2014.


\bibitem{cava:overview} F.~Cavalletti. \newblock An Overview of $L^{1}$ optimal transportation on metric measure spaces.
\newblock {\em Book Chapter}, to appear in ``Measure Theory in Non-Smooth Spaces",
\newblock edited by N. Gigli, De Gruyter Open. 



\bibitem{cavhue:regular}
F.~Cavalletti and M.~Huesmann.
\newblock Self-intersection of optimal geodesics.
\newblock {\em Bull. London Math. Soc.},  {46}:653--656, 2014.  

\bibitem{CM1}
F.~Cavalletti and A.~Mondino.
\newblock Sharp and rigid isoperimetric inequalities in metric-measure spaces with lower Ricci curvature bounds.
\newblock {\em Invent. Math.},  to appear, arXiv:1502.06465.

\bibitem{CM2} 
F.~Cavalletti and A.~Mondino. \newblock Sharp geometric and functional inequalities in metric measure spaces with lower Ricci curvature bounds.
\newblock {\em Geom. Topol.}, to appear, arXiv:1502.06465.

\bibitem{CM3} 
F.~Cavalletti and A.~Mondino. \newblock Optimal maps in essentially non-branching spaces.
\newblock {\em Commun. Contemp. Math.}, to appear, arXiv:1609.00782


\bibitem{cavasturm:MCP}
F.~Cavalletti and K.-T. Sturm.
\newblock Local curvature-dimension condition implies measure-contraction  property.
\newblock {\em J. Funct. Anal.},  {262}:5110--5127, 2012.

\bibitem{cheeger:Lip}
J.~Cheeger.
\newblock Differentiability of Lipschitz functions on metric measure spaces, 
\newblock {\em Geom. Funct. Anal.},   {3}(9):428--517, 1999.


\bibitem{CorderoMassTransportAndGaussianInqs}
D.~Cordero-Erausquin.
\newblock Some applications of mass transport to {G}aussian-type inequalities.
\newblock {\em Arch. Ration. Mech. Anal.}, 161(3):257--269, 2002.

\bibitem{corderomccann:brescamp}
D.~Cordero-Erausquin, R.~J. McCann, and M.~Schmuckenschl\"ager.
\newblock A {R}iemannian interpolation inequality \`a la {B}orell, {B}rascamp and {L}ieb.
\newblock {\em Invent. Math.},  {146}:219--257, 2001.


\bibitem{dengsturm}
Q.~Deng and K.-T.~Sturm.
\newblock Localization and tensorization properties of the curvature-dimension condition for metric measure spaces, II.
\newblock {\em J. Funct. anal.},   {260}:3718--3725, 2011.

\bibitem{EKS-EntropicCD} M.~Erbar, K.~Kuwada, and K.-T.~Sturm.
\newblock On the equivalence of the entropic curvature-dimension condition and {B}ochner's inequality on metric measure spaces.
\newblock {\em Invent. Math.}, 201(3):993--1071, 2015.


\bibitem{EvansSurveyOT}
L.~C. Evans.
\newblock Partial differential equations and {M}onge-{K}antorovich mass
  transfer.
\newblock In {\em Current developments in mathematics, 1997 ({C}ambridge,
  {MA})}, pages 65--126. Int. Press, Boston, MA, 1999.

\bibitem{EvansGangbo}
L.~C. Evans and W.~Gangbo.
\newblock Differential equations methods for the {M}onge-{K}antorovich mass
  transfer problem.
\newblock {\em Mem. Amer. Math. Soc.}, 137(653):viii+66, 1999.

\bibitem{FeldmanMcCann-Manifold}
M.~Feldman and R.~J. McCann.
\newblock Monge's transport problem on a {R}iemannian manifold.
\newblock {\em Trans. Amer. Math. Soc.}, 354(4):1667--1697 (electronic), 2002.


\bibitem{Fre:measuretheory4}
D.~H. Fremlin.
\newblock {\em Measure Theory}, volume~4.
\newblock Torres Fremlin, 2002.


\bibitem{gigli:laplacian} N.~Gigli.
\newblock On the differential structure of metric measure spaces and  applications.
\newblock {\em Mem. Amer. Math. Soc.}  {236}(no. 1113), 2015.


\bibitem{GSR:maps} N.~Gigli, T.~Rajala and K.-T.~Sturm. 
\newblock Optimal Maps and Exponentiation on Finite-Dimensional Spaces with Ricci Curvature Bounded from Below.
\newblock {\em J. Geom. Anal.},  {26}:2914--2929, 2016.


\bibitem{GromovGeneralizationOfLevy}
M.~Gromov.
\newblock {P}aul {L}\'evy's isoperimetric inequality.
\newblock preprint, I.H.E.S., 1980.

\bibitem{GromovBook}
M.~Gromov.
\newblock {\em Metric Structures for Riemannian and Non-Riemannian spaces}.
\newblock Birkh\"auser, 2001.

\bibitem{Gromov-Milman}
M.~Gromov and V.~D. Milman.
\newblock Generalization of the spherical isoperimetric inequality to uniformly
  convex {B}anach spaces.
\newblock {\em Compositio Math.}, 62(3):263--282, 1987.

\bibitem{ConvexAnalysisBookI}
J.-B. Hiriart-Urruty and C.~Lemar{\'e}chal.
\newblock {\em Convex analysis and minimization algorithms. {I}}, volume 305 of
  {\em Grundlehren der Mathematischen Wissenschaften [Fundamental Principles of
  Mathematical Sciences]}.
\newblock Springer-Verlag, Berlin, 1993.


\bibitem{KLS}
R.~Kannan, L.~Lov{\'a}sz, and M.~Simonovits.
\newblock Isoperimetric problems for convex bodies and a localization lemma.
\newblock {\em Discrete Comput. Geom.}, 13(3-4):541--559, 1995.

\bibitem{Klartag} B.~Klartag.
\newblock Needle decompositions in Riemannian geometry.
\newblock {\em Mem. Amer. Math. Soc}, to appear, arXiv:1408.6322.

\bibitem{Ledoux-Book}
M.~Ledoux.
\newblock {\em The concentration of measure phenomenon}, volume~89 of {\em
  Mathematical Surveys and Monographs}.
\newblock American Mathematical Society, Providence, RI, 2001.



\bibitem{lottvillani:length}
J.~Lott and C.~Villani.
\newblock Hamilton--Jacobi semigroup on length spaces and applications.
\newblock {\em J. Math. Pures Appl.}  {88}:219--229, 2007.

\bibitem{LottVillani-WeakCurvature}
J.~Lott and C.~Villani.
\newblock Weak curvature conditions and functional inequalities.
\newblock {\em J. Funct. Anal.}, 245(1):311--333, 2007.

\bibitem{lottvillani:metric}
J.~Lott and C.~Villani.
\newblock Ricci curvature for metric-measure spaces via optimal transport.
\newblock {\em Ann. of Math.},  {169}(3):903--991, 2009.

 
 \bibitem{McCannConvexityPrincipleForGases}
R.~J. McCann.
\newblock A convexity principle for interacting gases.
\newblock {\em Adv. Math.}, 128(1):153--179, 1997.

\bibitem{McCannGuillenLectureNotes}
R.~J. McCann and N.~Guillen.
\newblock Five lectures on optimal transportation: geometry, regularity and
  applications.
\newblock In {\em Analysis and geometry of metric measure spaces}, volume~56 of
  {\em CRM Proc. Lecture Notes}, pages 145--180. Amer. Math. Soc., Providence,
  RI, 2013.


\bibitem{miranda:bvmetric} M.~Miranda Jr..
\newblock Functions of bounded variation on ``good'' metric spaces.
\newblock {\em J. Math. Pures Appl.},  {82}:975--1004, 2003.

\bibitem{EMilmanSharpIsopInqsForCDD}
E.~Milman.
\newblock Sharp isoperimetric inequalities and model spaces for the
  curvature-dimension-diameter condition.
\newblock {\em J. Eur. Math. Soc. (JEMS)}, 17(5):1041--1078, 2015.

\bibitem{EMilmanNegativeDimension}
E.~Milman.
\newblock Beyond traditional curvature-dimension {I}: new model spaces for
  isoperimetric and concentration inequalities in negative dimension.
\newblock {\em Trans. Amer. Math. Soc.}, to appear, arXiv:1409.4109.

\bibitem{Ohta1}  S.-I.~Ohta.
\newblock  On the measure contraction property of metric measure spaces,
\newblock  {\em Comment. Math. Helv.},  {82}:805--828, 2007.

\bibitem{Ohta-CDforFinsler} S.-I.~Ohta.
\newblock Finsler interpolation inequalities.
\newblock {\em Calc. Var. Partial Differential Equations}, 36(2):211--249, 2009.

\bibitem{OhtaNeedles} S.-I.~Ohta.
\newblock Needle decompositions and isoperimetric inequalities in {F}insler geometry.
\newblock arXiv:1506.05876, 2015.
  
\bibitem{Ohta-NegativeN} S.-I.~Ohta.
\newblock {$(K,N)$}-convexity and the curvature-dimension condition for negative {$N$}.
\newblock {\em J. Geom. Anal.}, 26(3):2067--2096, 2016.

\bibitem{Oliver-ExactPeano}
H.~W. Oliver.
\newblock The exact {P}eano derivative.
\newblock {\em Trans. Amer. Math. Soc.}, 76:444--456, 1954.

\bibitem{OttoVillaniHWI}
F.~Otto and C.~Villani.
\newblock Generalization of an inequality by {T}alagrand and links with the
  logarithmic {S}obolev inequality.
\newblock {\em J. Funct. Anal.}, 173(2):361--400, 2000.

\bibitem{PayneWeinberger}
L.~E. Payne and H.~F. Weinberger.
\newblock An optimal {P}oincar\'e inequality for convex domains.
\newblock {\em Arch. Rational Mech. Anal.}, 5:286--292, 1960.

\bibitem{Petrunin-CDforAlexandrov}
A.~Petrunin.
\newblock Alexandrov meets {L}ott-{V}illani-{S}turm.
\newblock {\em M\"unster J. Math.}, 4:53--64, 2011.


\bibitem{RachevRuschendorf-Book}
S.~T. Rachev and L.~R{\"u}schendorf.
\newblock {\em Mass transportation problems. {V}ol. {I}}.
\newblock Probability and its Applications (New York). Springer-Verlag, New York, 1998.

\bibitem{R2012} T.~Rajala.
\newblock {Interpolated measures with bounded densities in metric spaces satisfying the curvature-dimension conditions of Sturm},
\newblock{\em J. Funct. Anal.},  {263}:896--924, 2012.

\bibitem{R2016} T.~Rajala.
\newblock {Failure of the local-to-global property for $\CD(K,N)$ spaces},
\newblock{\em Ann. Sc. Norm. Super. Pisa Cl. Sci.},  {16}:45--68, 2016.


\bibitem{rajasturm:branch} T.~Rajala and K.-T. Sturm.
\newblock Non-branching geodesics and optimal maps in strong  $\mathsf{CD}(K,\infty)$-spaces.
\newblock {\em Calc. Var. Partial Differential Equations}, 50:831--846, 2014.


\bibitem{MVR} M.-K.~von Renesse.
\newblock {On local Poincar\'e via transportation},
\newblock{\em Math. Z.}, 259:21--31, 2008.


\bibitem{VonRenesseSturm} M.-K.~von Renesse and K.-T.~Sturm.
\newblock {Transport inequalities, gradient estimates, entropy and Ricci curvature},
\newblock{\em Comm. Pure Appl. Math.}, 58:923--940, 2005.



\bibitem{Schneider-Book}
R.~Schneider.
\newblock {\em Convex bodies: the {B}runn-{M}inkowski theory}, volume~44 of
  {\em Encyclopedia of Mathematics and its Applications}.
\newblock Cambridge University Press, Cambridge, 1993.

\bibitem{Srivastava} S.M.~Srivastava, 
\newblock A course on Borel sets,
\newblock{\em Graduate Texts in Mathematics}, Springer 1998.


\bibitem{sturm:I} K.-T.~Sturm.
\newblock On the geometry of metric measure spaces. {I}.
\newblock {\em Acta Math.},  {196}(1):65--131, 2006.

\bibitem{sturm:II} K.-T.~Sturm.
\newblock On the geometry of metric measure spaces. {II}.
\newblock {\em Acta Math.},  {196}(1):133--177, 2006.


\bibitem{UrbasLectureNotesOT}
J.~Urbas.
\newblock Mass transfer problems.
\newblock Lecture notes, University of Bonn, 1998.


\bibitem{Vil:topics}
C.~Villani.
\newblock {\em Topics in optimal transportation}, volume~58 of {\em Graduate
  Studies in Mathematics}.
\newblock American Mathematical Society, Providence, RI, 2003.

\bibitem{Vil}
C.~Villani.
\newblock {\em Optimal transport - old and new}, volume 338 of {\em Grundlehren
  der Mathematischen Wissenschaften [Fundamental Principles of Mathematical
  Sciences]}.
\newblock Springer-Verlag, Berlin, 2009.


\end{thebibliography}
\end{document}